\documentclass[11pt]{amsart}

\usepackage{amssymb}
\usepackage{amsmath}
\usepackage{amsfonts}
\usepackage{amsthm}
\usepackage{stmaryrd}
\usepackage[all]{xy}
\usepackage{mathrsfs}
\usepackage{graphicx}
\usepackage{hyperref}
\usepackage{color}
\usepackage{multirow}
\usepackage{pstricks}
\usepackage{auto-pst-pdf}

\usepackage[margin=1.0in]{geometry}

\numberwithin{equation}{subsection}

\theoremstyle{plain}
\newtheorem{theorem}[subsection]{Theorem}

\newtheorem{conj}[subsection]{Conjecture}
\newtheorem{cor}[subsection]{Corollary}
\newtheorem{lemma}[subsection]{Lemma}
\newtheorem{prop}[subsection]{Proposition}

\theoremstyle{definition}

\newtheorem{defn}[subsection]{Definition}
\newtheorem{example}[subsection]{Example}

\newtheorem{notation}[subsection]{Notation}

\newtheorem{hypo}[subsection]{Hypothesis}

\theoremstyle{remark}

\newtheorem{remark}[subsection]{Remark}

\newcommand{\ra}{\rightarrow}
\newcommand{\xra}{\xrightarrow}
\newcommand{\hra}{\hookrightarrow}

\def\AAA{\mathbb{A}}

\def\CC{\mathbb{C}}

\def\FF{\mathbb{F}}
\def\GG{\mathbb{G}}
\def\HH{\mathbb{H}}
\def\II{\mathbb{I}}

\def\NN{\mathbb{N}}

\def\PP{\mathbb{P}}
\def\QQ{\mathbb{Q}}
\def\RR{\mathbb{R}}
\def\SSS{\mathbb{S}}
\def\TT{\mathbb{T}}

\def\ZZ{\mathbb{Z}}

\def\bff{\mathbf{f}}

\def\bfi{\mathbf{i}}

\def\bfA{\mathbf{A}}
\def\bfB{\mathbf{B}}
\def\bfC{\mathbf{C}}

\def\bfX{\mathbf{X}}

\def\calA{\mathcal{A}}

\def\calD{\mathcal{D}}
\def\calE{\mathcal{E}}

\def\calG{\mathcal{G}}

\def\calI{\mathcal{I}}

\def\calL{\mathcal{L}}

\def\calO{\mathcal{O}}
\def\calP{\mathcal{P}}

\def\calS{\mathcal{S}}
\def\calT{\mathcal{T}}

\def\calX{\mathcal{X}}

\def\gothd{\mathfrak{d}}
\def\gothe{\mathfrak{e}}

\def\gothh{\mathfrak{h}}

\def\gothl{\mathfrak{l}}

\def\gothp{\mathfrak{p}}
\def\gothq{\mathfrak{q}}

\def\gotht{\mathfrak{t}}

\def\gothF{\mathfrak{F}}

\def\gothH{\mathfrak{H}}

\def\gothR{\mathfrak{R}}

\def\scrD{\mathscr{D}}

\def\scrL{\mathscr{L}}

\def\rmM{\mathrm{M}}

\def\ttS{\mathtt{S}}
\def\ttT{\mathtt{T}}

\def\ii{\mathbf{i}}

\DeclareMathOperator{\Art}{Art}
\DeclareMathOperator{\Aut}{Aut}

\DeclareMathOperator{\End}{End}
\DeclareMathOperator{\Gal}{Gal}
\DeclareMathOperator{\Ker}{Ker}
\DeclareMathOperator{\Coker}{Coker}

\DeclareMathOperator{\Hom}{Hom}

\DeclareMathOperator{\Lie}{Lie}
\DeclareMathOperator{\Pic}{Pic}

\DeclareMathOperator{\Rec}{Rec}
\DeclareMathOperator{\Res}{Res}
\DeclareMathOperator{\Spec}{Spec}

\newcommand{\Q}{\mathbb{Q}}
\newcommand{\Z}{\mathbb{Z}}
\newcommand{\R}{\mathbb{R}}
\newcommand{\C}{\mathbb{C}}
\newcommand{\G}{\mathbb{G}}
\newcommand{\cO}{\mathcal{O}}
\newcommand{\kb}{\underline{k}}

\newcommand{\Fpb}{\overline{\FF}_p}

\newcommand{\tSigma}{\widetilde{\Sigma}}

\newcommand{\cD}{\mathcal{D}} 
\newcommand{\cl}{\mathrm{cl}}
\newcommand{\tcD}{\tilde{\cD}}

\newcommand{\bSh}{\mathbf{Sh}}

\newcommand{\dR}{\mathrm{dR}}

\newcommand{\Sym}{\mathrm{Sym}} 
\newcommand{\es}{\mathrm{es}}
\newcommand{\pr}{\mathrm{pr}}

\newcommand{\ord}{\mathrm{ord}}

\newcommand{\lra}{\mathrm{\longrightarrow}}
\newcommand{\im}{\mathrm{Im}}

\newcommand{\Ver}{\mathrm{Ver}}

\newcommand{\ab}{\mathrm{ab}}
\newcommand{\ad}{\mathrm{ad}}
\newcommand{\bbalpha}{\boldsymbol{\alpha}}

\newcommand{\bfSh}{\mathbf{Sh}}

\newcommand{\der}{\mathrm{der}}

\newcommand{\et}{\mathrm{et}}

\newcommand{\Fp}{\FF_p}
\newcommand{\Fr}{\mathrm{Fr}}
\newcommand{\Frob}{\mathrm{Frob}}
\newcommand{\geom}{\mathrm{geom}}
\newcommand{\gothRec}{\gothR\mathrm{ec}}
\newcommand{\GL}{\mathrm{GL}}

\newcommand{\Iw}{\mathrm{Iw}}

\newcommand{\Nm}{\mathrm{Nm}}

\newcommand{\cris}{\mathrm{cris}}

\newcommand{\PGL}{\mathrm{PGL}}
\newcommand{\Qp}{\QQ_p}

\newcommand{\rig}{\mathrm{rig}}

\newcommand{\sym}{\mathrm{sym}}
\newcommand{\Sh}{\mathrm{Sh}}

\newcommand{\SL}{\mathrm{SL}}

\newcommand{\tr}{\mathrm{tr}}
\newcommand{\Tr}{\mathrm{Tr}}
\newcommand{\ttAV}{\mathtt{AV}}
\newcommand{\ur}{\mathrm{ur}}
\newcommand{\Zp}{\ZZ_p}
\newcommand{\cHom}{\mathcal{H}om}

\begin{document}

\title{On Goren-Oort Stratification for quaternionic Shimura varieties}

\author{Yichao Tian}
\email{yichaot@math.ac.cn}
\address{Morningside Center of Mathematics, Chinese Academy of Sciences,  55 Zhong Guan Cun East Road, Beijing, 100190, China.}

\author{Liang Xiao}

\email{liang.xiao@uconn.edu}

\address{196 Auditorium Road, UConn Department of Mathematics, Unit 3009, Storrs, CT 06269-3009, U.S.A.}

\thanks{Y.T. is partially supported by the National Natural Science Foundation of China (No. 11321101), and L.X.  by the Simons Collaboration Grant \#278433.}

\begin{abstract}
Let $F$ be a totally real field in which a prime $p$ is unramified.
We define the Goren-Oort stratification of the characteristic $p$ fiber of a quaternionic Shimura variety of maximal level at $p$.
We show that each stratum is a $(\PP^1)^r$-bundle over other quaternionic Shimura varieties (for an appropriate integer $r$). 
As an application, we give a necessary condition for the ampleness of a modular line bundle on a quaternionic Shimura variety in characteristic $p$. 
\end{abstract}

\maketitle
\tableofcontents

\section{Introduction}
This paper is intended as the first in a series \cite{tian-xiao2, tian-xiao3}, in which we study the Goren-Oort stratification for quaternionic Shimura varieties.  The purpose of this paper is to give a global description of the strata, saying  that they are in fact $(\PP^1)^r$-bundles over (the special fiber of) other quaternionic Shimura varieties for a certain integer $r$.
We fix $p > 2$ a prime number.

\subsection{Motivation: the case of modular curves}
\label{S:modular curve}
Let $N \geq 5$ be an integer prime to $p$.  Let $\calX$ denote the  modular curve with level $\Gamma_1(N)$; it admits a smooth integral model $\bfX$ over $\ZZ[1/N]$.
We are interested in the special fiber $X: = \bfX \otimes_{\ZZ[1/N]} \FF_p$. The curve $X$ has a natural stratification by the supersingular locus $X^\mathrm{ss}$ and the ordinary locus $X^\ord$.  In concrete terms, $X^\mathrm{ss}$ is defined as the zero locus of the Hasse invariant $h \in  H^0(X, \omega^{\otimes(p-1)})$, where $\omega^{\otimes(p-1)}$ is the sheaf for weight $p-1$ modular forms.
The following deep result of Deuring and Serre (see e.g. \cite{serre}) gives an intrinsic description of $X^\mathrm{ss}$.

\begin{theorem}[(Deuring, Serre)]
\label{T:Deuring-Serre}
Let  $\AAA^\infty$  denote the ring of finite adeles over $\QQ$, and $\AAA^{\infty, p}$ its prime-to-$p$ part.
We have a bijection of sets:
\[
\big\{\overline \FF_p\textrm{-points of } X^\mathrm{ss} \big\} \longleftrightarrow B^\times_{p,\infty} \backslash B_{p,\infty}^\times(\AAA^{\infty}) / K_1(N) B_{p,\infty}^\times(\Zp)
\]
equivariant under the prime-to-$p$ Hecke correspondences,
where $B_{p, \infty}$ is the quaternion algebra over $\QQ$ which ramifies at exactly two places: $p$ and $\infty$, $B_{p, \infty}^\times(\Zp)$ is the maximal open compact subgroup of $B_{p, \infty}^\times(\Qp)$, and $K_1(N)$ is an open compact subgroup of $\GL_2(\AAA^{\infty,p}) = B^\times_{p, \infty}(\AAA^{\infty,p})$ given by
\[
K_1(N) = \Big\{\big(\begin{smallmatrix}
a & b \\ c & d
\end{smallmatrix}\big) \in \GL_2(\widehat \ZZ^{(p)})
\; \Big|\;
c \equiv 0, d\equiv 1 \pmod N
 \Big\}, \textrm{ where } \widehat \ZZ^{(p)} = \prod_{l \neq p} \ZZ_l.
\]
\end{theorem}

The original proof of this theorem uses the fact that all supersingular elliptic curves over $\overline\FF_p$ are isogenous and the quasi-endomorphism ring is exactly $B_{p, \infty}$.
We however prefer to understand the result as: certain special cycles of the special fiber of the Shimura variety for $\GL_2$ is just the special fiber of the Shimura variety for $B_{p, \infty}^\times$.

The aim of this paper is to generalize this theorem  to the case of quaternionic Shimura varieties.  For the purpose of simple presentation in this introduction, we focus on the case of Hilbert modular varieties.  We will indicate how to modify the result to adapt to general cases.

\subsection{Goren-Oort Stratification}
\label{S:stratification of HMV}
Let $F$ be a totally real field, and let $\calO_F$ denote its ring of integers.  We assume that $p$ is \emph{unramified} in $F$.
Goren and Oort \cite{goren-oort} defined a stratification of the special fiber of the Hilbert modular variety $X_{\GL_2}$.
More precisely, let $\AAA_F^\infty$ denote the ring of finite adeles of $F$ and $\AAA_F^{\infty, p}$ its prime-to-$p$ part.
We fix an open compact subgroup $K^p \subset \GL_2(\AAA_F^{\infty, p})$.
Let $\calX_{\GL_2}$ denote the {\it Hilbert modular variety} (over $\QQ$) with tame level $K^p$. Its complex points are given by
\[
\calX_{\GL_2}(\CC) = \GL_2(F)\; \backslash\; \big({\gothh^\pm}^{[F:\QQ]} \times \GL_2(\AAA_F^{\infty})\big) \;/\; \big(K^p \times \GL_2(\calO_{F, p})\big),
\]
where $\gothh^\pm :=\CC \backslash \RR$ and $\calO_{F, p} := \calO_F \otimes_\ZZ \Zp$.
The Hilbert modular variety $\calX_{\GL_2}$ admits an integral model $\bfX_{\GL_2}$ over $\ZZ_{(p)}$ and let $X_{\GL_2}$ denote its special fiber over $\overline \FF_p$.

Since $p$ is unramified in $F$, we may and will identify the $p$-adic embeddings of $F$ with the homomorphisms of $\calO_F$ to $\overline \FF_p$, i.e. $\Hom(F, \overline \QQ_p) \cong \Hom(\calO_F, \overline \FF_p)$.  Let $\Sigma_\infty$ denote this set. (We shall later identify the  $p$-adic embeddings with the real embeddings, hence the subscript $\infty$.)
Under the latter description,
the absolute Frobenius $\sigma$ acts on $\Sigma_\infty$ by taking an element $\tau$ to the composite $\sigma \tau:\calO_F \xrightarrow{\tau} \overline \FF_p \xrightarrow{x\mapsto x^p} \overline \FF_p$.
This action decomposes  $\Sigma_\infty$ into a disjoint union of cycles, parametrized by all $p$-adic places of $F$.

Let $\calA$ denote the universal abelian variety over $X_{\GL_2}$.  The sheaf of invariant differential 1-forms $\omega_{\calA/X_{\GL_2}}$ is then locally free of rank one as a module over
\[
\calO_F \otimes_\ZZ \calO_{X_{\GL_2}} \cong \bigoplus_{\tau \in \Sigma_\infty} \calO_{X_{\GL_2, \tau}},
\]
where  $\cO_{X_{\GL_2, \tau}}$ is the direct summand on which $\calO_F$ acts through $\tau : \calO_F \to \overline \FF_p$.
We then write accordingly $\omega_{\calA/X_{\GL_2}} = \bigoplus_{\tau \in \Sigma_\infty} \omega_\tau$;  each $\omega_\tau$ is locally free of rank one over $\cO_{X_{\GL_2}}$.

\begin{defn}
The Verschiebung map induces an $\calO_F$-morphism $\omega_{A/X_{\GL_2}} \to \omega_{A^{(p)}/ X_{\GL_2}}$, which further induces a homomorphism $h_\tau: \omega_\tau \to \omega^{\otimes p}_{\sigma^{-1}\tau}$ for each $\tau \in \Sigma_\infty$.
This map then defines a global section $h_\tau \in H^0(X_{\GL_2}, \omega_\tau^{\otimes -1} \otimes \omega^{\otimes p}_{\sigma^{-1}\tau})$, which we call the \emph{partial Hasse invariant at} $\tau$.
We use $X_\tau$ to denote the zero locus of $h_\tau$.
For a subset $\ttT \subseteq \Sigma_\infty$, we put $X_\ttT = \bigcap_{\tau \in \ttT} X_\tau$.  These $X_\ttT$'s give the \emph{Goren-Oort stratification} of $X_{\GL_2}$.

An alternative definition of $X_\ttT$ is given as follows: $z \in X_\ttT(\overline \FF_p)$ if and only if $\Hom(\alpha_p, A_z[p])$ under the action of $\calO_F$ has eigenvalues given by those embeddings $\tau \in \ttT$.  We refer to \cite{goren-oort} for the proof of equivalence and a more detailed discussion.
\end{defn}

It is proved in \cite{goren-oort} that each $X_\tau$ is a smooth and proper divisor and these divisors intersect transversally.  Hence $X_\ttT$ is smooth of codimension $\#\ttT$ for any subset $\ttT \subseteq \Sigma_\infty$; it is  proper if $\ttT \neq \emptyset$.

\subsection{Description of the Goren-Oort strata}
\label{S:intro GO-strata}
The goal of this paper is to give a \emph{global} description of the Goren-Oort strata.

Prior to our paper, most works focus on the $p$-divisible groups of the abelian varieties, which often provides a good access to the local structure of Goren-Oort strata, e.g. dimensions and smoothness.  Unfortunately, there have been little understanding of the global geometry of $X_\ttT$, mostly in low dimension.  We refer to the survey article \cite{andreatta-goren} for a historical account.
Recently, 
D. Helm made a break-through progress  in
\cite{helm, helm-PEL} by taking advantage of the moduli problem; he was able to describe the global geometry of certain analogous strata of the special fibers of  Shimura varieties of type $U(2)$.

Our proof of the main theorem of this paper is, roughly speaking, to complete Helm's argument to cover all cases for the $U(2)$-Shimura varieties and then transfer the results from the unitary side to the quaternionic side.

Rather than stating our main theorem in an abstract way, we prefer to give more examples to indicate the general pattern.

When $F=\QQ$, this is discussed in Section~\ref{S:modular curve}.
For $F \neq \QQ$, we fix an isomorphism $\CC \cong \overline \QQ_p$ and hence identify $\Sigma_\infty = \Hom(F, \overline \QQ_p)$ with the set of real embeddings of $F$.

\subsubsection{$F$ real quadratic and $p$ inert in $F$}
Let $\infty_1$ and $\infty_2$ denote the two real embeddings of $F$, and $\tau_1$ and $\tau_2$ the corresponding $p$-adic embeddings (via the fixed isomorphism $ \CC \cong \overline \QQ_p$).
Then our main Theorem~\ref{T:main-thm} says that each $X_{\tau_i}$ is a $\PP^1$-bundle over the special fiber of the discrete Shimura variety $\overline\Sh_{B_{\infty_1, \infty_2}^\times}$, where $B_{\infty_1, \infty_2}$ stands for the quaternion algebra over $F$ which ramifies at both archimedean places. The intersection $X_{\tau_1} \cap X_{\tau_2}$ is isomorphic to the special fiber of the discrete Shimura variety $\overline\Sh_{ B_{\infty_1, \infty_2}^\times}(\Iw_p)$, where $(\Iw_p)$ means to take Iwahori level structure at $p$ instead.
The two natural embeddings of $X_{\tau_1} \cap X_{\tau_2}$ into $X_{\tau_1}$ and $X_{\tau_2}$ induces two morphisms $\overline\Sh_{ B_{\infty_1, \infty_2}^\times}(\Iw_p) \to \overline\Sh_{ B_{\infty_1, \infty_2}^\times}$; this gives (certain variant of) the Hecke correspondence at $p$ (see Theorem~\ref{T:link and Hecke operator}).

We remark here that it was first proved in \cite{bachmat-goren} that the one-dimensional strata are disjoint unions of $\PP^1$ and the number of such $\PP^1$'s is also computed in \cite{bachmat-goren}.  This computation relies on the intersection theory and does not provide a natural parametrization as we gave above.  Our proof will be  different from theirs.  One can easily recover their counting from our natural parametrization.

\subsubsection{Quaternionic Shimura varieties}
Before proceeding, we clarify our convention on quaternionic Shimura varieties.

For $\ttS$ an even subset of archimedean and $p$-adic places of $F$, we use $B_\ttS$ to denote the quaternion algebra over $F$ which ramifies exactly at those places in $\ttS$.
We fix an identification $B_\ttS^\times(\AAA^{\infty, p}) \cong \GL_2(\AAA^{\infty, p})$.
We fix a maximal open compact subgroup $B_\ttS^\times(\calO_{F,p})$ of $B_\ttS^\times(F \otimes_\QQ \Qp)$.
We use $\ttS_\infty$ to denote the subset of archimedean places of $\ttS$.
The Shimura variety $\calS h_{B_\ttS^\times}$ for the algebraic group $\Res_{F/\QQ} B_\ttS^\times$ has complex points
\[
\calS h_{B_\ttS^\times}(\CC) = B_\ttS^\times(F) \; \backslash \; \big({\gothh^\pm}^{[F:\QQ]-\#\ttS_\infty} \times B_\ttS^\times(\AAA_F^{\infty}) \big)\; /  \; \big( K^p \times B_\ttS^\times(\calO_{F,p})\big).
\]
Here and later, the tame level $K^p$ is uniformly matched up for all quaternionic Shimura varieties.

Unfortunately, $\calS h_{B_\ttS^\times}$ itself does not possess a moduli interpretation.
We follow the construction of Carayol \cite{carayol} to relate it with a unitary Shimura variety $Y$ and ``carry over" the integral model of $Y$.
The assumption $p >2$ comes from the verification of the extension property for the integral canonical model following Moonen \cite[Corollary~3.8]{moonen96}.

In any case, we use $\overline\Sh_{B_\ttS^\times}$ to denote the special fiber of the Shimura variety over $\overline \FF_p$.
When we take the Iwahori level structure at $p$ instead, we write $\overline\Sh_{B_\ttS^\times}(\Iw_p)$.

\subsubsection{$F$ real quartic and $p$ inert in $F$}
Let $\infty_1, \dots, \infty_4$ denote the four real embeddings of $F$, labeled so that the corresponding $p$-adic embeddings $\tau_1, \dots, \tau_4$ satisfy $\sigma \tau_i = \tau_{i+1}$ with the convention that $\tau_i = \tau_{i\,(\mathrm{mod}\; 4)}$.  We list the description of the strata as follows.

\begin{center}
\begin{tabular}{|c|c|}
\hline
Strata & Description\\
\hline
$X_{\tau_i}$ for each $i$ & $\PP^1$-bundle over $\overline\Sh_{B^\times_{\{\infty_{i-1}, \infty_i\}}}$\\

\hline
$X_{\{\tau_{i-1},\tau_i\}}$ for each $i$ &  $\overline\Sh_{B^\times_{\{\infty_{i-1}, \infty_i\}}}$\\

\hline
$X_{\{\tau_1, \tau_3\}}$ and $X_{\{\tau_2, \tau_4\}}$ & $(\PP^1)^2$-bundle over $\overline\Sh_{B^\times_{\{\infty_1, \infty_2, \infty_3, \infty_4\}}}$\\
\hline
$X_{\ttT}$ with $\#\ttT = 3$ & $\PP^1$-bundle over $\overline\Sh_{B^\times_{\{\infty_1, \infty_2, \infty_3, \infty_4\}}}$\\
\hline
$X_{\{\tau_1, \tau_2, \tau_3, \tau_4\}}$ & $\overline\Sh_{B^\times_{\{\infty_1, \infty_2, \infty_3, \infty_4\}}}(\Iw_p)$\\
\hline
\end{tabular}
\end{center}

In particular, we point out that for a codimension $2$ stratum, its shape depends on whether the two chosen $\tau_i$'s are adjacent in the cycle  $\tau_1 \to \dots \to \tau_4 \to \tau_1$.

\medskip

\subsubsection{$F$ general totally real of degree $g$ over $\QQ$ and $p$ inert in $F$}
\label{SS:p inert case}
As before, we label the real embeddings of $F$ by $\infty_1, \dots, \infty_g$ such that the corresponding $p$-adic embeddings $\tau_1, \dots, \tau_g$ satisfy $\sigma \tau_i = \tau_{i+1}$ with the convention that $\tau_i = \tau_{i \,(\mathrm{mod}\; g)}$.
The general statement for Goren-Oort strata takes the following form: for a subset $\ttT \subseteq \Sigma_\infty$, the stratum $X_\ttT$ is isomorphic to a $(\PP^1)^r$-bundle over the special fiber of another quaternion Shimura variety $\overline \Sh_{B_{\ttS(\ttT)}^\times}$.  We now explain, given $\ttT$, what $\ttS(\ttT)$ and $r$ are.
\begin{itemize}
\item
When $\ttT \subsetneq \Sigma_\infty$, we construct $\ttS(\ttT)$ as follows: if $\tau \notin \ttT$ and $\sigma^{-1}\tau, \dots, \sigma^{-m}\tau \in \ttT$, we put $\sigma^{-1}\tau, \dots, \sigma^{-2 \lceil m/2 \rceil}\tau$ into $\ttS(\ttT)$.  In other words, we always have $\ttT \subseteq \ttS(\ttT)$, and $\ttS(\ttT)$ contains the additional element $\sigma^{-m-1}\tau$ if and only if the corresponding $m$ is odd.
The number $r$ is the cardinality of $\ttS(\ttT) - \ttT$.
\item
When $\ttT = \Sigma_\infty$, $r$ is always $0$; for $\ttS(\ttT)$, we need to distinguish the parity:
\begin{itemize}
\item if $\#\Sigma_\infty$ is odd, we put $\ttS(\ttT) = \Sigma \cup \{p\}$;
\item if $\#\Sigma_\infty$ is even, we put $\ttS(\ttT) = \Sigma$ and we put an Iwahori level structure at $p$.
\end{itemize}
\end{itemize}

\subsubsection{$F$ general totally real and $p$ is unramified  in $F$}
\label{S:F general p unramified}
The general principle is: \emph{different places above $p$ work ``independently'' in the recipe of describing the strata (e.g. which places of the quaternion algebra are ramified); so we just take the ``product" of all recipes for different $p$-adic places.}

More concretely, let $p\calO_F = \gothp_1 \cdots \gothp_d$ be the prime ideal factorization.
We use $\Sigma_\infty$ to denote the set of all archimedean embeddings of $F$, which is identified with the set of  $p$-adic embeddings.  We use $\Sigma_{\infty/\gothp_i}$ to denote those archimedean embeddings or equivalently $p$-adic embeddings that give rise to the $p$-adic place $\gothp_i$.
Given any subset $\ttT \in \Sigma_\infty$, we put $\ttT_{\gothp_i} = \ttT \cap \Sigma_{\infty/\gothp_i}$.
Applying the recipe in \ref{SS:p inert case} to each $\ttT_{\gothp_i}$ viewed as a subset of $\Sigma_{\infty/\gothp_i}$, we get a set of places $\ttS(\ttT_{\gothp_i})$ and a nonnegative number $r_{\gothp_i}$.
We put $\ttS(\ttT) = \cup_{i=1}^d \ttS(\ttT_{\gothp_i})$ and $r = \sum_{i=1}^d r_{\gothp_i} =  \sum_{i=1}^d \#(\ttS(\ttT_{\gothp_i}) - \ttT_{\gothp_i})$.  Then $X_\ttT$ is a 
$(\PP^1)^r$-bundle over $\overline\Sh_{B_{\ttS(\ttT)}^\times}$ (with possibly some Iwahori level structure at appropriate places above $p$).

We also prove analogous results on the global description of the Goren-Oort strata on general quaternionic Shimura varieties (Theorem~\ref{T:main-thm}).  We refer to the content of the paper for the statement.
The modification we need to do in the general quaternionic case is that one just ``ignores" all ramified archimedean places and apply the above recipe formally to the set $\Sigma_\infty$ after ``depriving all ramified archimedean places".


\subsection{Method of the proof}
We briefly explain the idea behind the proof.
The first step is to translate the question into an analogous question on (the special fiber of) unitary Shimura varieties.
We use $X'$ to denote the special fiber of the unitary Shimura variety we start with, over which we have the universal abelian variety $A'$.  Similar to the Hilbert case,   we have naturally defined analogous Goren-Oort stratification given by divisors $X'_\tau$.  We consider $X'_\ttT = \cap_{\tau \in \ttT}X'_\tau$.

The idea is to prove the following sequence of isomorphisms $X'_\ttT \xleftarrow\cong Y'_\ttT \xrightarrow{\cong} Z'_\ttT$, where
$Z'_\ttT$ is the $(\PP^1)$-power bundle over the special fiber of another unitary Shimura variety, which comes with a universal abelian variety $B'$; $Y'_\ttT$ is the moduli space classifying both $A'$ and $B'$ together with a quasi-isogeny $A' \to B'$ of certain fixed type (with very small $p$-power degree); and the two morphisms are just simply forgetful morphisms.  We defer the characterization of the quasi-isogeny to the content of the paper.
To prove the two isomorphisms above, we simply check that the natural forgetful morphisms are bijective on the closed points and induce isomorphisms on the tangent spaces.

\begin{remark}
We point out that we have been deliberately working with the special fiber over the algebraic closure $\overline \FF_p$.  This is because the description of the stratification is \emph{not} compatible with the action of the Frobenius.
In fact, a more rigorous way to formulate theorem is to compare the special fiber of the Shimura variety associated to $\GL_2(F) \times_{F^\times} E^\times$ and that to $B_{\ttS(\ttT)}^\times \times_{F^\times} E^\times$.  The homomorphism from the Deligne torus into the two $E^\times$'s are in fact different, causing the incompatibility of the Frobenius action.  (See Corollary~\ref{C:main-thm-product} for the corresponding statement.)  The result about quaternionic Shimura variety is obtained by comparing geometric connected components of the corresponding Shimura varieties, in which we lose the Frobenius action.  See Remark~\ref{R:quaternionic Shimura reciprocity not compatible} for more discussion.
\end{remark}

\subsection{Ampleness of automorphic line bundles}

An immediate application of the study of the global geometry of the Goren-Oort stratification is to give a necessary condition (hopefully also sufficient) for an automorphic line bundle to be ample.
As before, we take $F$ to be a totally real field of degree $g$ in which $p$ is inert for simplicity.
Let $X^*_{\GL_2}$ denote the special fiber of the minimal compactification of the  Hilbert modular variety.
We label all $p$-adic embeddings as $\tau_1, \dots, \tau_g$ with subindices considered modulo $g$, such that $\sigma \tau_i = \tau_{i+1}$.
Then $\omega_{\tau_1}, \dots, \omega_{\tau_g}$ form a basis of the group of automorphic line bundles.  The class $[\omega_{\tau_i}]$ in $\Pic(X)_\QQ : = \Pic(X) \otimes_\ZZ \Q$ of each $\omega_{\tau_i}$ extends to a class in $\Pic(X^*_{\GL_2})_\QQ$, still denoted by $[\omega_{\tau_i}]$.
For a $g$-tuple $\underline k = (k_1, \dots, k_g) \in \ZZ^g$, we put $[\omega^{\underline k}]  = \sum_{i=1}^g k_i[\omega_{\tau_i}]$.  Probably slightly contrary to the common intuition from the case of modular forms, we prove the following.

\begin{theorem}
\label{T:introduction ample}
If the rational class of line bundle $[\omega^{\underline k}]$ is ample, then
\begin{equation}
\label{E:ampleness condition GL2}
p k_i > k_{i-1} \quad \textrm{for all }i; \quad \textrm{(and all }k_i >0).
\end{equation}
\end{theorem}
Here we put the second condition in parentheses because it automatically follows from the first condition.
This theorem is proved in Theorem~\ref{T:ampleness}.  
When $F$ is a real quadratic field, Theorem~\ref{T:introduction ample} is proved in \cite[Theorem~8.1.1]{andreatta-goren}. 

To see that the condition \eqref{E:ampleness condition GL2} is necessary, we simply restrict to each of the GO-strata $X_{\tau_i}$, which is a $\PP^1$-bundle as we discussed before.  Along each of the $\PP^1$-fiber, the line bundle $\omega^{\underline k}$ restricts to $\calO(pk_i - k_{i-1})$.  The condition~\eqref{E:ampleness condition GL2} is clear.

We do expect the condition in Theorem~\ref{T:introduction ample}
to be necessary (which was proved for Hilbert modular surface in \cite{andreatta-goren}), but we are not able to prove it due to a combinatorics complication.

\subsection{Forthcoming works in this series}
We briefly advertise the other papers of this series to indicate the potential applications of the technical result in this paper.
In the subsequent paper \cite{tian-xiao2}, we discuss an application to the classicality of overconvergent Hilbert modular forms, following the original proof of R. Coleman.
In the third paper \cite{tian-xiao3}, we
show that certain generalizations of the Goren-Oort strata realize Tate classes of the special fiber of  certain Hilbert modular varieties, and hence verify the Tate Conjecture under some genericity hypothesis.

\subsection{Structure of the paper}
In Section~\ref{Section:Sh Var}, we review some basic facts about integral models of Shimura varieties, which will be used to relate the quaternionic Shimura varieties with the unitary Shimura varieties.
One novelty is that we include a discussion about the ``canonical model" of certain discrete Shimura varieties, this can be treated uniformly together with usual Shimura varieties.
In Section~\ref{Section:Integral-model},
we construct the integral canonical model for quaternionic Shimura varieties, following the Carayol \cite{carayol}.  However, we tailor many of the choices (e.g. the auxiliary CM field, signatures) for our later application.
In Section~\ref{Section:defn of GOstrata}, we define the Goren-Oort stratification for the unitary Shimura varieties and transfer them to the quaternionic Shimura varieties; this is a straightforward generalization of the work of Goren and Oort \cite{goren-oort}.
In Section~\ref{Section:GO-geometry}, we give the global description of Goren-Oort stratification.  The method is very close to that was used in \cite{helm}.
In Section~\ref{Section:GO divisors}, we give more detailed description for Goren-Oort divisors, including a necessary condition for an automorphic line bundle to be ample, and a structure theorem relating the Goren-Oort stratification along a $\PP^1$-bundle morphism provided by Theorem~\ref{T:main-thm-unitary}.
In Section~\ref{Section:links}, we further study some structure of the Goren-Oort strata which will play an important role in the forthcoming paper \cite{tian-xiao3}.

\subsection*{Acknowledgments}

We thank Ahmed Abbes, Matthew Emerton, and David Helm for useful discussions.
We started working on this project when we were attending a workshop held at the Institute of Advance Study at Hongkong University of Science and Technology in December 2011.
The  hospitality of the institution and the well-organization provided us a great environment for brainstorming ideas.
We especially thank the organizers Jianshu Li and Shou-wu Zhang, as well as the staff at IAS of HKUST.
We also thank Fields Institute and the Morningside Center; the authors discussed the project while both attending conferences at these two institutes.

\subsection{Notation}\label{S:Notation-for-the-paper}

\subsubsection{}

For a scheme $X$ over a ring $R$ and a ring homomorphism $R \to R'$, we use $X_{R'}$ to denote the base change $X \times_{\Spec R} \Spec R'$.

For a field $F$, we use $\Gal_F$ to denote its Galois group.


For a number field $F$, we use $\AAA_F$ to denote its ring of adeles, and $\AAA_F^\infty$ (resp. $\AAA_F^{\infty, p}$) to denote its finite adeles (resp. prime-to-$p$ finite adeles).
When $F = \QQ$, we suppress the subscript $F$ from the notation.
We use superscript $\cl$ to mean closure in certain topological groups. For example, $F^{\times, \cl}$ means the closure of $F^\times$ inside $\AAA_F^{\infty, \times}$ or $ \AAA_F^{\infty, p}$ (depending the situation).
We put $\widehat \ZZ = \prod_l \ZZ_l$ and $\widehat \ZZ^{(p)} = \prod_{l \neq p} \ZZ_l$.

For each finite place $\gothp$ of $F$, let $F_\gothp$ denote the completion of $F$ at $\gothp$ and $\calO_\gothp$  its valuation ring, which has uniformizer $\varpi_\gothp$ and residue field $k_\gothp$. (When $F_\gothp$ is unramified over $\QQ_p$, we take $\varpi_\gothp$ to be $p$.)

We normalize the Artin map $\Art_F: F^\times \backslash \AAA^\times_F \to \Gal_F^\ab$ so that for each finite prime $\gothp$, the local uniformizer at $\gothp$ is mapped to a geometric Frobenius at $\gothp$.


\subsubsection{} For $A$  an abelian scheme over a scheme $S$, we denote by $A^{\vee}$ the dual abelian scheme, by $\Lie(A)$ the Lie algebra of $A$, and by $\omega_{A/S}$ the module of invariant $1$-differential forms of $A$ relative to $S$.
We sometimes omit $S$ from the notation when the base is clear.

For a finite $p$-group scheme or a $p$-divisible group $G$ over a perfect field $k$ of characteristic $p$, we use $\calD(G)$ to denote its \emph{covariant} Dieudonn\'e module.
For an abelian variety $A$ over $k$, we write $\calD_A$ for $\calD(A[p])$ and write $\tilde \calD_A$ for $\calD(A[p^\infty])$.

\subsubsection{}
\label{SS:notation-F}
Throughout this paper, we fix a totally real field $F$ of degree $g>1$ over $\QQ$.
Let $\Sigma$ denote the set of places of $F$, and $\Sigma_\infty$  the subset of archimedean  places, or equivalently, all real embeddings of $F$.

 We fix a prime number $p$ which is unramified in $F$. Let $\Sigma_p$ denote the set of places of $F$ above $p$.
We fix
an isomorphism $\iota_p: \CC \xrightarrow{\cong} \overline \QQ_p$ so that we identify $\Sigma_\infty$ as the set of $p$-adic embeddings of $F$.
For each $\gothp \in \Sigma_p$, we use $\Sigma_{\infty/\gothp}$ to denote the subset of $p$-adic embeddings  $\tau\in \Sigma_{\infty}$ which induce the $p$-adic place $\gothp$.
Since $p$ is unramified, each $\tau$ induces an embedding $\calO_F \hookrightarrow W(\overline \FF_p)$.  
Post-composition with the Frobenius $\sigma$ on the latter induces an action of $\sigma$ on the set of $p$-adic embeddings and makes each $\Sigma_{\infty/\gothp}$ into one cycle. We use $\sigma \tau$ to denote this action, i.e. $\sigma  \tau = \sigma\circ  \tau$.

\subsubsection{}
\label{SS:notation for E}
We will consider a CM extension $E$ over $F$, in which all places above $p$ are unramified.  Let $\Sigma_{E, \infty}$ denote the set of complex embeddings of $E$.
For $\tau \in \Sigma_\infty$, we often use $\tilde \tau$ to denote a/some complex embedding of $E$ extending $\tau$, and we write $\tilde \tau^c$ for its complex conjugate.
Using the isomorphism $\iota_p$ above, we view $\tilde \tau$ and $\tilde \tau^c$ as $p$-adic embeddings of $E$.

Under the natural two-to-one map $\Sigma_{E, \infty} \to \Sigma_\infty$, we use $\Sigma_{E, \infty/\gothp}$ to denote the preimage of $\Sigma_{\infty/\gothp}$.
In case when $\gothp$ splits as $\gothq \gothq^c$ in $E/F$, we use $\Sigma_{E, \infty/\gothq}$ to denote the set of complex embeddings $\tilde \tau$ such that $\iota_p \circ \tilde \tau$ induces the $p$-adic place $\gothq$.

\subsubsection{}\label{SS:notation-S}
For $\ttS$ an even subset of places of $F$, we denote by $B_\ttS$  the quaternion algebra over $F$ ramified at $\ttS$.
Let $\Nm_{B_\ttS/F}: B_\ttS \to F$ denote the reduced norm and $\Tr_{B_\ttS/F}: B_\ttS \to F$ the reduced trace.

We use the following lists of algebraic groups.  Let $G_\ttS$ denote the algebraic group $\Res_{F/\Q}B_\ttS^\times$.
Let $E$ be the CM extension of $F$ above and we put $T_{E,\tilde \ttS} =\Res_{E/\Q}\GG_m$; see Subsection~\ref{S:CM extension} for the meaning of subscript $\tilde \ttS$.
We put $\widetilde G_{\tilde \ttS} = G_\ttS \times T_{E, \tilde \ttS}$ and $G''_{\tilde \ttS} = G_\ttS \times_Z T_{E,\tilde \ttS}$, which is the quotient of $\widetilde G_{\tilde \ttS}$ by the subgroup $Z = \Res_{F/\Q}\GG_m$ embedded as $z \mapsto (z, z^{-1})$. et $G'_{\tilde \ttS}$ denote the subgroup of $G''_{\tilde \ttS}$ consisting of elements $(g, e)$ such that $\Nm_{B_\ttS/F}(g)\cdot \Nm_{E/F}(e) \in \GG_m$.

We put $\ttS_{\infty}=\Sigma_{\infty}\cap \ttS$.
For each place $\gothp\in \Sigma_p$, we set $\ttS_{\infty/\gothp}=\Sigma_{\infty/\gothp}\cap \ttS$.

\section{Basics of Shimura Varieties}
\label{Section:Sh Var}
We first collect some basic facts on integral canonical models of Shimura varieties.
Our main references are \cite{deligne1, deligne2,  milne-book,kisin}. (Our convention follows \cite{milne-book,kisin}.)   We focus on how to transfer integral canonical models of Shimura varieties from one group to another group.
This is  well known to the experts. We include the discussion here for completeness.
One novelty of this section, however, is that we give an appropriate definition of ``canonical model" for certain discrete Shimura varieties, so that the construction holds uniformly for both usual Shimura varieties and these zero-dimensional ones.
This will be important for later applications to transferring  the description of the Goren-Oort strata between Shimura varieties for different groups.

\begin{notation}
\label{N:condition on G}

Fix a prime number $p$.  Fix an isomorphism $\iota: \CC \xrightarrow{\cong }\overline \QQ_p$.
We use $\overline \QQ$ to denote the algebraic closure of $\QQ$ inside
$\CC$ (which is then identified with the algebraic closure of $\QQ$ inside $\overline \QQ_p$ via $\iota$).

In this section, let $G$ be a connected reductive group over $\QQ$.
We use $G(\RR)^+$ to denote the neutral connected component of $G(\RR)$.
We put $G(\QQ)^+ = G(\RR)^+ \cap G(\QQ)$.
We use $G^\ad$ to denote the adjoint group and $G^\der$ its derived subgroup.
We use $G(\RR)_+$ to denote the preimage of $G^\ad(\RR)^+$ under the natural homomorphism $G(\RR) \to G^\ad(\RR)$. We put $G(\QQ)_+ = G(\RR)_+ \cap G(\QQ)$.

For $S$ a torus over $\Qp$, let $S(\Zp)$ denote the maximal open compact subgroup of $S(\Qp)$.


\end{notation}

\subsection{Shimura varieties over $\CC$}
\label{A:Shimura varieties}
Put $\SSS =\Res_{\CC/\RR}\GG_m$.
For a real vector space $V$, a \emph{Deligne homomorphism} $h: \SSS \to \GL(V)$ induces a direct sum decomposition $V_\CC = \oplus_{a, b \in \ZZ} V^{a,b}$ such that $z \in \SSS(\RR) \cong \CC^\times$ acts on $V^{a,b}$ via the character $z^{-a} \bar z^{-b}$.
Let $r$ denote the $\CC$-homomorphism $\GG_{m, \CC} \to \SSS_\CC$ such that $z^{-a} \bar z^{-b} \circ r = (x \mapsto x^{-a})$.

A \emph{Shimura datum} is a pair $(G, X)$ consisting of a connected reductive  group $G$ over $\QQ$ and a $G(\RR)$-conjugacy class $X$ of homomorphisms $h: \SSS \to G_\RR$ satisfying the following conditions:
\begin{itemize}
\item[(SV1)]
for $h \in X$, only characters $z/\bar z, 1, \bar z / z$ occur in the representation of $\SSS(\RR) \cong \CC^\times$ on $\Lie(G^\ad)_\CC$ via $\mathrm{Ad} \circ h$;
\item[(SV2)]
for $h \in X$,
$\mathrm{Ad}(h(i))$ is a Cartan involution on $G^\ad_\RR$; and
\item[(SV3)]
$G^\ad$ has no $\QQ$-factor $H$ such that $H(\RR)$ is compact.
\end{itemize}
The $G(\RR)$-conjugacy class $X$ of $h$
admits the structure of a complex manifold.
Let $X^+$ denote a fixed connected component of $X$.

A pair $(G,X)$ satisfying only (SV1) (SV2) and the following (SV3)' is called a \emph{weak Shimura datum}. 
\begin{itemize}
\item[(SV3)']
$G^\ad(\RR)$ is compact (and hence connected by \cite[p.277]{borel}; this forces the image of $h$ to land in the center $Z_\RR$ of $G_\RR$).
\end{itemize}

For an open compact subgroup $K \subseteq G(\AAA^\infty)$, we define the \emph{Shimura variety} for $(G, X)$ with level $K$ to be the quasi-projective variety $\Sh_K(G,X)_\CC$, whose $\CC$-points are
\[
\Sh_K(G, X)(\CC) := G(\QQ) \backslash
X \times G(\AAA^\infty) / K \cong
G(\QQ)_+ \backslash X^+ \times G(\AAA^\infty) / K.
\]
When $(G,X)$ is a weak Shimura datum, $\Sh_K(G,X)_\CC$ is just a finite set of points.

\subsection{Reflex field}
Let $(G,X)$ be a (weak) Shimura data.
The \emph{reflex field}, denoted by $E= E(G, X)$, is the field of definition of the conjugacy class of the composition $h \circ r: \GG_{m, \CC} \to \SSS_\CC \to G_\CC$.
It is a subfield of $\CC$, finite over $\QQ$.
We refer to   \cite{deligne1} for the definition of the canonical model $\Sh_K(G,X)$ of $\Sh_K(G,X)_\CC$ over this reflex field $E$.
We assume from  now on, all (weak) Shimura varieties we consider in this section admit canonical models.
(In fact, {\it loc. cit.} excludes the case when $(G,X)$ is a weak Shimura datum. We will give the meaning of the canonical model in this case in Subsection~\ref{S:integral model weak Shimura datum} later.)

We will always assume that $K$ is the product $K^p K_p$ of an open compact subgroup $K^p$ of $G(\AAA^{\infty, p})$ and an open compact subgroup $K_p$ of $G(\Qp)$.
Taking the inverse limit over the open compact subgroups $K^p$, we have $\Sh_{K_p}(G, X) := \varprojlim_{K^p} \Sh_{K^pK_p}(G, X)$. This is actually a scheme locally of finite type over $E$ carrying  a natural (right) action of $G(\AAA^{\infty,p})$.

\subsection{Extension property}
\label{S:extension property}
Let $\calO$ be the ring of integers in a finite extension of $\Q_p$.  A  scheme  $X$ over $\calO$
is said to have the \emph{extension property} over $\calO$ if for
any smooth $\calO$-scheme $S$, a map $S \otimes \mathrm{Frac}(\calO) \to X$ extends to $S$.  (Such an extension is automatically unique if it exists by the normality of $S$.)  Note that this condition is weaker than the one given in \cite[2.3.7]{kisin} but is enough to ensure the uniqueness.

 The chosen isomorphism $\CC \cong \overline \QQ_p$ identifies $E$ as a subfield of $\overline \QQ_p$. Let $E_\wp$ denote the $p$-adic completion of $E$, $\calO_\wp$ its valuation ring with $\FF_\wp$ as the residue field.  Let $E_\wp^\ur$ be the maximal unramified extension of $E_\wp$ and $\calO_\wp^\ur$ its valuation ring.

An \emph{integral canonical  model} $\Sh_{K_p}(G, X)_{\calO_\wp}$ of $\Sh_{K_p}(G,X)$ over $\calO_{\wp}$ is  an $\cO_{\wp}$-scheme  $\Sh_{K_p}(G,X)_{\calO_\wp}$, which is an inverse limit of smooth $\cO_{\wp}$-schemes $\Sh_{K_{p}K^p}(G,X)_{\calO_{\wp}}$ with finite \'etale transition maps as $K^p$ varies, such that
\begin{itemize}
\item there is an isomorphism $\Sh_K(G,X)_{\calO_\wp} \otimes_{\calO_\wp} E_\wp \cong \Sh_K(G,X) \otimes_{E} E_\wp$ for each $K$ compatible with transition maps as $K$ varies, and 
\item $\Sh_{K_p}(G,X)_{\calO_\wp}=\varprojlim_{K^p}\Sh_{K^pK_p}(G,X)_{\calO_{\wp}}$ satisfies the extension property.
\end{itemize}

\vspace{10pt}

Existence of integral canonical models of Shimura varieties of abelian type with hyperspecial level structure was proved by Kisin \cite{kisin}.
Unfortunately, our application requires, in some special cases, certain non-hyperspecial level structures, as well as certain non-quasi-split groups at $p$.
We have to establish the integral canonical models in two steps: we first prove the existence for some group $G'$ with the same derived and adjoint groups as $G$ (as is done in Section~\ref{Section:Integral-model}); we then reproduce a  variant of an argument of Deligne to show that the integral canonical model for the Shimura variety for $G'$ gives that of $G$.  The second step is well known at least for regular Shimura varieties when $K_p$ is hyperspecial (\cite{kisin}). Our limited contribution here is to include some non-hyperspecial case and to cover the case of discrete Shimura varieties, in a uniform way.

\begin{hypo}
\label{H:hypo on G}
Let $(G,X)$ be a (weak) Shimura datum.
From now on, we assume that the derived subgroup $G^\mathrm{der}$ is simply-connected, which will be the case when we apply the theory later.
Let $Z$ denote the center of $G$.
Let $\nu: G \twoheadrightarrow T$ denote the maximal abelian quotient of $G$.
We fix an open compact subgroup $K_p$ of $G(\Qp)$ such that $\nu(K_p) = T(\Zp)$ and $K_p \cap Z(\Qp) = Z(\Zp)$.
\end{hypo}

\subsection{Geometric connected components}
\label{A:geometric connected components}
We put $T(\RR)^\dagger = \mathrm{Im}(Z(\RR) \to T(\RR))$ and $T(\QQ)^\dagger = T(\RR)^\dagger \cap T(\QQ)$.  Put $T(\QQ)^{?, (p)} = T(\QQ)^? \cap T(\Zp)$ for $? = \emptyset$ or $\dagger$.
Let $Y$ denote the finite quotient $T(\RR) / T(\RR)^\dagger$,
which is isomorphic to $T(\QQ) / T(\QQ)^\dagger$ because $T(\QQ)$ is dense in $T(\RR)$.
The morphism $\nu: G \to T$ then induces a natural map
\begin{align}
\label{E:nu map}
&
\xymatrix@C=15pt{\nu\colon
G(\QQ)_+ \backslash
X^+ \times G(\AAA^\infty) / K
\ar[r]^-\nu &
T(\QQ)^\dagger \backslash
T(\AAA^\infty) / \nu(K) \cong
T(\QQ)^{\dagger, (p)} \backslash 
T(\AAA^{\infty, p}) / \nu(K^p)
}\\
\nonumber
&\qquad\qquad\qquad \qquad\qquad\qquad\qquad\qquad\qquad\cong T(\QQ)^{(p)}\backslash Y \times T(\AAA^{\infty, p}) / \nu (K^p).
\end{align}

If $(G, X)$ is a Shimura datum, this map induces an isomorphism \cite[Theorem~5.17]{milne-book} on the set of geometric connected components $\pi_0\big(
\Sh_K(G, X)_{\overline \QQ} \big)\cong
T(\QQ)^{(p)}\backslash Y \times T(\AAA^{\infty, p}) / \nu (K^p)$.

Taking inverse limit gives a bijection
\[
\pi_0(\Sh_{K_p}(G, X)_{\overline \QQ}) \cong T(\QQ)^{\dagger, (p), \cl} \backslash T(\AAA^{\infty,p})   \cong T(\QQ)^{(p), \cl} \backslash Y \times T(\AAA^{\infty,p}),
\]
Here and later, the superscript $\cl$ means taking closure in $T(\AAA^{\infty,p})$ (or in appropriate topological groups).

\subsection{Reciprocity law}
\label{A:reciprocity law}
Let $(G,X)$ be a (weak) Shimura datum.
The composite
\[
\xymatrix{
\nu hr: \GG_{m, \CC} \ar[r]^-{r} & \SSS_\CC \ar[r]^-{h} &
G_\CC \ar[r]^-{\nu} & T_\CC
}
\]
does not depend on the choice of $h$ (in the conjugacy class) and is defined over the reflex field $E$.
The \emph{Shimura reciprocity map} is given by
\[
\Rec(G, X):\
\Res_{E/\QQ}(\GG_m) \xrightarrow{\Res_{E/\QQ}(\nu hr)} T_E \xrightarrow{N_{E/\QQ}} T.
\]

We normalize the Artin reciprocity map $\Art_E: \AAA^\times_E / (E^{\times}E_\RR^{\times, +})^\cl \xrightarrow{\cong} \Gal_E^\ab$ so that the local parameter at a finite place $\gothl$ is mapped to a geometric Frobenius at $\gothl$, where $E_\RR^{\times, +}$ is the identity connected component of $E_\RR^\times$.
We denote the unramified Artin map at $\wp$ by $\Art_\wp: E_\wp^\times / \calO_\wp^\times \to \Gal_{E_\wp}^{\ab, \mathrm{ur}}$ (again normalized so that a uniformizer is mapped to the geometric Frobenius).

The morphism $\Rec(G, X)$ induces a natural homomorphism
\[
\xymatrix@C=10pt{\gothRec=
\gothRec(G, X): \
\Gal_E^{\ab} &&
\ar[ll]_-{\Art_E}^-\cong
(E^\times E_\RR^{\times, +})^\cl \backslash \AAA^\times_E\ar[rrr]^-{\Rec(G,X)} &&& T(\QQ)^\cl \backslash Y \times T(\AAA^{\infty}).
}
\]
When $(G,X)$ is a Shimura datum, the Shimura reciprocity law \cite[Section~13]{milne-book} says that the action of $\sigma \in \Gal_E$ on $\pi_0(\Sh_{K_p}(G,X)_{\overline \QQ}) \cong T(\QQ)^\cl \backslash Y \times T(\AAA^\infty) / T(\Zp)$ is given by multiplication by $\gothRec(G,X)(\sigma)$.  As a corollary, $\pi_0(\Sh_{K_p}(G,X)_{\overline \QQ}) = \pi_0(\Sh_{K_p}(G,X)_{E_\wp^\ur})$, i.e. the geometric connected components are seen over an unramified extension of $E_\wp$.

The action of $\Gal_{E_\wp}^{\ab, \ur} = \Gal_{\FF_\wp}$ on the geometric connected component is then given by multiplication by the image of the Galois group element under the following map:
\begin{align}
\nonumber
\gothRec_\wp = \gothRec_\wp(G, X)\colon&
\Gal_{\FF_\wp}
\xrightarrow[\cong]{\Art_\wp}
\widehat{E_\wp^\times / \calO_\wp^\times}\longrightarrow \\
\label{E:reciprocity-at-p} &(E^\times E_\RR^{\times, +})^\cl \backslash \AAA^\times_E / \calO_\wp^\times \xrightarrow{\Rec(G,X)} T(\QQ)^{(p), \cl} \backslash Y \times T(\AAA^{\infty,p}),
\end{align}
where $\widehat{E_\wp^\times / \calO_\wp^\times}$ denotes the profinite completion of $E_\wp^\times / \calO_\wp^\times$.

\subsection{Integral canonical model for weak Shimura datum}
\label{S:integral model weak Shimura datum}
When $(G,X)$ is a weak Shimura datum, the associated Shimura variety is, geometrically, a finite set of points. We define its canonical model by specifying the action of $\Gal_E$.
The key observation is that condition (SV3)' ensures that the morphism $\Rec(G,X)$ factors as
\[
\xymatrix@C=50pt{
\Res_{E/\QQ}(\GG_m) \ar[r]^-{\Res_{E/\QQ}(hr)} \ar@{-->}[dr]_{\Rec_Z(G,X)}
&
 Z_E \ar[r]^{\Res_{E/\QQ}(\nu)} \ar[d]^{N_{E/\QQ}}
& T_E \ar[d]^{N_{E/\QQ}}
\\
& Z \ar[r]^{\nu} & T.
}
\]
We consider the natural homomorphism
\[
\xymatrix@C=10pt{\gothRec_Z: \
\Gal_E^{\ab} &&
\ar[ll]_-{\Art_E}^-\cong
(E^\times E_\RR^{\times, +})^\cl \backslash \AAA^\times_E\ar[rrr]^-{\Rec_Z(G,X)} &&& Z(\QQ)^\cl Z(\RR) \backslash Z(\AAA) \cong Z(\QQ)^\cl \backslash Z(\AAA^{\infty}).
}
\]
We define the \emph{canonical model} $\Sh_K(G,X)$ to be the (pro-)$E$-scheme whose base change to $\CC$ is isomorphic to $\Sh_K(G,X)_\CC$, such that every $\sigma \in \Gal_E$ acts on its $\overline \QQ$-points by multiplication by $\gothRec_Z(\sigma)$.
In comparison to Subsection~\ref{A:reciprocity law}, we have $\nu(\sigma(x)) = \gothRec(G,X)(\sigma)\cdot \nu( x)$ for any $x \in \Sh_K(G,X)(\overline \QQ)$.

Since $\Sh_K(G,X)$ is just a finite union of spectra of some finite extensions of $E$, it naturally admits an integral canonical model over $\calO_\wp$ by taking the corresponding valuation rings. With the map $\gothRec_\wp$ as defined in \eqref{E:reciprocity-at-p}, we have $\nu(\sigma(x)) = \gothRec_\wp(G,X)(\sigma)\cdot \nu(x)$ for any closed point $x \in \Sh_K(G,X)_{\calO_\wp}$ and $\sigma \in \Gal_{\FF_\wp}$.

\begin{notation}
\label{N:hypo on G}

We put $K^\der_p = K_p \cap G^\der(\Qp)$. Let $K^\ad_p$ denote the image of $K_p$ in $G^\ad(\Qp)$.
Set $G^?(\QQ)^{(p)} = G^?(\QQ) \cap K^?_p$ for $? = \emptyset, \ad$ and $\der$; they are the subgroups of $p$-integral elements.  Put $G^\ad(\QQ)^{+, (p)} = G^\ad(\RR)^+ \cap G^\ad(\QQ)^{(p)}$ and $G^?(\QQ)_+^{(p)} = G^?(\RR)_+ \cap G^?(\QQ)^{(p)}$ for $? = \emptyset$ or $\der$.

\end{notation}

\subsection{A group theoretic construction}
Before proceeding, we recall a pure group theoretic construction.
See \cite[\S~2.0.1]{deligne2} for more details.

Let $H$ be a group equipped with an action $r$ of a group $\Delta$, and $\Gamma \subset H$ a $\Delta$-stable subgroup.  Suppose given a $\Delta$-equivariant map $\varphi: \Gamma \to \Delta$, where $\Delta$ acts on itself by inner automorphisms, and suppose that for $\gamma \in \Gamma$, $\varphi(\gamma)$ acts on $H$ as inner conjugation by $\gamma$.

Given the data above, we can first define the semi-product $H \rtimes \Delta$ using the action $r$.
The conditions above imply that the natural map $\gamma \mapsto (\gamma, \varphi(\gamma)^{-1})$ embeds $\Gamma$ as a normal subgroup of $H \rtimes \Delta$.
We define the \emph{star extension} $H \ast_\Gamma \Delta$ to be the quotient of $H \rtimes \Delta$ by this subgroup.

Two typical examples we will encounter later are
\[
G^\der(\AAA^{\infty, p}) \ast _{G^\der(\QQ)^{(p)}} G(\QQ)^{(p)} \cong  G^\der(\AAA^{\infty, p}) \cdot G(\QQ)^{(p)} \quad \textrm{and} \quad G^\der(\AAA^{\infty,p}) \ast _{G^\der(\QQ)^{(p)}} G^\ad(\QQ)^{(p)}.
\]

\subsection{The connected components of the integral model}
\label{A:connected integral model}
Let $(G,X)$ be a (weak) Shimura datum.
Suppose that there exists an integral canonical model $\Sh_{K_p}(G,X)_{\calO_\wp}$.
For $K^p$ an open compact subgroup of $G(\AAA^{\infty, p})$, let $\Sh_{K^pK_p}(G, X)^\circ_{\calO_\wp^\ur}$ denote 
the open and closed subscheme whose $\CC$-points consists of the preimage of $\{1\}$ under the $\nu$-map in \eqref{E:nu map}.  When $(G,X)$ is a Shimura datum, this gives a connected component of $\Sh_{K^pK_p}(G,X)_{\calO_\wp^\ur}$.
We put 
\begin{equation}
\label{E:connected components of Shimura varieties}
\Sh_{K_p}(G, X)^\circ_{\calO_\wp^\ur} = \varprojlim_{K^p} \Sh_{K^pK_p}(G,X)^\circ_{\calO_\wp^\ur} \quad \textrm{and} \quad \Sh_{K_p}(G, X)^\circ_{\overline \FF_\wp} = \Sh_{K_p}(G, X)^\circ_{\calO_\wp^\ur} \otimes_{\calO_\wp^\ur} \overline \FF_\wp.
\end{equation}
Note that the set of $\CC$-points of $\Sh_{K_p}(G, X)^\circ_{\calO_\wp^\ur}$ is nothing but $G^\der(\QQ)^{(p), \cl}_+ \backslash (X^+ \times G^\der(\AAA^{\infty, p}))$.
When $(G,X)$ is a Shimura datum, strong approximation shows that this is a projective limit of connected complex manifold.  In any case, this implies that $\Sh_{K_p}(G,X)^\circ_{\calO_\wp^\ur}$ depends only on $X$, the groups $G^\der$ and $G^\ad$ (as opposed to the group $G$), and the subgroups $K_p^\der$ and $K_p^\ad$ (as opposed to $K_p$).

We also point out that \eqref{E:nu map} gives rise to a natural map
\begin{equation}
\label{E:nu-map on Sh var}
\nu\colon
\pi_0(\Sh_{K_p}(G, X)_{\calO_\wp^\ur}) =
\pi_0(\Sh_{K_p}(G, X)_{\overline \FF_\wp}) =  \pi_0(\Sh_{K_p}(G,X)_{\overline \QQ})
\longrightarrow T(\QQ)^{(p), \cl} \backslash Y \times T(\AAA^{\infty, p}).
\end{equation}
By abuse of language, we call \eqref{E:connected components of Shimura varieties} the geometric connected components of the Shimura varieties, and the target of \eqref{E:nu-map on Sh var} the set of connected components (although this is not the case if $(G,X)$ is a weak Shimura datum).

The Shimura varieties
$\Sh_{K_p}(G,X)_{{\boldsymbol ?}}$ for ${\boldsymbol ?} = \calO_\wp^\ur$ and $\overline \FF_\wp$ admit the following actions.
\begin{enumerate}
\item
The natural right action of $G(\AAA^{\infty, p})$ on $\Sh_{K_p}(G,X)_{\overline \QQ}$ extends to a right action on $\Sh_{K_p}(G,X)_\textbf{?}$.
The subgroup $Z(\QQ)^{(p)} := Z(\QQ) \cap G(\Zp)$ acts trivially.
So the right multiplication action above factors through $ G(\AAA^{\infty, p})\big/ Z(\QQ)^{(p), \cl}$.
The induced action on the set of connected components is given by $\nu: G(\AAA^{\infty,p})\big/ Z(\QQ)^{(p),\cl}\to T(\QQ)^{\dagger,(p),\cl} \backslash T(\AAA^{\infty,p})$.

\item There is a right action $\rho$ of $G^\ad(\QQ)^{+,(p)}$ on $\Sh_{K_p}(G,X)_{\calO_\wp^\ur}$ such that the induced map on $\CC$-points is given by, for $g \in G^\ad(\QQ)^{+,(p)}$,
\[
\xymatrix@C=40pt@R=0pt{
\rho(g)\colon
G(\QQ)_+^\cl \backslash X^+ \times G(\AAA^\infty) / K_p
\ar[r] &
G(\QQ)_+^\cl \backslash X^+ \times G(\AAA^\infty)  /K_p
\\
 [x, a] \ar@{|->}[r]
& [g^{-1}x, \mathrm{int}_{g^{-1}}(a)].
}
\]
Here note that $K_p$ is stable under the conjugation action of $K^\ad_p$ and hence of $G^\ad(\QQ)^{+,(p)}$.
One extends the action $\rho(g)$ to the integral model and hence to the special fiber using the extension property.  Moreover, this action preserves the connected component $\Sh_{K_p}(G, X)^\circ_{\boldsymbol ?}$.

\item
For an element $g \in G(\QQ)_+^{(p)}$, the two actions above coincide.  Putting  them together, we have a right action of the group
\begin{equation}
\label{E:calG}
\calG := \big(G(\AAA^{\infty,p}) \big/ Z(\QQ)^{(p), \cl}\big) \ast_{G(\QQ)^{(p)}_+/Z(\QQ)^{(p)}} G^\ad(\QQ)^{+,(p)}
\end{equation}
on
$\Sh_{K_p}(G, X)^\circ_{\boldsymbol ?}$.
The induced action on the set of connected components is given by 
\[
\nu \ast \mathrm{triv}: \calG \twoheadrightarrow T(\QQ)^{\dagger,(p),\cl} \backslash T(\AAA^{\infty, p}),
\] i.e., $\nu$ on the first factor and trivial on the second factor.

\item The Galois group $\Gal(E_\wp^\ur/E_\wp)$ acts on $\Sh_{K_p}(G)_{\boldsymbol ?}$, according to \eqref{E:reciprocity-at-p} (and Subsection~\ref{S:integral model weak Shimura datum}).

\end{enumerate}

Let $\calE_{G, \wp}$ denote the subgroup of $\calG \times \Gal(E_\wp^\ur / E_\wp)$ consisting of pairs $(g, \sigma)$ such that $(\nu \ast \mathrm{triv})(g)$ is equal to $\gothRec_\wp(\sigma)^{-1}$ in $T(\QQ)^{\dagger,(p),\cl} \backslash T(\AAA^{\infty,p})$.
Then by the discussion above, the group $\calE_{G,\wp}$ acts on the connected component $\Sh_{K_p}(G,X)^\circ_{\boldsymbol ?}$.


Conversely, knowing $\Sh_{K_p}(G,X)^\circ_{\boldsymbol ?}$ together with the action of $\calE_{G,\wp}$, we can recover the integral model $\Sh_{K_p}(G,X)_{\calO_\wp}$ or its special fiber $\Sh_{K_p}(G,X)_{\FF_\wp}$ of the Shimura variety as follows.
We consider the (pro-)scheme
$\Sh_{K_p}(G,X)^\circ_{\calO_\wp^\ur} \times_{\calE_{G,\wp}} \big(\calG \times \Gal(E_\wp^\ur / E_\wp) \big)$.
Since this is a projective limit of \emph{quasi-projective} varieties, by Galois descent,
it is the base change of a projective system of varieties
$\Sh_{K_p}(G,X)_{\calO_\wp}$ from  $\calO_\wp$ to $\calO_\wp^\ur$.
The same argument applies to the special fiber.

In general, for a finite unramified extension $\tilde E_{\tilde \wp}$ of $E_\wp$, we put $\calE_{G, \tilde E_{\tilde \wp}}$ to be the subgroup of $\calE_{G, \wp}$ consisting of elements whose second coordinate lives in $\Gal(E_\wp^\ur/\tilde E_{\tilde \wp})$.  Knowing the action of $\calE_{G, \tilde E_{\tilde \wp}}$ on $\Sh_{K_p}(G,X)^\circ_{\calO_\wp^\ur}$ or $\Sh_{K_p}(G,X)^\circ_{\overline \FF_\wp}$ allows one to descend the integral model to $\Sh_{K_p}(G,X)_{\calO_{\tilde E_{\tilde \wp}}}$.

\subsection{Transferring mathematical objects}
\label{S:transfer math obj}
One can slightly generalize the discussion above to $\calE_{G, \tilde E_{\tilde \wp}}$-equivariant mathematical objects over the Shimura variety.
More precisely, for $? = \FF_\wp, \calO_\wp$,  by a \emph{mathematical object} $\calP$ over $\Sh_{K^p}(G,X)_?$, we mean,
for each sufficiently small open compact subgroups $K^p$ of $G(\AAA^{\infty,p})$, we have a (pro-)scheme or a vector bundle (with a section) $\calP_{K^p}$ over $\Sh_{K_pK^p}(G,X)_{\boldsymbol ?}$, such that, for any subgroup $K_1^p \subseteq K_2^p$, $\calP_{K_1^p}$ is the base change of $\calP_{K_2^p}$ along the natural morphism $\Sh_{K_pK_1^p}(G,X)_{\boldsymbol ?} \to \Sh_{K_pK_2^p}(G,X)_{\boldsymbol ?}$.
We say $\calP$ is \emph{$\calG \times \Gal(E_\wp^\ur/\tilde E_{\tilde \wp})$-equivariant} if $\calP$ carries an action of   $\calG \times \Gal(E_\wp^\ur/\tilde E_{\tilde \wp})$ that is compatible with the actions on the Shimura varieties.

Similarly, a \emph{mathematical object} $\calP^\circ$ over $\Sh_{K^p}(G, X)^\circ_{?^\ur}$ is a (pro-)scheme or a vector bundle (with a section) as above, over the connected Shimura variety $\Sh_{K^p}(G, X)^\circ_{?^\ur}$, viewed as a pro-scheme.  It is called $\calE_{G, \tilde E_{\tilde \wp}}$-equivariant, if it carries an action of the group compatible with the natural 
group action on the base Shimura variety.

Similar to the discussion above, we have the following.
\begin{cor}
\label{C:mathematical objects equivalence}
There is a natural equivalence of categories between the category of $\calG \times \Gal(E^\ur_\wp / \tilde E_{\tilde \wp})$-equivariant mathematical objects $\calP$ over the tower of Shimura varieties $\Sh_{K_pK^p}(G,X)_?$, and the category of mathematical objects $\calP^\circ$ over $\Sh_{K^p}(G,X)^\circ_{?^\ur}$, equivariant for the action of $\calE_{G, \tilde E_{\tilde \wp}}$.
\end{cor}
\begin{proof}
As above, given $\calP$, we can recover $\calP^\circ$ by taking inverse limit with respect to the open compact subgroup $K^p$ and then restricting to the connected component $\Sh_{K_p}(G,X)^\circ_{?^\ur}$.  Conversely, we can recover $\calP$ from $\calP^\circ$ through the isomorphism $
\calP_{?^\ur} \cong \calP^\circ \times_{\calE_{G, \tilde E_{\tilde \wp}}} (\calG \times \Gal(E_\wp^\ur/\tilde E_{\tilde \wp}))$
and then use Galois descent if needed.
\end{proof}

\begin{remark}
\label{R:no Galois action}
If one does not consider the Galois action, Theorem~\ref{T:structure of calE_G,p} below implies that
\[
\Sh_{K_p}(G,X)_{\calO_\wp^\ur} \cong \Sh_{K_p}(G,X)_{\calO_\wp^\ur}^\circ \times_{\big(G^\der(\AAA^{\infty,p}) \ast_{G^\der(\QQ)^{(p)}_+} G^\ad(\QQ)^{+,(p)} \big)}\calG,
\]
and the same applies to the mathematical objects.
\end{remark}

\begin{lemma}
\label{L:nu(G(Q)->>T(Q)}
We have $\nu(G(\QQ)_+^{(p)}) = T(\QQ)^{\dagger, (p)}$.
\end{lemma}
\begin{proof}
By Subsection~\ref{A:geometric connected components}, we have $\nu(G(\QQ)_+) = T(\QQ)^\dagger$.
The lemma follows from taking the kernels of the following morphism of exact sequences
\[
\xymatrix{
1 \ar[r] &
G^\der(\QQ)_+ \ar[r] \ar@{->>}[d] &
G(\QQ)_+ \ar[r] \ar[d] &
T(\QQ)^\dagger \ar[r] \ar[d] & 1\\
1 \ar[r] &
G^\der(\QQ_p)/K^\der_p \ar[r] &
G(\QQ_p)/K_p \ar[r] &
T(\Qp) / T(\Zp) \ar[r] & 1
.}
\]
Here, the left vertical arrow is surjective by the strong approximation theorem for the simply-connected group $G^\der(\QQ)$.
The bottom sequence is exact because the corresponding sequences are exact both for $\Qp$ (because $H^1(\Qp, G^\der)=0$) and for $\Zp$ (by Hypothesis~\ref{H:hypo on G}).
\end{proof}

The following structure theorem for $\calE_{G, \wp}$ is the key to transfer integral canonical models of Shimura varieties for one group to that for another group.

\begin{theorem}
\label{T:structure of calE_G,p}
For a finite unramified extension $\tilde E_{\tilde \wp}$ of $E_\wp$, we have a natural short exact sequence.
\begin{equation}
\label{E:structure of E_p}
1 \longrightarrow G^\der(\AAA^{\infty,p}) \ast_{G^\der(\QQ)^{(p)}_+} G^\ad(\QQ)^{+,(p)}
\longrightarrow \calE_{G, \tilde E_{\tilde \wp}} \longrightarrow \Gal(E_\wp^\ur / \tilde E_{\tilde \wp}) \longrightarrow 1.
\end{equation}
\end{theorem}
\begin{proof}
By the definition of $\calE_{G,\tilde E_{\tilde \wp}}$, it fits into the following short  exact sequence
\[
1 \longrightarrow \Ker\Big (\widetilde \calG \to T(\QQ)^{\dagger,(p), \cl} \big \backslash  T(\AAA^{\infty,p}) \Big)
\longrightarrow \calE_{G, \tilde E_{\tilde \wp}} \longrightarrow \Gal(E_\wp^\ur / \tilde E_{\tilde \wp}) \longrightarrow 1.
\]
By Lemma~\ref{L:nu(G(Q)->>T(Q)}, the kernel above is isomorphic to
\begin{equation}
\label{E:expression of kernel}
\big(\,
\big(G(\QQ)_+^{(p)}   G^\der(\AAA^{\infty,p})\big)^\cl 
\big/ Z(\QQ)^{(p), \cl}
\big) \ast_{G(\QQ)_+^{(p)} / Z(\QQ)^{(p)}} G^\ad(\QQ)^{+,(p)},
\end{equation}
where both closures are taken   inside $G(\AAA^{\infty, p})$.

We claim that we can remove the two completions.
Indeed, put $Z' = Z \cap G^\der$ and $Z'(\QQ)^{(p)} = Z'(\QQ) \cap Z(\QQ)^{(p)}$; the latter is a finite group. Consider the commutative diagram of exact sequences
\begin{tiny}
\[
\xymatrix@R=15pt@C=10pt{
1 \ar[r] & Z'(\QQ)^{(p)} \ar[r] \ar@{=}[d] & Z(\QQ)^{(p)} \times G^\der(\AAA^{\infty,p}) \ar[d] \ar[r] & G(\QQ)_+^{(p)} G^\der(\AAA^{\infty,p}) \ar[d] \ar[r] & T(\QQ)^{\dagger,(p)} \big/ \mathrm{Im}\big(Z(\QQ)^{(p)} \to T(\QQ)^{(p)}\big) \ar[r] \ar[d] & 1
\\
1 \ar[r] & Z'(\QQ)^{(p), \cl} \ar[r] &{Z(\QQ)^{(p), \cl}} \times G^\der(\AAA^{\infty,p}) \ar[r] & \big(G(\QQ)^{(p)}_+ G^\der(\AAA^{\infty,p})\big)^\cl \ar[r] & T(\QQ)^{\dagger, (p), \cl} \big/\mathrm{Im}\big({Z(\QQ)^{(p), \cl}} \to {T(\QQ)^{(p), \cl}}\big) \ar[r] & 1.
}
\]\end{tiny}
By diagram chasing, it suffices to prove that the right vertical arrow is an isomorphism.
Since the kernel of $Z \to T$ is finite, \cite[\S~2.0.10]{deligne2} implies that $\mathrm{Im}({Z(\QQ)^\cl} \to {T(\QQ)^\cl}) \cong (\mathrm{Im}(Z(\QQ) \to T(\QQ)))^\cl$ and the right vertical arrow is an isomorphism.

Now, the exact sequence \eqref{E:structure of E_p} follows from a series of tautological isomorphisms
\begin{align*}
&\big( G(\QQ)_+^{(p)} \cdot  G^\der(\AAA^{\infty,p}) / Z(\QQ)^{(p)} \big) \ast_{G(\QQ)^{(p)}_+ / Z(\QQ)^{(p)}} G^\ad(\QQ)^{+,(p)} \\ \cong\ & \Big[ \big(G^\der(\AAA^{\infty,p}) \ast_{G^\der(\QQ)_+^{(p)}} G(\QQ)_+^{(p)}) / Z(\QQ)^{(p)} \Big] \ast_{G(\QQ)_+^{(p)} / Z(\QQ)^{(p)}} G^\ad(\QQ)^{+,(p)} \\  \cong\ & \Big[ G^\der(\AAA^{\infty,p}) \ast_{G^\der(\QQ)_+^{(p)}} \big( G(\QQ)_+^{(p)} / Z(\QQ)^{(p)} \big) \Big] \ast_{G(\QQ)_+^{(p)} / Z(\QQ)^{(p)}} G^\ad(\QQ)^{+,(p)} \\ \cong \ & G^\der(\AAA^{\infty,p}) \ast_{G^\der(\QQ)_+^{(p)}} G^\ad(\QQ)^{+, (p)}.\tag*{\qedhere} 
\end{align*}
\end{proof}

\begin{cor}
\label{C:Sh(G)^circ_Zp independent of G}
Let $\varphi: G \to G'$ be a homomorphism of two reductive groups over $\QQ$ satisfying Hypothesis~\ref{H:hypo on G}, which induces isomorphisms between the derived and adjoint groups as well as isomorphisms $G^{?}(\Q)^{(p)}\cong G'^{?}(\QQ)^{(p)}$ for $?=\der,\ad$.
A $G^\ad(\RR)^+$-conjugacy class $X^+$ of homomorphisms $h: \SSS \to G_\RR$ induces a $G'^\ad(\RR)^+$-conjugacy class $X'^+$ of homomorphisms $h': \SSS \to G'_\RR$.
Put $X = G(\RR) \cdot X^+$ and $X' = G'(\RR) \cdot X'^+$.
Then, for any field $\tilde E_{\tilde \wp}$ containing both $E_\wp$ and $E'_{\wp'}$ and unramified over them, there exist a natural isomorphism of groups $\calE_{G, \tilde E_{\tilde \wp}} \xrightarrow{ \cong} \calE_{G', \tilde E_{\tilde \wp}}$ and
a natural isomorphism of geometric connected components of Shimura varieties
$\Sh_{K_p}(G, X)_{\tilde E_{\tilde \wp}^\ur}^\circ \cong \Sh_{K'_p}(G', X')_{\tilde E_{\tilde \wp}^\ur}^\circ$, equivariant for
the natural actions of the  groups $\calE_{G, \tilde E_{\tilde \wp}} \xrightarrow{ \cong} \calE_{G', \tilde E_{\tilde \wp}}$.

As a corollary, if the Shimura variety for one of $G$ or $G'$ admits an integral canonical model and both $E_\wp$ and $E'_{\wp'}$ are unramified extensions of $\QQ_p$, then the other Shimura variety admits an integral canonical model.

Moreover, when there are canonical integral models, we have an equivalence of categories between the category of $\calG \times \Gal(E_\wp^\ur/\tilde E_{\tilde \wp})$-equivariant mathematical objects $\calP$ over the tower of Shimura varieties $\Sh_{K_pK^p}(G,X)_?$ (for $? = \calO_{\tilde \wp}$ or $\FF_{\tilde \wp}$) and the  category of $\calG' \times \Gal(E_\wp^\ur/\tilde E_{\tilde \wp})$-equivariant mathematical objects $\calP'$ over the tower of Shimura varieties $\Sh_{K'_pK'^p}(G',X')_{?'}$ (for $?' = \calO_{\tilde \wp'}$ or $\FF_{\tilde \wp'}$).
\end{cor}
\begin{proof}
The first part follows from Theorem~\ref{T:structure of calE_G,p} and the discussion in Subsection~\ref{A:connected integral model}.
For the second part, the existence of integral canonical model over $\tilde E_{\tilde \wp}$ follows from the first part and the discussion at the end of Subsection~\ref{A:connected integral model}.  The extension property allows one to further descend the integral canonical model to $\calO_\wp$ (or $\calO_{\wp'}$).
The last part follows from Corollary~\ref{C:mathematical objects equivalence}.
\end{proof}

\section{Integral canonical models of quaternionic Shimura varieties}\label{Section:Integral-model}

Classically, the integral model for a quaternionic Shimura variety is defined by passing to a unitary Shimura variety, as is done in the curve case by  Carayol \cite{carayol}.
As pointed out earlier that we will encounter some groups which are not quasi-split at $p$, Kisin's general work \cite{kisin} unfortunately does not apply.
We have to work out a generalization of Carayol's  construction for completeness. This will also be useful later when discussing the construction of the Goren-Oort stratification.

We tailor the choice of the unitary group to our application of Helm's isogeny trick later. In particular, we will assume that certain places above $p$ to be inert in the CM extension.

\subsection{Quaternionic Shimura varieties}
\label{S:quaternionic-shimura-varieties}
Recall the notation from \ref{SS:notation-F}.
Let $\ttS$ be an even subset of places of $F$.
Put $\ttS_\infty = \ttS \cap \Sigma_\infty$ and $\ttS_p = \ttS\cap \Sigma_p$.
Let $B_\ttS$ be the quaternion algebra over $F$ ramified precisely at $\ttS$.
Let $G_{\ttS}$ denote the reductive group $\Res_{F/\QQ}(B^\times_\ttS)$.
Then $G_{\ttS,\RR}$ is isomorphic to
\[
\prod_{\tau \in \ttS_\infty} \HH^\times \times
 \prod_{\tau \in\Sigma_\infty -\ttS_\infty}
\GL_{2, \RR}.
\]
We define the {Deligne homomorphism} to be $h_\ttS: \SSS \to G_{\ttS,\RR}$, sending $z = x +\ii y$ to $(z_{G_\ttS}^\tau)_{\tau \in \Sigma_\infty}$,
where $z_{G_\ttS}^\tau = 1$ if $\tau \in \ttS_\infty$ and
$z_{G_\ttS}^\tau = \big(\begin{smallmatrix}
x&y\\ -y&x
\end{smallmatrix}\big)$ if
$\tau \in \Sigma_\infty -\ttS_\infty$.
Let $\gothH_\ttS$ denote the $G_\ttS(\RR)$-conjugacy class of the homomorphism $h_\ttS$; it is isomorphic to the product of $\#(\Sigma_\infty - \ttS_\infty)$ copies of $\gothh^\pm = \PP^1(\CC) - \PP^1(\RR)$.
We put $\gothH_\ttS^+ = (\gothh^+)^{\Sigma_\infty - \ttS_\infty}$, where $\gothh^+$ denotes the upper half plane.

We will consider the following type of open compact subgroups of $G_\ttS(\AAA^\infty)$: $K = K^pK_p$, where $K^p$ is an open compact subgroup of $B^\times_\ttS(\AAA_F^{\infty, p})$ and $K_p = \prod_{\gothp \in \Sigma_p} K_\gothp$ with $K_\gothp$ an open compact subgroup of $B_\ttS^\times(F_\gothp)$.

From this point onward, we write $\Sh_K(G)$ instead of $\Sh_K(G,X)$ for Shimura varieties when the choice of $X$ is clear.  Associated to the data above,
there is a Shimura variety $\Sh_K(G_\ttS)$ whose $\CC$-points are
\[
\Sh_K(G_\ttS)(\CC) = G_\ttS(\QQ) \backslash ( \gothH_\ttS \times  G_\ttS(\AAA^\infty))/ K.
\]
The \emph{reflex field} $F_\ttS$ is a subfield of $\CC$ characterized as follows:
an element $\sigma \in \Aut(\CC/\QQ)$ fixes $F_\ttS$ if and only if the subset $\ttS_\infty$ of $\Sigma_\infty $ is preserved under the action of $\sigma$ by post-composition.
Following Subsection~\ref{A:Shimura varieties}, we put $\Sh_{K_p}(G_\ttS) = \varprojlim_{K^p} \Sh_{K^pK_p}(G_\ttS)$. (Note that the level structure at $p$ is fixed in the inverse limit.)

Put $T_F = \Res_{F/\QQ}\GG_m$.
The reduced norm on $B_\ttS$ induces a homomorphism $\Nm = \Nm_{B_\ttS / F}: G_\ttS \to T_F$.
This homomorphism induces a map
\[
\pi_0^\geom(\Sh_K(G_\ttS)) \longrightarrow  T_F(\QQ) \backslash ( T_F(\AAA^\infty) \times \{\pm 1\}^g) / \Nm (K),
\]
which is an isomorphism if $\ttS_\infty \subsetneq \Sigma_\infty$.
We will make the Shimura reciprocity law (Subsection~\ref{A:reciprocity law}) explicit for $\Sh_K(G_\ttS)$ later when it is in use.

\subsection{Level structure at $p$}\label{S:level-structure-at-p}

We fix an isomorphism $\iota_p: \CC \simeq \overline \QQ_p$.
For each $\gothp \in \Sigma_p$, let $\Sigma_{\infty/\gothp}$ denote the subset of $\Sigma_\infty$ consisting of real embeddings which, when composed with $\iota_p$, induce the $p$-adic place $\gothp$.
We put $\ttS_{\infty/\gothp} = \ttS \cap \Sigma_{\infty/\gothp}$. Similarly, we can view the reflexive field $F_\ttS$ as a subfield of $\overline \QQ_p$ via $\iota_p$, which induces a $p$-adic place $\wp$ of $F_\ttS$.  We use $\calO_\wp$ to denote the valuation ring and $k_\wp$ the residue field.

In this paper, we  always make the following assumption on $\ttS$:

\begin{hypo}
\label{H:B_S-splits-at-p}
If $B_\ttS$ does not split at a $p$-adic place $\gothp$ of $F$, then $\ttS_{\infty/\gothp}= \Sigma_{\infty/\gothp}$.
\end{hypo}

For each $\gothp\in \Sigma_{p}$, we now specify the level structure $K_{\gothp}\subset B_{\ttS}^\times(F_\gothp)$ of $\Sh_{K}(G_{\ttS})$ to be considered in this paper. We distinguish four types of the prime $\gothp \in \Sigma_p$:

\begin{itemize}

\item {\it Types $\alpha$ and $\alpha^\sharp$:} $B_{\ttS}$ splits at $\gothp$ and the cardinality $\# (\Sigma_{\infty/\gothp}-\ttS_{\infty/\gothp})$ is even.
We fix an identification $B^{\times}_{\ttS}(F_{\gothp})\simeq \GL_2(F_{\gothp})$. We take $K_{\gothp}$ to be
\begin{itemize}
\item
either $\GL_2(\cO_{\gothp})$, or
\item
$\Iw_\gothp =
\big(
\begin{smallmatrix}
\calO_\gothp ^\times &  \calO_\gothp\\
\gothp\calO_\gothp & \calO_\gothp^\times
\end{smallmatrix}
\big)$ which we allow only when $\Sigma_{\infty/\gothp} = \ttS_{\infty/\gothp}$.

\end{itemize}
We name the former case as type $\alpha$ and the latter as type $\alpha^\sharp$.  (Under our definition, when $\Sigma_{\infty/\gothp} = \ttS_{\infty/\gothp}$, the type of $\gothp$ depends on the choice of the level structure.)

\item {\it Type $\beta$:}
$B_{\ttS}$ splits at $\gothp$ and
the cardinality $\#(\Sigma_{\infty/\gothp}-\ttS_{\infty/\gothp})$ is odd.
We fix an identification $B_{\ttS}^{\times}(F_{\gothp})\simeq \GL_2(F_{\gothp})$. We take $K_{\gothp}$ to be $\GL_2(\cO_{\gothp})$.

\item {\it Type $\beta^\sharp$:}
$B_\ttS$ ramifies at $\gothp$ and $\ttS_{\infty/\gothp} = \Sigma_{\infty/\gothp}$. In this case, $B_{\ttS}\otimes_{F}F_{\gothp}$ is the division quaternion algebra $B_{F_{\gothp}}$ over $F_{\gothp}$.  Let $\cO_{B_{F_{\gothp}}}$ be the maximal order of $B_{F_{\gothp}}$. We take $K_{\gothp}$ to be $\cO_{B_{F_{\gothp}}}^{\times}$.
\end{itemize}

The aim of this section is to construct an integral canonical model of $\Sh_K(G_{\ttS})$ over $\calO_\wp$ with $K_{p}=\prod_{\gothp|p}K_{\gothp}$ specified above.
For this, we need to introduce an auxiliary CM extension and a unitary group.

\subsection{Auxiliary CM extension}
\label{S:CM extension}

We choose a CM extension $E$ over $F$ such that
\begin{itemize}
\item
every place in $\ttS$ is inert in $E/F$; and
\item
a place $\gothp \in \Sigma_p$ is split in $E/F$ if it is of type $\alpha$ or $\alpha^\sharp$, and is inert in $E/F$ if it is of type $\beta$ or $\beta^\sharp$.
\end{itemize}
We remark that our construction slightly differs from \cite{carayol} in that Carayol requires all places above $p$ to split in $E/F$.
For later convenience, we fix some totally negative element $\gothd \in \calO_F$ coprime to $p$ so that $E = F(\sqrt{\gothd})$.  (The construction will be independent of such a choice.)

Let $\Sigma_{E, \infty}$ denote the set of complex embeddings of $E$.  We have a natural two-to-one map $\Sigma_{E, \infty} \to \Sigma_\infty$.
For each $\tau \in \Sigma_\infty$, we often use $\tilde \tau $ to denote a complex embedding of $E$ extending $\tau$, whose complex conjugate is denoted by $\tilde \tau^c$.

We fix a choice of a subset $\tilde \ttS_\infty \subseteq \Sigma_{E,\infty}$ which consists of, for each $\tau \in \ttS_\infty$, a choice exactly one lift $\tilde \tau\in \Sigma_{E,\infty}$.
This choice is equivalent to a collection of the numbers $s_{\tilde \tau} \in \{0,1,2\}$ for all $\tilde \tau \in \Sigma_{E, \infty}$ such that
\begin{itemize}
\item
if $\tau \in \Sigma_\infty - \ttS_\infty$, we have $s_{\tilde \tau} = 1$ for all lifts $\tilde \tau $ of $\tau$;
\item
if $\tau \in \ttS_\infty$ and $\tilde \tau$ is the lift in $\tilde \ttS_\infty$, we have $s_{\tilde \tau} = 0$ and $s_{\tilde \tau^c} = 2$.
\end{itemize}

We put $\tilde\ttS=(\ttS,\tilde\ttS_{\infty})$.  
Consider  the  torus  $T_{E, \tilde \ttS} = \Res_{E/\Q}\GG_m$ together with the   following choice of the Deligne homomorphism:
\[
\xymatrix@R=0pt@C=50pt{
h_{E, \tilde \ttS}\colon
\SSS(\RR) = \CC^\times \ar[r] &
T_{E, \tilde \ttS}(\RR) = \bigoplus_{\tau \in \Sigma_\infty} (E \otimes_{F, \tau}\RR)^\times \simeq \bigoplus_{\tau\in\Sigma_\infty} \CC^\times\\
z\ar@{|->}[r] & (\bar z_{E, \tau})_\tau.
}
\]
Here $\bar z_{E, \tau} = 1$ if $\tau \in \Sigma_\infty -\ttS_\infty$ and $\bar z_{E, \tau} = \bar z$ otherwise, where, in the latter case, the isomorphism $(E \otimes_{F, \tau}\RR)^\times \simeq \CC^\times$ is given by the lift $\tilde \tau \in \tilde \ttS_\infty$.
The reflex field $E_{\tilde \ttS}$ is the subfield of $\CC$ corresponding to the subgroup of $\Aut(\C/\Q)$ which stabilizes the set $\tilde \ttS_\infty \subset \Sigma_{E, \infty}$. It contains $F_\ttS$ as a subfield.
The isomorphism $\iota_p: \CC \simeq \overline \QQ_p$ determines a $p$-adic place $\tilde \wp$ of $E_{\tilde \ttS}$. We use $\calO_{\tilde \wp}$ to denote the valuation ring of the completion of $E_{\tilde \ttS}$ at $\tilde \wp$, and $k_{\tilde \wp}$ its residue field. 
Note  that $[k_{\tilde \wp}: \FF_p]$ is always even whenever there is a place $\gothp\in \Sigma_p$ of type $\beta$.

We take the level structure $K_E$ to be $K_E^pK_{E,p}$, where $K_{E, p} = (\calO_E \otimes_\ZZ \ZZ_p)^\times$ and $K_E^p$ is an open compact subgroup of $\AAA_E^{\infty,p,\times}$.
This gives rise to a Shimura variety $\Sh_{K_E}(T_{E, \tilde \ttS})$ and its limit $\Sh_{K_{E,p}}(T_{E, \tilde \ttS}) = \varprojlim_{K_E^p}\Sh_{K_{E,p}K_E^p}(T_{E, \tilde \ttS})$. They have integral canonical models $\bfSh_{K_E}(T_{E, \tilde \ttS})$ and $\bfSh_{K_{E,p}}(T_{E, \tilde \ttS})$ over $\calO_{\tilde \wp}$ as specified in Subsection \ref{S:integral model weak Shimura datum}.

We also consider the product group $ G_\ttS \times T_{E, \tilde \ttS}$ with the product Deligne homomorphism
\[
\tilde h_{\tilde \ttS} = h_\ttS \times h_{E, \tilde \ttS} \colon \SSS(\R) = \CC^\times \longrightarrow (G_\ttS\times T_{E, \tilde \ttS})(\RR).
\]
This gives rise to the product Shimura varieties:
\begin{align*}
\Sh_{K \times K_E}(G_\ttS \times T_{E,\tilde \ttS}) &= \Sh_K(G_\ttS) \times_{F_{\ttS,\wp}}
\Sh_{K_E}(T_{E, \tilde \ttS});\\
\Sh_{K_p \times K_{E,p}}(G_\ttS \times T_{E,\tilde \ttS}) &= \Sh_{K_p}(G_\ttS) \times_{F_{\ttS,\wp}}
\Sh_{K_{E,p}}(T_{E, \tilde \ttS}) .
\end{align*}

Let $Z = \Res_{F/\QQ}\GG_m$ denote the center of $G_\ttS$.
Put $G''_{\tilde \ttS} = G_\ttS \times_Z T_{E, \tilde \ttS}$ which is the quotient of $G_\ttS \times T_{E, \tilde \ttS}$ by $Z$ embedded anti-diagonally as $z \mapsto (z, z^{-1})$.
The corresponding Deligne homomorphism $h''_{\tilde \ttS}: \SSS(\RR) \to G''_{\tilde \ttS}(\R)$ is the one induced by $\tilde h_{\tilde \ttS}$.
We will consider open compact subgroups $K'' \subseteq G''_{\tilde \ttS}(\AAA^\infty)$ of the form $K''^pK''_p$, where $K''^p$ is an open compact subgroup of $G''_{\tilde \ttS}(\AAA^{\infty,p})$ and $K''_p $ is an open compact subgroup of $G''_{\tilde \ttS}(\Qp)$.
Finally, the $G''_{\tilde \ttS}(\RR)$-conjugacy class of $h''_{\tilde \ttS}$ can be canonically identified with $\gothH_\ttS$.
We then get the Shimura variety $\Sh_{K''}(G''_{\tilde \ttS})$ and its limit $\Sh_{K''_p}(G''_{\tilde \ttS})$ over the reflex field $E_{\tilde \ttS}$.  The set of $\CC$-points of $\Sh_{K''}(G''_{\tilde \ttS})$ is
\[
\Sh_{K''}(G''_{\tilde \ttS})(\CC) = G''_{\tilde \ttS}(\QQ)\backslash (\gothH_\ttS \times G''_{\tilde \ttS}(\AAA^\infty) ) / K''.
\]

\subsection{Unitary Shimura varieties}
\label{S:unitary-shimura}
We now introduce the unitary group. Consider the morphism
\[
\xymatrix@R=0pt@C=10pt{
\nu = \Nm_{B/F} \times \Nm_{E/F}: &
G''_{\tilde \ttS}=G_\ttS \times_Z T_E \ar[rr]&&
T = \Res_{F/\QQ} \GG_m\\
& (g, z) \ar@{|->}[rr]&&
\Nm(g) z\bar z.
}
\]
Viewing $\GG_m$ naturally as a subgroup of $T = \Res_{F/\QQ}\GG_m$,
we define $G'_{\tilde \ttS}$ to be the reductive group $\nu^{-1}(\GG_m)$. This will be our auxiliary unitary group, whose associated Shimura variety will provide $\Sh_K(G_{ \ttS})$ an integral canonical model.
We will occasionally use the algebraic group $G'_{{\tilde \ttS}, 1} = \Ker(\nu)$, but we view it as a reductive group over $F$.

Note that the Deligne homomorphism $h''_{\tilde \ttS} : \SSS(\RR) \to G''_{\tilde \ttS}(\R)$ factors through a homomorphism $h'_{\tilde \ttS}: \SSS(\R) \to G'_{\tilde \ttS}(\R)$.
The $G'_{\tilde \ttS}(\RR)$-conjugacy class of $h'_{\tilde \ttS}$ can be canonically identified with $\gothH_\ttS$.

We will consider open compact subgroups of $G'_{\tilde \ttS}(\AAA^\infty)$ of the form $K' = K'_pK'^p$,
where $K'_p$ is an open compact subgroup of $G'_{\tilde \ttS}(\QQ_p)$ (to be specified later in Subsection~\ref{S:level-structure}) and $K'^p$ is an open compact subgroup of $G'_{\tilde \ttS}(\AAA^{\infty, p})$.
We will always take $K'^p$ to be sufficiently small so that $K'$ is \emph{neat} and hence the moduli problem we encounter later would be representable by a fine moduli space.
Given the data above,
we have a Shimura variety $\Sh_{K'}(G'_{\tilde \ttS})$ whose $\CC$-points are given by
\[
\Sh_{K'}(G'_{\tilde \ttS})(\CC) = G'_{\tilde \ttS}(\QQ) \backslash (\gothH_\ttS \times  G'_{\tilde \ttS}(\AAA^\infty) )/ K'.
\]
The Shimura variety $\Sh_{K'}(G'_{\tilde \ttS})$ is defined over the reflex field $E_{\tilde \ttS}$.
We put  $\Sh_{K'_p}(G'_{\tilde \ttS}) = \varprojlim_{K'^p} \Sh_{K'_pK'^p}(G'_{\tilde \ttS})$.

The upshot is the following lemma, which verifies the conditions listed in Corollary~\ref{C:Sh(G)^circ_Zp independent of G}.  This allows us to bring the integral canonical models of the unitary Shimura varieties to that of the quaternionic Shimura varieties.

\begin{lemma}
\label{L:compatibility of derived group and adjoint group}
The natural diagram of  morphisms of groups
\begin{equation}\label{E:morphism-of-groups}
G_{\ttS}\leftarrow G_{\ttS}\times T_{E, \tilde \ttS}\ra G''_{\tilde \ttS} =  G_{\ttS}\times_{Z}T_{E, \tilde \ttS}\leftarrow G'_{\tilde \ttS}
\end{equation}
\begin{itemize}
\item[(1)] are compatible with the Deligne homomorphisms, and
\item[(2)] induce isomorphisms on their associated derived and adjoint groups.
\end{itemize}
\end{lemma}
\begin{proof}
Straightforward.
\end{proof}

\subsection{PEL Shimura data}\label{S:PEL-Shimura-data}
We put $D_\ttS = B_\ttS \otimes_F E$. It is isomorphic to $\rmM_2(E)$ under Hypothesis~\ref{H:B_S-splits-at-p}. This is a quaternion algebra over $E$ equipped with an involution $l \to \bar l$ given by the tensor product of the natural involution on $B_\ttS$ and the complex conjugation on $E$.
Let $D_\ttS^\sym$ denote the subsets of \emph{symmetric} elements, i.e. those elements $\delta \in D_\ttS$ such that $\delta = \bar \delta$.
For any element $\delta\in (D_{\ttS}^\sym)^{\times}$, we  can define a new involution on $D_\ttS$ given by $l \mapsto l^* = \delta^{-1}\bar l \delta$. In the following Lemma~\ref{L:property-PEL-data}, we will specify a convenient choice of such a $\delta$.

Let $V$ be the underlying $\QQ$-vector space of $D_\ttS$,
with the natural left $D_\ttS$-module structure.
Define  a pairing $\psi_{E}: V\times V\ra E$ on $V$  by
\begin{equation}\label{Equ:pairing-E}
\psi_E(v, w) = \Tr_{D_\ttS/E}(\sqrt{\gothd} \cdot v \delta w^*), \quad \quad
v, w \in V.
\end{equation}
It is easy to check that $\psi_E$ is skew-hermitian over $E$ for $*$, i.e. $\overline{\psi_{E}(v,w)}=-\psi_E(w,v)$ and $\psi_E(lv, w) = \psi_E(v, l^*w)$ for $l \in D_\ttS$ and $v, w \in V$. We define the bilinear form
$$\psi=\Tr_{E/\Q}\circ \psi_{E}\colon V\times V\lra \Q,$$
which is skew-symmetric and hermitian for $*$.
One checks easily that the subgroup consisting of elements $l \in D^\times_\ttS$ satisfying $\psi(vl, wl) = c(l)\psi(v, w)$ for some $c(l) \in \QQ^\times$ is exactly the subgroup $G'_{\tilde \ttS} \subset D^\times_\ttS$.
\emph{We make the above right action of $G'_{\tilde \ttS}$ on $V$ into a left action by taking the transpose action.} This is different from the convention used in \cite{carayol} which used the inverse action, because taking the transpose action is naturally compatible with the setup of Hilbert modular varieties (see Subsection~\ref{S:comparison-Hilbert}), and is also compatible with our earlier choice of Deligne's homomorphism.

The group $G'_{\tilde \ttS}$ is identified  with the $D_\ttS$-linear unitary of group of $V$ with similitudes in $\QQ^\times$, i.e. for each $\Q$-algebra $R$, we have
\begin{equation}\label{Equ:description-G'}
G'_{\tilde \ttS}(R)=\{g\in \End_{D_{\ttS}\otimes_{\Q}R}(V\otimes_{\Q}R) = D_\ttS^{\mathrm{op}}\;|\; \psi(v\,{}^t\!g, w\,{}^t\!g)=c(g)\psi(v,w)\; \text{with }c(g)\in R^{\times}\}.
\end{equation}

We describe $D_{\ttS,\gothp} = D_\ttS\otimes_{F}F_{\gothp}$ by distinguishing three cases according to the types of $\gothp\in \Sigma_{p}$ in Subsection~\ref{S:level-structure-at-p}:
\begin{itemize}
\item {\it Types $\alpha$ or $\alpha^\sharp$:} In this case, the place $\gothp$ splits into two primes $\gothq$ and $\bar{\gothq}$ in $E$.  We have natural isomorphisms $F_\gothp \cong E_{\gothq} \cong E_{\bar \gothq}$. We fix an isomorphism $B_{\ttS}\otimes_{F}F_{\gothp}\simeq \rmM_2(F_{\gothp})$ as above, then $D_{\ttS,\gothp} \simeq  \rmM_2(E_{\gothq})\oplus \rmM_2(E_{\bar{\gothq}})$. Under these identification, we put $\cO_{B_{\ttS}, \gothp}=\rmM_2(\cO_{\gothp})$ and $\cO_{D_{\ttS, \gothp}}=\rmM_2(\cO_{\gothq})\oplus \rmM_2(\cO_{\bar\gothq})$.

\item {\it Type $\beta$:} In this case, the place $\gothp$ is inert in  $E/F$ and let $\gothq$ denote the  unique place in $E$ above $\gothp$. Using the fixed isomorphism  $B_{\ttS}\otimes_{F}F_{\gothp}\simeq \rmM_2(F_{\gothp})$, we have
$D_{\ttS,\gothp}\simeq \rmM_2(E_{\gothq})$. We put $\cO_{B_{\ttS}, \gothp}=\rmM_2(\cO_{\gothp})$ and $\cO_{D_{\ttS}, \gothp}=\rmM_2(\cO_{\gothq})$.

\item {\it Type $\beta^\sharp$:} Let $\gothq$ be the unique place in $E$ above $\gothp$. The division quaternion algebra $B_{F_{\gothp}}=B_{\ttS}\otimes_{F}F_{\gothp}$ over $F_{\gothp}$ is generated by an element $\varpi_{B_{F_\gothp}}$ over $E_{\gothq}$, with the relations $\varpi_{B_{F_\gothp}}^2 = p$ and $\varpi_{B_{F_\gothp}} a = \bar a \varpi_{B_{F_\gothp}}$ for $a \in E_\gothq$. We identify $B_{F_\gothp} \otimes_{F_\gothp} E_\gothq$ with $ \rmM_2(E_\gothq)$ via the map
\begin{equation}
\label{E:involution-on-quaternion-embedding}
(a+b\varpi_{B_{F_\gothp}}) \otimes c \longmapsto
\big( \begin{smallmatrix}
ac & bc\\ p\bar b c & \bar a c
\end{smallmatrix} \big).
\end{equation}
This also identifies $D_{\ttS ,\gothp}$ with $\rmM_2(E_\gothq)$. We put $\cO_{B_{\ttS}, \gothp}=\cO_{B_{F_\gothp}}$, and take $\cO_{D_{\ttS}, \gothp}$ to be the preimage of $\rmM_2(\cO_\gothq)$ in $D_{\ttS}\otimes_F F_{\gothp}$.

\end{itemize}
We put $\calO_{D_\ttS, p} = \prod_{\gothp \in \Sigma_p} \calO_{D_\ttS, \gothp}$.

\begin{lemma}
\label{L:property-PEL-data}
\begin{itemize}
\item[(1)] We can choose the symmetric element $\delta \in (D_\ttS^\sym)^\times$ above such that
\begin{itemize}
\item[(a)] $\delta \in \calO_{D_\ttS, \gothp}^\times$ for each $\gothp \in \Sigma_p$ not of type $\beta^\sharp$,
and $\delta \in
\big(
\begin{smallmatrix} p^{-1} &0\\0&1
\end{smallmatrix}
\big) \calO_{D_\ttS, \gothp}^\times$ for each $\gothp \in \Sigma_p$ of type $\beta^\sharp$, and
\item[(b)]
the following (symmetric) bilinear form on $V_\RR$ is positive definite.
\[
(v, w) \mapsto \psi\big(v, w\cdot h'_{\tilde \ttS}(\ii)\big).
\]
\end{itemize}

\item[(2)] Through the homomorphism $h'_{\tilde \ttS}: \SSS(\R) \to G'_{\tilde \ttS}(\R)$, $h'_{\tilde \ttS}(\ii)$ acts on the vector space $V_\RR$ and gives it a Hodge structure of type $\{(-1, 0), (0, -1)\}$.
For $l \in D_\ttS$, we have
\[
\tr(l; V_\CC / F^0V_\CC) =  \big( \sum_{\tilde \tau \in \Sigma_{E,\infty}}s_{\tilde \tau} \tilde \tau \big) (\Tr_{D_\ttS/E}(l)).
\]
The reflex field $E_{\tilde \ttS}$ is the subfield of $\CC$ generated by these traces for all $l \in D_\ttS$.

\item[(3)] With the choice of $\delta$ in (1), the group $G'_{\tilde \ttS,1}$ is unramified at $\gothp \in \Sigma_p$ not of type $\beta^\sharp$,  and it  is non-quasi-split at $\gothp \in \Sigma_p$ of type $\beta^\sharp$.  Moreover, $\calO_{D_\ttS, p}$ is a maximal $*$-invariant lattice of $D_\ttS(\QQ_p)$ (up to scaling).
\end{itemize}
\end{lemma}
\begin{proof}
(1) Since $F$ is dense in $F \otimes_\QQ \QQ_p \oplus F \otimes_\QQ \RR$, the symmetric elements of $V$ is dense in the symmetric elements of $V \otimes_\QQ \Qp \oplus V \otimes_\QQ \RR$.  The conditions at places above $p$ is clearly open and non-empty; so are the conditions at archimedean places, which follows from the same arguments in \cite[2.2.4]{carayol}.

(2) This follows from the same calculation as in \cite[2.3.2]{carayol}.

(3) 
We first remark that $G'_{\tilde\ttS, 1, F_\gothp}$ does not depend on the particular choice of $\delta$, and hence we may use a convenient $\delta$ to ease the computation.  We discuss each of the types separately.

If $\gothp$ is of type $\alpha$ or $\alpha^\sharp$, $G'_{\tilde\ttS,1}(F_\gothp)$ is isomorphic to the kernel of $\GL_2(F_\gothp) \times_{F_\gothp^\times}(E_{\gothq}^\times \times E_{\bar \gothq}^\times) \to F_\gothp^\times$ given by $(l, x, y) \mapsto \Nm(l)xy$.
Hence $l \mapsto (l, \Nm(l)^{-1}, 1)$ induces an isomorphism $\GL_2(F_\gothp) \to G'_{\ttS,1}(F_\gothp)$.  They are of course unramified.

If $\gothp$ is of type $\beta$, when we identify $D_{\ttS, \gothp}$ with $\rmM_2(E_\gothq)$, the convolution $l \mapsto \bar l$ is given by $\big(\begin{smallmatrix}
a& b\\ c& d
\end{smallmatrix} \big) \mapsto \big(\begin{smallmatrix}
\bar d& -\bar b\\ -\bar c& \bar a
\end{smallmatrix} \big)$ for $a, b, c, d \in E_\gothq$.
We take the element $\delta$ to be $\big(
\begin{smallmatrix}
0&1\\ 1&0
\end{smallmatrix}
\big)$.
The Hermitian form on $\rmM_2(E_\gothq)$ is then given by
\[
\langle v,w
\rangle
= \tr_{\rmM_2(E_\gothq) / E_\gothq}(v\bar w\delta) = -a\bar b' + b \bar a' +c \bar d' - d \bar c',
\quad v=
\big(\begin{smallmatrix}
a& b\\ c& d
\end{smallmatrix} \big)\textrm{ and }
w=
\big(\begin{smallmatrix}
a'& b'\\ c'& d'
\end{smallmatrix} \big).
\]
One checks easily that $\gothe = \big(\begin{smallmatrix}
1 &0 \\0&0
\end{smallmatrix}
\big)$ is invariant under the $*$-involution.  So $D_{\ttS, \gothp}$ is isomorphic to $(\gothe D_{\ttS,\gothp})^{\oplus 2}$ as a $*$-Hermitian space and $G'_{\tilde \ttS,1,F_\gothp}$ is the unitary group for $\gothe D_{\ttS, \gothp}$.  It is clear from the expression above that $\gothe D_{\ttS, \gothp}$ is a hyperbolic plane (\cite[Example~3.2]{minguez}).  Hence $G'_{\tilde \ttS,1, F_\gothp}$ being the unitary group of such Hermitian space is unramified.

If $\gothp$ is of type $\beta^\sharp$, the identification of $D_{\ttS, \gothp}$ with $\rmM_2(E_\gothq)$ using \eqref{E:involution-on-quaternion-embedding} implies that the convolution $l \mapsto \bar l$ is given by
\[
\big(\begin{smallmatrix}
a& b\\ c& d
\end{smallmatrix} \big) \mapsto \big(\begin{smallmatrix}
\bar a& -\bar c/p\\ -p\bar b& \bar d
\end{smallmatrix} \big) \quad \textrm{ for }a, b, c, d \in E_\gothq.
\]
We take the element $\delta$ to be $\big(
\begin{smallmatrix} p^{-1} &0\\0&1
\end{smallmatrix}
\big)$.
The Hermitian form  on $\rmM_2(E_\gothq)$ is then given by
\begin{equation}
\label{E:Hermitian-type-gamma}
\langle v,w\rangle
= \Tr_{\rmM_2(E_\gothq) / E_\gothq}( v\bar w\delta) = a\bar a'/p -b \bar b' - c\bar c'/p + d \bar d',
\quad v=
\big(\begin{smallmatrix}
a& b\\ c& d
\end{smallmatrix} \big)\textrm{ and }
w=
\big(\begin{smallmatrix}
a'& b'\\ c'& d'
\end{smallmatrix} \big).
\end{equation}
Similar to above, $\gothe = \big(
\begin{smallmatrix}
1&0\\0&0
\end{smallmatrix}
\big)$ is invariant under $*$-involution, and $D_{\ttS, \gothp}$ is isomorphic to $(\gothe
D_{\ttS, \gothp})^{\oplus 2}$ as $*$-Hermitian spaces.
The unitary group $G'_{\tilde \ttS,1,F_\gothp}$ is just the usual unitary group of $\gothe
D_{\ttS, \gothp}$.  But the Hermitian form there takes the form of $a\bar a'/p - b\bar b'$, which is a typical example of anisotropic plane (\cite[Example~3.2]{minguez}).
So $G'_{\tilde \ttS,1,F_\gothp}$ is a non-quasi-split unitary group.

To see that $\calO_{D_\ttS,p}$ is a maximal $*$-stable lattice, it suffices to prove it for $\calO_{D_\ttS, \gothp}$ for each $\gothp \in \Sigma_p$.  When $\gothp$ is of type $\alpha, \alpha^\sharp$, or $\beta$, this is immediate. When $\gothp$ is of type $\beta^\sharp$, we write $\delta$ as $\big(
\begin{smallmatrix} p^{-1} &0\\0&1
\end{smallmatrix}
\big) u$ for $u \in  \calO_{D_\ttS, \gothp}^\times$.  The involution $*$ is given by
\[
\big(\begin{smallmatrix}
a& b\\ c& d
\end{smallmatrix} \big) \mapsto
u^{-1}\big(
\begin{smallmatrix} p &0\\0&1
\end{smallmatrix}
\big) \big(\begin{smallmatrix}
\bar a& -\bar c/p\\ -p\bar b& \bar d
\end{smallmatrix} \big)
\big(
\begin{smallmatrix} p^{-1} &0\\0&1
\end{smallmatrix}
\big) u  = u^{-1} \big(\begin{smallmatrix}
\bar a& -\bar c\\ -\bar b& \bar d
\end{smallmatrix} \big) u
 \quad \textrm{ for }a, b, c, d \in E_\gothq.
\]
It is then clear that $\calO_{D_\ttS,\gothp}$ is a maximal $*$-stable lattice.
\end{proof}

\subsection{Level structures at $p$ in the unitary case}
\label{S:level-structure}

We specify our choice for $K'_p$ corresponding to the level structure $K_p=\prod_{\gothp|p}K_{\gothp}\subset \prod_{\gothp|p} (B_{\ttS}\otimes_{F}F_{\gothp})^{\times}$ considered in Subsection \ref{S:level-structure-at-p}.

By \eqref{Equ:description-G'}, giving an element $g_p\in G_{\tilde \ttS}'(\Q_p)$ is equivalent to giving tuples $(g_{\gothp})_{\gothp\in \Sigma_{p}}$ with $g_{\gothp}\in \End_{D_{\ttS}\otimes_{F}F_{\gothp}}(V\otimes_{F}F_{\gothp})$ such that there exists  $\nu(g_p)\in \Q_{p}^{\times}$ independent of $\gothp$ satisfying
$$
\psi_{E, \gothp}(g_{\gothp}v, g_{\gothp}w)=\nu(g_p)\psi_{E, \gothp}(v,w), \quad  \forall v,w \in V\otimes_F F_{\gothp},
$$
where $\psi_{E, \gothp}$ is the base change of $\psi_E$ to $V\otimes_F F_{\gothp}= D_{\ttS,\gothp}$. In the following, we will give a chain of lattices $\Lambda_{\gothp}^{(1)}\subseteq \Lambda_{\gothp}^{(2)}$ in $D_{\ttS,\gothp}$ for each $\gothp$, and define $K_p'\subseteq G_{\tilde \ttS}'(\Q_p)$ to be the subgroup consisting of the elements $(g_{\gothp})_{\gothp\in \Sigma_{p}}$ with $g_{\gothp}$ belonging to the stabilizer of $\Lambda^{(1)}_{\gothp}\subseteq \Lambda_{\gothp}^{(2)}$ and with $\nu(g_p)\in \ZZ_{p}^{\times}$ independent of $\gothp$.

\begin{itemize}
\item When $\gothp$ is of type $\alpha$, we take $\Lambda_\gothp^{(1)}=\Lambda_{\gothp}^{(2)}$ to be $\calO_{D_\ttS, \gothp}$.
\item When $\gothp$ is of type $\alpha^\sharp$, we take
\[
\Lambda_\gothp^{(1)} =
\left( \begin{smallmatrix}
\calO_\gothq & \gothq \\
\calO_\gothq & \gothq
\end{smallmatrix} \right)\oplus\left( \begin{smallmatrix}
\calO_{\bar \gothq} & \calO_{\bar \gothq} \\
\calO_{\bar \gothq} & \calO_{\bar \gothq}
\end{smallmatrix}\right)
\quad \textrm{ and } \quad \Lambda_\gothp^{(2)} =\left( \begin{smallmatrix}
\calO_\gothq & \calO_\gothq \\
\calO_\gothq & \calO_\gothq
\end{smallmatrix}\right)
\oplus\left(
\begin{smallmatrix}
\bar \gothq^{-1} &\calO_{\bar \gothq}   \\
\bar \gothq^{-1} & \calO_{\bar \gothq}
\end{smallmatrix}\right).
\]

\item When  $\gothp$ is of type $\beta$, we take $\Lambda_{\gothp}^{(1)} =\Lambda_{\gothp}^{(2)} =\calO_{D_{\ttS},\gothp}$.

\item When $\gothp$ is of type $\beta^\sharp$, we take
\[
\Lambda_\gothp^{(1)} = \left(
\begin{smallmatrix}
\gothq & \calO_\gothq\\
\gothq & \calO_\gothq
\end{smallmatrix}
\right)
\subseteq
\Lambda_\gothp^{(2)} = \left(
\begin{smallmatrix}
\calO_\gothq & \calO_\gothq\\
\calO_\gothq & \calO_\gothq
\end{smallmatrix}
\right).
\]
Note that, these two lattices are dual to each other under the Hermitian form \eqref{E:Hermitian-type-gamma}.
\end{itemize}
Similarly, we give the level structure at $p$ for the Shimura variety associated to the group $G''_{\tilde \ttS}$: take $K''_{ p}$ to be the image of $K_p \times K_{E,p} $ under the natural map $(G_\ttS \times T_{E,\tilde \ttS})(\Qp) \to G''_{\tilde \ttS}(\Qp)$.

\begin{lemma}
\label{L:compatibility of derived group and adjoint group2}
The Shimura data for $G_\ttS, G_\ttS \times T_{E,\tilde \ttS}, G''_{\tilde \ttS},$ and $G'_{\tilde \ttS}$ satisfy Hypothesis~\ref{H:hypo on G}.  Moreover,
The natural diagram of  morphisms of groups
\begin{equation}\label{E:morphism-of-groups2}
G_{\ttS}\leftarrow G_{\ttS}\times T_{E, \tilde \ttS}\ra G''_{\tilde \ttS} = G_{\ttS}\times_{Z}T_{E,\tilde \ttS}\leftarrow G'_{\tilde \ttS}
\end{equation}
induce isomorphisms on the $p$-integral points of the derived and adjoint groups.
\end{lemma}
\begin{proof}
This is  straightforward from definition.
In fact,  both $K_p^\ad = \prod_{\gothp \in \Sigma_p}K_\gothp^\ad$ and $K_p^\der = \prod_{\gothp \in \Sigma_p} K_\gothp^\der$ are products and we give the description case by case:
\begin{itemize}
\item if $\gothp$ is of type $\alpha$ or $\beta$, then $K^\der_\gothp  = \SL_2(\calO_\gothp)$ and $K^\ad_\gothp = \PGL_2( \calO_\gothp)$;
\item if $\gothp$ is of type $\alpha^\sharp$, then $K^\der_\gothp  = \SL_2( \calO_\gothp) \cap \big(
\begin{smallmatrix}
\calO_\gothp^\times & \calO_\gothp\\ \gothp & \calO_\gothp^\times
\end{smallmatrix}
\big)$ and $K^\ad_\gothp = \big(
\begin{smallmatrix}
\calO_\gothp^\times & \calO_\gothp\\ \gothp & \calO_\gothp^\times
\end{smallmatrix}
\big) / \calO_\gothp^\times$; and
\item if $\gothp$ is of type $\beta^\sharp$, $K^\der_\gothp$ and $K^\ad_\gothp$ are the maximal compact open subgroups of $(B_\ttS^\times)^\der(F_\gothp)$ and $(B_\ttS^\times)^\ad(F_\gothp)$, respectively.\qedhere
\end{itemize}
\end{proof}

\begin{cor}
\label{C:comparison of shimura varieties}
The natural morphisms between Shimura varieties
\begin{equation}
\label{E:morphisms of Shimura varieties}
\Sh_{K_p}(G_\ttS)\longleftarrow\Sh_{K_p \times K_{E, p}}(G_\ttS \times T_{E,\tilde \ttS}) \longrightarrow \Sh_{K''_p}(G''_{\tilde \ttS}) \longleftarrow\Sh_{K'_p}(G'_{\tilde \ttS})
\end{equation}
induce isomorphisms on the geometric connected components.  Moreover, the groups $\calE_{G, \tilde \wp}$ defined in \ref{A:connected integral model} (and made explicit below) are isomorphic for each of the groups; and \eqref{E:morphisms of Shimura varieties} is equivariant for the actions of $\calE_{G, \tilde \wp}$'s on the geometric connected components.
Moreover, if one of the Shimura varieties admits an integral canonical model; so do the others.
\end{cor}
\begin{proof}
This follows from Corollary~\ref{C:Sh(G)^circ_Zp independent of G} for which the conditions are verified in Lemmas~\ref{L:compatibility of derived group and adjoint group} and \ref{L:compatibility of derived group and adjoint group2}.
\end{proof}

\subsection{Structure groups for connected Shimura varieties}
\label{S:structure group}
In order to apply the machinery developed in Section~\ref{Section:Sh Var}, we now make explicit the structure groups $\calG$ in \eqref{E:calG} and $\calE_{G, \tilde \wp}$ in Subsection~\ref{A:connected integral model} in the case of our interest.

We use $\calG_\ttS$ (resp. $\calG'_{\tilde \ttS}$, $\calG''_{\tilde \ttS}$)  to denote the group defined in \eqref{E:calG} for $G = G_\ttS$ (resp. $ G'_{\tilde \ttS}$, $G''_{\tilde \ttS}$).
Explicitly, since the center of $G_\ttS$ is $\Res_{F/\QQ}\GG_m$, we have $G^\ad_\ttS(\QQ) = B_\ttS^\times / F^\times$.  Taking the positive and $p$-integral part as in Lemma~\ref{L:nu(G(Q)->>T(Q)}, we have $G_\ttS^\ad(\QQ)^{+, (p)} = B_\ttS^{\times, >0, (p)} / \calO_{F, (p)}^\times$ where the superscript $>0$ means to take the elements whose reduced norm is positive for all real embeddings.
It follows that $\calG_\ttS = G_\ttS(\AAA^{\infty, p}) / \calO_{F, (p)}^{\times, \cl}$.  The same argument applies to  $G''_{\tilde \ttS}$ whose center is $\Res_{E/\QQ}\GG_m$, and shows that  
$\calG''_{\tilde \ttS} = G''_{\tilde \ttS}(\AAA^{\infty, p}) \big/ \calO_{E,(p)}^{\times, \cl}$.
Determination of $\calG'_{\tilde \ttS}$ is more subtle.
By Lemmas~\ref{L:compatibility of derived group and adjoint group} and \ref{L:compatibility of derived group and adjoint group2}, we have $(G'_{\tilde \ttS})^\ad(\QQ)^{+,(p)} = (G''_{\tilde \ttS})^\ad(\QQ)^{+,(p)}$.
So if we use $Z'_{\tilde \ttS}$ to denote the center of $G'_{\tilde \ttS}$, then we have
\begin{align}
\label{E:structure group description}
\calG'_{\tilde \ttS}
 &= 
\big(G'_{\tilde \ttS}(\AAA^{\infty,p})  \big/ Z'_{\tilde \ttS}(\QQ)^{(p), \cl}\big) \ast_{G'_{\tilde \ttS}(\QQ)^{(p)}_+/Z'_{\tilde \ttS}(\QQ)^{(p)}} (G'_{\tilde \ttS})^\ad(\QQ)^{+,(p)}
\\
\nonumber
&= 
\big(G'_{\tilde \ttS}(\AAA^{\infty,p})  \big/ Z'_{\tilde \ttS}(\QQ)^{(p), \cl}\big) \ast_{G'_{\tilde \ttS}(\QQ)^{(p)}_+/Z'_{\tilde \ttS}(\QQ)^{(p)}} \big( G''_{\tilde \ttS}(\QQ)_+^{(p)} / \calO_{E, (p)}^\times \big)
\\
\nonumber
 &= G'_{\tilde \ttS}(\AAA^{\infty, p}) G''_{\tilde \ttS}(\QQ)_+^{(p)}  \big/ \calO_{E, (p)}^{\times, \cl}.
\end{align}
The subgroup $G'_{\tilde \ttS}(\AAA^{\infty, p}) G''_{\tilde \ttS}(\QQ)_+^{(p)}$ can be characterized by the following commutative diagram of exact sequence as the pullback of the right square.
\begin{equation}
\label{E:description of G'' G'}
\xymatrix{
1 \ar[r] & \ar@{=}[d] G'_{\tilde
 \ttS,1}(\AAA^{\infty,p}) \ar[r] &G''_{\tilde \ttS}(\QQ)_+^{(p)} G'_{\tilde \ttS}(\AAA^{\infty,p}) \ar[r] \ar@{^{(}->}[d] &
\calO_{F, (p)}^\times (\AAA^{\infty,p})^\times\ar[r] \ar@{^{(}->}[d] & 1
\\1 \ar[r]&
G'_{{\tilde \ttS},1}(\AAA^{\infty,p}) \ar[r] & G''_{\tilde \ttS}(\AAA^{\infty,p}) \ar[r] & 
(\AAA_F^{\infty,p})^\times \ar[r] &1.
}
\end{equation}

We use $\calE_{G, \ttS, \wp}$ to denote
the group $\calE_{G, \tilde \wp}$ defined in Subsection~\ref{A:connected integral model}.  
As an abstract group, it is isomorphic for all groups $G_\ttS$, $G'_{\tilde \ttS}$, and $G''_{\tilde \ttS}$.
But we point out that it is important (see Remark~\ref{R:quaternionic Shimura reciprocity not compatible}) to know how they sit as subgroups  of $\calG_\ttS \times \Gal_{k_\wp}$, $\calG'_{\tilde \ttS} \times \Gal_{k_{\tilde \wp}}$ and $\calG''_{\tilde \ttS} \times \Gal_{k_{\tilde \wp}}$, respectively, according to the Shimura reciprocity map.

\subsection{Integral models of unitary Shimura varieties}\label{S:integral-unitary}

We choose a finite extension $k_0$ of $k_{\tilde \wp}$ that contains all residual fields $k_\gothq$ for any $p$-adic place $\gothq$ of $E$.  Then the ring of Witt vectors $W(k_0)$ may be viewed as a subring of $\overline \QQ_p$, containing $\calO_{\tilde \wp}$ as a subring.

We fix an order $\calO_{D_\ttS}$ of $D_\ttS$ stable under the involution $l \mapsto l^*$ such that  $\calO_{D_\ttS} \otimes_{\calO_F} \calO_{F,p} \simeq \calO_{D_\ttS, p}$.  Recall that $V$ is the abstract $\QQ$-vector space $D_\ttS$. We choose and fix an $\calO_{D_\ttS}$-lattice $\Lambda$ of $V$ such that,
\begin{itemize}
\item
for each $\gothp \in \Sigma_p$, we have $\Lambda \otimes_{\calO_F} \calO_\gothp \cong \Lambda_\gothp^{(1)}$, and
\item if we put $\widehat{\Lambda}^{(p)} : = \Lambda \otimes_\ZZ \widehat \ZZ^{(p)}$ as a lattice of $V \otimes_\QQ \AAA^{\infty, p}$, we have
\begin{equation}
\label{E:Lambda-dual}
\widehat \Lambda^{(p)} \subseteq \widehat \Lambda^{(p),\vee}\textrm{ under the bilinear form } \psi, \textrm{ or equivalently, } \psi(\widehat \Lambda^{(p)}, \widehat \Lambda^{(p)}) \subseteq \widehat{\ZZ}^{(p)}.
\end{equation}
\end{itemize}
We call such $\Lambda$ \emph{admissible}.

\begin{theorem}
\label{T:unitary-shimura-variety-representability}
Let $K'_p$ be the open compact subgroup of $G_{\tilde \ttS}'(\Q_p)$ considered in Subsection \ref{S:level-structure}, and $K'^p\subset G_{\tilde \ttS}'(\AAA^{\infty, p})$ be  sufficiently small so that $K'=K'^pK'_p$  is neat.
Then there exists a unique \emph{smooth} quasi-projective scheme $\bfSh_{K'}(G'_{\tilde \ttS})$ over $W(k_0)$
representing the functor that sends a locally noetherian $W(k_0)$-scheme $S$ to the set of isomorphism classes of tuples $(A, \iota, \lambda,  \alpha_{K'})$, as described as follows.
\begin{itemize}
\item[(a)] $A$ is an abelian scheme over $S$ of dimension $4g$
equipped with an embedding $\iota: \calO_{D_\ttS} \to \End_S(A)$ such that the characteristic polynomial of the endomorphism $\iota(b)$ on $\Lie(A/S)$ for $b \in \calO_E$  is given by
\[
\prod_{\tilde \tau \in \Sigma_{E,\infty}} \big(x - \tilde \tau(b)\big) ^{2s_{\tilde \tau}}.
\]

\item[(b)] $\lambda:A \to A^\vee$ is a polarization of $A$, such that
\begin{itemize}
\item[(b1)] the Rosati involution associated to $\lambda$ induces the involution $l \mapsto l^*$ on $\calO_{D_\ttS}$,

\item[(b2)] $(\Ker \lambda)[p^\infty] $ is a  finite flat closed subgroup scheme contained in $\prod_{\gothp \textrm{ of type } \beta^\sharp}A[\gothp]$ such that $(\Ker\lambda)[p^\infty] \cap A[\gothp]$ for each $\gothp$ of type $\beta^\sharp$ has rank $ (\#k_{\gothp})^{4}$ and

\item[(b3)]  the cokernel of $\lambda_*: H_1^\dR(A / S) \to H_1^\dR(A^\vee/S)$ is a locally free module of rank two over
\[
\bigoplus_{\gothp \textrm{ of type  }\beta^\sharp} \calO_S \otimes_{\ZZ_p} (\calO_E \otimes_{\calO_F} k_\gothp).
\]
\end{itemize}

\item[(c)] $ \alpha_{K'}$ is a pair  $( \alpha^p_{K'^p}, \alpha_p)$ defined as follows:

\begin{itemize}
\item[(c1)] For each connected component $S_i$ of $S$, we choose a geometric point $\bar{s}_i$, and let $T^{(p)}(A_{\bar{s}_i})$ be the product of $l$-adic Tate modules of $A$ at $\bar{s}_i$ for all $l\neq p$. Then $\alpha^p_{K'^p}$ is a collection of $\pi_1(S_i, \bar{s}_i)$-invariant $K'^p$-orbit of pairs $(\alpha^p_i, \nu(\alpha^p_i))$, where $\alpha^p_i$ is an $\calO_{D_\ttS}\otimes_\ZZ \widehat{\ZZ}^{(p)}$-linear isomorphism $\widehat \Lambda^{(p)} \xrightarrow{\sim} T^{(p)}( A_{\bar s_i})$ and $\nu(\alpha^p_i)$ is an isomorphism $\widehat{\Z}^{(p)}\xra{\sim} \widehat \Z^{(p)}(1)$ such that the following diagram commute:
\[
\xymatrix{
\widehat \Lambda^{(p)}\times \widehat \Lambda^{(p)}\ar[rr]^-{\psi}\ar[d]_{\alpha^p_i\times \alpha_i^p} &&\widehat{\Z}^{(p)}\ar[d]^{\nu(\alpha_i^p)}\\
T^{(p)}(A_{\bar s_i})\times T^{(p)}(A_{\bar s_i}) \ar[rr]^-{\lambda-\mathrm{Weil}} && \widehat{\Z}^{(p)}(1).
}
\]

\item[(c2)] For each prime $\gothp\in \Sigma_p$ of type $\alpha^\sharp$, let $\gothq$ and $\bar\gothq$ be the two primes of $E$ above $\gothp$. Then $\alpha_p$ is a collection of $\calO_{D_\ttS}$-stable closed finite flat subgroups $\alpha_{\gothp}=H_{\gothq}\oplus H_{\bar{\gothq}} \subset A[\gothq]\oplus  A[\bar\gothq]$ of order $(\#k_\gothp)^4$ such that $H_\gothq$ and $H_{\bar\gothq}$ are dual to each other  under the perfect pairing
\[A[\gothq]\times A[\bar\gothq]\ra \mu_p\]\
induced by the polarization $\lambda$.

\end{itemize}
\end{itemize}
By Galois descent, the moduli space $\bfSh_{K'}(G'_{\tilde \ttS})$ can be defined over $\calO_{ \tilde \wp}$.
Moreover, if the ramification set $\ttS_{\infty}$ is non-empty, $\bfSh_{K'}(G'_{\tilde \ttS})$ is projective.
\end{theorem}
We will postpone the proof of this theorem after Notation~\ref{N:notation-reduced}. 
The intuition behind the proof is the following. It is well known that the corresponding moduli problem of hyperspecial level is representable by a quasi-projective smooth scheme over $W(k_0)$. 
In our situation, the  hyperspecial level occurs merely at primes $\gothp\in \Sigma_p$ of type $\alpha^{\#}$ or $\beta^{\#}$, but the condition $\Sigma_{\infty/\gothp}=\ttS_{\infty/\gothp}$ for such primes $\gothp$ implies that the extra levels at those primes are representable by finite \'etale maps over the hyperspecial moduli. Hence, the resulting moduli problem is still smooth.

\subsection{Deformation theory}\label{S:deformation}
 We recall briefly the crystalline deformation theory of abelian varieties due to Serre-Tate and Grothendieck-Messing. This will be used in the proof of Theorem~\ref{T:unitary-shimura-variety-representability}.

We start with a general situation. Let $S$ be a $\Z_p$-scheme on which $p$ is locally nilpotent, and $S_0\hra S$ be a closed immersion whose ideal sheaf $\calI$ is equipped with a divided power structure compatible with that on $p\Z_p$, e.g. $S_0= \Spec k\hra S= \Spec k[\epsilon]/ (\epsilon^2)$ with $k$ a perfect field of characteristic $p$. Let $(S_0/\Z_p)_{\cris}$ be the crystalline site of $S_0$ over $\Spec \Z_p$, and $\cO^\cris_{S_0/\Z_p}$ be the structure sheaf.
Let $A_0$ be  an abelian scheme over $S_0$, and $H^{\cris}_1(A_0/S_0)$ be the \emph{dual} of the relative crystalline cohomology $H^1_{\cris}(A_0/S_0)$ (or isomorphically $H^{\cris}_1(A_0/S_0)=H^1_{\cris}(A_0^\vee/S_0)$). Then $H^{\cris}_1(A_0/S_0)$ is a crystal of locally  free $\cO^\cris_{S_0/\Z_p}$-modules whose evaluation $H^{\cris}_1(A_0/S_0)_{S}$ at the pd-embedding  $S_0\hra S$ is a locally free $\cO_S$-module. We have a canonical isomorphism $H^{\cris}_1(A_0/S_0)_{S}\otimes_{\cO_S}\cO_{S_0}\simeq H^{\dR}_1(A_0/S_0)$, which is the \emph{dual} of the relative de Rham cohomology of $A_0/S_0$. For each abelian scheme $A$  over $S$ with $A\times_S S_0\simeq A_0$, we have a canonical Hodge filtration
\[
0\ra \omega_{A^\vee/S}\ra H^{\cris}_1(A_0/S_0)_{S}\ra \Lie(A/S)\ra 0.
\]
Hence, $\omega_{A^\vee/S}$ gives rise to a local direct factor of  $H^{\cris}_1(A_0/S_0)_{S}$ that lifts the subbundle $\omega_{A_0^\vee/S_0}\subseteq H_1^{\dR}(A_0/S_0)$. Conversely,
the  theory of deformations of abelian schemes says that  knowing this lift of subbundle is also enough to recover $A$ from $A_0$.
More precisely, let $\ttAV_{S}$ be the category of abelian schemes over $S$, and let $\ttAV^+_{S_0}$ denote the category of pairs $(A_0, \omega)$, where $A_0$ is an abelian scheme over $S_0$ and $\omega$ is a subbundle of $H_1^\cris( A_0/S_0)_{S}$ that lifts $\omega_{A_0^\vee/S_0} \subseteq H_1^\dR(A_0/S_0)$.
The main theorem of the crystalline deformation theory (cf. \cite[pp. 116--118]{grothendieck},   \cite[Chap. II \S 1]{mazur-messing}) says that \emph{the natural functor $\ttAV_{S} \to \ttAV_{S_0}^+$ given by $A\mapsto (A\times_{S}S_0, \omega_{A^\vee/S})$ is an equivalence of categories.}

Let $A$ be a deformation of $A_0$ corresponding to a direct factor $\omega\subseteq H_1^{\cris}(A_0/S_0)_S$  that lifts $\omega_{A_0^\vee/S_0}$. If $A_0$ is equipped with an action $\iota_0$ by a certain algebra $R$, then $\iota_0$ deforms to an action $\iota$ of $R$ on $A$ if and only if $\omega_{S}\subseteq H_1^{\cris}(A_0/S_0)_S$ is $R$-stable. Let $\lambda_0:A_0\ra A_0^\vee$ be a polarization. Then $\lambda_0$ induces a natural alternating pairing \cite[5.1]{bbm}
\[
\langle\ ,\ \rangle_{\lambda_0}\colon H^{\cris}_1(A_0/S_0)_S\times H^{\cris}_1(A_0/S_0)_S\ra \cO_S,
\]
which is perfect if $\lambda_0$ is prime-to-$p$. Then there exists a (necessarily unique) polarization $\lambda: A\ra A^\vee$ that lifts $\lambda_0$ if and only if $\omega_S$ is isotropic for $\langle\ ,\ \rangle_{\lambda_0}$ by \cite[2.1.6.9, 2.2.2.2, 2.2.2.6]{lan}.

\begin{notation}
\label{N:notation-reduced}
Before going to the proof of Theorem~\ref{T:unitary-shimura-variety-representability}, we introduce some notation. Recall that we have an isomorphism $\cO_{D_{\ttS},p}\simeq \rmM_2(\cO_{E}\otimes \Z_p)$. We denote by $\gothe\in \cO_{D_{\ttS},p}$ the element corresponding to
 $\bigl(
\begin{smallmatrix}1 &0\\0&0\end{smallmatrix}
\bigr)$
 in $\rmM_2(\cO_{E}\otimes \Z_p)$.
 For $S$  a $W(k_0)$-scheme and $M$  an  $\cO_S$-module locally free of finite rank equipped with an action of $\cO_{D_{\ttS},p}$, we call $M^{\circ}:=\gothe M$ \emph{the reduced part} of $M$. We have $M=(M^{\circ})^{\oplus 2}$ by Morita equivalence. Moreover, the $\calO_E$-action induces a canonical decomposition
 \[
M^{\circ}=\bigoplus_{\tilde \tau\in \Sigma_{E, \infty}} M^{\circ}_{\tilde\tau},
\]
where $\calO_E$ acts on each factor $M^\circ_{\tilde \tau}$ by $\tilde \tau: \calO_E \to W(k_0)$.

Let $A$ be an abelian scheme over $S$ carrying an action of $\cO_{D_{\ttS}}$.
The construction above gives rise to locally free  $\calO_S$-modules $\omega^\circ_{A/S}$, $\Lie(A/S)^\circ$,  and $ H_1^\dR(A/S)^\circ $, which are  of rank $\frac12\dim A$, $\frac12\dim A$, and $\dim A$, respectively. We call them the \emph{reduced invariant differential $1$-forms}, the \emph{reduced Lie algebra},  and the \emph{reduced de Rham homology} of $A$ respectively.
For each $\tilde \tau \in \Sigma_{E, \infty}$, we have a \emph{reduced Hodge filtration} in $\tilde \tau$-component
\begin{equation}\label{Equ:reduced-Hodge}
0\ra \omega_{A^\vee/S, \tilde \tau}^{\circ}\ra  H_1^\dR(A/S)^\circ_{\tilde \tau}\ra \Lie(A/S)^{\circ}_{\tilde \tau}\ra 0.
\end{equation}
If the abelian scheme $A$ comes from the moduli problem in Theorem~\ref{T:unitary-shimura-variety-representability}, the dimensions of these three factors are $2-s_{\tilde \tau}$, $2$,  and $s_{\tilde \tau}$, respectively, where the number $s_{\tilde \tau}$ is defined in Subsection~\ref{S:CM extension}.
\end{notation}

\begin{proof}[Proof of Theorem~\ref{T:unitary-shimura-variety-representability}]
The representability of $\bfSh_{K'}(G'_{\tilde \ttS})$ by a quasi-projective scheme over $W(k_0)$ is well known (cf. for instance  \cite[1.4.13, 2.3.3, 7.2.3.10]{lan}). 
To show the smoothness of $\bfSh_{K'}(G'_{\tilde \ttS})$, it suffices to prove that it  is formally smooth over $W(k_0)$.
 Let $R$ be a noetherian $W(k_0)$-algebra, $I\subset R$ be an ideal with $I^2=0$, and $R_0=R/I$.  We need to show that, every point $x_0=(A_0, \iota_0, \lambda_0, \alpha_{K',0})$ of $\bfSh_{K'}(G'_{\tilde \ttS})$ with values in $R_0$ lifts to an $R$-valued point $x$ of $\bfSh_{K'}(G'_{\tilde \ttS})$.
  We apply the deformation theory recalled in  \ref{S:deformation}. The relative crystalline homology $H_1^{\cris}(A_0/R_0)$ is naturally equipped with an action of $\cO_{D_{\ttS}}\otimes \Z_p$. Let $H_1^{\cris}(A_0/R_0)^{\circ}:=\gothe H_1^{\cris}(A_0/R_0)$ be its reduced part, and $H_1^{\cris}(A_0/R_0)^{\circ}_{R}$ be its evaluation on $R$. This is a  free $R\otimes\cO_E$-module of rank $4[F:\Q]$, and we have  a canonical decomposition
\[
H_1^{\cris}(A_0/R_0)^{\circ}_{R}= \bigoplus_{\tilde \tau\in \Sigma_{E,\infty}} H_1^{\cris}(A_0/R_0)^{\circ}_{R,\tilde \tau}.
\]
 The polarization $\lambda_0$ on $A_0$ induces a pairing
\begin{equation}\label{Equ:pairing-R}
H_1^{\cris}(A_0/R_0)_{R, \tilde \tau}^{\circ}\times H_1^{\cris}(A_0/R_0)_{R,\tilde \tau^c}^{\circ}\lra R,
\end{equation}
which is perfect for $\tilde \tau\in \Sigma_{E, \infty/\gothp}$ with $\gothp$ not of type $\beta^\sharp$.
 By the deformation theory \ref{S:deformation}, giving a deformation of $(A_0, \iota_0)$ to  $R$  is equivalent to giving, for each $\tilde \tau \in \Sigma_{E, \infty}$, a direct summand $\omega_{R, \tilde \tau}^{\circ}\subseteq H_1^{\cris}(A_0/R_0)^{\circ}_{R, \tilde \tau}$  which lifts $\omega_{A_0^\vee/R_0, \tilde \tau}^{\circ}$. Let $\gothp\in \Sigma_{p}$ with $\tilde \tau\in \Sigma_{E, \infty/\gothp}$. We distinguish several cases:

\begin{itemize}
\item If $\tilde \tau$ restricts to $\tau \in \ttS_\infty$, $\Lie(A_0/R_0)^{\circ}_{\tilde \tau}$ has rank $s_{\tilde \tau} \in \{0,2\}$ by the determinant condition (a). By duality or the Hodge filtration \eqref{Equ:reduced-Hodge}, $\omega_{A_0^\vee/R_0, \tilde \tau}^\circ$ has rank $2-s_{\tilde{\tau}}$, i.e. $\omega_{A_0^\vee/R_0, \tilde{\tau}}^\circ=0$ when $s_{\tilde \tau} =2$ and $\omega_{A_0^\vee/R_0, \tilde \tau}^{\circ}\cong H_1^{\dR}(A_0/R_0)^{\circ}_{\tilde \tau}$ when $s_{\tilde \tau} = 0$. Therefore, $\omega_{R, \tilde \tau}^{\circ}=0$ or $\omega_{R, \tilde \tau}^{\circ}=H_1^{\cris}(A_0/R_0)_{R, \tilde \tau}^{\circ}$  is the unique lift in these cases respectively.

\item If $\tilde \tau$ restricts to $\tau\in \Sigma_{\infty}-\ttS_{\infty}$, then $\omega_{A_0^\vee/R_0,\tilde \tau}^{\circ}$ and $\omega_{A_0^\vee/R_0, \tilde \tau^c}^{\circ}$ are both of rank 1 over $R_0$,  and we have $\omega_{A_0^\vee/R_0, \tilde \tau}^\circ=(\omega_{A_0^\vee/R_0, \tilde \tau^c}^\circ)^{\perp}$ under the perfect pairing between $H_1^{\dR}(A_0/R_0)^{\circ}_{\tilde \tau}$ and $H_1^{\dR}(A_0/R_0)_{\tilde \tau^c}^\circ$ induced by $\lambda_0$.  (Note that $\tau \in \Sigma_{\infty/\gothp}- \ttS_\infty$ means that $\gothp$ is not of type $\beta^\sharp$ and hence the Weil pairing is perfect.)
Within each pair $\{\tilde \tau, \tilde \tau^c\}$, we can take an arbitrary direct summand  $\omega_{R,\tilde \tau }^{\circ} \subseteq H_1^{\cris}(A_0/R_0)_{R, \tilde \tau}^\circ$ which lifts $\omega_{A_0^\vee/R_0, \tilde \tau}^{\circ}$, and let $\omega_{R, \tilde \tau^c}^\circ$ be the orthogonal complement of $\omega_{R, \tilde \tau^c}^\circ$ under the perfect pairing \eqref{Equ:pairing-R}. By the Hodge filtration \eqref{Equ:reduced-Hodge}, such choices of $(\omega_{R, \tilde \tau}^\circ, \omega_{R, \tilde \tau^c}^{\circ})$ form a torsor under the group
\[
\Hom_{R_0}(\omega_{A_0^\vee/R_0, \tilde \tau}^{\circ}, \Lie(A_0)_{\tilde \tau}^\circ)\otimes I\cong \Lie(A_0)_{\tilde \tau}^{\circ}\otimes_{R_0} \Lie(A_0)_{\tilde \tau^c}^\circ \otimes I,
\]
where in the second isomorphism, we have used the fact that $\Lie(A_0^\vee)^\circ_{\tilde \tau}\simeq \Lie(A_0)^\circ_{\tilde \tau^c}$.

\end{itemize}

We take  liftings $\omega_{R, \tilde \tau}^\circ$ for each $\tilde \tau \in \Sigma_{E, \infty}$ as above, and let $(A, \iota)$ be the corresponding deformation to $R$ of $(A_0, \iota_0)$. It is clear that $\bigoplus_{\tau\in \Sigma_{\infty}}(\omega_{R,\tilde{\tau}}^{\circ}\oplus \omega_{R, \tilde{\tau}^c}^{\circ})$ is isotropic for the pairing on $H^{\cris}_1(A_0/R_0)^{\circ}_R$ induced by $\lambda_0$. Hence,  the polarization $\lambda_0$ lifts uniquely to a polarization $\lambda: A\ra A^\vee$ satisfying condition (b1) in the statement of the Theorem. By the criterion of flatness by fibers \cite[11.3.10]{ega}, $\Ker(\lambda)$ is a finite flat group scheme over $R$, and the condition (b2) is thus satisfied. Condition (b3) follows from the fact that the morphism $\lambda_*:H_1^{\dR}(A/R)\ra H_1^{\dR}(A/R)$ is the same as $\lambda_{0,*}: H_1^{\cris}(A_0/R_0)_R\ra H_1^{\cris}(A_0^\vee/R_0)_R$ under the canonical isomorphism $H_1^{\dR}(B/R)\simeq  H_1^{\cris}(B_0/R_0)_R$ for $B=A_0, A_0^\vee$.

We have to show moreover that the level structure $\alpha_{K',0}=(\alpha^p_0, \alpha_{p,0})$ extends uniquely to $A$. It is clear for $\alpha^{p}_0$. For $\alpha_{p,0}$, let $H_0=\prod_{\gothp\text{ of type } \alpha^\sharp} \alpha_{\gothp}$ be the product of the closed subgroups in the data of $\alpha_{p,0}$.
 Let  $f_0:A_0\ra B_0=A_0/H_0$ be the canonical isogeny. It suffices to show that $B_0$ and $f_0$ deform to $R$. The abelian variety $B_0$ is equipped with an induced action of $\cO_{D_{\ttS}}$, a polarization $\lambda_{B_0}$ satisfying conditions (a) and (b). The isogeny $f_0$ induces canonical isomorphisms $H^{\cris}_{1}(A_0/R_0)_{R, \tilde \tau}^\circ\cong H_1^{\cris}(B_0/R_0)_{R,\tilde \tau}^\circ$ for $\tilde \tau \in \Sigma_{E,\infty/\gothp}$ with $\gothp$ not of type $\alpha^\sharp$. So for such primes $\gothp$ and $\tilde \tau \in \Sigma_{\infty/\gothp}$, the liftings $\omega^{\circ}_{R,\tilde \tau}$ chosen above give the liftings of $\omega_{B_0^\vee/R_0, \tilde \tau}^\circ\subset H_1^\dR(B_0/R_0)_{\tilde \tau}^\circ$.
For $\tilde \tau\in \Sigma_{\infty/\gothp}$ with $\gothp$ of type $\alpha^\sharp$, we note that at each closed point $x$ of $R_0$, $\omega_{B_0^\vee/k_x, \tilde \tau}^\circ$ is either trivial or isomorphic to the whole $H_1^{\dR}(B_0/k_x)_{\tilde \tau}^{\circ}$ as in the case for $A_0$. Hence the same holds for $R_0$ in place of $k_x$.  Therefore, $\omega_{B_0^\vee/R_0, \tilde \tau}^\circ$ admits a unique lift to a direct summand of $H^{\cris}_1(B_0/R_0)_{R, \tilde \tau}^{\circ}$. Such choices of liftings of $\omega_{B_0^\vee/R_0,\tilde \tau}^{\circ}$ give rise to a deformation $B/R$ of $B_0/R_0$.
It is clear that $f_0 : A_0 \to B_0$ also lifts to an isogeny $f: A \to B$.  Then the kernel of $f$ gives the required lift of $H_0$.
This concludes the proof of the smoothness of $\bfSh_{K'}(G'_{\tilde \ttS})$.

The dimension of $\bfSh_{K'}(G'_{\tilde \ttS})$ follows from the calculation of the tangent bundle of $\bfSh_{K'}(G'_{\tilde \ttS})$ as the following corollary shows.
When the ramification set $\ttS_{\infty}$ is non-empty, it is a standard argument to use valuative criterion to check that $\bfSh_{K'}(G'_{\tilde \ttS})$ is proper.
We will postpone the proof to Proposition~\ref{Prop:smoothness}, where a more general statement is proved. (One can check that there is no loop-hole in our argument.)
\end{proof}

\begin{cor}\label{C:deformation}
Let  $S_0\hra S$ be a closed immersion of locally noetherian $k_0$-schemes with ideal sheaf $\calI$ such that $\calI^2=0$.  Let $x_0=(A_0, \iota_0, \lambda_0, \bar{\alpha}_{K',0} )$ be an $S_0$-valued point of $\bfSh_{K'}(G'_{\tilde \ttS})$. Then the set-valued sheaf  of local deformations of $x_0$ to $S$  form a torsor under the group
\[
\bigoplus_{\tau\in \Sigma_{\infty}-\ttS_{\infty}} \bigl(\Lie(A_0)_{\tilde \tau}^{\circ}\otimes \Lie(A_0)_{\tilde \tau ^c}^{\circ}\bigr)\otimes \calI.
\]
In particular, the tangent bundle $\calT_{\bfSh_{K'}(G'_{\tilde \ttS})}$ of $\bfSh_{K'}(G_{\tilde \ttS}')$ is canonically isomorphic to
\[
 \bigoplus_{\tau\in \Sigma_{\infty}-\ttS_{\infty}} \Lie(\bfA')^\circ_{\tilde \tau}\otimes \Lie(\bfA')^\circ_{\tilde \tau^c}
\]
where $\bfA' = \bfA'_{{\tilde \ttS}, K'}$ denotes the universal abelian scheme over $\bfSh_{K'}(G'_{\tilde \ttS})$.
\end{cor}
\begin{proof}
A deformation of $x_0$ is determined by the liftings $\omega^\circ_{S, \tilde \tau}\subseteq H^{\cris}_1(A_0/S_0)_{S, \tilde \tau}^\circ$ of  $\omega_{A_0^\vee/S_0, \tilde \tau}^\circ$ for $\tilde \tau\in \Sigma_{E,\infty}$.
From the proof of Theorem~\ref{T:unitary-shimura-variety-representability}, we see that the choices for  $\omega_{S, \tilde \tau}^\circ$ are unique if $\tilde \tau$ restricts to $\tau\in \ttS_{\infty}$.
For $\tau\in \Sigma_{\infty}-\ttS_{\infty}$, the possible liftings $\omega_{S, \tilde \tau}^\circ$ and $\omega_{S, \tilde \tau^c}^\circ$ determine each other, and  form a torsor under the group
\[
\Hom_{\cO_{S_0}}(\omega_{A_0^\vee/S_0, \tilde \tau}^\circ, \Lie(A_0)_{\tilde \tau}^\circ) \otimes_{\cO_{S_0}}\calI\simeq \Lie(A_0)_{\tilde \tau}^\circ \otimes \Lie(A_0)^\circ_{\tilde \tau^c}\otimes_{\cO_{S_0}} \calI.
\]
The statement for the local lifts of $x_0$ to $S$ follows immediately. Applying this to the universal case, we obtain the second part of the Corollary.
\end{proof}

\begin{remark}
We remark that the moduli space $\bfSh_{K'}(G'_{\tilde \ttS})$ does not depend on the choice of the admissible lattice $\Lambda$ in Subsection~\ref{S:integral-unitary}; but the  universal abelian scheme $\bfA'$ does in the following way.
If $\Lambda_1$ and $\Lambda_2$ are two admissible lattices, we put $\widehat{\Lambda}_i^{(p)}: = \Lambda_i \otimes_\ZZ \widehat{\ZZ}^{(p)}$, and we use $\bfA'_i$ to denote the corresponding universal abelian variety over  $\bfSh_{K'}(G'_{\tilde \ttS})$ and $\bar \bbalpha^p_{K'^p, i}$ to denote  the universal level structure (away from $p$), for $i = 1,2$.

Then
 there is a natural prime-to-$p$ quasi-isogeny $\eta: \bfA'_1 \dashrightarrow \bfA'_2$ such that
\[
\xymatrix@C=50pt{
\widehat \Lambda^{(p)}_1 \ar@{-->}[d]
\ar[r]^-{\bar \bbalpha^p_{K'^p,1}}_-\cong &
T^{(p)}(
\bfA'_1)
\ar@{-->}[d]^{T^{(p)}(\eta)}\\
\widehat \Lambda^{(p)}_2
\ar[r]^-{\bar \bbalpha^p_{K'^p, 2}}_-\cong &
T^{(p)}(
\bfA'_2)
}
\]
is a commutative diagram up the action of $K'^p$, where the left vertical arrow is the isogeny of lattices inside $V \otimes_\QQ \AAA^{\infty,p}$.
(For more detailed discussion, see \cite[1.4.3]{lan}.)
\end{remark}

\begin{cor}
\label{C:integral-model-quaternion}
The integral model $\bfSh_{K'}(G'_{\tilde \ttS})$ defined in Theorem~\ref{T:unitary-shimura-variety-representability} gives an  integral canonical model $\bfSh_{K'_p}(G'_{\tilde \ttS})$ of $\Sh_{K'_p}(G'_\ttS)$.  Consequently, the quaternionic  Shimura variety $\Sh_{K_p}(G_\ttS)$ admits an integral canonical model over $\calO_{\wp}$.  Similarly, the Shimura varieties $\Sh_{K_p \times K_{E,p}}(G_\ttS \times T_{E, \tilde \ttS})$ and $\Sh_{K''_p}(G''_{\tilde \ttS})$ both admit integral canonical models over $\calO_{\tilde \wp}$.
The geometric connected components of these integral canonical models are canonically isomorphic.
\end{cor}
\begin{proof}
We first assume that $\ttS_\infty \neq \Sigma_\infty$. We need  to verify that for any smooth $\cO_{\tilde \wp}$-scheme $S$,  any morphism $s_0:S\otimes_{\cO_{\tilde \wp}} E_{\tilde \ttS, \tilde \wp}\ra \Sh_{K'_p}(G'_{\tilde \ttS})$ extends to a morphism $s: S\ra \bfSh_{K'_p}(G'_{\tilde \ttS})$.
Explicitly, we have to show that a tuple $(A, \iota,\lambda_, \alpha^p\alpha_p)$ over $S\otimes_{\cO_{\tilde \wp}}E_{{\tilde \ttS},\tilde \wp}$ extends to a similar tuple over $S$. Here, $\alpha^p\alpha_p$ is the projective limit in $K'^{p}$ of  level structures $\alpha^p_{K'^p} \alpha_p$  as in Theorem~\ref{T:unitary-shimura-variety-representability}(c).
The same arguments as \cite[Corollary 3.8]{moonen96} apply to proving the existence of  extension of $A$, $\iota$, $\lambda$, and the prime-to-$p$ level structure $\alpha^p$.  It remains to extend the level structure $\alpha_p$.
Let $\bfSh_{\widetilde {K}'_{p}}(G'_{\tilde \ttS})$ denote the  similar moduli space as $\bfSh_{K'_{p}}(G_{\tilde \ttS}')$ by forgetting the $p$-level structure $\alpha_p$. 
We have seen in the proof of Theorem~\ref{T:unitary-shimura-variety-representability} that there is no local deformation of $\alpha_p$, which means the forgetful map  $\bfSh_{K'_{p}}(G'_{\tilde \ttS})\ra \bfSh_{\widetilde {K}'_{p}}(G'_{\tilde \ttS})$ is finite and \'etale.
The discussion above shows that $\bfSh_{\widetilde {K}'_{p}}(G'_{\tilde \ttS})$ satisfies the extension property. Hence,  there exists a morphism $\tilde s: S\ra \bfSh_{\widetilde{K}'_p}(G'_{\tilde \ttS})$ such that the square of the following   diagram
 \[
 \xymatrix{
 S\otimes_{\cO_{\tilde \wp}}E_{\tilde \ttS,\tilde \wp}\ar[r]^{s_0}\ar[d]& \bfSh_{K'_p}(G'_{\tilde \ttS})\ar[d]\\
 S\ar@{-->}[ur]^{s} \ar[r]^{\tilde s} &\bfSh_{\widetilde{K}'_p}(G'_{\tilde \ttS})
 }
 \]
 is commutative.
 We have to show that there exists a  map $s$ as the dotted arrow that makes the  whole diagram commutative. Giving such a map $s$ is equivalent to giving a section of the finite \'etale  cover $S\times_{\bfSh_{ \widetilde{K}'_{\ttS}}(G'_{\tilde \ttS})}\bfSh_{{K}'_{\ttS}}(G'_{\tilde \ttS})\ra S$ extending the section corresponding to $s_0$. Since a section of a finite \'etale cover of separated schemes is an open and closed immersion, the existence of $s$ follows immediately.
    The existence of integral canonical models for $\Sh_{K_p}(G_{\ttS})$, $\Sh_{K''_p}(G''_{\tilde \ttS})$ and $\Sh_{K_p\times K_{E,p}}(G_{\ttS}\times T_{E, \tilde \ttS})$  follows from Corollary~\ref{C:Sh(G)^circ_Zp independent of G}.
    
When $\ttS_\infty = \Sigma_\infty$, we need to show that the action of the arithmetic Frobenius $\sigma_{\tilde \wp}$  on the moduli space $\bfSh_{K'}(G'_{\tilde \ttS})$ is given  by  the  reciprocity law as in Subsection~\ref{S:integral model weak Shimura datum}.
Put $n_{\tilde \wp} = [k_{\tilde \wp}: \FF_p]$.
Let $\gothRec_{Z'}: \Gal_{E_{\tilde \ttS}} 
\to Z'(\QQ)^\cl\backslash Z'(\AAA^\infty)/Z'(\Zp)$ denote the reciprocity map defined in Subsection~\ref{S:integral model weak Shimura datum}, where $Z'$ is the center of $G'_{\tilde \ttS}$ which is the algebraic group associated to the subgroup of  $E^\times$ consisting of elements with norm to $F^{\times}$ lying in $\Q^{\times}$.
By definition, $\gothRec_{Z'}(\sigma_{\tilde \wp})$ is the image of $\varpi_{\tilde \wp}^{-1}$ under the composite of
\[
\gothRec_{Z', \tilde \wp}\colon E_{\tilde\ttS,\tilde \wp}^\times / \calO_{\tilde \wp}^\times \xrightarrow{\Rec_{Z'}(G'_{\tilde \ttS}, \gothH_\ttS)} Z'(\Qp) /Z'(\Zp)
\]
and the natural map  $Z'(\Qp)/Z'(\Zp) \ra Z'(\QQ)^\cl\backslash Z'(\AAA^\infty)/Z'(\Zp)$.
Explicitly, one has 
\[
Z'(\QQ_p) = \big\{ \big( (x_\gothp)_{\gothp \in \Sigma_p}, y\big)\,\big|\, y \in \QQ_p^\times,\ x_\gothp \in E_{\gothp}^\times, \textrm{ and }\Nm_{E_\gothp/F_\gothp}(x_\gothp) = y\big\}.
\]
We note that there is no $p$-adic  primes of $F$ of type $\beta$, and the valuation of $y$ determines the valuation of $x_\gothp$ for $\gothp$ of type $\beta^{\sharp}$. 
For each prime $\gothp\in \Sigma_p$ of type $\alpha$ or $\alpha^{\sharp}$, choose a place $\gothq$ of $E$ above $\gothp$, then the map 
$((x_{\gothp})_{\gothp}, y)\mapsto (\mathrm{val}_p(y), (\mathrm{val}_{p}(x_{\gothq}))_{\gothp})$ defines an isomorphism  
$$\xi:Z'(\Qp)/Z'(\Zp)\xrightarrow\cong \ZZ \times \prod_{\gothp \textrm{ of type $\alpha$ or $\alpha^{\sharp}$}} \ZZ,$$
 where we have written $x_{\gothp}=(x_{\gothq},x_{\bar\gothq})$ for each prime $\gothp\in \Sigma_p$ of type $\alpha$ or $\alpha^{\sharp}$.
By the definition of $\gothRec_{Z', \tilde \wp}$ in Subsection~\ref{S:integral model weak Shimura datum} using $h'_{\tilde \ttS}$, we see that $\xi\circ\gothRec_{Z', \tilde \wp}(\varpi_{\tilde \wp}^{-1})$ is equal to 
\begin{equation}
\label{E:image of rec}
\big(-n_{\tilde \wp}, (-\#\tilde \ttS_{\infty / \bar \gothq} \cdot n_{\tilde \wp} / f_\gothp)_{\gothp}\big),
\end{equation}
where $n_{\tilde \wp} = [k_{\tilde \wp}: \FF_p]$ and $f_\gothp$ is the inertia degree of $\gothp $ in $F/\QQ$. (Note that the homomorphism $h_{E, \tilde \ttS}$ sends $z$ to $\bar z$ with respect to the embeddings $\tilde \tau \in \tilde \ttS_\infty$.)

On the other hand, $\sigma_{\tilde \wp}$ takes a closed point $x = (A, \iota, \lambda, \alpha_{K'})$ of $\bfSh_{K'}(G'_{\tilde \ttS})_{\overline \FF_p}$ to $\sigma_{\tilde \wp}(x) = (\sigma_{\tilde \wp}^*(A), \iota', \lambda', \Frob_{\tilde \wp}\circ \alpha_{K'} )$, where $\sigma_{\tilde \wp}^*(A)$ denotes the pullback of $A$ via the Frobenius $\sigma_{\tilde \wp}=\sigma^{n_{\tilde\wp}}$ on the residue field $\kappa(x)$, equipped with the induced  $\cO_{D_{\ttS}}$-action and the polarization, and $\Frob_{\tilde \wp}:A\ra \sigma_{\tilde \wp}^*A$ is the relative Frobenius map.
For a $p$-adic prime  $\gothp$ of $F$ (or of $E$),  denote by  $\tcD(A)_{\gothp}$ the covariant Dieudonn\'e  module of $A[\gothp^{\infty}]$.
We observe that, if $\gothp$ is a prime of $F$ of type $\beta$, then 
\[
\tilde\calD(\sigma^*_{\tilde \wp}(A))_\gothp=p^{-n_{\tilde \wp}/2} V^{n_{\tilde \wp}}\tilde \calD(A)_\gothp,
\]
and if $\gothp$ is of type $\alpha$ or $\alpha^{\sharp}$ with $\gothq$ a place of $E$ above $\gothp$, then 
\[
 \tilde \calD(\sigma^*_{\tilde \wp}(A))_\gothq
 =
 p^{-\#\tilde \ttS_{\infty / \bar \gothq} \cdot n_{\tilde \wp} / f_\gothp}V^{n_{\tilde \wp}}\tilde \calD(A)_\gothq.
\]
Let  $g_{p}\in Z'(\QQ_p)$ be an element such that $\xi(\bar g_p)$ is given by \eqref{E:image of rec}, where $\bar g_p$ denotes the image of $g_p$ in $Z'(\QQ_p)/Z'(\ZZ_p)$.  
Then via the isogeny $\Frob_{\tilde \wp}:A\ra \sigma_{\tilde\wp}^*A$, which corresponds to $V^{n_{\tilde \wp}}: \tcD(A)\ra \tcD(\sigma^*_{\tilde\wp}A)$, $\tcD(\sigma_{\tilde \wp}^*A)$ is identified  with the lattice $g_p\tcD(A)$ of $\tcD(A)[1/p]$. 
This agrees with the computation of  $\gothRec_{Z', \tilde \wp}(\varpi_{\tilde \wp}^{-1})$ above.
\end{proof}


The rest of this section is devoted to understanding how to pass the universal abelian varieties on $\bfSh_{K'_p}(G'_{\tilde \ttS})$ to other Shimura varieties, as well as natural partial Frobenius morphisms among these varieties and their compatibility with the abelian varieties.

\subsection{Actions on universal abelian varieties in the unitary case}
\label{S:abel var in unitary case}

We need to extend the usual tame Hecke algebra action on the universal abelian variety $\bfA'_{K'_p}$ over  $\bfSh_{K'_p}(G'_{\tilde \ttS})$ to the action of a slightly bigger group $\widetilde G_{\tilde \ttS}: = G'_{\tilde \ttS}(\AAA^{\infty, p}) G''_{\tilde \ttS}(\QQ)_+^{(p)}$.
We follow \cite{saito}. Take an element $\tilde g \in \widetilde G_{\tilde \ttS}$ and  let $K'^p_1$ and $K'^p_2$ be two open compact subgroups of $G'_{\tilde \ttS}(\AAA^{\infty, p})$ such that $\tilde g^{-1} K'^p_1 \tilde g \subseteq K'^p_2$ (note that $G''_{\tilde \ttS}$ normalizes $G'_{\tilde \ttS}$).
We put $K'_i = K'^p_iK'_p$ for 
$i=1,2$.
Then starting from the universal
abelian variety $\bfA'_{K'_1}$ together with the tame level structure $\bar \bbalpha_{K'^p_1}^p: \widehat \Lambda^{(p)} \xrightarrow{\cong}
T^{(p)} \bfA'_{K'_1}$, we may obtain an abelian variety $\bfB'$ over $\bfSh_{K'_1}(G'_{\tilde \ttS})$, together with a prime-to-$p$ quasi-isogeny $\eta: \bfA'_{K'_1} \to \bfB' $ and a tame level structure such that the following diagram commutes
\[
\xymatrix@C=40pt{
& \widehat \Lambda^{(p)} \ar[r]^-{\bar \bbalpha^p_{K'^p_1}}_-\cong
\ar@{-->}[d]
&
T^{(p)}(
\bfA'_{K'_1})
\ar@{-->}[d]^{T^{(p)}(\eta)}
\\
\widehat \Lambda^{(p)} \ar[r]^-{\cdot \tilde g^{-1}}_-\cong
&
\tilde g\widehat \Lambda^{(p)} \ar[r]^-\cong
&
T^{(p)}(\bfB'),
}
\]
where the left vertical arrow is the natural quasi-isogeny as lattices inside $V \otimes_\QQ \AAA^{\infty,p}$.
Since $\tilde g^{-1}K'^p_1 \tilde g \subseteq K'^p_2$, we may take the $K'^p_2$-orbit of the composite of the bottom homomorphism as the tame level structure on $\bfB'$.

One can easily transfer other data in the moduli problem of Theorem~\ref{T:unitary-shimura-variety-representability} to $\bfB'$, \emph{except for the polarization} which we make the modification as follows:
since $\tilde g \in G'_{\tilde \ttS}(\AAA^{\infty, p}) G''_{\tilde \ttS}(\QQ)_+^{(p)}$, we have $\nu(\tilde g) \in \calO_{F, (p)}^{\times,>0} \cdot \AAA_\QQ^{\infty, p,\times} = \calO_{F, (p)}^{\times,>0} \cdot \widehat \ZZ^{(p),\times}$.
We can then write $\nu(\tilde g)$ as the product $\nu^+_{\tilde g} \cdot u$ for $\nu^+_{\tilde g} \in \calO_{F, (p)}^{\times,>0}$ and $u \in \widehat \ZZ^{(p),\times}$.
In fact, $\nu^+_{\tilde g}$ is uniquely determined by this restriction.  We take the polarization on $\bfB'$ to be the composite of a sequence of quasi-isogenies:
\[
\lambda_{\bfB'}: \bfB' \xleftarrow{\quad} \bfA'_{K'_1} \xrightarrow{\nu^+_{\tilde g} \lambda_{\bfA'}} \big(\bfA'_{K'_1}\big)^\vee \xrightarrow{\quad} \bfB'^\vee.
\]
Such modification ensures that $\bfB'$ satisfies condition (c1) of Theorem~\ref{T:unitary-shimura-variety-representability}.

The moduli problem then implies that $\bfB' \cong (H_{ \tilde g})^*(\bfA'_{K'_2})$ for a uniquely determined morphism $H_{\tilde g}:\bfSh_{K'_1}(G'_{\tilde \ttS})\to
\bfSh_{K'_2}(G'_{\tilde \ttS})$. This gives the action of $\widetilde G_{\tilde \ttS}$ on $\Sh_{K'_p}(G'_{\tilde \ttS})$.
Moreover, we have a quasi-isogeny
\[
H_{\tilde g}^\#: \bfA'_{K'_1} \xrightarrow{\eta} \bfB' \cong
(H_{\tilde g})^*(\bfA'_{K'_2})\]
giving rise to an equivariant action of $\widetilde G_{\tilde \ttS}$ on the universal abelian varieties $\bfA'_{K'_p}$ over $\bfSh_{K'_p}(G'_{\tilde \ttS})$.

One easily checks that the action of diagonal $ \calO_{E,(p)}^\times$ on the Shimura variety $\bfSh_{K'_p}(G'_{\tilde \ttS})$ is trivial, and hence we have an action of $\calG'_{\tilde \ttS}= \widetilde G_{\tilde \ttS} / \calO_{E,(p)}^{\times, \cl}$ on $\bfSh_{K'_p}(G'_{\tilde \ttS})$.  However, the action of $\calO_{E,(p)}^\times$ on the universal abelian variety $\bfA'_{K'_p}$ is \emph{not} trivial.  So the latter does not carry a natural action of $\calG'_{\tilde \ttS}$.  So our earlier framework for Shimura varieties does not apply to the universal abelian varieties directly.
However, we observe that, by the construction at the end of Subsection~\ref{A:connected integral model},
\[
\bfSh_{K''_p}(G''_{\tilde \ttS}) = \bfSh_{K'_p}(G'_{\tilde \ttS}) \times_{\calG'_{\tilde \ttS}} \calG''_{\tilde \ttS} = \bfSh_{K'_p}(G'_{\tilde \ttS}) \times_{\widetilde G_{\tilde \ttS}} G''_{\tilde \ttS}(\AAA^{\infty,p}).
\]
So 
\[
\bfA''_{K''_p}: = \bfA'_{K'_p}\times_{\calG'_{\tilde \ttS}} \calG''_{\tilde \ttS} = \bfA'_{K'_p} \times_{\widetilde G_{\tilde \ttS}} G''_{\tilde \ttS}(\AAA^{\infty,p})
\]
gives a natural family of abelian variety over $\bfSh_{K''_p}(G''_{\tilde \ttS})$.
We will not discuss families of abelian varieties over the quaternionic Shimura variety $\bfSh_{K_p}(G_\ttS)$ (except when $\ttS =\emptyset$).

\subsection{Automorhpic $l$-adic systems on $\bfSh_{K''_p}(G''_{\tilde \ttS})$ and its geometric interpretation}

By a \emph{multiweight}, we mean a tuple $(\underline{k}, w) =((k_\tau)_{\tau\in\Sigma_\infty}, w) \in \NN^{[F:\Q]} \times \NN$  such that $ k_\tau \geq 2$ and $w \equiv k_\tau \pmod 2$ for each $\tau$.
We also fix a section of the natural map $\Sigma_{E, \infty} \to \Sigma_\infty$, that is to fix a extension $\tilde \tau$ to $E$ of each real embedding $\tau \in \Sigma_\infty$ of $F$. Use $\tilde \Sigma$ to denote the image of this section.
In this subsection, we use $\tilde \tau$ to denote this chosen lift of $\tau$.
We fix a subfield $L$ of $\overline \QQ \subset \CC$ containing all embeddings of $E$, as the coefficient field.

Let $\gothl$ be a finite place of $L$ over a prime $l$ with $l \neq p$.
Fix an isomorphism $\iota_\gothl:\CC \simeq \overline L_\gothl  $.

Consider the injection
\begin{small}
\[
G''_{\tilde \ttS} \times_\QQ L =
\big( \Res_{F/\QQ}(B_\ttS^\times) \times_{\Res_{F/\QQ}\GG_m} \Res_{E/\QQ}\GG_m \big) \times_\QQ L \hookrightarrow
\Res_{E/\QQ} D^\times_\ttS \times_\QQ L
 \cong
\prod_{\tau \in \Sigma_\infty}
\GL_{2, L, \tilde \tau} \times
\GL_{2, L, \tilde \tau^c},
\]
\end{small}
where $E^\times$ acts on $\GL_{2, L, \tilde \tau}$ (resp. $\GL_{2, L, \tilde \tau^c}$) through $\tilde \tau$ (resp. $\tilde \tau^c$).
For a multiweight $(\underline k, w)$, we consider the following representation of $G''_{\tilde \ttS} \times_\QQ L$:
\[
\rho''^{(\underline{k}, w)}_{\tilde \ttS, \widetilde \Sigma} = \bigotimes_{\tau \in \widetilde \Sigma} \rho_\tau^{(k_\tau, w)} \circ \check\pr_{\tilde \tau} \quad \textrm{for }i = 1, 2 \textrm{ with} \quad
\rho_\tau^{(k_\tau, w)} =
\Sym^{k_\tau-2} \otimes
\det{}^{\frac{w-k_\tau}{2}}
,
\]
where $\tau$ is the restriction of $\tilde \tau$ to $F$, and $\check\pr_{\tilde \tau}$  is the \emph{contragradient} of the natural projection to the $\tilde \tau$-component of $G''_{\tilde \ttS} \times_\QQ L \hookrightarrow \Res_{E/\QQ} D_\ttS^\times \times_\QQ L$.
Note that $\rho''^{(\underline k, w)}_{\tilde \ttS, \widetilde \Sigma}$ is trivial on the maximal anisotropic $\R$-split subtorus of the center of $G''_{\tilde \ttS}$, i.e. $\Ker(\Res_{F/\Q}\GG_m \to \GG_m)$.
By \cite[Ch. III, \S 7]{milne}, $\rho''^{(\underline k, w)}_{\tilde \ttS, \widetilde \Sigma}$ corresponds to a lisse $L_\gothl$-sheaf $\scrL''^{(\underline{k}, w)}_{\tilde \ttS, \widetilde \Sigma}$ over the Shimura variety $\bfSh_{K''}(G''_{\tilde \ttS})$ compatible as the level structure changes.

We now give a geometric interpretation of this automorphic $l$-adic sheaf on $\bfSh_{K''}(G''_{\tilde \ttS})$.
For this, we fix an isomorphism  $D_\ttS\simeq \rmM_2(E)$ and let $\gothe = \big(
\begin{smallmatrix}
1&0\\0&0
\end{smallmatrix}\big)\in \rmM_2(\cO_E)$ denote the idempotent element.
Let $\bfA'' = \bfA''_{\tilde \ttS, K''}$ denote
the natural family of abelian varieties constructed in Subsection~\ref{S:abel var in unitary case}.
Let $V(\bfA'')$ denote the $l$-adic Tate module of $\bfA''$.
We then have a decomposition
\[
V(\bfA'') \otimes_{\QQ_l} L_\gothl \cong \bigoplus_{\tau \in\Sigma_\infty}
\big(V(\bfA'')_{\tilde \tau} \oplus
V(\bfA'')_{\tilde \tau^c} \big) = \bigoplus_{\tau \in\Sigma_\infty}
\big(V(\bfA'')^{\circ, \oplus 2}_{\tilde \tau} \oplus
V(\bfA'')^{\circ, \oplus 2}_{\tilde \tau^c} \big),
\]
where $V(\bfA'')_{\tilde \tau}$ (resp. $V(\bfA'')_{\tilde \tau^c}$) is the component where $\calO_E$ acts through $\iota_\gothl \circ\tilde \tau$ (resp. $\iota_\gothl \circ \tilde \tau^c$), and $V(\bfA'')_{\tilde \tau}^\circ: = \gothe  V(\bfA'')_{\tilde \tau}$ (resp. $V(\bfA'')_{\tilde \tau^c}^\circ := \gothe  V(\bfA'')_{\tilde \tau^c}$) is a lisse $L_\gothl$-sheaf of rank $2$. For a multiweight $(\underline k, w)$, we put
\[
\calL^{(\kb,w)}_{\widetilde\Sigma}(\bfA'')=\bigotimes_{\tau \in \widetilde \Sigma} \bigg( \Sym^{k_\tau -2} V(\bfA'')_{\tilde \tau}^{\circ, \vee} \otimes (\wedge^2
V(\bfA'')_{\tilde \tau}^{\circ, \vee})^{\frac{w-k_\tau}2}
\bigg).
\]
Note that the duals on the Tate modules mean that we are  essentially taking the relative first \'etale \emph{cohomology}.
The moduli interpretation implies that we have a canonical isomorphism
\[
\scrL''^{(\underline k, w)}_{\tilde \ttS, \widetilde \Sigma} \cong
\calL_{\widetilde \Sigma}^{(\underline k, w)}(\bfA''_{\tilde \ttS}).
\]

\subsection{Twisted Partial Frobenius}
\label{S:partial Frobenius}

The action of the twisted partial Frobenius and its compatibility with the GO-strata description will be important to later applications in \cite{tian-xiao2}.
We first define the twisted partial Frobenius on the universal abelian scheme $\bfA'_{\tilde \ttS} = \bfA'_{\tilde \ttS, K'}$ over the unitary Shimura variety $\bfSh_{K'}(G_{\tilde \ttS}')$.

Fix $\gothp \in \Sigma_p$.
We define an action of
$\sigma_\gothp $ on $\Sigma_{E,\infty}$ as follows: for $\tilde\tau\in \Sigma_{E,\infty}$, we put
\begin{equation}\label{E:defn-sigma-gothp}
\sigma_{\gothp}\tilde\tau=\begin{cases}\sigma \circ\tilde\tau & \text{if }\tilde \tau\in \Sigma_{E,\infty/\gothp},\\
\tilde\tau &\text{if }\tilde\tau\notin \Sigma_{E,\infty/\gothp},
\end{cases}
\end{equation}
where $\Sigma_{E,\infty/\gothp}$ denotes the lifts of places in $\Sigma_{\infty/\gothp}$.
Note that $\sigma_{\gothp}$ induces a natural action on $\Sigma_{\infty}$,  and $\prod_{\gothp\in \Sigma_p}\sigma_{\gothp}=\sigma$ is the Frobenius action.
Let $\sigma_{\gothp}\tilde \ttS$ denote the image of $\tilde \ttS$ under $\sigma_{\gothp}$. Note that a prime $\gothp'\in \Sigma_p$ has the same type with respect to $\tilde \ttS$ or $\sigma_{\gothp}\tilde \ttS$.
We fix an isomorphism  $B_{\sigma_{\gothp}\ttS}\otimes\AAA^{\infty}\simeq B_{\ttS}\otimes\AAA^{\infty}$, which  induces in turns an  isomorphism $G'_{\sigma_\gothp\tilde \ttS}(\AAA^{\infty})\simeq G'_{\tilde \ttS}(\AAA^{\infty})$. We regard $K'$ as an open subgroup of $G'_{\sigma_{\gothp}\tilde \ttS}(\AAA^{\infty})$ and have a well-defined unitary Shimura variety $\bfSh_{K'}(G'_{\sigma_{\gothp}\tilde \ttS})$.
We also point out that the reflex fields for $G'_{\tilde \ttS}$ and for $G'_{\sigma_\gothp \tilde \ttS}$ have the same completion at $\tilde \wp$.

Let $S$ be a locally noetherian $k_{\tilde \wp}$-scheme and let $(A, \iota, \lambda, \bar \alpha_{K'})$ be an $S$-point on $\bfSh_{K'}(G'_{\tilde \ttS})_{k_{\tilde \wp}}$.
We will define a new  $S$-point $(A',\iota', \lambda', \bar{\alpha}_{K'})$ on $\bfSh_{K'}(G'_{\sigma^2_\gothp \tilde \ttS})_{k_{\tilde \wp}}$ as follows.
The kernel of the relative $p^2$-Frobenius $\Fr_{A}\colon A \to A^{(p^2/S)}$ carries an action of $\calO_F$, and we denote by  $\Ker_{\gothp^2}$ its $\gothp$-component. We put $A'=(A/\Ker_{\gothp^2})\otimes_{\cO_F} \gothp$ with its induced action by $\cO_{D_{\ttS}}$. 
It also comes equipped with a quasi-isogeny $\eta$ given by the composite 
\[
\eta: A \longrightarrow A/\Ker_{\gothp^2} \longleftarrow (A/\Ker_{\gothp^2})\otimes_{\cO_F} \gothp  =A'.
\]
It induces canonical isomorphisms of $p$-divisible groups $A'[\gothq^{\infty}]\simeq A[\gothq^{\infty}]$ for $\gothq\in \Sigma_{p}$ with $\gothq\neq \gothp$, and  $A'[\gothp^{\infty}]\simeq A[\gothp^{\infty}]^{(p^2)}$. 
From this, one can easily check the signature condition for $A'$.
We define the polarization $\lambda'$ to be the quasi-isogeny defined by the
composite
\begin{equation}
\label{E:polarization for partial frob}
A' \xleftarrow{\ \eta\ } A \xrightarrow{\ \lambda\ } A^\vee \xleftarrow{\ \eta^\vee\ } A'^\vee.
\end{equation}
We have to check that $\lambda'$ is a genuine isogeny, and it verifies condition Theorem~\ref{T:unitary-shimura-variety-representability}(b) on $\lambda'$ at prime $\gothp$. By flatness criterion by fibers, it suffices to do this after base change to every geometric point of $S$. We may thus suppose that $S=\Spec(k)$ for an algebraically closed field $k$ of characteristic $p$. Let $\tcD(A)_{\gothp}$ be the covariant Dieudonn\'e module of the $p$-divisible group $A[\gothp^{\infty}]$, and define $\tcD(A')_{\gothp}$ similarly. By definition, we have
\[
\tcD(A')_{\gothp}=p\tcD(A/\Ker_{\gothp^2})_{\gothp}=pV^{-2}\tcD(A)_{\gothp}=p^{-1}F^2\tcD(A)_{\gothp},
\]
where $pV^{-2}\tcD(A)_{\gothp}$ means the inverse image of $\tcD(A)_{\gothp}$ under  the bijective endomorphism $V^{2}$ on $\tcD(A)_{\gothp}[1/p]$. Applying the Dieudonn\'e functor to \eqref{E:polarization for partial frob}, we get
\[
\lambda'_*: 
\tcD(A')_{\gothp}=pV^{-2}\tcD(A)_{\gothp}
\xleftarrow{\ \eta_*\ }
\tcD(A)_{\gothp} \xrightarrow{\ \lambda_*\ } \tcD(A^\vee)_{\gothp}
\xleftarrow{\ \eta^\vee_*\ }\tilde \calD(A'^\vee) = p^{-1}F^2\tcD(A^\vee)_{\gothp}.
\]
Now it is easy to see that $\lambda'$ is an isogeny, and  the condition in Theorem~\ref{T:unitary-shimura-variety-representability}(b) on $\lambda'$ follows from that on $\lambda$.
The tame level structure $\bar\alpha'_{K'}$ is given by the composition
\[
\widehat\Lambda^{(p)} \xrightarrow{\alpha_{K'}} T^{(p)}A \xrightarrow \cong
T^{(p)}(A/\Ker_{\gothp^2})
\xleftarrow \cong
T^{(p)}((A / \Ker_{\gothp^2}) \otimes_{\calO_F} \gothp) = T^{(p)}(A').
\]
We are left to define the subgroups $\alpha'_{\gothp'}$ for all $\gothp' \in \Sigma_p$ of type $\alpha^\sharp$.
The definition is clear for $\gothp'\neq \gothp$, since $A'[\gothp'^{\infty}]$ is canonically identified with $A[\gothp'^{\infty}]$. Assume thus $\gothp' = \gothp$ is  of type $\alpha^\#$.  In the data of $\alpha'_\gothp = H'_\gothq \oplus H'_{\bar \gothq}\subseteq A'[\gothp]$, the subgroup $H'_{\bar\gothq}$ is determined as the orthogonal complement of $H'_{\gothq}$ under the Weil-pairing on $A[\gothp]$. Therefore, it suffices to construct $H'_{\gothq}$, or equivalently an $\cO_{D_{\ttS}}$-isogeny $f':A'\ra B'=A'/H'_{\gothq}$ with kernel in $A'[\gothq]$ of degree $\#k_{\gothp}^2$.
Let $f:A\ra B=A/H_{\gothq}$ be the isogeny  given by $\alpha_{\gothp}$.
We write $\Ker_{\gothp^2, B}$ for the $\gothp$-component of the kernel of the relative $p^2$-Frobenius $B \to B^{(p^2)}$.
It is easy to see that we have a natural isogeny $f_{\gothp^2}: A / \Ker_{\gothp^2} \to B / \Ker_{\gothp^2, B}$.
Then $H'_\gothq$ is defined to be the kernel of
\[
f_{\gothp^2} \otimes 1:\ A' = (A / \Ker_{\gothp^2}) \otimes \gothp \longrightarrow (B / \Ker_{\gothp^2, B}) \otimes \gothp =: B',
\]
and $\alpha'_\gothp$ is the direct sum of $H'_\gothq$ and its orthogonal dual  $H'_{\bar \gothq}$.

To sum up, we obtain a morphism
\begin{equation}\label{E:twist-partial-Frob}
\gothF'_{\gothp^2}: \bfSh_{K'}(G'_{\tilde \ttS})_{k_{\tilde \wp}} \to \bfSh_{K'}(G'_{\sigma^2_\gothp \tilde \ttS})_{k_{\tilde \wp}}.
\end{equation}
In all cases, we call the morphism $\gothF'_{\gothp^2}$ the \emph{twisted partial Frobenius map} on the unitary Shimura varieties.
Moreover, if $\bfA'_{\tilde \ttS}$ and $\bfA'_{\sigma^2_{\gothp}\tilde \ttS}$ are respectively the universal abelian schemes over $\bfSh_{K'}(G'_{\tilde \ttS})$ and $\bfSh_{K'}(G'_{\sigma^2_{\gothp}\tilde \ttS})$, 
we have the following universal quasi-isogeny:
\begin{equation}\label{E:universal-isog-Frob}
\eta'_{\gothp^2}\colon
\bfA'_{\tilde \ttS, k_{\tilde \wp}}
\longrightarrow \gothF'^{*}_{\gothp^2}(\bfA'_{\sigma^2_{\gothp}\tilde \ttS, k_{\tilde \wp}}).
\end{equation}
It is clear from the definition that $(\gothF'_{\gothp^2}, \eta'_{\gothp^2})$'s for different $\gothp \in \Sigma_p$ commute with each other.  

Let $S_p: \bfSh_{K'}(G_{\tilde\ttS}')\ra \bfSh_{K'}(G'_{\tilde\ttS})$ be the automorphism defined by $(A,\iota,\lambda,\bar\alpha_{K'})\mapsto (A,\iota,\lambda, p\bar\alpha_{K'})$. It is clear that $S_p^*\bfA'_{\tilde\ttS}\cong \bfA'_{\tilde\ttS}$. Hence, $S_p$ induces an automorphism of the cohomology groups $H^{\star}_{\rig}(\bfSh_{K'}(G'_{\tilde\ttS}), \scrD^{(\kb,w)}_{\tSigma}(\bfA'_{\tilde\ttS,k_0}))$, still denoted by $S_p$. If  
\[
F^2_{\bfSh_{K'}(G'_{\tilde \ttS})_{k_{\tilde \wp}}/k_{\tilde \wp}}\colon
\bfSh_{K'}(G'_{\tilde \ttS})_{k_{\tilde \wp}} \longrightarrow
\bfSh_{K'}(G'_{\sigma^2\tilde \ttS})_{k_{\tilde \wp}}\simeq \bfSh_{K'}(G'_{\tilde\ttS})_{k_{\tilde\wp}}^{(p^2)}
\]
denotes the relative $p^2$-Frobenius, then we have 
$
S_p^{-1}\circ F^2_{\bfSh_{K'}(G'_{\tilde \ttS})_{k_{\tilde \wp}}/k_{\tilde \wp}} = \prod_{\gothp\in\Sigma_p} \gothF_{\gothp^2}.
$
Similarly, if $[p]: \bfA'^{(p^2)}_{\tilde \ttS} \ra \bfA'^{(p^2)}_{\tilde \ttS}$ denotes the multiplication by $p$ and  
\[
F^2_{A}
\colon \bfA'_{\tilde \ttS, k_{\tilde \wp}} \to (F^2_{\bfSh_{K'}(G'_{\tilde \ttS})_{k_{\tilde \wp}}/k_{\tilde \wp}})^*
(\bfA'_{\sigma^2\tilde \ttS, k_{\tilde \wp}}) \cong \bfA'^{(p^2)}_{\tilde \ttS, k_{\tilde \wp}}
\]
denotes the $p^2$-Frobenius homomorphism,  we have
$
[p]^{-1}\circ F_{A}^2 = \prod_{\gothp \in \Sigma_p}
\eta'_{\gothp^2}.
$

Finally, we note that all the discussions above are equivariant with respect to the action of the Galois group and the action of $\widetilde G_{\tilde \ttS} = G''_{\tilde\ttS}(\QQ)^{+, (p)} G'_{\tilde\ttS}(\AAA^{\infty, p}) \simeq
G''_{\sigma_\gothp^2\tilde\ttS}(\QQ)^{+, (p)}
G'_{\sigma_\gothp^2\tilde\ttS}(\AAA^{\infty, p})$ when passing to the limit. (The isomorphism follows from the description of the group $\widetilde G_{\tilde \ttS}$ in \eqref{E:description of G'' G'}.)
So applying $-\times_{\widetilde G_{\tilde \ttS}} G''_{\tilde \ttS}(\AAA^{\infty, p})$ to the construction gives the following.

\begin{prop}
\label{P:product of partial Frobenius}
Let $\bfA''_{\tilde \ttS}$ denote the natural family of abelian varieties over $\bfSh_{K''_p}(G''_{\tilde \ttS})$. We identify the level structure for $G''_{\tilde \ttS}$ with that of $G''_{\sigma_\gothp^2\tilde \ttS}$ similarly.
Then for each $\gothp \in \Sigma_p$, we have a $G''_{\tilde \ttS}(\AAA^{\infty, p})$-equivariant  natural \emph{twisted partial Frobenius morphism} and an quasi-isogeny of family of abelian varieties:
\[
\gothF''_{\gothp^2}\colon
\bfSh_{K''_p}(G''_{\tilde \ttS})_{k_{\tilde \wp}} \longrightarrow
\bfSh_{K''_p}(G''_{\sigma_\gothp^2\tilde \ttS})_{k_{\tilde \wp}} \quad \textrm{and} \quad
\eta''_{\gothp^2}
\colon
\bfA''_{\tilde \ttS, k_{\tilde \wp}}
\longrightarrow \gothF''^{*}_{\gothp^2}(\bfA''_{\sigma^2_{\gothp}\tilde \ttS, k_{\tilde \wp}}).
\]
This induces a natural $G''_{\tilde \ttS}(\AAA^{\infty, p})$-equivariant homomorphism of \'etale cohomology groups:
\[
\xymatrix{
H^*_\et\big(
\bfSh_{K''_p}(G''_{\sigma_\gothp^2\tilde \ttS})_{\overline \FF_p}, \calL_{\tilde \Sigma}^{(\underline k, w)}(\bfA''_{\sigma_\gothp^2 \tilde \ttS})\big)
\ar[rr]^-{\gothF''^* _{\gothp^2}}
\ar[drr]_{\Phi_{\gothp^2}}
&&
H^*_\et\big(
\bfSh_{K''_p}(G''_{\tilde \ttS})_{\overline \FF_p}, \calL_{\tilde \Sigma}^{(\underline k, w)}(\gothF''^* _{\gothp^2}\bfA''_{\sigma_\gothp^2 \tilde \ttS})\big)
\ar[d]^{\eta''^*_{\gothp^2}}
\\
&&
H^*_\et\big(
\bfSh_{K''_p}(G''_{\tilde \ttS})_{\overline \FF_p}, \calL_{\tilde \Sigma}^{(\underline k, w)}(\bfA''_{\tilde \ttS})\big).
}
\]
Moreover, we have an equality of morphisms
\[
\prod_{\gothp \in \Sigma_p} \Phi_{\gothp^2} = S_p^{-1} \circ F^2
\colon
H^*_\et\big(
\bfSh_{K''_p}(G''_{\sigma^2\tilde \ttS})_{\overline \FF_p}, \calL_{\tilde \Sigma}^{(\underline k, w)}(\bfA''_{\sigma^2 \tilde \ttS})\big) \longrightarrow
H^*_\et\big(
\bfSh_{K''_p}(G''_{\tilde \ttS})_{\overline \FF_p}, \calL_{\tilde \Sigma}^{(\underline k, w)}(\bfA''_{ \tilde \ttS})\big),
\]
where $F^2$ is  the relative $p^2$-Frobenius and $S_p$ is the Hecke action given by multiplication by $\underline p^{-1}$.
Here $\underline p$ is the idele element which is $p$ at all places above $p$ and $1$ elsewhere.
\end{prop}

\begin{proof}
This is clear from the construction.
\end{proof}

\subsection{Comparison with the Hilbert modular varieties}\label{S:comparison-Hilbert}
 When $\ttS=\emptyset$, we have $G_{\emptyset}=\Res_{F/\Q}(\GL_{2,F})$ and $K_p=\GL_2(\cO_{F}\otimes_{\Z} \Z_p)$. It is well known that  $\Sh_{K_p}(G_{\emptyset})=\varprojlim_{K^p}\Sh_{K^pK_p}(G_{\emptyset})$ is  a projective system of Shimura varieties defined over $\Q$, and it parametrizes polarized Hilbert-Blumenthal abelian varieties  (HBAV for short) with prime-to-$p$ level structure.    Using this moduli interpretation, one can construct  an integral canonical model  over $\Z_p$ of $\Sh_{K_p}(G_{\emptyset})$ as in \cite{rap,lan}. By the uniqueness of the integral canonical model, we know that this classical integral model is isomorphic to $\bfSh_{K_p}(G_{\emptyset})$ constructed in Corollary~\ref{C:integral-model-quaternion}. For this ``abstract'' isomorphism to be useful in applications, we need to relate the universal HBAV $\calA$ on $\bfSh_{K_p}(G_{\emptyset})$ and the  abelian scheme $\bfA''_\emptyset$ on $\bfSh_{K''_p}(G''_\emptyset)$ constructed at the end of Subsection~\ref{S:abel var in unitary case}.

Let $G_{\emptyset}^\star \subseteq G_{\emptyset}$ be the inverse image of $\G_{m,\Q}\subseteq T_{F}=\Res_{F/\Q}(\G_{m,\Q})$ via the determinant map $\nu: G_{\emptyset}\ra T_{F}$.  The homomorphism $h_{\emptyset}: \C^{\times}\ra G_{\emptyset}(\R)$ factors through $G^\star_\emptyset(\R)$.
We can talk about the Shimura variety associated to $(G^\star_{\emptyset}, h_{\emptyset})$. We put $K^\star _p=K_p\cap G^\star_{\emptyset}(\Q_p)$. Then by Corollary~\ref{C:Sh(G)^circ_Zp independent of G}, $\Sh_{K_p}(G_{\emptyset})$ and $\Sh_{K_{p}^\star}(G^\star_{\emptyset})$ have isomorphic neutral connected components $\Sh_{K_p}(G_{\emptyset})^\circ_{\Q_{p}^{\ur}}\simeq \Sh_{K^*_p}(G^\star_\emptyset)^{\circ}_{\Q_p^{\ur}}$.
The Shimura variety $\Sh_{K^\star_p}(G^\star_\emptyset)$ is of PEL-type, and the universal abelian scheme on $\Sh_{K^\star_p}(G^\star_\emptyset)^\circ_{\Q_p^{\ur}}$ is identified with that on $\Sh_{K_p}(G_{\emptyset})^\circ_{\Q_p^{\ur}}$ via the isomorphism above.
Actually, if $K^p\subseteq \GL_2(\AAA_F^{\infty,p})$ is an open compact subgroup  such that $\det(K^p\cap \cO_{F}^{\times})=\det(K^p)\cap \cO_{F}^{\times, +}$,  where $\cO^{\times, +}_{F}$ denotes the set of totally positive units of $F$, then $\Sh_{K^pK_p}(G_{\emptyset})$ is isomorphic to  a finite union of $\Sh_{K^{\star p}K^\star_{p}}(G^\star_{\emptyset})$ for some appropriate tame level structure $K^{\star p}$'s (see \cite[4.2.1]{hida} or \cite[Proposition 2.4]{tian-xiao2}).

We now describe the Hilbert moduli problem that defines an integral canonical model of $\Sh_{K^\star_p}(G_{\emptyset}^\star)$. Let $\widehat{\cO}_F^{(p)}=\prod_{v\nmid p\infty}\cO_{F_v}$, and put $\widehat\Lambda_{F}^{(p)}=\widehat{\cO}_F^{(p)}e_1\oplus \widehat{\cO}_F^{(p)}\gothd_F^{-1} e_{2}$. We endow $\widehat \Lambda_{F}^{(p)}$ with the  symplectic form
\[
\psi_F( a_1 e_1+a_2e_2, b_1e_1+b_2e_2)=\Tr_{F/\Q}(a_2b_1-a_1b_2)\in \widehat{\Z}^{(p)}\quad
\]
{for }$a_1, b_1\in \widehat{\cO}_F^{(p)}$, and $ a_2,b_2\in \gothd_F^{-1}\widehat{\cO}^{(p)}_F$.
It is an elementary fact that every rank two free $\widehat{\cO}_F^{(p)}$-module together with a perfect $\widehat{\Z}^{(p)}$-linear $\cO_F$-hermitian symplectic form is isomorphic to $(\widehat\Lambda^{(p)}_F, \psi_F)$.

Let $K^\star_p=\GL_2(\cO_{F}\otimes \Z_p)\cap G^\star_{\emptyset}(\Q_p)$ be as above. For an  open compact subgroup $K^{\star p}$, we put $K^\star=K^{\star p}K_p^\star$. Assume that $K^{\star p}$ stabilizes the lattice $\widehat \Lambda_F^{(p)}$. We consider the functor that associates to each  locally noetherian $\Z_p$-scheme $S$ the set of isomorphism classes of quadruples $(A,\iota,\lambda, \alpha_{K^{\star p}})$, where
\begin{enumerate}
\item $(A,\iota)$ is  a HBAV, i.e. an abelian scheme $A/S$ of dimension $[F:\QQ]$ equipped with a homomorphism $\iota: \cO_F\hra \End_S(A)$,

\item  $\lambda : A\ra A^\vee$ is an $\cO_F$-linear $\Z_{(p)}^{\times}$-polarization in the sense of \cite[1.3.2.19]{lan},

\item choosing a geometric point $\bar s_i$ for each connected component $S_i$ of $S$,  $\alpha_{K^{\star p}}$ is a collection of $\pi_1(S_i,\bar{s_i})$-invariant $K^{\star p}$-orbits of $\widehat{\cO}_F^{(p)}$-linear isomorphisms $\widehat\Lambda^{(p)}_F\xra{\sim} T^{(p)}(A_{\bar{s_i}})$, sending the symplectic pairing $\psi_F$ on the former to the $\lambda$-Weil pairing on the latter for some identification of $\widehat \ZZ^{(p)}$ with $\widehat \ZZ^{(p)}(1)$.
\end{enumerate}
This functor is representable by a quasi-projective and smooth scheme $\bfSh_{K^\star}(G^\star_{\emptyset})$ over $\Z_p$ such that $\bfSh_{K^\star}(G_{\emptyset}^\star)_{\Q_p}\simeq \Sh_{K^\star}(G_{\emptyset}^\star)$ \cite{rap} and \cite[1.4.1.11]{lan}. By the same arguments of \cite[Corollary~3.8]{moonen96}, it is easy to see that $\bfSh_{K^\star_p}(G_{\emptyset}^\star)$ satisfies the extension property \eqref{S:extension property}.
This then gives rise to an integral canonical model $\bfSh_{K_p}(G_\emptyset)$ of $\Sh_{K_p}(G_\emptyset)$.
We could pull back the universal abelian variety $\calA^\star$ over $\bfSh_{K^\star_p}(G^\star_\emptyset)$ to a family of abelian variety over $\bfSh_{K_p}(G_\emptyset)$ using \cite[4.2.1]{hida} cited above.  But we prefer to do it more canonically  following the same argument as in Subsection~\ref{S:abel var in unitary case}.  More precisely, there is a natural equivariant action of $\widetilde G^\star := G_\emptyset(\Q)^{(p)}_+ \cdot G^\star(\AAA^{\infty, p})$ on the universal abelian variety $\calA^\star$ over $\bfSh_{K^\star_p}(G^\star _\emptyset)$.
Then 
\begin{equation}
\label{E:A from A*}
\calA : = \calA^\star \times_{\widetilde G^\star} \GL_2(\AAA^{\infty,p})
\end{equation}
gives a natural family of abelian variety over $ \bfSh_{K_p}(G_\emptyset)$.

The natural homomorphism $\GL_{2,F}\ra \GL_{2,F}\times_{F^{\times}}E^{\times}$ induces a closed immersion of algebraic groups $G^\star_{\emptyset}\ra G'_{\emptyset}$ compatible with the Deligne's homomorphisms $h_{\emptyset}$ and $h_{\emptyset}'$.
 Therefore, one obtains a map of (projective systems of) Shimura varieties $f: \Sh_{K^\star_p}(G^\star_{\emptyset})\ra \Sh_{K_p'}(G'_{\emptyset})$ which induces an isomorphism of neutral connected component $\Sh_{K^\star_p}(G^\star_{\emptyset})^{\circ}_{\Q_p^{\ur}}\simeq \Sh_{K'_p}(G'_{\emptyset})^{\circ}_{\Q_p^{\ur}}$. We will extend $f$ to  a map of integral models $\bfSh_{K^\star_p}(G^\star_{\emptyset})\ra\bfSh_{K_p'}(G'_{\emptyset})$.

Now, let $E$ be a CM extension of $F$ unramified at $p$ as before. In the process of constructing the pairing $\psi$ on $D_\emptyset$, we may take $\delta_\emptyset$ to be $\big(\begin{smallmatrix}
0 & -1/\sqrt{\gothd} \\ 1/\sqrt{\gothd} &0
\end{smallmatrix}\big)$ which is coprime to $p$, where $\gothd$ is the totally negative element chosen in \ref{S:PEL-Shimura-data}.
It is easy to check that it satisfies the conditions in Lemma~\ref{L:property-PEL-data}(1),
and  the $*$-involution given by $\delta_{\emptyset}$ on $D_\emptyset=\rmM_2(E)$ is given by $\big(\begin{smallmatrix}
a&b\\c&d
\end{smallmatrix}\big) \mapsto \big(\begin{smallmatrix}
\bar a&\bar c\\\bar b&\bar d
\end{smallmatrix}\big)$ for $a,b, c, d \in E$.
The $\ast$-hermitian pairing on $\rmM_2(E)$ is given by
\begin{align*}
\psi(v, w) &= \Tr_{\rmM_2(E) / \QQ}\Big( v \bar w
\big(\begin{smallmatrix}
0 & -1 \\ 1 &0
\end{smallmatrix}\big)
\Big), \textrm{ for } v = \big(\begin{smallmatrix}
a_v & b_v \\ c_v & d_v
\end{smallmatrix}\big) \textrm{ and } w = \big(\begin{smallmatrix}
a_w & b_w \\ c_w & d_w
\end{smallmatrix}\big) \in \rmM_2(E)\\
&=\Tr_{E/\QQ} \big(
b_v \bar a_w -a_v \bar b_w + d_v \bar c_w - c_v \bar d_w \big).
\end{align*}
In defining the PEL data for $G'_\emptyset$, we  take the $\calO_{D_\emptyset} $-lattice $\Lambda$ to be
$
\begin{pmatrix}
\calO_E & \gothd_F^{-1}\calO_E \\
\calO_E & \gothd_F^{-1}\calO_E
\end{pmatrix}
$. Clearly $\widehat\Lambda^{(p)}=\Lambda\otimes_{\Z}\widehat{\Z}^{(p)}$ satisfies $\widehat\Lambda^{(p)} \subseteq \widehat\Lambda^{(p),\vee}$ for the bilinear form $\psi$ above.
Moreover, if we equip $\Lambda_{F}^{(p)} \otimes_{\calO_F} \calO_E$ with the symplectic form $\psi_{E}=\psi_F(\Tr_{E/F}(\bullet), \Tr_{E/F}(\bullet))$, then $(\widehat\Lambda^{(p)},\psi) $ is isomorphic to $ ((\widehat{\Lambda}_{F}^{(p)} \otimes_{\calO_F} \calO_E)^{\oplus 2}, \psi_E^{\oplus 2})$ as a $\ast$-hermitian symplectic $\rmM_2(\calO_E)$-module.

\begin{prop}
\label{P:integral-HMV-unitary}
For any open compact subgroup $K'^p$ of $G'_\emptyset(\AAA^{\infty, p})$, we put $K^{\star p} = K'^p \cap G^\star_\emptyset(\AAA^{\infty, p})$. Then we have a canonical morphism
$$
\bff: \bfSh_{K^{\star p}K^\star_p}(G^\star _\emptyset) \ra \bfSh_{K'^pK'_p}(G'_\emptyset)
$$
such that, if $\calA$ and $\bfA'_\emptyset$ denote respectively  the universal abelian scheme  on $\bfSh_{K^{\star p}K^\star _p}(G^\star _{\emptyset})$ and  that  on $\bfSh_{K'^pK'_p}(G'_{\emptyset})$, then we have an isomorphism of abelian schemes $\bff^*\bfA'_\emptyset \cong (\calA\otimes_{\cO_F}\cO_E)^{\oplus 2}$ compatible with the natural action of $\rmM_2(\cO_E)$ and polarizations on both sides.
By passing to the limit, the morphism $\bff$ induces an isomorphism  between the integral models of connected Shimura varieties $\bfSh_{K_p^\star }(G^\star _\emptyset)_{\Zp^\ur}^\circ \cong \bfSh_{K_p'}(G'_\emptyset )_{\Zp^\ur}^\circ$.

\end{prop}
\begin{proof}
By Galois descent, it is enough to work over $W(k_0)$ for $k_0$ in Theorem~\ref{T:unitary-shimura-variety-representability}.
Let $S$ be a locally noetherian $W(k_0)$-scheme, and $x=(A,\iota, \lambda, \alpha_{K^{\star p}})$ be an $S$-valued point of $\bfSh_{K^{\star p}K^\star _p}(G_{\emptyset}^\star )$. We define its image  $f(x)=(A', \iota', \lambda', \alpha_{K'^p})$ as follows.  We take $A'=(A\otimes_{\cO_F}\cO_E)^{\oplus 2}$ equipped with the naturally induced action $\iota'$ of $\rmM_2(\cO_E)$. It is clear that $\Lie(A')_{\tilde\tau}$  is an $\cO_{S}$-module locally free of rank $2$ for all $\tilde\tau\in \Sigma_{E,\infty}$. The   prime-to-$p$ polarization $\lambda'$ on $A'$ is defined to be
\[
\lambda': A'\xra{\sim} (A\otimes_{\cO_F}\cO_E)^{\oplus 2}\xra{(\lambda\otimes 1)^{\oplus 2}} (A^\vee\otimes_{\cO_F}\cO_E)^{\oplus 2}\cong A'^\vee \otimes \delta_{E/F} \to A'^\vee,
\]
where $\delta_{E/F}$ is the different ideal of $E$ over $F$. We define the $K'^p$-level structure  to be the $K'^p$-orbit of the isomorphism
\[
\alpha_{K'^p}\colon \widehat{\Lambda}^{(p)}\xra {\cong} (\widehat{\Lambda}^{(p)}_F\otimes_{\cO_F}\cO_E)^{\oplus 2}\xra{\alpha_{K^{\star p}}^{\oplus 2}} \big (T^{(p)}(A_{\bar s})\otimes_{\cO_F}\cO_E\big)^{\oplus 2}\simeq T^{(p)}(A'_{\bar s}).
\]
By the discussion before the Proposition, it is clear that $\alpha_{K'^p}$ sends the symplectic form $\psi$ on the left hand side to the $\lambda'$-Weil pairing on the right. This defines the morphism $\bff$ from $\bfSh_{K^{\star p}K^\star _p}(G^\star _{\emptyset})$ to $\bfSh_{K'^pK'_p}(G'_{\emptyset})$. By looking at the complex uniformization, we  note that $\bff$ extends the morphism $f: \Sh_{K^{\star p}K^\star _p}(G^\star _{\emptyset})_{\Q_p}\ra \Sh_{K'^pK'_p}(G'_{\emptyset})_{\Q_p}$ defined previously by group theory. Since both $\bfSh_{K^\star _p}(G^\star _{\emptyset})$ and $\bfSh_{K'_{p}}(G'_{\emptyset})$ satisfy the extension property \ref{S:extension property}, it follows that $\bff$ induces an isomorphism   $\bfSh_{K_p^\star }(G^\star _\emptyset)_{\Zp^\ur}^\circ \simeq \bfSh_{K_p'}(G'_\emptyset )_{\Zp^\ur}^\circ$.
\end{proof}

   \begin{cor}\label{C:integral-HMV-unitary}
Let $\calA$ denote the universal HBAV over $\bfSh_{K_p}(G_{\emptyset})$, and  $\bfA''_\emptyset$ be the family of abelian varieties over $\bfSh_{K''_p}(G''_{\emptyset})$  defined in Subsection~\ref{S:abel var in unitary case}.  Then under the natural morphisms of Shimura varieties
\begin{equation}
\label{E:morphisms of Shimura varieties HMV}
\bfSh_{K_p}(G_\emptyset) \xleftarrow{\ \mathbf{pr}_1\ }
\bfSh_{K_p \times K_{E,p}}
(G_\emptyset \times T_{E, \emptyset}) \xrightarrow{\ \bbalpha \ }
\bfSh_{K''_p}(G''_\emptyset),
\end{equation}
one has an isomorphism of abelian schemes over $\bfSh_{K_p \times K_{E,p}}
(G_\emptyset \times T_{E, \emptyset})$
\begin{equation}
\label{E:comparison abelian varieties over HMV and unitary}
\bbalpha^*\bfA''_\emptyset \cong (\mathbf{pr}_1^*\calA\otimes_{\cO_F}\cO_E)^{\oplus 2}
\end{equation}
compatible with the action of $\rmM_2(\cO_E)$ and prime-to-$p$ polarizations.
   \end{cor}
\begin{proof}
This follows from the constructions of $\calA$ and $\bfA''_\emptyset$ and the proposition above.
\end{proof}

\subsection{Comparison of the twisted partial Frobenius}

Keep the notation as in Subsection~\ref{S:comparison-Hilbert}.
The Shimura variety $\bfSh_{K^\star }(G_\emptyset^\star )_{\Fp}$ also admits a twisted partial Frobenius $\Phi_{\gothp^2}$ for each $\gothp \in \Sigma_p$ which we define as follows.
Let $S$ be a locally noetherian $ \FF_p$-scheme.  Given an $S$-point $(A, \iota, \lambda, \alpha_{K^{\star p}})$ of $\bfSh_{K^\star }(G_\emptyset^\star)_{\Fp}$, we associate a new point $(A', \iota', \lambda', \alpha'_{K^{\star p}})$:
\begin{itemize}
\item
$A' = A / \Ker_{\gothp^2} \otimes_{\calO_F} \gothp$, where $\Ker_{\gothp^2}$ is the $\gothp$-component of the kernel of relative Frobenius homomorphism $\Fr^2_A: A \to A^{(p^2)}$; it is equipped with the induced $\calO_F$-action $\iota'$;
\item
using the natural quasi-isogeny $\eta: A \to A'$,
$\lambda'$ is given by the composite of quasi-isogenies $A' \xleftarrow{\eta} A \xrightarrow{\lambda} A^\vee \xrightarrow{\eta^\vee} A'^\vee$ (which is a $\ZZ_{(p)}^\times$-isogeny by the same argument as in Subsection~\ref{S:partial Frobenius});
\item
$\bar \alpha'_{K^{\star p}}$ is
the composite $\widehat\Lambda_F^{(p)} \xrightarrow{\bar \alpha_{K^{\star p}}} T^{(p)}(A) \xleftarrow{\eta} T^{(p)}(A')$.
\end{itemize}

The construction above gives rise to a \emph{twisted partial Frobenius morphism}
\[
\gothF^\star_{\gothp^2}: \bfSh_{K^\star }(G^\star _\emptyset)_{\FF_p} \longrightarrow
\bfSh_{K^\star }(G^\star _\emptyset)_{\FF_p} \quad \textrm{and} \quad
\eta_{\gothp^2}: \calA_{\FF_p} \to (\gothF^\star_{\gothp^2})^*\calA_{\FF_p}.
\]
Using the formalism of Shimura varieties (Corollary~\ref{C:Sh(G)^circ_Zp independent of G} and more specifically \eqref{E:A from A*}), it gives rise to a \emph{twisted partial Frobenius morphism}
\[
\gothF^\emptyset_{\gothp^2}: \bfSh_{K}(G_\emptyset)_{\FF_p} \longrightarrow
\bfSh_{K}(G_\emptyset)_{\FF_p} \quad \textrm{and} \quad
\eta_{\gothp^2}^\emptyset: \calA_{\FF_p} \to (\gothF^\emptyset_{\gothp^2})^*\calA_{\FF_p}.
\]

\begin{cor}
The 
twisted partial Frobenius morphism $\gothF''_{\gothp^2}$ on $\bfSh_{K''_p}(G''_\emptyset)_{\FF_p}$ and the twisted partial Frobenius $\gothF^\emptyset_{\gothp^2}$ on $\bfSh_{K_p}(G_\emptyset)_{\Fp}$ are compatible, in the sense that there exists a morphism $\tilde \gothF_{\gothp^2}$ so that both squares in the commutative diagram are Cartesian.
\[
\xymatrix{
\bfSh_{K_p}(G_\emptyset)_{\Fp}
\ar[d]^{\gothF^\emptyset_{\gothp^2}} & \ar[l]_-{\mathbf{pr}_1}
\bfSh_{K_p \times K_{E,p}}
(G_\emptyset \times T_{E, \emptyset})_{\Fp} \ar[r]^-{ \bbalpha  }
\ar[d]^{\tilde \gothF_{\gothp^2}}
 &
\bfSh_{K''_p}(G''_\emptyset)_{\Fp}
\ar[d]^{\gothF''_{\gothp^2}}
\\
\bfSh_{K_p}(G_\emptyset)_{\Fp} & \ar[l]_-{\mathbf{pr}_1}
\bfSh_{K_p \times K_{E,p}}
(G_\emptyset \times T_{E, \emptyset})_{\Fp} \ar[r]^-{ \bbalpha  } &
\bfSh_{K''_p}(G''_\emptyset)_{\Fp}
}
\]
Moreover, $\eta^\emptyset_{\gothp^2}$ is compatible with $\eta''_{\gothp^2}$ in the sense that the following diagram commutes.
\[
\xymatrix{
\bbalpha^*\bfA''_{\emptyset,\FF_p} \ar[d]^{\bbalpha^*(\eta'' _{{\gothp^2}})} \ar[r]^-{\eqref{E:comparison abelian varieties over HMV and unitary}} & (\mathbf{pr}_1^* \calA_{\FF_p} \otimes_{\calO_F} \calO_E)^{\oplus 2} \ar[d]_{\eta^\emptyset_{\gothp^2} \otimes 1}\\
\bbalpha^*\gothF''^*_{\gothp^2} \bfA''_{\emptyset, \FF_p}
\ar[r]^-{\eqref{E:comparison abelian varieties over HMV and unitary}} & \big(\mathbf{pr}_1^*\gothF_{\gothp^2}^{\emptyset,*}( \calA_{\FF_p}) \otimes_{\calO_F} \calO_E \big)^{\oplus 2}
}
\]
\end{cor}
\begin{proof}
This follows from the definition of the partial Frobenii in various situations and the comparison Proposition~\ref{P:integral-HMV-unitary} above.
\end{proof}

\section{Goren-Oort Stratification}
\label{Section:defn of GOstrata}

We define an analog of the Goren-Oort stratification on the special fibers of quaternionic Shimura varieties.  This is first done for unitary Shimura varieties and then pulled back to the quaternionic ones.  Unfortunately, the definition a priori depends on the auxiliary choice of CM field (as well as the signatures $s_{\tilde \tau}$).
In the case of Hilbert modular variety, we show that our definition of the GO-strata agrees with Goren-Oort's original definition  in \cite{goren-oort} (and hence does not depend on the auxiliary choice of data).

\subsection{Notation}
\label{S:GO-notation}
Keep  the notation as in the previous sections.
Let $k_0$ be a finite extension of $\Fp$ containing all residue fields of $\calO_E$ of characteristic $p$.
Let $X':=\bfSh_{K'}(G'_{\tilde \ttS})_{k_0}$ denote the base change  to $k_0$ of the Shimura variety $\bfSh_{K'}(G'_{\tilde \ttS})$ considered in Theorem~\ref{T:unitary-shimura-variety-representability}.

Recall that $\gothe \in \calO_{D_\ttS,p}$ corresponds to $\big(
\begin{smallmatrix}
1&0\\0&0
\end{smallmatrix}
\big)$ when identifying  $\calO_{D_\ttS,p}$ with $\rmM_2(\calO_{E,p})$.
For an abelian scheme $A$ over a locally noetherian $k_0$-scheme $S$ carrying an action of $\cO_{D_{\ttS}}$, we have the reduced module of invariant differential 1-forms $\omega_{A/S}^{\circ}$, the reduced Lie algebra $\Lie(A/S)^\circ$, and the reduced de Rham homology $H^{\dR}_1(A/S)^{\circ}$ defined in Subsection~\ref{N:notation-reduced}. Their $\tilde{\tau}$-components $\omega_{A/S, \tilde{\tau}}^{\circ}$, $\Lie(A/S)_{\tilde{\tau}}^{\circ}$ and $H_1^{\dR}(A/S)^{\circ}_{\tilde{\tau}}$ for $\tilde{\tau}\in \Sigma_{E,\infty}$, fit in an exact sequence, called the \emph{reduced Hodge filtration}, 
\[
0\ra \omega_{A^\vee/S,\tilde{\tau}}^{\circ}\ra H^{\dR}_1(A/S)^{\circ}_{\tilde{\tau}}\ra \Lie(A/S)^\circ_{\tilde{\tau}}\ra 0.
\]
Let $A^{(p)}$ denote the base change of $A$ via the absolute Frobenius on $S$.  The Verschiebung  $\Ver:A^{(p)}\ra A$ and the Frobenius morphism $\Fr : A\ra A^{(p)}$ induce respectively  maps of coherent sheaves on $S$:
\[
F_A: H_1^{\dR}(A/S)^{\circ, (p)}\ra H_1^{\dR}(A/S)^{\circ}\quad \text{and } V_A: H_1^{\dR}(A/S)^{\circ}\ra H_1^{\dR}(A/S)^{\circ, (p)},
\]
which are compatible with the action of $\cO_E$. Here, for a coherent $\cO_S$-module,  $M^{(p)}$ denotes the base change $M\otimes_{\cO_S, F_{\mathrm{abs}}} \cO_S$.   If there is no confusion, we drop the subscript $A$ from the notation and simply write $F$ and $V$ for the two maps. (Although the letter $F$ also stands for the totally real field, we think this should not cause confusion.) Moreover, we have 
\[
\Ker(F)=\im(V)=(\omega_{A^\vee/S}^{\circ})^{(p)}\quad \textrm{ and } \quad \im(F)=\Ker(V)\cong \Lie(A^{(p)}/S)^{\circ}.\]

 
Let $(A, \iota, \lambda, \alpha_{K'})$ be an $S$-valued point of $X'=\bfSh_{K'}(G'_{\tilde \ttS})_{k_0}$. By Kottwitz' determinant condition \ref{T:unitary-shimura-variety-representability}(a), for each $\tilde{\tau}\in \Sigma_{E,\infty}$, $\Lie(A/S)^{\circ}_{\tilde{\tau}}$ is a locally free $\cO_{S}$-module of rank $s_{\tilde{\tau}}$.  (The numbers $s_{\tilde \tau}$ are defined as in Subsection~\ref{S:CM extension}.)
By duality, this implies that $\omega_{A^\vee/S,\tilde{\tau}}^{\circ}$ is locally  free of rank  $s_{\tilde{\tau}^c}=2-s_{\tilde{\tau}}$.
Moreover, when $\tau \in \Sigma_{\infty/\gothp}$ with $\gothp$ not of type $\beta^\sharp$, the universal polarization $\lambda$ induces an isomorphism of locally free $\cO_{S}$-modules
\begin{equation}\label{Equ:duality-omega}
\omega_{A^\vee/S,\tilde{\tau}}^\circ\cong \omega_{A/S,\tilde{\tau}^c}^\circ.
\end{equation}

\subsection{Essential Frobenius and essential Verschiebung}
\label{N:essential frobenius and verschiebung}
We now define two very important morphisms: essential Frobenius and essential Verschiebung; we will often encounter later their variants for crystalline homology and Dieudonn\'e modules, for which we shall simply refer to the similar construction given here.

Let $(A,\iota, \lambda, \alpha_{K'})$ be as above. 
For $\tilde \tau \in \Sigma_{E, \infty}$ lifting a place $\tau \in \ttS_\infty$,
we define the \emph{essential Frobenius} to be
\begin{align}
\label{E:defintion of Fes}
F_\es =F_{A,\es, \tilde \tau}: (H^{\dR}_1(A/S)_{\sigma^{-1}\tilde{\tau}}^{\circ})^{ (p)}&=H_1^{\dR}(A^{(p)}/S)^{\circ}_{\tilde{\tau}}\longrightarrow H_1^{\dR}(A/S)^{\circ}_{\tilde{\tau}}
\\
\nonumber
x&\longmapsto{
\left\{
\begin{array}{ll}
F(x) & \textrm{when }s_{\sigma^{-1} \tilde \tau} = 1\textrm{ or }2;\\
V^{-1}(x) & \textrm{when }s_{\sigma^{-1} \tilde \tau} = 0.\\
\end{array}
\right.}
\end{align}
Note that in the latter case, the morphism $V: H_1^{\dR}(A/S)^{\circ}_{\tilde{\tau}} \xra{\sim} H_1^{\dR}(A^{(p)}/S)^{\circ}_{\tilde{\tau}}$ is an isomorphism by Kottwitz' determinant condition.

Similarly, we define the \emph{essential Verschiebung} to be 
\begin{align}
\label{E:definition of Ves}
V_\es =V_{A,\es, \tilde \tau}: H_1^{\dR}(A/S)^{\circ}_{\tilde{\tau}}&\longrightarrow
H_1^{\dR}(A^{(p)}/S)^{\circ}_{\tilde{\tau}}
=(H_1^{\dR}(A/S)_{\sigma^{-1}\tilde{\tau}}^{\circ})^{(p)}
\\
\nonumber
x&\longmapsto{
\left\{
\begin{array}{ll}
V(x) & \textrm{when }s_{\sigma^{-1} \tilde \tau} = 0\textrm{ or }1;\\
F^{-1}(x) & \textrm{when }s_{\sigma^{-1} \tilde \tau} = 2.\\
\end{array}
\right.}
\end{align}
Here, in the latter case, the morphism $F: H_1^{\dR}(A^{(p)}/S)^{\circ}_{\tilde{\tau}} \to H_1^{\dR}(A/S)^{\circ}_{\tilde{\tau}}$ is an isomorphism.

When no confusion arises, we may suppress the subscript $A$ and/or $\tilde \tau$ from  $F_{A, \es, \tilde\tau}$ and $V_{A,\es,\tilde\tau}$.

Thus, if $s_{\sigma^{-1}\tilde \tau} =0$ or $2$, both $F_{\es, \tilde \tau}:H_1^{\dR}(A^{(p)}/S)^{\circ}_{\tilde{\tau}}\ra H_1^{\dR}(A/S)^{\circ}_{\tilde{\tau}}$ and $V_{\es, \tilde \tau}: H_1^{\dR}(A/S)^{\circ}_{\tilde{\tau}}\ra  H_1^{\dR}(A^{(p)}/S)^{\circ}_{\tilde{\tau}}$  are isomorphisms, and both $F_{\es, \tilde \tau}V_{\es, \tilde \tau}$ and $V_{\es, \tilde \tau}F_{\es, \tilde \tau}$ are isomorphisms.  
When $s_{\sigma^{-1}\tilde \tau}=1$, we usually prefer to write the usual Frobenius and Verschiebung.

We will also use composites of Frobenii and Verschiebungs:
\begin{align}
\label{E:Ves n}
V_{\es, \tilde \tau}^n: &H_1^\mathrm{dR}(A/S)^\circ_{\tilde \tau}\xra{V_{\es, \tilde \tau}} H^{\dR}_1(A^{(p)}/S)_{\tilde{\tau}}^{\circ}\xra{V_{\es, \sigma^{-1}\tilde \tau}^{(p)}} 
\cdots \xra{V_{\es, \sigma^{1-n}\tilde \tau}^{(p^{n-1})}}
H^{\dR}_1(A^{(p^{n})}/S)_{\tilde{\tau}}^{\circ},
\\
\label{E:Fes n}
F_{\es, \tilde \tau}^n:& H^{\dR}_1(A^{(p^{n})}/S)_{\tilde{\tau}}^{\circ}
\xra{F_{\es, \sigma^{1-n} \tilde \tau}^{(p^{n-1})}}
 H^{\dR}_1(A^{(p^{n-1})}/S)_{\tilde{\tau}}^{\circ} \xra{F_{\es, \sigma^{2-n}\tilde \tau}^{(p^{n-2})}}
  \cdots \xra{F_{\es,\tilde\tau}}
 H_1^\mathrm{dR}(A/S)^\circ_{\tilde \tau}.
\end{align}

Suppose now $S=\Spec(k)$ is the spectrum of a perfect field of characteristic $p>0$. Let $\tcD_{A}$ denote the \emph{covariant} Dieudonn\'e module of $A[p^\infty]$. We have a canonical decomposition $\tcD_A=\bigoplus_{\tilde\tau\in \Sigma_{E,\infty}}\tcD_{A,\tilde\tau}$. 
We put $\tcD_{A,\tilde\tau}^{\circ}=\gothe\cdot \tcD_{A,\tilde\tau}$. Then we  define the essential Frobenius and  essential Verschiebung
\[
F_\es = F_{A,\es, \tilde \tau}: \tcD^{\circ}_{A,\sigma^{-1}\tilde\tau}\ra \tcD^{\circ}_{A,\tilde\tau}\quad \text{and}\quad V_{\es} = V_{A, \es, \tilde \tau}:\tcD_{A,\tilde\tau}^{\circ}\ra \tcD^{\circ}_{A,\sigma^{-1}\tilde\tau}
\]
 in the same way as in \eqref{E:defintion of Fes} and \eqref{E:definition of Ves} for $H_1^{\dR}(A/S)^{\circ}_{\tilde\tau}$. 
The morphisms $F_{A,\es, \tilde \tau}$ and $V_{A,\es, \tilde \tau}$ on $H_1^{\dR}(A/S)^{\circ}_{\tilde\tau}$ can be recovered from those on $\tcD^{\circ}_{A,\tilde\tau}$ by reduction modulo $p$.


\begin{notation}
\label{N:n tau}
For $\tau \in \Sigma_\infty-\ttS_\infty$, we define $n_{\tau} = n_{\tau, \ttS}\geq 1$ to be the  integer such that $\sigma^{-1}\tau, \dots, \sigma^{-n_\tau+1}\tau \in \ttS_\infty$ and $\sigma^{-n_{\tau}}\tau\notin \ttS_{\infty}$.
\end{notation}

\subsection{Partial Hasse invariants}\label{S:partial-Hasse}

For each $\tilde \tau$ lifting a place $\tau \in \Sigma_\infty-\ttS_\infty$, we must have $s_{\tilde \tau} = 1$. So in definition of $V_{\es, \tilde \tau}^{n_\tau}$  in \eqref{E:Ves n}, all morphisms are isomorphisms except the last one. Similarly, in the definition
 of $F_{\es, \tilde \tau}^{n_\tau}$  in \eqref{E:Fes n}, all morphisms are isomorphisms except the first one.
It is clear that $V_{\es, \tilde \tau}^{n_\tau} F_{\es, \tilde \tau}^{n_\tau} =F_{\es, \tilde \tau}^{n_\tau} V_{\es, \tilde \tau}^{n_\tau}=0$, coming from the composition of $V^{(p^{n_\tau-1})}_{\sigma^{n_\tau-1}\tilde \tau}$ and $F^{(p^{n_\tau-1})}_{\sigma^{n_\tau-1}\tilde \tau}$ in both ways. 
Note  also that the cokernels of $V_{\es, \tilde \tau}^{n_{\tau}}$ and $F_{\es, \tilde \tau}^{n_{\tau}}$ are both locally free $\cO_{X'}$-modules  of rank $1$.

The restriction of $V_{\es, \tilde \tau}^{n_\tau}$ to the line bundle $\omega^{\circ}_{A^\vee/S,\tilde{\tau}}$ induces a homomorphism
\begin{equation*}
h_{\tilde{\tau}}(A):
\omega_{A^\vee/S,\tilde{\tau}}^{\circ}\lra \omega^\circ_{A^{\vee, (p^{n_{\tau}})}/S,\tilde{\tau}}=(\omega^{\circ}_{A^\vee/S,\sigma^{-n_{\tau}}\tilde{\tau}})^{\otimes p^{n_{\tau}}}.
\end{equation*}
Applied to the universal case, this gives rise to a global section
\begin{equation}\label{Equ:partial-hasse}
h_{\tilde{\tau}}\in \Gamma(X', (\omega^{\circ}_{\bfA'^\vee/X', \sigma^{-n_{\tau}}\tilde{\tau}})^{\otimes p^{n_{\tau}}}\otimes( \omega_{\bfA'^\vee/X', \tilde{\tau}}^{\circ})^{\otimes (-1)}),
\end{equation}
where $\bfA'$ is the universal abelian scheme over $X'=\bfSh_{K'}(G'_{\tilde \ttS})_{k_0}$.
We call $h_{\tilde \tau}$ \emph{the $\tilde{\tau}$-partial Hasse invariant}.
With $\tilde{\tau}$ replaced by $\tilde{\tau}^c$ everywhere, we can define similarly a partial Hasse invariant $h_{\tilde{\tau}^c}$.
They are  analogs  of the partial Hasse invariants in the unitary case.

\begin{lemma}\label{Lemma:partial-Hasse}
Let $(A,\iota, \lambda, \bar{\alpha}_{K'})$ be an $S$-valued point of $X'$ as above.  Then the following statements are equivalent for $\tilde \tau$ lifting $\tau \in \Sigma_\infty - \ttS_\infty$:
\begin{enumerate}
\item We have $h_{\tilde{\tau}}(A)=0$.

\item The image of $F_{\es, \tilde \tau}^{n_\tau}: H_{1}^{\dR}(A^{(p^{n_{\tau}})}/S)^{\circ}_{\tilde{\tau}}\ra H_1^{\dR}(A/S)_{\tilde{\tau}}^{\circ}$ is $\omega_{A^\vee/S,\tilde{\tau}}^{\circ}$.

\item We have $h_{\tilde{\tau}^c}(A)=0$.

\item The image of $F_{\es, \tilde \tau^c}^{n_\tau}:H_{1}^{\dR}(A^{(p^{n_{\tau}})}/S)^{\circ}_{\tilde{\tau}^c}\ra H_1^{\dR}(A/S)_{\tilde{\tau}^c}^{\circ}$ is $\omega_{A^\vee/S, \tilde{\tau}^c}^{\circ}$.
\end{enumerate}

\end{lemma}

\begin{proof} The equivalences $(1)\Leftrightarrow(2)$ and $(3)\Leftrightarrow (4)$ follow from the fact that the image of $F$ coincides with the kernel of $V$. We prove now $(2)\Leftrightarrow (4)$. Let $\gothp\in \Sigma_{p}$ be the prime above $p$  so that $\tau\in \Sigma_{\infty/\gothp}$.  Since $\Sigma_{\infty/\gothp}\neq \ttS_{\infty/\gothp}$, 
$\gothp$ can not be of type $\beta^\sharp$ by Hypothesis \ref{H:B_S-splits-at-p}. We consider the following diagram:
\[
\xymatrix@C=5pt{
H_1^\dR(A^{(p^{n_{\tau}})}/S)^\circ_{\tilde{\tau}}
\ar@/^10pt/[d]^{F_{\es, \tilde \tau}^{n_\tau}}
& \times & H_1^\dR(A^{(p^{n_{\tau}})}/S)^{\circ}_{\tilde{\tau}^c}
\ar@/_10pt/[d]_{F_{\es, \tilde \tau^c}^{n_\tau}}
 \ar[rrr]^-{\langle\ , \ \rangle} &&& \calO_{S}
 \\
H_1^\dR(A/S)^{\circ}_{\tilde{\tau}}
\ar@/^10pt/[u]^{V_{\es, \tilde \tau}^{n_\tau}}
&\times&
H_1^\dR(A/S)^{\circ}_{\tilde{\tau}^c}
\ar@/_10pt/[u]_{V_{\es, \tilde \tau^c}^{n_\tau}}
\ar[rrr]^-{\langle\ , \ \rangle} &&& \calO_{S}, \ar@{=}[u]
}
\]
where the pairings $\langle\ , \ \rangle$ are induced by the polarization $\lambda$, and they are perfect because $\gothp$ is not of type $\beta^\sharp$. We have $\langle F_{\es, \tilde \tau}^{n_\tau} x, y\rangle=\langle x,V_{\es, \tilde \tau^c}^{n_\tau} y \rangle^{\sigma^{n_\tau}}$. It follows that
\[
(\omega_{A^\vee/S,\tilde{\tau}}^{\circ})^\perp=\omega_{A^\vee/S,\tilde{\tau}^c}^{\circ},\quad
\text{and}\quad \im(F_{\es, \tilde \tau}^{n_\tau})^\perp=\im(F_{\es, \tilde \tau^c}^{n_\tau}),\]
where $\perp$ means the orthogonal complement under $\langle\ ,\ \rangle$. Therefore, we have
\begin{align*}
& \text{(2) } \omega_{A^\vee/S,\tilde{\tau}}^{\circ}=\im(F_{\es, \tilde \tau}^{n_\tau})
\Longleftrightarrow  (\omega_{A^\vee/S,\tilde{\tau}}^{\circ})^\perp=\im(F_{\es, \tilde \tau}^{n_\tau}
)^\perp
\Longleftrightarrow   \text{(4) } \omega_{A^\vee/S, \tilde{\tau}^c }^{\circ}= \im (F_{\es, \tilde \tau^c}^{n_\tau}).
\end{align*}

\end{proof}

\begin{defn}\label{Defn:GO-strata}
 We fix a section $ \tau \mapsto \tilde{\tau}$ of the natural restriction map $\Sigma_{E,\infty}\ra \Sigma_{\infty}$. Let  $\ttT\subset \Sigma_{\infty}-\ttS_{\infty}$ be a subset.
We put $\bfSh_{K'}(G'_{\tilde \ttS})_{k_0,\emptyset}=X'$, and $X'_\ttT : =\bfSh_{K'}(G'_{\tilde \ttS})_{k_0, \ttT}$ to be the closed subscheme of $\bfSh_{K'}(G'_{\tilde \ttS})_{k_0}$ defined as the vanishing locus of $\{h_{\tilde{\tau}}: \tau\in \ttT\}$. Passing to the limit, we put 
 \[
 \bfSh_{K'_p}(G'_{\tilde \ttS})_{k_0,\ttT}:=\varprojlim_{K'^p}\bfSh_{K'^pK'_p}(G'_{\tilde \ttS})_{k_0,\ttT}
 \]
 We call $\{\bfSh_{K'}(G'_{\tilde \ttS})_{k_0,\ttT}: \ttT\subset \Sigma_{\infty}-\ttS_{\infty}\}$ (resp. $\{\bfSh_{K'_p}(G'_{\tilde \ttS})_{k_0,\ttT}: \ttT\subset \Sigma_{\infty}-\ttS_{\infty}\}$) the \emph{Goren-Oort stratification} (or GO-stratification for short) of $\bfSh_{K'}(G'_{\tilde \ttS})_{k_0}$ (resp. $\bfSh_{K'_p}(G'_{\tilde \ttS})_{k_0}$).

\end{defn}

By Lemma~\ref{Lemma:partial-Hasse}, the GO-strata $X'_\ttT$ do not depend on the choice of the section $\tau\mapsto \tilde{\tau}$. 

\begin{prop}\label{Prop:smoothness}
 For any subset $\ttT\subseteq \Sigma_\infty - \ttS_\infty$, the closed GO-stratum $X'_\ttT\subseteq X'$ is smooth of codimension $\#\ttT$, and the tangent bundle $\calT_{X'_\ttT}$ is the subbundle
\[
\bigoplus_{\tau\in \Sigma_\infty - (\ttS_\infty \cup \ttT)}
\bigl( \Lie(\bfA')^{\circ}_{\tilde{\tau}}\otimes \Lie(\bfA')_{\tilde{\tau}^c}^{\circ}\bigr)|_{X'_\ttT}\subseteq \bigoplus_{\tau\in \Sigma_\infty - \ttS_\infty}\bigl(\Lie(\bfA')^\circ_{\tilde{\tau}}\otimes
\Lie(\bfA')^\circ_{\tilde{\tau}^c}\bigr)|_{X'_\ttT},
\]
where the latter is identified with the restriction to $X'_\ttT$ of the tangent bundle of $X'$ computed in  Corollary~\ref{C:deformation}.
Moreover, $X'_\ttT$ is proper if $\ttS_{\infty}\cup \ttT$ is non-empty.
\end{prop}
\begin{proof}
 We follow the same strategy as in \cite[Proposition 3.4]{helm}. First,  the same argument as \cite[Lemma 3.7]{helm} proves the non-emptyness of $X'_{\ttT}$.
We now proceed as in the proof of Corollary~\ref{C:deformation}. Let $S_0\hra S$ be a closed immersion of locally noetherian $k_0$-schemes whose ideal of definition $\calI$ satisfies $\calI^2=0$.  Consider an $S_0$-valued point $x_0=(A_0, \iota_0, \lambda_0, \bar{\alpha}_{K'})$ of $X'_{\ttT}$. To prove the smoothness of $X'_{\ttT}$, it suffices to show that, locally for the Zariski topology on $S_0$, there exists $x\in X'_{\ttT}(S)$ lifting  $x_0$. By Lemma \ref{Lemma:partial-Hasse}, we have, for every $\tau \in \ttT$,
$$
\omega_{A_0^\vee/S_0,\tilde{\tau}}^{\circ}=F_{\es, \tilde \tau}^{n_\tau}
(H_1^{\dR}(A_0^{(p^{n_{\tau}})}/S_0)^{\circ}_{\tilde{\tau}}).
$$
The reduced ``crystalline homology'' $H_1^{\cris}(A_0/S_0)_{S}^{\circ}$ is equipped with natural operators $F$ and $V$, lifting the corresponding operators on $H_1^{\dR}(A_0/S_0)^{\circ}$.
We define the composite of essential Frobenius
\[
\tilde{F}_{\es, \tilde \tau}^{n_\tau}: H_1^{\cris}(A_0^{(p^{n_{\tau}})}/S_0)^{\circ}_{S, \tilde{\tau}}\ra H_1^{\cris}(A_0/S_0)^{\circ}_{S, \tilde{\tau}}
\]
in the same manner as $F_{\es, \tilde \tau}^{n_\tau}$ on $H^{\dR}_1(A_0^{(p^{n_{\tau}})}/S_0)^{\circ}_{\tilde{\tau}}$ in Notation~\ref{N:essential frobenius and verschiebung}.
Let $\tilde{\omega}_{A_0^\vee/S_0, \tilde{\tau}}^{\circ}$ denote the image of $\tilde{F}_{\es, \tilde \tau}^{n_\tau}$ for $\tau\in \ttT$. This is a local direct factor of $H_1^{\cris}(A_0/S_0)_{S, \tilde{\tau}}^{\circ}$ that lifts $\omega_{A_0^\vee/S_0, \tilde{\tau}}^{\circ}$.  As in the proof of Theorem~\ref{T:unitary-shimura-variety-representability}, specifying a deformation $x\in X'(S)$ of $x_0$ to $S$ is equivalent to giving a local direct summand $\omega_{S,\tilde{\tau}}^{\circ}\subseteq H_1^{\cris}(A_0/S_0)_{S, \tilde{\tau}}^{\circ}$ that lifts $\omega_{A_0^\vee/S_0, \tilde{\tau}}^{\circ}$ for each $\tau \in \Sigma_{\infty}-\ttS_{\infty}$. By Lemma \ref{Lemma:partial-Hasse}, such a deformation $x$ lies in $X'_{\ttT}$ if and only if $\omega_{S,\tilde{\tau}}^{\circ}=\tilde{\omega}_{A_0^\vee/S_0, \tilde{\tau}}^{\circ}$ for all $\tau \in \ttT$. 
Therefore, to give a deformation of $x_0$ to $S$ in $X'_{\ttT}$, we just need to specify the liftings $\omega_{S,\tilde{\tau}}^{\circ}$ of $\omega_{A_0^\vee/S_0, \tilde{\tau}}^{\circ}$ for $\tau\in \Sigma_{\infty}-(\ttS_{\infty}\cup \ttT)$. 
The set-valued sheaf of liftings $\omega_{S, \tilde{\tau}}^{\circ}$ for $\tau\in \Sigma_{\infty}-(\ttS_{\infty}\cup \ttT)$ is a torsor under the group
\[
\cHom_{\cO_{S_0}}(\omega_{A_0^\vee/S_0, \tilde{\tau}}^\circ, \Lie(A_0)_{\tilde{\tau}}^\circ)\otimes_{\cO_{S_0}}\calI\simeq \Lie(A_0)^\circ_{\tilde{\tau}}\otimes_{\cO_{S_0}} \Lie(A_0)^\circ_{\tilde{\tau}^c}\otimes_{\cO_{S_0}}\calI.
\]
Here, in the last isomorphism, we have used \eqref{Equ:duality-omega}.
The statement for the tangent bundle of $X'_{\ttT}$ now follows immediately.

 It remains to prove the properness of $X'_{\ttT}$ when  $\ttS_{\infty}\cup \ttT$ is non-empty.
  The arguments are similar to those in \cite[Proposition 3.4]{helm}.
 We use the valuative criterion of properness. 
 Let $R$ be a discrete valuation ring containing $\Fpb$ and $L$ be its fraction field. Let $x_L=(A_L, \iota, \lambda, \bar{\alpha}_{K'})$ be an $L$-valued point of $X'_{\ttT}$.
  We have to show that $x_L$ extends to  an $R$-valued point $x_R\in X'_{\ttT}$ up to a finite extension of $L$. 
  By Grothendieck's semi-stable reduction theorem, we may assume that, up to a finite extension of $L$, the N\'eron model $A_R$ of $A_L$ over $R$ has a semi-stable reduction. 
  Let $\overline{A}$ be the special fiber of $A_R$, and $\TT\subset \overline{A}$ be its torus part. 
  Since the  N\'eron model is canonical, the action of $\cO_{D_{\ttS}}$ extends uniquely to $A_R$, and hence to $\TT$. 
  The rational cocharacter group $X_*(\TT)_{\Q}:=\Hom(\GG_m,\TT)\otimes_{\Z}\Q$ is a $\Q$-vector space of dimension   at most $\dim(\overline A)= 4 g=\frac{1}{2}\dim_{\Q}(D_{\ttS})$, and equipped with an induced action of $D_{\ttS}\cong \rmM_{2}(E)$.
By the classification of $\rmM_{2}(E)$-modules, $X_*(\TT)_{\Q}$ is either $0$ or isomorphic  to  $E^{\oplus 2}$. 
 In the latter case, we have $X_*(\TT)_{\Q}\otimes L\cong \Lie(A_L)$, and  the trace of the action of  $b\in E$ on $X_*(\TT)_{\Q}$ is $2\sum_{\tilde\tau\in \Sigma_{E}}\tilde\tau(b)$, which implies that $\ttS_{\infty}=\emptyset$. 
Therefore, if   $\ttS_{\infty}\neq \emptyset$, $\TT$ has to be trivial and $A_R$ is an abelian scheme over $R$ with generic fiber $A_L$. 
The polarization $\lambda$ and level structure $\bar{\alpha}_{K'}$ extends uniquely to $A_R$ due to the canonicality of N\'eron model.
 We obtain thus a point $x_R\in X'(R)$ extending $x_L$.
  Since $X'_{\ttT}\subseteq X'$ is a closed subscheme, we see easily that $x_R\in X'_{\ttT}$.
     Now consider the case  $\ttS_{\infty}=\emptyset$ but  $\ttT$ is non-empty.
      If $X_*(\TT)_{\Q}\cong E^{\oplus 2}$, then the abelian part of $\overline{A}$ is trivial.
       Since the action of Verschiebung on $\omega_{\TT}$ is an isomorphism, the point $x_L$ cannot lie in any $X'_{\ttT}$ with $\ttT$ non-empty. Therefore, if $\ttT\neq \emptyset$, $\TT$ must be trivial, and we conclude as in the case $\ttS_{\infty}\neq \emptyset$.
\end{proof}

\begin{remark}
It seems that $X'_{\ttT}$ is still proper if $\ttS$ is  non-empty. But we do not know a convincing algebraic argument. 
\end{remark}

\subsection{GO-stratification of connected Shimura varieties}
\label{S:GO-stratum connected Shimura variety}
From the definition, it is clear that the GO-stratification on $\bfSh_{K'_p}(G'_{\tilde \ttS})_{k_0}$ is compatible with the action (as described in Subsection~\ref{S:abel var in unitary case}) of the group $\calG'_{\tilde \ttS}$ (introduced in Subsection~\ref{S:structure group}). 
By Corollary~\ref{C:mathematical objects equivalence}, for each $\ttT \subseteq \Sigma_\infty -\ttS_\infty$, there is a natural scheme 
\[
\bfSh_{K'_p}(G'_{\tilde \ttS})^\circ_{\overline \FF_p, \ttT} \subseteq \bfSh_{K'_p}(G'_{\tilde \ttS})^\circ_{\overline \FF_p}
\]
equivariant for the action of $\calE_{G, k_0}$.  We call them the \emph{Goren-Oort stratification} for the connected Shimura variety.

Using the identification of connected Shimura variety in Corollary~\ref{C:comparison of shimura varieties} together with Corollary~\ref{C:mathematical objects equivalence}, we obtain \emph{Goren-Oort stratum}
$\bfSh_{K_p}(G_\ttS)_{k_0, \ttT} \subseteq \bfSh_{K_p}(G_\ttS)_{k_0}$ and $\bfSh_{K''_p}(G''_{\tilde \ttS})_{k_0, \ttT} \subseteq \bfSh_{K''_p}(G''_{\tilde \ttS})_{k_0}$, for each subset $\ttT \subseteq \Sigma_\infty- \ttS_\infty$.
Explicitly, for the latter case, we have
\[
\bfSh_{K''_p}(G''_{\tilde \ttS})_{k_0, \ttT}:=
\bfSh_{K'_p}(G'_{\tilde \ttS})_{k_0, \ttT} \times_{\widetilde G_{\tilde \ttS}} G''_{\tilde \ttS}(\AAA^{\infty,p}).
\]
Alternatively, in terms of the natural family of abelian varieties $\bfA''_{\tilde \ttS}$, the stratum $\bfSh_{K''_p}(G''_{\tilde \ttS})_{k_0, \ttT}$ is the common zero locus of partial Hasse-invariants
\[
h_{\tilde \tau}: \omega^\circ_{\bfA''_{\tilde \ttS,k_0}/\bfSh_{K''_p}(G''_{\tilde \ttS})_{k_0}, \tilde \tau} \longrightarrow \big( \omega^\circ_{\bfA''_{\tilde \ttS,k_0}/\bfSh_{K''_p}(G''_{\tilde \ttS})_{k_0}, \sigma^{-n_\tau} \tilde\tau}\big)^{\otimes p^{n_\tau}}
\]
for all $\tilde \tau$ lifting $\tau \in \ttT$.

\begin{theorem}
When $\ttS=\emptyset$, the GO-stratification on $\bfSh_{K_p}(G_{\emptyset})_{k_0}$ defined above agrees with the original definition given in \cite{goren-oort}.
Moreover, the for each subset $\ttT \subseteq \Sigma_\infty$, under the morphisms \eqref{E:morphisms of Shimura varieties HMV}, we have
\[
\mathbf{pr}_1^*(\bfSh_{K_p}(G_{\emptyset})_{k_0, \ttT}) = 
\bbalpha^*(\bfSh_{K''_p}(G''_{\emptyset})_{k_0,\ttT}),
\]
where $\bfSh_{K_p}(G_{\emptyset})_{k_0, \ttT}$ denotes the GO-stratum for $\ttT$ defined in {\it loc. cit.}
\end{theorem}
\begin{proof}
Put $X=\bfSh_{K_p}(G_{\emptyset})_{k_0}$ for simplicity.
By Proposition~\ref{P:integral-HMV-unitary}, we have an isomorphism of abelian varieties $\bbalpha^*\bfA''_{\emptyset}=(\mathbf{pr}_1^*(\calA)\otimes_{\cO_F}\cO_E)^{\oplus 2}$ on $\bfSh_{K_p\times K_{E,p}}(G_\emptyset \times T_{E, \emptyset})$.
 Let
$\omega_{\calA_{k_0}^\vee/X}=\bigoplus_{\tau\in \Sigma_{\infty}}\omega_{\calA_{k_0}^\vee/X, \tau}$
be the canonical decomposition,
where $\omega_{\calA_{k_0}^\vee/X,\tau}$ is the local direct factor on which $\cO_F$ acts via $\iota_p\circ\tau: \cO_{F}\ra\ZZ_p^\ur\twoheadrightarrow \Fpb$. Then we have a canonical isomorphism of line bundles over $\bfSh_{K_p\times K_{E,p}}(G_\emptyset \times T_{E, \emptyset})_{k_0}$
\[
\bbalpha^*\omega_{\bfA''^\vee_{k_0}/X, \tilde{\tau}}^{\circ}\simeq \mathbf{pr}_1^*\omega_{\calA_{k_0}^\vee/X,\tau},
\]
for either lift $\tilde{\tau}\in \Sigma_{E,\infty}$ of $\tau$.
Via these identifications, the (pullback of) partial Hasse invariant $\bbalpha^*(h_{\tilde{\tau}})$ defined in \eqref{Equ:partial-hasse} coincides with the pullback  via $\mathbf{pr}_1$ of the  partial Hasse invariant $h_{\tau}\in \Gamma(X, \omega_{\calA/X, \sigma^{-1}\tau}^{\otimes p}\otimes \omega_{\calA/X, \tau}^{\otimes -1})$ defined in \cite{goren-oort}. Therefore, for any $\ttT\subset \Sigma_{\infty}$, the pullback along $\mathbf{pr}_1$ of the GO-strata $X_{\ttT}\subseteq X$ defined by the vanishing of $\{h_{\tau}:\tau\in \ttT\}$ is the same as the pullback along $\bbalpha$ of the GO-stratum defined by $\{h_{\tilde{\tau}}:\tau\in \ttT\}$.
\end{proof}

\begin{remark}
It would be interesting to know, in general, whether the GO-strata on quaternionic Shimura varieties depend on the auxiliary choice of CM field $E$.
\end{remark}

To understand the ``action" of the twisted partial Frobenius on the GO-strata, we need the following.

\begin{lemma}
\label{L:partial Frobenius vs partial Hasse inv}
Let $x = (A, \iota, \lambda, \bar \alpha_{K'}) $ be a point of $X'$ with  values in a locally noetherian $k_0$-scheme $S$, and $\gothF'_{\gothp^2}(x) = (A', \iota', \lambda', \bar \alpha'_{K'})$ be the image of $x$ under the twisted partial Frobenius at $\gothp$  \eqref{S:partial Frobenius}    (which lies on another Shimura variety).  Then
 $h_{\tilde \tau} (x) = 0$ if and only if $h_{\sigma_{\gothp}^2\tilde \tau}(\gothF'_{\gothp^2}(x))=0$.
 \end{lemma}
\begin{proof}
The statement is clear if $\tilde\tau\notin \Sigma_{E,\infty/\gothp}$, since $\gothF'_{\gothp^2}$ induces a canonical isomorphism of $p$-divisible groups $A[\gothq^{\infty}]\simeq A'[\gothq^{\infty}]$ for $\gothq\in \Sigma_p$ with $\gothq\neq \gothp$. Consider the case $\tilde\tau \in \Sigma_{E,\infty/\gothp}$. We claim that there exists an isomorphism 
\[
H^{\dR}_1(A'/S)^{\circ}_{\tilde\tau}\cong (H^{\dR}_1(A/S)^{\circ}_{\sigma^{-2}\tilde\tau})^{(p^2)}
\]
compatible with the action of $F$ and $V$ on both sides with $\tilde\tau\in \Sigma_{E,\infty/\gothp}$ varying. By Lemma~\ref{Lemma:partial-Hasse}, the Lemma follows from the claim immediately. It remains thus to prove the claim.
Note that the $\gothp$-component of de Rham homology  
$$
H^{\dR}_1(A'/S)_{\gothp}:=\bigoplus_{\tilde\tau\in \Sigma_{E,\infty/\gothp}}H^{\dR}_1(A'/S)^{\circ, \oplus 2}_{\tilde\tau}
$$
is canonically isomorphic to the evaluation at the trivial pd-thickening $S\hra S$ of the reduced covariant Dieudonn\'e crystal  of $A'[\gothp^{\infty}]$, which we denote by $\cD(A'[\gothp^{\infty}])_S$.  
By definition of $\gothF'_{\gothp^2}$, the $p$-divisible group $A'[\gothp^{\infty}]\cong (A/\Ker_{\gothp^2})[\gothp^{\infty}]$ is  isomorphic to the quotient of $A[\gothp^{\infty}]$ by its kernel of $p^2$-Frobenius $A[\gothp^{\infty}]\ra (A[\gothp^{\infty}])^{(p^2)}$. Therefore, by functoriality of Dieudonn\'e crystals, one has $\cD(A'[\gothp^{\infty}])_S=\cD(A[\gothp^\infty])_S^{(p^2)}$, whence the claim.
\end{proof}

One deduces immediately 
\begin{cor}
For $\overline{\Sh}_{\tilde \ttS} = \bfSh_{K'_p}(G'_{\tilde \ttS})_{k_0}$ and $\bfSh_{K''_p}(G''_{\tilde \ttS})_{k_0}$, the twisted partial Frobenius map $\gothF_{\gothp^2}: \overline{\Sh}_{\tilde \ttS} \to \overline{\Sh}_{\sigma_\gothp^2 \tilde \ttS}$ takes the subvariety $\overline{\Sh}_{\tilde \ttS, \ttT}$ to $\overline{\Sh}_{\sigma_\gothp^2\tilde  \ttS, \sigma_\gothp^2 \ttT}$ for each $\ttT\subseteq \Sigma_{\infty}-\ttS_{\infty}$. 
\end{cor}

\section{The global geometry of the GO-strata: Helm's isogeny trick}\label{Section:GO-geometry}

In this section,  we will prove that each closed GO-stratum  of the special fiber of the unitary Shimura variety  defined in Definition~\ref{Defn:GO-strata} is a $(\PP^1)^N$-bundle over the special fiber of another unitary Shimura variety for some appropriate integer $N$.  This then allows us to deduce the similar result for the case of quaternionic Shimura varieties.

This section is largely inspired by Helm's pioneer work \cite{helm}, where he considered the case when $p$ splits in $E_0/\QQ$ and $\ttS$ is ``sparse'' (we refer to \textit{loc. cit.} for the definition of sparse subset; essentially, this means that, for any $\tau \in \Sigma_\infty$, $\tau$ and $\sigma\tau$ cannot belong to $\ttS$ simultaneously.)

\subsection{The associated quaternionic Shimura data for a GO-stratum}\label{S:quaternion-data-T}

We first introduce the recipe for  describing general GO-strata.  We recommend first reading the light version of the same recipe in the special case of Hilbert modular varieties, as explained in the introduction~\ref{S:intro GO-strata}, before diving into the general but more complicated definition below.

Keep the notation as in the previous sections.
Let $\ttT$ be a subset of $\Sigma_{\infty}-\ttS_{\infty}$.
Our main theorem will say that the Goren-Oort stratum $\bfSh_K(G_\ttS)_{\overline \FF_p, \ttT}$ is a $(\PP^1)^N$-bundle over $\bfSh_{K_\ttT}(G_{\ttS(\ttT)})_{\overline{\FF}_p}$ for some $N \in \ZZ_{\geq0}$, some even subset $\ttS(\ttT)$ of places of $F$, and open compact subgroup  $K_{\ttT}\subseteq G_{\ttS(\ttT)}(\AAA^{\infty})$.

 We describe the set $\ttS(\ttT)$ now. For each prime $\gothp\in \Sigma_{p}$, we put $\ttT_{/\gothp}=\ttT\cap \Sigma_{\infty/\gothp}$. We define first a subset $\ttT'_{/\gothp}\subseteq \Sigma_{\infty/\gothp}\cup \{\gothp\} $ containing $\ttT_{/\gothp}$ which depends on the types of $\gothp$ as in Subsection~\ref{S:level-structure-at-p} and we   put
 \begin{equation}\label{Equ:defn-S-T}
\ttT'=\coprod_{\gothp\in \Sigma_p}\ttT'_{/\gothp},\quad\text{and }\quad\ttS(\ttT)=\ttS \sqcup \ttT'.
\end{equation}
We separate the discussion into several cases:
\begin{itemize}
\item If $\gothp$ is of type $\alpha^\sharp$ or type $\beta^\sharp$ for $\bfSh_{K}(G_{\ttS})$, we put $\ttT'_{/\gothp} = \emptyset$.

\item If $\gothp$ is of type $\alpha$ for $\ttS$, i.e. $(\Sigma_{\infty/\gothp}-\ttS_{\infty/\gothp})$ has even cardinality. We distinguish two cases:
\begin{itemize}

\item (Case $\alpha 1$) $\ttT_{/\gothp}\subsetneq \Sigma_{\infty/\gothp}-\ttS_{\infty/\gothp}$. We  write $\ttS_{\infty/\gothp}\cup \ttT_{/\gothp}=\coprod C_i$ as a disjoint union  of chains. Here, by a chain $C_i$, we mean that there exists $\tau_i\in \ttS_{\infty/\gothp}\cup \ttT_{/\gothp}$ and an integer $m_i\geq 0$ such that $C_i=\{\sigma^{-a}\tau_i: 0\leq a\leq m_i\}$ belong to $\ttS_{\infty/\gothp} \cup \ttT_{/\gothp}$ and $\sigma\tau_i, \sigma^{-m_i-1}\tau_i\notin (\ttS_{\infty/\gothp}\cup \ttT_{/\gothp})$.
We put $\ttT'_{/\gothp} = \coprod_i C'_i$, where
\[
C'_i: =
\left\{ 
\begin{array}{ll}
C_i \cap \ttT_{/\gothp} & \textrm{if } \#(C_i \cap \ttT_{/\gothp}) \textrm{ is even};
\\
(C_i\cap \ttT_{/\gothp})\cup \{\sigma^{-m_i-1}\tau_i\}& \textrm{if } \#(C_i \cap \ttT_{/\gothp}) \textrm{ is odd}.
\end{array}
\right.
\]

For example, if $\Sigma_{\infty/\gothp}=\{\tau_0, \sigma^{-1}\tau_0,\dots, \sigma^{-9}\tau_0\}$, $\ttS_{\infty/\gothp}=\{\sigma^{-2}\tau_0, \sigma^{-6}\tau_0\}$, and $\ttT_{/\gothp}=\{\sigma^{-3}\tau_0, \sigma^{-5}\tau_0, \sigma^{-7}\tau_0\}$, then $\ttS_{\infty/\gothp} \cup \ttT_{/\gothp}$ is separated into two chains $C_1 = \{\sigma^{-2}\tau_0, \sigma^{-3} \tau_0\}$ and $C_2 = \{\sigma^{-5}\tau_0, \sigma^{-6}\tau_0, \sigma^{-7}\tau_0\}$. We have
$
\ttT'_{/\gothp}=\{\sigma^{-3}\tau_0,\sigma^{-4}\tau_0, \sigma^{-5}\tau_0, \sigma^{-7}\tau_0\}.
$
An alternative way to understand the partition of $\ttT_{/\gothp}$ is to view it as a subset of $\Sigma_{\infty/\gothp} - \ttS_{\infty/\gothp}$ with the cycle structure inherited from $\Sigma_{\infty/\gothp}$. Then $C_i \cap \ttT_{/\gothp}$ is just to group elements of $\ttT_{/\gothp}$ into connected subchains.

\item (Case $\alpha 2$) $\ttT_{/\gothp}= \Sigma_{\infty/\gothp}-\ttS_{\infty/\gothp}$. We put $\ttT'_{/\gothp}=\ttT_{/\gothp}$.
\end{itemize}

\item If $\gothp$ is of type $\beta$ for $\ttS$, i.e. $(\Sigma_{\infty/\gothp}-\ttS_{\infty/\gothp})$ has  odd cardinality and $B_{\ttS}$ splits at $\gothp$. We distinguish two cases:
\begin{itemize}
\item (Case $\beta 1$)  $\ttT_{/\gothp}\subsetneq \Sigma_{\infty/\gothp}-\ttS_{\infty/\gothp}$. In this case, we define $\ttT'_{/\gothp}$ using the same rule as in Case $\alpha 1$.

\item (Case $\beta 2$)  $\ttT_{/\gothp}=\Sigma_{\infty/\gothp}-\ttS_{\infty/\gothp}$.   We put $\ttT'_{\gothp}=\ttT_{/\gothp}\cup \{\gothp\}$.
\end{itemize}


\end{itemize}

In either case, we put $\ttT'_{\infty/\gothp} = \ttT'_{/\gothp} \cap \Sigma_\infty$. It is equal to $\ttT'_{/\gothp}$ unless in case $\beta2$.

 It is easy to see that each $\ttT'_{/\gothp}$ has even cardinality. Therefore, $\ttS(\ttT)$ is also an even set, and it defines a quaternion algebra $B_{\ttS(\ttT)}$ over $F$. Note that $\ttS(\ttT)$ still satisfies  Hypothesis~\ref{H:B_S-splits-at-p}.

 Let $G_{\ttS(\ttT)}=\Res_{F/\Q}(B_{\ttS(\ttT)}^{\times})$ be the algebraic group over $\Q$ associated to $B_{\ttS(\ttT)}^{\times}$. We fix an isomorphism $B_{\ttS}\otimes_{F}F_{\gothl}\simeq B_{\ttS(\ttT)}\otimes_{F}F_{\gothl}$ whenever $\{\gothl\}\cap \ttS=\{\gothl\}\cap\ttS(\ttT)$. We  define an open compact subgroup  $K_{\ttT}=K_{\ttT}^pK_{\ttT, p}\subseteq G_{\ttS(\ttT)}(\AAA^{\infty})$ determined by $K$ as follows.
\begin{itemize}
\item We put $K_{\ttT}^p=K^p$. This makes sense, because $B_{\ttS}\otimes_{F}F_{\gothl}\simeq B_{\ttS(\ttT)}\otimes_{F}F_{\gothl}$ for any finite place $\gothl$ prime to $p$.

\item For $K_{\ttT, p}=\prod_{\gothp\in \Sigma_p}K_{\ttT, \gothp}$, we take $K_{\ttT,\gothp}=K_{\gothp}$, unless we are in case $\alpha 2$ or $\beta2 $.
\begin{itemize}

\item If $\gothp$ is of type  $\alpha 2$ for $\Sh_{K}(G_{\ttS})$, we have $B_{\ttS(\ttT)}\otimes_{F}F_{\gothp}\simeq B_{\ttS}\otimes_{F}F_{\gothp}\simeq \rmM_2(\cO_{F_{\gothp}})$.
 We take $K_{\ttT, \gothp}=K_{\gothp}$ if $\ttT_{/\gothp}=(\Sigma_{\infty/\gothp}-\ttS_{\infty/\gothp})=\emptyset$, and $K_{\ttT, \gothp}=\Iw_{\gothp}$ if $\ttT_{/\gothp}\neq \emptyset$.

\item If we are in case $\beta 2$ (and $\beta^\sharp$), $B_{\ttS(\ttT)}$ is ramified at $\gothp$. We take $K_{\ttT, \gothp}=\cO_{B_{F_{\gothp}}}^{\times}$, where $\cO_{B_{F_{\gothp}}}$ is the unique maximal order of the division algebra over $F_{\gothp}$ with invariant $1/2$.
\end{itemize}
\end{itemize}

The level $K_{\ttT}$ fits into the framework considered in Subsection~\ref{S:level-structure-at-p}. We obtain thus a quaternionic Shimura variety $\Sh_{K_{\ttT}}(G_{\ttS(\ttT)})$, and its integral model $\bfSh_{K_{\ttT}}(G_{\ttS(\ttT)})$ is given by Corollary~\ref{C:integral-model-quaternion}. Note that
\begin{itemize}
\item if we are in case $\alpha 1$ above,  then $\gothp$ is of type $\alpha$ for the Shimura variety $\Sh_{K_{\ttT}}(G_{\ttS(\ttT)})$;

\item if we are in case $\alpha 2$ above, then $\gothp$ is of type $\alpha^\sharp$ for $\Sh_{K_{\ttT}}(G_{\ttS(\ttT)})$ unless $\gothp$ is of type $\alpha$ for $\Sh_{K}(G_{\ttS})$ and $\ttT_{/\gothp}=\Sigma_{\infty/\gothp}-\ttS_{\infty/\gothp}=\emptyset$, in which case $\gothp$ remains of type $\alpha$ for $\Sh_{K_{\ttT}}(G_{\ttS(\ttT)})$;

\item if we are in case $\beta 1$, then $\gothp$ is of type $\beta$ for $\Sh_{K_{\ttT}}(G_{\ttS(\ttT)})$;

\item if we are in case $\beta2$ or $\beta^\sharp$ above, then $\gothp$ is of type $\beta^\sharp$ for $\Sh_{K_{\ttT}}(G_{\ttS(\ttT)})$.

\end{itemize}

\begin{theorem}\label{T:main-thm}
For a subset $\ttT\subseteq \Sigma_{\infty}-\ttS_{\infty}$,
the GO-stratum $\bfSh_{K}(G_\ttS)_{\overline \FF_p, \ttT}$ is
isomorphic to a $(\PP^1)^{I_{\ttT}}$-bundle over
$\bSh_{K_{\ttT}}(G_{\ttS(\ttT)})_{\overline \FF_p}$, where $\ttS(\ttT)$ is as described above and the index set is given by
$$
I_{\ttT}=\ttS(\ttT)_{\infty}-(\ttS_{\infty}\cup \ttT)=\bigcup_{\gothp\in \Sigma_p}(\ttT'_{\infty/\gothp}-\ttT_{/\gothp}).
$$
Moreover, this isomorphism is compatible with the action of  $G_{\ttS}(\AAA^{\infty,p})$, when taking the limit over open compact subgroups $K^p\subseteq G_{\ttS}(\AAA^{\infty,p})$.
\end{theorem}

As an example, the situation $I_{\ttT}=\emptyset$ happens exactly when $\ttT'_{\infty/\gothp}=\ttT_{/\gothp}$ for every $\gothp$, which means that, for each $\gothp$,  either $\ttT_{/\gothp}$ is divided into a union of chains of \emph{even} number of ``consecutive" elements of $\Sigma_{\infty/\gothp}-\ttS_{\infty/\gothp}$, or $\ttT_{/\gothp}=\Sigma_{\infty/\gothp}$. 

Theorem~\ref{T:main-thm} will follow from the analogous statements (Theorem~\ref{T:main-thm-unitary} and Corollary~\ref{C:main-thm-product}) in the unitary case.  But note Remark~\ref{R:quaternionic Shimura reciprocity not compatible}.

\subsection{The signatures at infinity for the unitary Shimura varieties}
\label{S:tilde S(T)}
In order to describe the unitary Shimura data associated to $\Sh_{K_{\ttT}}(G_{\ttS(\ttT)})$ as in Subsections~\ref{S:CM extension} and \ref{S:unitary-shimura}, we need to pick a lift $\tilde \ttS(\ttT)$ of the set $\ttS(\ttT)$ to embeddings of $E$.
More precisely, we will define a subset $\tilde \ttS(\ttT)_\infty = \coprod_{\gothp \in \Sigma_p} \tilde \ttS(\ttT)_{\infty/\gothp}$, where $\tilde \ttS(\ttT)_{\infty/\gothp}$ consists of exactly one lift $\tilde \tau \in \Sigma_{E, \infty}$ for each $\tau \in \ttS(\ttT)_{\infty/\gothp}$.
Then we put $\tilde \ttS(\ttT) = (\ttS(\ttT), \tilde \ttS(\ttT)_\infty)$.
So we just need to assign such choices of lifts.
\begin{itemize}
\item
When $\tau \in \ttS(\ttT)_{\infty/\gothp}$ belongs to $\ttS_{\infty/\gothp}$, we choose its lift $\tilde \tau \in \Sigma_{E, \infty}$ be the one that belongs to $\tilde \ttS$.
\end{itemize}
We now specify our choices of the lifts in $\tilde\ttS(\ttT)_{\infty/\gothp}$ for the elements of   $ \ttT'_{\infty/\gothp}$, which are collectively denoted by $\tilde \ttT'_{/\gothp}$.  We separate into cases and use freely the notation from Subsection~\ref{S:quaternion-data-T}.
There is nothing to do if $\gothp$ is of type $\alpha^\sharp$ or type $\beta^\sharp$ (for $\ttS$).

\begin{itemize}
\item $\gothp$ is of type $\alpha$ (for $\ttS$). In this case, $\gothp$ splits into two primes $\gothq$ and $\gothq^c$ in $E$.
For a place $\tau \in \Sigma_{\infty/\gothp}$, we use $\tilde \tau$ to denote its lift to $\Sigma_{E, \infty}$ which corresponds to the $p$-adic place $\gothq$.
\begin{itemize}
\item (Case $\alpha1$) 
For a chain  $C_i=\{\sigma^{-a}\tau_i, 0\leq a\leq m_{i}\}\subseteq \ttS_{\infty/\gothp}\cup\ttT_{/\gothp}$ and  the corresponding set $C'_i =\{ \sigma^{-a_1}\tau_i, \dots, \sigma^{-a_{r_i}}\tau_i\}$ as defined in \ref{S:quaternion-data-T} with some  $0\leq a_{1}< \dots<a_{r_i}\leq m_i+1$ (note that $r_i$ is always even by construction),  we put 
\[
\tilde C'_i = \{ \sigma^{-a_1}\tilde \tau_i, \sigma^{-a_2}\tilde \tau^c_i, \sigma^{-a_3}\tilde \tau_i, \dots, \sigma^{-a_{r_i}}\tilde \tau^c_i\}.
\]
Set $\tilde \ttT'_{/\gothp} = \coprod_i \tilde C'_i$.
\item (Case $\alpha2$) 
We need to fix $\tau_0 \in \ttT_{/\gothp} = \Sigma_{\infty/\gothp}-\ttS_{\infty/\gothp}$ and write $\ttT_{/\gothp}$ as $\{\sigma^{-a_1}\tau_0, \dots, \sigma^{-a_{2r}}\tau_0\}$ for integers $0 = a_1 < \cdots < a_{2r} \leq f_\gothp-1$.  We put
\[
\tilde \ttT'_{/\gothp} = \{\sigma^{-a_1}\tilde \tau_0, \sigma^{-a_2}\tilde \tau^c_0, \sigma^{-a_3}\tilde \tau_0, \dots, \sigma^{-a_{2r}}\tilde \tau^c_0\}.
\]

\end{itemize}
\item $\gothp$ is of type $\beta$ (for $\ttS$).  In this case, $\gothp$ is inert in $E/F$, and  we do not have  a canonical choice for the  lift $\tilde \tau$ of an embedding $\tau$.
\begin{itemize}
\item (Case $\beta1$) 
In this case, we fix a partition of  the preimage of $C'_i$ under the map $\Sigma_{E, \infty/\gothp} \to \Sigma_{\infty/\gothp}$ into two chains $\tilde C''_i \coprod \tilde C''^c_i$, where
\[
\tilde C''_i = \{\sigma^{-a_1} \tilde \tau_i, \dots, \sigma^{-a_{r_i}} \tilde \tau_i\}, \quad \textrm{and} \quad \tilde C''^c_i = \{\sigma^{-a_1} \tilde \tau^c_i, \dots, \sigma^{-a_{r_i}} \tilde \tau^c_i\}.
\]
Here, the choice of $\tilde\tau_i$ is arbitrary, and  $r_i$ is always even by construction.
We put
\[
\tilde C'_i :=\{ \sigma^{-a_1}\tilde \tau_i, \sigma^{-a_2}\tilde \tau^c_i, \sigma^{-a_3}\tilde \tau_i, \dots, \sigma^{-a_{r_i}}\tilde \tau^c_i\}.
\]
Finally, we set $\tilde \ttT'_{/\gothp} = \coprod_i \tilde C'_i$.
\item (Case $\beta2$) 
We fix an element $\tilde \tau_0 \in \Sigma_{E, \infty/\gothp}$.
Then the preimage of $\ttT'_{/\gothp}$ under the natural map $\Sigma_{E, \infty/\gothp} \to \Sigma_{\infty/\gothp}$ can be written as $\{ \sigma^{-a_1}\tilde \tau_0, \dots, \sigma^{-a_{2r}}\tilde \tau_0\}$ (where $r = \# ( \Sigma_\infty - \ttS_\infty)$ is odd), with $0=a_1< \cdots < a_{2r}\leq 2f_\gothp-1$ and $a_{r+i} = a_i + f_\gothp$ for all $i$.  We put
\[
\tilde \ttT'_{/\gothp} = \{ \sigma^{-a_1} \tilde \tau_0, \sigma^{-a_3} \tilde \tau_0, \sigma^{-a_5} \tilde \tau_0, \dots, \sigma^{-a_{2r-1}} \tilde \tau_0\}.
\]
Since $r$ is odd, this consists exactly one lift of each element of $\ttT'_{\infty/\gothp}$.

\end{itemize}
\end{itemize}

Now, we can assign integers $s_{\ttT, \tilde \tau}$ according to $\tilde \ttS(\ttT)$:
\begin{itemize}
\item
if $\tau \in \Sigma_\infty - \ttS(\ttT)_\infty$, we have $s_{\tilde \tau} = 1$ for all lifts $\tilde \tau $ of $\tau$;
\item
if $\tau \in \ttS(\ttT)_\infty$ and $\tilde \tau$ is the lift in $\tilde \ttS(\ttT)_\infty$, we have $s_{\ttT,\tilde \tau} = 0$ and $s_{\ttT,\tilde \tau^c} = 2$.
\end{itemize}
We put $\tilde \ttT' = \cup_{\gothp \in \Sigma_p} \tilde \ttT'_{/\gothp}$ and $\tilde \ttT'^c$ the complex conjugations of the elements in $\tilde \ttT'$.
\\

Now we compare the PEL data for the Shimura varieties for $G'_{\tilde \ttS}$ and $G'_{\tilde \ttS(\ttT)}$.  We fix an
isomorphism $\theta_{\ttT}:D_{\ttS}\ra D_{\ttS(\ttT)}$ that induces isomorphism between $\cO_{D_{\ttS},\gothp}$ and $\cO_{D_{\ttS(\ttT)},\gothp}$ for each $\gothp\in \Sigma_p$, where $\cO_{D_{\ttS},\gothp}$ and $\cO_{D_{\ttS(\ttT)},\gothp}$ are respectively fixed  maximal orders of $D_{\ttS}\otimes_F F_{\gothp}$ and $D_{\ttS(\ttT)}\otimes_{F}F_{\gothp}$ as  in \ref{S:PEL-Shimura-data}.

\begin{lemma}
\label{L:compare D_S with D_S(T)}
Let $\delta_{\ttS}\in (D_{\ttS}^\sym)^{\times}$ be an element satisfying Lemma~\ref{L:property-PEL-data}(1). Then there exists an element $\delta_{\ttS(\ttT)}\in (D_{\ttS(\ttT)}^{\sym})^{\times}$  satisfying the same condition with $\ttS$ replaced by $\ttS(\ttT)$ such that, if  $*_{\ttS}: l\mapsto \delta_{\ttS}^{-1}\bar{l}\delta_{\ttS}$ and $*_{\ttS(\ttT)}: l\mapsto \delta_{\ttS(\ttT)}^{-1}\bar{l}\delta_{\ttS(\ttT)}$ denote  the involutions on $D_{\ttS}$ and  $D_{\ttS(\ttT)}$ induced by $\delta_{\ttS}$ and $\delta_{\ttS(\ttT)}$ respectively, then $\theta_{\ttT}$ induces an isomorphism of algebras with positive involutions $(D_{\ttS},*_{\ttS})\xra{\sim} (D_{\ttS(\ttT)}, *_{\ttS(\ttT)})$.
\end{lemma}
 \begin{proof}
 We choose first an arbitrary ${\delta}'_{\ttS(\ttT)}\in (D_{\ttS(\ttT)}^{\sym})^{\times}$ satisfying Lemma~\ref{L:property-PEL-data}(1). Let $*'_{\ttS(\ttT)}$ denote the involution $l\mapsto (\delta'_{\ttS(\ttT)})^{-1}\bar{l}\delta'_{\ttS(\ttT)}$ on $D_{\ttS(\ttT)}$.
 By Skolem-Noether theorem, there exists $g\in D^{\times}_{\ttS(\ttT)}$ such that $\theta_{\ttT}(x)^{*'_{\ttS(\ttT)}}=g\theta_{\ttT}(x^{*_\ttS})g^{-1}$ for all $x\in D_{\ttS}$. Since both $*_{\ttS}^2$ and $*'^2_{\ttS(\ttT)}$ are identity, we get $g^{*'_{\ttS(\ttT)}}=g\mu$ for some $\mu\in E^{\times}$ with $\bar\mu\mu=1$.
  By Hilbert 90, we can write $\mu={\lambda}/{\bar\lambda}$ for some $\lambda\in E^{\times}$. Up to replacing $g$ by $g\lambda$, we may assume that $g^{*'_{\ttS(\ttT)}}=g$, or equivalently, $\overline{\delta'_{\ttS(\ttT)}g}=\delta'_{\ttS(\ttT)}g$ and hence $\delta'_{\ttS(\ttT)}g \in (D_{\ttS(\ttT)}^\sym)^\times$.
   Note that we  still have the freedom to modify $g$ by an element of $F^{\times}$ without changing $*_{\ttS(\ttT)}$.
   We claim that, up to such a modification on $g$, $\delta_{\ttS(\ttT)}=\delta'_{\ttS(\ttT)}g$ will answer the question.
    Indeed, by construction, $\theta_{\ttT}$ is an $*$-isomorphism, i.e.  $\theta_{\ttT}(x)^{*_{\ttS(\ttT)}}=\theta_{\ttT}(x^{*_\ttS})$.
Note that $\theta_{\ttT}$ sends $\cO_{D_{\ttS},\gothp}$ isomorphically to $\cO_{D_{\ttS(\ttT)},\gothp}$ for every $\gothp\in \Sigma_p$, and both lattices are invariant under the involutions $*_\ttS$ and $*'_{\ttS(\ttT)}$, respectively. 
So up to modifying $g$ by an element of $F^{\times}$, we may assume that $g\in \cO_{D_{\ttS(\ttT)},\gothp}^{\times}$ for all $\gothp\in \Sigma_p$.
Then it is clear that $\delta_{\ttS(\ttT)}$ satisfies Lemma~\ref{L:property-PEL-data}(1)(a), since so does $\delta'_{\ttS(\ttT)}$ by assumption.
It remains to prove that, up to multiplying $g$ by an element of $\calO_{F,(p)}^{\times}$, the hermitian form 
     $$
(v,w)\mapsto \psi_{\delta_{\ttS(\ttT)}}(v,w h'_{\tilde\ttS(\ttT)}(\bfi))=\Tr_{D_{\ttS(\ttT),\R}/\R}(\sqrt{\gothd}vh'_{\tilde\ttS(\ttT)}(\bfi)^{-1}\bar{w}\delta_{\ttS(\ttT)})
     $$
     on $D_{\ttS(\ttT),\R}:=D_{\ttS(\ttT)}\otimes_{\Q}\R$ is positive definite, where $\psi_{\delta_{\ttS(\ttT)}}$  is the $*_{\ttS(\ttT)}$-hermitian alternating form on $D_{\ttS(\ttT)}$ defined as in Subsection \ref{S:PEL-Shimura-data}.
     Since the elements $\delta_{\ttS}$ and  $\delta'_{\ttS(\ttT)}$ satisfy similar positivity conditions by assumption, we get two semi-simple $\R$-algebras with positive involution $(D_{\ttS,\R}, *_{\ttS})$ and $(D_{\ttS(\ttT),\R}, *_{\ttS(\ttT)})$.
By \cite[Lemma~2.11]{kottwitz}, there exists an element $b\in D^{\times}_{\ttS(\ttT), \R}$ such that  $b\theta_{\ttT}(x^{*_{\ttS}})b^{-1}=(b\theta_{\ttT}(x)b^{-1})^{*'_{\ttS(\ttT)}}$. It follows that  $g=b^{*'_{\ttS(\ttT)}}b\lambda$ with $\lambda\in (F\otimes_{\Q} \R)^{\times}$. 
Up to multiplying  $g$ by an element of $\calO_{F,(p)}^{\times}$, we may assume that $\lambda$ is totally positive so that $\lambda=\xi^2$ with $\xi\in (F\otimes_{\Q} \R)^{\times}$.
 Then, up to replacing $b$ by $b\xi$, we have   $g=b^{*'_{\ttS(\ttT)}}b$. Then the positivity of the form $\psi_{\delta_{\ttS(\ttT)}}$ follows immediately from the positivity of $\psi_{\delta'_{\ttS(\ttT)}}$, and the fact that $\psi_{\delta_{\ttS(\ttT)}}(v,wh'_{\tilde\ttS(\ttT)}(\bfi))=\psi_{\delta'_{\ttS(\ttT)}}(bv, bwh'_{\tilde\ttS(\ttT)}(\bfi))$. 
	        \end{proof}

\begin{lemma}
\label{L:comparison of Hermitian space}
We keep the choice of $\delta_{\ttS(\ttT)}$ as in Lemma~\ref{L:compare D_S with D_S(T)}.
Then there exists an isomorphism $\Theta_\ttT: D_\ttS(\AAA^{\infty,p}) \to D_{\ttS(\ttT)}(\AAA^{\infty, p})$ of skew $*$-Hermitian spaces compatible with the actions of $D_\ttS$ and $D_{\ttS(\ttT)}$, respectively.
\end{lemma}
\begin{proof}
This could be done explicitly.  We however prefer a sneaky quick proof.
Under Morita equivalence, we are essentially working with two-dimensional Hermitian spaces and the associated unitary groups.  It is well known that, over a nonarchimedean local field, there are exactly two Hermitian spaces and the associated unitary groups are not isomorphic (see e.g. \cite[3.2.1]{minguez}).  In our situation, we know that $G'_{\tilde \ttS,1, v} \cong G'_{\tilde \ttS(\ttT),1,v}$ for any place $v \nmid p\infty$.  It follows that the associated Hermitian spaces at $v$ are isomorphic.  The lemma follows.
\end{proof}

\begin{cor}
\label{C:identify structure groups}
The isomorphisms $\theta_\ttT$ and $\Theta_\ttT$ induce an isomorphism $\theta'_\ttT:\calG'_{\tilde \ttS} \xrightarrow{\cong}\calG'_{\tilde \ttS(\ttT)}$. Moreover $\theta'_\ttT \times \mathrm{id}$ takes the subgroup $\calE_{G, \ttS, k_0} \subset\calG'_{\tilde \ttS} \times \Gal_{k_0}$ to the subgroup $\calE_{G, \ttS(\ttT), k_0}\subset\calG'_{\tilde \ttS(\ttT)} \times \Gal_{k_0}$.
\end{cor}
\begin{proof}
The first statement follows from the description of the two groups in \eqref{E:structure group description} and \eqref{E:description of G'' G'} and the interpretation of these groups as certain automorphic groups of the skew $*$-Hermitian spaces.
The second statement follows from the description of both subgroups in Subsection~\ref{S:structure group} and the observation that the choice of signatures in Subsection~\ref{S:tilde S(T)} ensures that the reciprocity map $\gothRec_{k_0}$ for both Shimura data are the \emph{same} at $p$.
\end{proof}

\subsection{Level structure of $\bfSh_{K'_{\ttT}}(G'_{\tilde \ttS(\ttT)})$}
\label{S:level structure tilde ttS(ttT)}
We now specify the level structure $K_{\ttT}'\subseteq G'_{\tilde \ttS(\ttT)}(\AAA^{\infty})$.
\begin{itemize}
\item For the prime-to-$p$ level,
since $\Theta_\ttT$ induces an isomorphism $G_{\tilde \ttS(\ttT)}'(\AAA^{\infty,p}) \simeq G_{\tilde \ttS}'(\AAA^{\infty, p})$, the subgroup $K'^{p}\subseteq G_{\tilde \ttS}'(\AAA^{\infty,p})$ corresponds   to a subgroup $K_{\ttT}'^{p}\subseteq G_{\tilde \ttS(\ttT)}'(\AAA^{\infty,p})$.

\item For $K'_{\ttT, p}$, we take it as the open compact subgroup of $G_{\tilde \ttS(\ttT)}'(\Q_p)$ corresponding to $K_{\ttT,p}\subseteq G_{\tilde \ttS(\ttT)}(\Q_p)$ by the rule in Subsection~\ref{S:level-structure}. 
According to the discussion there, it suffices to choose a chain of lattices $\Lambda^{(1)}_{\ttT,\gothp}\subset \Lambda^{(2)}_{\ttT,\gothp}$ in $D_{\ttS(\ttT)}\otimes_{F}F_{\gothp}$ for each $\gothp\in \Sigma_{p}$. Using the isomorphism $\theta_{\ttT}$,  we can identify $D_{\ttS(\ttT)}\otimes_F F_{\gothp}$ with $D_{\ttS}\otimes_F F_{\gothp}$, and hence with $\rmM_2(E \otimes_F F_\gothp)$.
\begin{itemize}
\item For $\gothp\in \Sigma_{p}$ with $K_{\ttT,\gothp}=K_{\gothp}$, we take $\Lambda_{\ttT, \gothp}^{(1)}\subseteq \Lambda_{\ttT,\gothp}^{(2)}$ to be the same as the chain $\Lambda_{\gothp}^{(1)}\subseteq \Lambda_{\gothp}^{(2)}$ for defining $K_{\gothp}'\subset G'_{\ttS}(\Q_p)$.
\item For $\gothp\in \Sigma_{p}$ with $K_{\ttT, \gothp}\neq K_{\gothp}$,  then $K_{\ttT, \gothp}$ is either Iwahori subgroup of $\GL_2(\cO_{F_{\gothp}})$  or $\cO_{B_{F_{\gothp}}}^{\times}$. We take then $\Lambda_{\ttT,\gothp}^{(1)}\subsetneq \Lambda_{\ttT,\gothp}^{(2)}$ to be the corresponding  lattice as in Subsection~\ref{S:level-structure} that defines the Iwahori level at $\gothp$.
\end{itemize}

\end{itemize}
We also specify  the lattices we use for both Shimura varieties: if $\Lambda_\ttS$ denotes the chosen lattice of $D_\ttS$, we choose the lattice of $D_{\ttS(\ttT)}$ to be $\Lambda_{\ttS(\ttT)} =\theta_\ttT(\Lambda_\ttS)$.
With these data, we have a unitary Shimura variety $\Sh_{K_{\ttT}'}(G_{\tilde \ttS(\ttT)}')$ over the reflex field $E_{\tilde \ttS(\ttT)}$, which is the field corresponding to the Galois group  fixing the subset $\tilde \ttS(\ttT) \subseteq \Sigma_{E, \infty}$.
  To construct an integral model of $\Sh_{K'_{\ttT}}(G'_{\tilde \ttS(\ttT)})$, we need to choose an order $\cO_{D_{\ttS(\ttT)}}$.
Let $\cO_{D_{\ttS}}$ be the order stable under $*$ and maximal at $p$ used to define the integral model $\bSh_{K'}(G_{\tilde \ttS}')$. We put $\cO_{D_{\ttS(\ttT)}}=\theta_{\ttT}(\cO_{D_{\ttS}})$.  For any $\gothp\in \Sigma_{p}$,   both $\cO_{D_{\ttS}, \gothp}$ and $\cO_{D_{\ttS(\ttT)}, \gothp}$ can be identified with $\rmM_2(\cO_{E}\otimes_{\cO_F}\cO_{F_{\gothp}})$.

We have now all the PEL-data needed for  Theorem~\ref{T:unitary-shimura-variety-representability}, which assures that $\Sh_{K_{\ttT}'}(G_{\tilde \ttS(\ttT)}')$ admits an integral model $\bSh_{K_{\ttT}'}(G_{\tilde \ttS(\ttT)}')$ over $W(k_0)$. Using $\bSh_{K_{\ttT}'}(G_{\tilde \ttS(\ttT)}')$, we can construct an integral model $\bSh_{K_{\ttT}}(G_{\ttS(\ttT)})$ of the quaternionic Shimura variety  $\Sh_{K_{\ttT}}(G_{\ttS(\ttT)})$.

\begin{theorem}\label{T:main-thm-unitary}
For a subset $\ttT\subseteq \Sigma_{\infty}-\ttS_{\infty}$, let  $\bfSh_{K'}(G'_{\tilde \ttS})_{k_0,\ttT}\subseteq \bfSh_{K'}(G'_{\tilde \ttS})_{k_0}$ denote the GO-stratum defined in Definition~\ref{Defn:GO-strata}. Let $I_{\ttT}$ be as in Theorem~\ref{T:main-thm}. Then we have the following:

\begin{itemize}
\item[(1)]
(Description)
$\bfSh_{K'}(G'_{\tilde \ttS})_{k_0,\ttT}$ is isomorphic to a $(\PP^1)^{I_{\ttT}}$-bundle over $\bSh_{K'_{\ttT}}(G'_{\tilde \ttS(\ttT)})_{k_0}$.

\item[(2)]
(Compatibility of abelian varieties)
  Let $\pi_{\ttT}: \bfSh_{K'}(G'_{\tilde \ttS})_{k_0,\ttT}\ra \bSh_{K'_{\ttT}}(G'_{\tilde \ttS(\ttT)})_{k_0}$ denote the projection of the $(\PP^1)^{I_{\ttT}}$-bundle  in \emph{(1)}. The abelian schemes
  $\bfA'_{\tilde \ttS,k_0}$ and $\pi_{\ttT}^*\bfA'_{\tilde \ttS(\ttT),k_0}$ over $\bfSh_{K'}(G'_{\tilde \ttS})_{k_0}$ are isogenous,  where $\bfA'_{\tilde \ttS,k_0}$ and $\bfA'_{\tilde \ttS(\ttT),k_0}$ denote respectively the universal abelian varieties over $\bfSh_{K'}(G'_{\tilde \ttS})_{k_0,\ttT}$ and  over $\bSh_{K'_{ \ttT}}(G'_{\tilde \ttS(\ttT)})_{k_0}$.
  
  \item[(3)]
(Compatibility with Hecke action)
Taking the limit over the open compact subgroups $K'^p\subseteq G'_{\tilde \ttS}(\AAA^{\infty,p})$, the isomorphism as well as the isogeny of abelian varieties are compatible with the action of the Hecke correspondence given by $\widetilde G_{\tilde \ttS} = G''_{\tilde \ttS}(\QQ)^{+,(p)} G'_{\tilde \ttS}(\AAA^{\infty, p}) \cong \widetilde G_{\tilde \ttS(\ttT)}$.

\item[(4)]
(Compatibility with partial Frobenius)
The description in (1) is compatible with the action of the twisted partial Frobenius (Subsection~\ref{S:partial Frobenius}) in the sense that we have a commutative diagram
\[
\xymatrix@C=50pt{
\bfSh_{K'}(G'_{\tilde \ttS})_{k_0, \ttT} \ar[r]_-{\xi^\mathrm{rel}} \ar[rd]_{\pi_{\ttT}}
\ar@/^15pt/[rr]^-{\gothF'_{\gothp^2, \tilde \ttS}}
&
\gothF'^*_{\gothp^2, \tilde \ttS(\ttT)}(
\bfSh_{K'}(G'_{\sigma_\gothp^2\tilde \ttS})_{k_0, \sigma_\gothp^2\ttT})
) \ar[d] \ar[r]_-{\gothF'^*_{\gothp^2, \tilde \ttS(\ttT)}}
&
\bfSh_{K'}(G'_{\sigma_\gothp^2\tilde \ttS})_{k_0, \sigma_\gothp^2\ttT} \ar[d]^{\pi_{\sigma_{\gothp}^2\ttT}}
\\
&
\bfSh_{K'_\ttT}(G'_{\tilde \ttS(\ttT)})_{k_0}
\ar[r]^-{\gothF'_{\gothp^2, \tilde \ttS(\ttT)}} &
\bfSh_{K'_{\sigma_\gothp^2\ttT}}(G'_{\sigma_\gothp^2(\tilde \ttS(\ttT))})_{k_0},
}
\]
where  the square is Cartesian, we added subscripts to the partial Frobenius to indicate the corresponding base scheme, and the morphism $\xi^\mathrm{rel}$ is a morphism whose restriction to each fiber $\pi_{\ttT}^{-1}(x)=(\PP^1_x)^{I_{\ttT}}$ is the product of the relative  $p^2$-Frobenius of the  $\PP^1$'s indexed by $I_{\ttT}\cap \Sigma_{\infty/\gothp}=\ttT'_{\infty/\gothp}-\ttT_{/\gothp}$, and  the identity on the other $\PP^1_x$'s.
\end{itemize}

\end{theorem}

The proof of this theorem will occupy the rest of this section and concludes in Subsection~\ref{S:End-of-proof}.
We first state a corollary.

\begin{cor}
\label{C:main-thm-product}
\begin{enumerate}
\item The Goren-Oort stratum $\bfSh_{K'_p}(G'_{\tilde \ttS})^\circ_{\overline \FF_p, \ttT}$ is  isomorphic to a $(\PP^1)^{I_\ttT}$-bundle over $\bfSh_{K'_p}(G'_{\tilde \ttS(\ttT)})^\circ_{\overline \FF_p}$, equivariant for the action of $\calE_{G, \ttS, \tilde \wp} \cong \calE_{G, \ttS(\ttT), \tilde \wp}$ (which are identified as in Corollary~\ref{C:identify structure groups}) with trivial action on the fibers.

\item The GO-stratum $\bfSh_{K''_p}(G''_{\tilde \ttS})_{k_0,\ttT}$  is isomorphic to a $(\PP^1)^{I_\ttT}$-bundle over $\bfSh_{K''_p}(G''_{\tilde \ttS(\ttT)})_{k_0}$, such that the natural projection $\pi_{\ttT}: \bfSh_{K''_p}(G''_{\tilde \ttS})_{k_0,\ttT} \to\bfSh_{K''_p}(G''_{\tilde \ttS(\ttT)})_{k_0}$ is equivariant for the tame Hecke action.

\item The abelian schemes $\bfA''_{\tilde \ttS,k_0}$ and $\pi_{\ttT}^*(\bfA''_{\tilde \ttS(\ttT),k_0})$  over $\bfSh_{K''_p}(G''_{\tilde \ttS})_{k_0, \ttT}$ are isogenous.

\item The following diagram is commutative
\[
\xymatrix@C=50pt{
\bfSh_{K''_p}(G''_{\tilde \ttS})_{k_0, \ttT} \ar[r]_-{\xi^\mathrm{rel}} \ar@/^15pt/[rr]^-{\gothF''_{\gothp^2,\tilde \ttS}} \ar[dr]_{\pi_{\ttT}}
&
\gothF_{\gothp^2, \tilde \ttS(\ttT)}^*(\bfSh_{K''_p}(G''_{\sigma_\gothp^2\tilde \ttS})_{k_0, \sigma_{\gothp}^2 \ttT} \ar[r]_-{\gothF''^*_{\gothp^2, \tilde \ttS(\ttT)}} \ar[d]
&
\bfSh_{K''_p}(G''_{\sigma_\gothp^2\tilde \ttS})_{k_0, \sigma_{\gothp}^2 \ttT} \ar[d]^{\pi_{\tilde \ttS(\ttT)}}
\\
&
\bfSh_{K''_{\ttT,p}}(G''_{\tilde \ttS(\ttT)})_{k_0}
\ar[r]^-{\gothF''_{\gothp^2, \tilde \ttS(\ttT)}} 
&
\bfSh_{K''_{\sigma_\gothp^2\ttT,p}}(G''_{\sigma_\gothp^2(\tilde \ttS(\ttT))})_{k_0}
}
\]
where the square is Cartesian, $\gothF''_{\gothp^2, \tilde \ttS}$ and $\gothF''_{\gothp^2, \tilde \ttS(\ttT)}$ denote the twisted partial Frobenii (Subsection~\ref{S:partial Frobenius}) on $\bfSh_{K''_p}(G''_{\tilde \ttS})_{k_0}$ and $\bfSh_{K''_{\ttT,p}}(G''_{\tilde \ttS(\ttT)})_{k_0}$ respectively,  and $\xi^\mathrm{rel}$ is a morphism whose restriction to each fiber $\pi_{\ttT}^{-1}(x)=(\PP^1_x)^{I_{\ttT}}$ is the product of  the relative  $p^2$-Frobenius of the  $\PP^1_x$'s indexed by $I_{\ttT}\cap \Sigma_{\infty/\gothp}=\ttT'_{\infty/\gothp}-\ttT_{/\gothp}$, and  the identity on the other $\PP^1_x$'s.
\end{enumerate}
\end{cor}
\begin{proof}
This is an immediate consequence of Corollary~\ref{C:main-thm-product} above.
The claims regarding the universal abelian varieties follows from the explicit construction of $\bfA''_{\tilde \ttS}$ and $\bfA''_{\tilde \ttS(\ttT)}$ in Subsection~\ref{S:abel var in unitary case}.
\end{proof}

\begin{remark}
\label{R:quaternionic Shimura reciprocity not compatible}
We emphasize that the analogous of Corollary~\ref{C:main-thm-product}(ii) for quaternionic Shimura varieties only holds over $\overline \FF_p$.
This is because the subgroups $\calE_{G, \ttS, \wp}$ and $\calE_{G, \ttS(\ttT), \wp}$, although abstractly isomorphic, sit in $\calG_\ttS \times \Gal_{k_0} \cong \calG_{\ttS(\ttT)} \times \Gal_{k_0}$ as \emph{different} subgroups. 
The two Deligne homomorphisms are different.
\end{remark}

The rest of this section is devoted to the proof of Theorem~\ref{T:main-thm-unitary}, which concludes in Subsection~\ref{S:End-of-proof}.

\subsection{Signature changes}\label{S:Delta-pm}
The basic idea of proving Theorem~\ref{T:main-thm-unitary} is to find a quasi-isogeny between the two universal abelian varieties $\bfA_{\tilde \ttS}$ and $\bfB: =\bfA_{\tilde \ttS(\ttT)}$ (over an appropriate base).
We view this quasi-isogeny as two genuine isogenies $ \bfA_{\tilde \ttS} \xrightarrow{\phi} \bfC \xleftarrow{\phi'} \bfB $ for some abelian variety $\bfC$. Each isogeny is characterized by the set of places $\tilde\tau\in \Sigma_{E,\infty}$ where the isogeny does not induce an isomorphism of the $\tilde\tau$-components of the de Rham cohomology of the abelian varieties.  We define these two subsets $\tilde \Delta(\ttT)^+$ and $\tilde \Delta(\ttT)^-$ of $\Sigma_{E, \infty}$ now, as follows.
As before, $ \tilde \Delta(\ttT)^\pm = \coprod_{\gothp \in \Sigma_p} \tilde \Delta(\ttT)_{/ \gothp}^\pm$ for subsets $\tilde \Delta(\ttT)_{/\gothp}^\pm \subseteq \Sigma_{E, \infty/\gothp}$.
When $\gothp$ is of type $\alpha^\sharp$ or $\beta^\sharp$ for $\ttS$, we set $\tilde \Delta(\ttT)_{/\gothp}^\pm =\emptyset$.
For the other two types, we use the notation in Subsection~\ref{S:tilde S(T)} in the corresponding cases (in particular our convention on $\tilde \tau$ and $a_j$'s):
\begin{itemize}
\item (Case $\alpha1$)  Put
\[
\tilde C^-_i: = \bigcup_{\substack{j \text{ odd}\\ 1\leq j\leq r_i}} \{\sigma^{-\ell}\tilde \tau_i: a_{j}\leq \ell\leq a_{j+1}-1\}.
\]
We set $\tilde \Delta(\ttT)_{/\gothp}^- = \coprod_i \tilde C^-_i$ and $\tilde \Delta(\ttT)_{/\gothp}^+ =(\tilde \Delta(\ttT)_{/\gothp}^-)^c$.
\item (Case $\alpha2$) 
Put
\[
\tilde \Delta(\ttT)_{/\gothp}^- : = \bigcup_{1\leq i\leq r} \big\{ \sigma^{-l}\tilde\tau_0: a_{2i-1} \leq l< a_{2i} \big\}; \quad \textrm{and}\quad
\tilde \Delta(\ttT)_{/\gothp}^+: = (\tilde \Delta(\ttT)_{/\gothp}^-)^c.
\]
\item (Case $\beta1$) 
Put
\[
\tilde C_i^-: = \bigcup_{\substack{j \text{ odd}\\ 1\leq j\leq r_i}} \{\sigma^{-\ell}\tilde \tau_i: a_{j}\leq \ell\leq a_{j+1}-1\}.
\]
We set $\tilde\Delta(\ttT)_{/\gothp}^- = \coprod_i \tilde C^-_i$ and $\tilde \Delta(\ttT)_{/\gothp}^+=(\tilde \Delta(\ttT)_{/\gothp}^-)^c$.  (Formally, this is the same recipe as in case $\alpha1$, but the choice of $\tilde \tau_i$ is less determined; see Subsection~\ref{S:tilde S(T)}.)
\item (Case $\beta2$) 
Put
\[
\tilde \Delta(\ttT)_{/\gothp}^- : = \bigcup_{1\leq i\leq r} \big\{ \sigma^{-l}\tilde\tau_0: a_{2i-1} \leq l< a_{2i} \big\}.
\]
\emph{Unlike in all other cases, we put $\tilde \Delta(\ttT)_{/\gothp}^+=\emptyset$.}
\end{itemize}
We always have $\tilde \Delta(\ttT)^+ \cap \tilde \Delta(\ttT)^- = \emptyset$.

\begin{notation}
We use $\ttT_E$ (resp. $\ttT'_E$) to denote the preimage of $\ttT$ (resp. $\ttT'$) under the map $\Sigma_{E, \infty} \to \Sigma_\infty$.
\end{notation}

The following two lemmas follow from the definition by a case-by-case check.
\begin{lemma}
\label{L:distance to T'}
For each $\tilde \tau  \in \tilde \Delta(\ttT)^+$ (resp. $\tilde \Delta(\ttT)^-$), let $n$ be the unique positive integer such that 
$\tilde \tau, \sigma^{-1}\tilde \tau, \dots, \sigma^{1-n} \tilde \tau$ all belong to $\tilde \Delta(\ttT)^+$ (resp. $\tilde \Delta(\ttT)^-$) but 
$\sigma^{-n}\tilde \tau$ does not.
Then, for this $n$, $\sigma^{-n}\tilde \tau \in \ttT'_E$.
Moreover, if $\tilde \tau$ also belongs to $\ttT'_E$, then $n$ equals to the number $n_\tau$ introduced in Subsection~\ref{S:partial-Hasse}.
\end{lemma}

\begin{lemma}
\label{L:property of Delta}
\emph{(1)} 
If both $\tilde \tau$ and $\sigma \tilde \tau$ belong to $\tilde \Delta(\ttT)^+$ (resp. $\tilde \Delta(\ttT)^-$), then $\tilde \tau|_F$ belongs to $\ttS_\infty$.

\emph{(2)}
If $\tilde\tau \in \tilde \Delta(\ttT)^-$ but $\sigma \tilde \tau \notin \tilde \Delta(\ttT)^-$, then $\tilde \tau \in \tilde \ttT'$.

\emph{(3)}
If $\tilde\tau \notin \tilde \Delta(\ttT)^-$ but $\sigma \tilde \tau \in \tilde \Delta(\ttT)^-$, then $\tilde \tau \in \tilde \ttT'^c$.
\end{lemma}

\subsection{Description of the strata $\bfSh_{K'}(G'_{\tilde \ttS})_{k_0, \ttT}$ via isogenies}
\label{S:moduli-Y_S}
To simplify the notation, we put $X'=\bfSh_{K'}(G'_{\tilde \ttS})_{k_0}$ and $X'_{\ttT}=\bfSh_{K'}(G'_{\tilde \ttS})_{k_0,\ttT}$ for a subset $\ttT\subseteq \Sigma_{\infty}-\ttS_{\infty}$. We will first prove statement (1) of Theorem~\ref{T:main-thm-unitary}. 
Following the idea of Helm \cite{helm}, we introduce  auxiliary moduli spaces $Y'_{\ttT}$ and $Z'_{\ttT}$ and   establish isomorphisms
\begin{equation}\label{E:sequence-moduli}
\xymatrix{
X'_\ttT  &Y'_\ttT\ar[l]^-{\cong}_-{\eta_1} \ar[r]^-{\eta_2}_-{\cong} &Z'_\ttT},
\end{equation}
where  $Z'_\ttT$ is a $(\PP^1)^{I_{\ttT}}$-bundle over the special fiber of  $\bSh_{K'_{\ttT}}(G'_{\tilde \ttS(\ttT)})_{k_0}$.

Recall that we have fixed an isomorphism $\theta_{\ttT}: (D_{\ttS},*_{\ttS})\xra{\sim}(D_{\ttS(\ttT)}, *_{\ttS(\ttT)})$ of simple algebras over $E$ with positive involution, and put $\cO_{D_{\ttS(\ttT)}}=\theta_{\ttT}(\cO_{D_{\ttS}})$. To ease the notation, we identify $\cO_{D_{\ttS(\ttT)}}$ with $\cO_{D_{\ttS}}$ via $\theta_{\ttT}$, and denote them by $\cO_D$ when there is no confusions.

We start now to describe $Y'_\ttT$: It is the moduli space over $k_0$ which attaches to a locally noetherian $k_0$-scheme $S$ the set of isomorphism classes of  $(A, \iota_A,\lambda_A,  \alpha_{K'}, B, \iota_B, \lambda_B, \beta_{K'_\ttT}, C, \iota_C; \phi_A, \phi_B)$, where
\begin{itemize}
\item[(i)]
$(A, \iota_A, \lambda_A, \alpha_{K'})$ is an element in $X'_\ttT(S)$.

\item[(ii)]
$(B, \iota_B, \lambda_B, \beta_{K'_\ttT})$ is an element in $\bfSh_{K'_\ttT}(G'_{\tilde \ttS(\ttT)})(S)$.

\item[(iii)]
 $C$ is an abelian scheme over $S$ of dimension $4g$, equipped with an embedding $\iota_C: \calO_D \to \End_S(C)$,

\item[(iv)]
 $\phi_A: A \to C$ is an $\calO_D$-isogeny whose kernel is killed by $p$, such that the induced map
$$
\phi_{A,*,\tilde{\tau}}: H_1^{\dR}(A/S)^\circ_{\tilde{\tau}} \to H_1^\dR(C/S)^\circ_{\tilde{\tau}}
$$
is an isomorphism  for $\tilde{\tau}\in \Sigma_{E,\infty}$ \emph{unless} $\tilde{\tau}\in \tilde\Delta(\ttT)^+$, in which case, we require that
\begin{equation}\label{E:condition-phi-A}
\Ker(\phi_{A,*,\tilde{\tau}})=\im(F_{A,\es, \tilde \tau}^{n}),
\end{equation}
where $n$ is as determined in Lemma~\ref{L:distance to T'}, and $F_{A,\es, \tilde \tau}^n$ is defined in \eqref{E:Fes n}.
(When $\tau$ itself belongs to $\ttT'$, the number $n$ equals to $n_\tau$ introduced in Subsection~\ref{S:partial-Hasse}. In this case, condition \eqref{E:condition-phi-A} is equivalent to saying that $\Ker(\phi_{A,*,\tilde\tau})=\omega^\circ_{A^\vee/S,\tilde\tau}$.)

\item[(v)]
$\phi_B: B\to C$ is an $\calO_D$-isogeny whose kernel is killed by $p$ such that
$\phi_{B,*,\tilde{\tau}}: H_1^\dR(B/S)_{\tilde{\tau}}^\circ \to H_1^\dR(C/S)_{\tilde{\tau}}^\circ$ is an isomorphism for $\tilde{\tau}\in \Sigma_{E,\infty}$ \emph{unless} $\tilde{\tau}\in \tilde{\Delta}(\ttT)^-$, in which case, we require that $\im(\phi_{B,*,\tilde{\tau}})$ is equal to $\phi_{A,*,\tilde{\tau}}(\im(F_{A,\es, \tilde \tau}^n))$, where $n$ is as determined in Lemma~\ref{L:distance to T'}. (Note that as $\tilde \tau \notin \tilde \Delta(T)^+$, $\phi_{A, *, \tilde \tau}$ is an isomorphism by (iv).)

\item[(vi)] The tame level structures are compatible, i.e. $T^{(p)}(\phi_A)\circ \alpha^p_{K'^p} = T^{(p)}(\phi_B) \circ \alpha^p_{K'^p_\ttT}$ as maps from $\widehat \Lambda^{(p)}_\ttS \cong \widehat \Lambda^{(p)}_{\ttS(\ttT)}$ to $T^{(p)}(C)$, modulo $K'^p$, if we identify the two lattices naturally as in Subsection~\ref{S:level structure tilde ttS(ttT)}.

\item[(vii)] If $\gothp$ is a prime of type $\alpha^\sharp$ for the original quaternionic Shimura variety $\Sh_{K}(G_{\ttS})$, then $\alpha_{\gothp}$ and $\beta_{\gothp}$ are compatible, i.e. $\phi_{A}(\alpha_{\gothp})=\phi_{B}(\beta_{\gothp})$, where $\alpha_{\gothp}\subseteq A[\gothp]$ denotes the closed finite flat group scheme given by Theorem~\ref{T:unitary-shimura-variety-representability}(c2). (Note that $\phi_A$ and $\phi_B$ induce isomorphisms between $A[\gothp]$, $C[\gothp]$, and $B[\gothp]$ for a prime $\gothp$ of type $\alpha^\sharp$ for $\ttS$.)

\item[(viii)] Let $\gothp$ be a prime in Case $\alpha2$, splitting into $\gothq \bar \gothq$ in $E$.
Then $\beta_\gothp = H_\gothq \oplus H_{\bar \gothq}$.  If $\phi_{\gothq}: B\ra B'_{\gothq}=B/H_{\gothq}$ the canonical isogeny. Then the kernel of  the  induced map $\phi_{\gothq,*}:H^{\dR}_1(B/S)^{\circ}_{\tilde{\tau}}\ra H^{\dR}_1(B'_{\gothq}/S)^\circ_{\tilde\tau}$   coincides with that of $\phi_{B,*}: H^{\dR}_1(B/S)^\circ_{\tilde\tau}\ra H^{\dR}_1(C/S)^\circ_{\tilde\tau}$ for all $\tilde\tau\in \tilde \Delta(\ttT)_{/\gothp}^-$.

\item[(ix)]
We have the following commutative diagram
\[
\xymatrix{
A \ar[r]^-{\phi_{A}}\ar[d]_{\lambda_A} & C & B   \ar[l]_-{\phi_B}\ar[d]^{\lambda_B}\\
A^\vee  &
C^\vee\ar[l]_-{\phi_A^\vee} \ar[r]^-{\phi_B^\vee}
&  B^\vee.
}
\]
\end{itemize}


\begin{remark}
 Compared to \cite{helm}, our moduli problem appears to be more complicated.
This is because we allow places above $p$ of $F$ to be inert in the CM extension $E$. It is clear that $B$ is quasi-isogenous to $A$. So when $S$ is the spectrum of a perfect field $k$, the \emph{covariant} Dieudonn\'e module $\tilde{\calD}_{B}$ is a $W(k)$-lattice in  $\tilde{\calD}_A[1/p]$. The complicated conditions (v) and (vi) can be better understood by looking at $\tilde{\calD}_{B}$ (See the proof of  Proposition~\ref{P:Y_S=X_S} below).

\end{remark}

\begin{prop}
\label{P:Y_S=X_S}
The natural forgetful functor
\[
\eta_1\colon (A, \iota_A,\lambda_A,  \alpha_{K'}, B, \iota_B, \lambda_B, \beta_{K'_\ttT}, C, \iota_C; \phi_A, \phi_B)
\mapsto (A, \iota_A,\lambda_A,  \alpha_{K'})
\]
induces an isomorphism $\eta_1\colon Y'_\ttT \to X'_\ttT$.
\end{prop}

\begin{proof}

By the general theory of moduli spaces of abelian schemes due to Mumford, $Y'_{\ttT}$ is representable by an $k_0$-scheme of finite type. Hence, to prove the proposition, it suffices to show that the natural map $Y'_{\ttT}\ra X'_{\ttT}$ induces a bijection on closed points and the tangent spaces at each closed point. The proposition will follow thus from Lemmas~\ref{L:Y_T=X_T-1} and \ref{L:Y_T-X_T-tangent} below.
This is a long and tedious book-keeping check, essentially following the ideas of \cite[Proposition~4.4]{helm}.
\end{proof}

\begin{lemma}\label{L:Y_T=X_T-1}
Let $x=(A, \iota_{A}, \lambda_{A}, \alpha_{K'}) \in X'_\ttT(k)$ for a perfect field $k$.
Then there exist \emph{unique} $(B, \iota_B,\lambda_{B}, \beta_{K'_\ttT}, C, \iota_C; \phi_A, \phi_B)$ such that $(A, \iota_{A}, \lambda_{A}, \alpha_{K'}, B, \iota_B,\lambda_{B}, \beta_{K'_\ttT}, C, \iota_C; \phi_A, \phi_B)\in Y'_{\ttT}(k)$.
\end{lemma}

\begin{proof}

We first recall some notation regarding Dieudonn\'e modules.
Let $\tilde \calD_A$ denote the \emph{covariant} Dieudonn\'e module of $A[p^\infty]$.
Then $\calD_A := \tilde \calD_A/p$ is  the covariant Dieudonn\'e module of $A[p]$.
Given the action of $\calO_D\otimes_{\Z}\Z_p\simeq \rmM_2(\cO_{E}\otimes_{\Z}\Z_p)$ on $A$, we have direct sum decompositions
\[
\tcD_A^\circ := \gothe \tcD_A = \bigoplus_{\tilde{\tau}\in \Sigma_{E,\infty}} \tcD_{A, \tilde{\tau}}^\circ, \quad
 \cD_A^\circ := \gothe \cD_A  = \bigoplus_{\tilde{\tau}\in \Sigma_{E,\infty}} \cD_{A, \tilde{\tau}}^{\circ},
\]
where $\gothe$ denotes the idempotent $\bigl(\begin{smallmatrix} 1&0\\0&0\end{smallmatrix}\bigr)$.
By the theory of Dieudonn\'e modules, we have canonical isomorphisms
\[
H^{\cris}_1(A/k)_{W(k)}\cong \tcD_{A},\quad H^{\dR}_1(A/k)\cong \cD_{A},
\]
compatible with all the structures.
For $\tilde{\tau}\in  \Sigma_{E,\infty}$, we have the Hodge filtration
$
0\ra \omega^{\circ}_{A^\vee,\tilde{\tau}}\ra \cD_{A,\tilde{\tau}}^{\circ}\ra \Lie(A)^{\circ}_{\tilde{\tau}}\ra 0.
$
We use $\tilde \omega^\circ_{A^\vee, \tilde \tau}$ to denote the preimage of $\omega^\circ_{A^\vee, \tilde\tau} \subseteq \cD^\circ_{A, \tilde\tau}$ under the reduction map $\tcD^\circ_{A, \tilde\tau} \twoheadrightarrow \cD^\circ_{A, \tilde\tau}$.

We first construct $C$ from $A$. For each  $\tilde{\tau}\in \Sigma_{E,\infty}$, we define a $W(k)$-module $M^{\circ}_{\tilde{\tau}}$ with $\tcD^\circ_{A,\tilde{\tau}}\subseteq M^{\circ}_{\tilde{\tau}}\subseteq p^{-1}\tcD_{A,\tilde{\tau}}^\circ$ as follows. 
We put $M^{\circ}_{\tilde{\tau}}=\tcD^{\circ}_{\tilde{\tau}}$ unless $\tilde{\tau}\in \tilde{\Delta}(\ttT)^+$.
In the exceptional case, let $n$ be the integer as determined in Lemma~\ref{L:distance to T'} (or as in property (v) of $Y'_{\ttT}$ above), and put
\[
M^{\circ}_{\tilde{\tau}}:=p^{-1}F_{A,\es}^n(\tcD^\circ_{A,\sigma^{-n}\tilde \tau}),
\]
where $ F_{A,\es}^n$ is the $n$-iteration of the essential Frobenius on $\tcD^{\circ}_{A}$ defined in Notation~\ref{N:essential frobenius and verschiebung}.

If $\tilde{\tau}\in \tilde \Delta(\ttT)^+ \cap \ttT_E$, then the number $n$ for $\tilde{\tau}$ in Lemma~\ref{L:distance to T'} coincides with $n_{\tau}$ introduced in Subsection~\ref{S:partial-Hasse}. Since the partial Hasse invariant $h_{\tilde{\tau}}(A)$ vanishes for any $\tilde{\tau}\in \ttT_E$ by the definition of $X'_{\ttT}$, we see that $M^{\circ}_{\tilde\tau}=p^{-1}\tilde\omega^\circ_{A,\tilde{\tau}}$ for $\tilde \tau \in \tilde \Delta(\ttT)^+ \cap \ttT_E$.

We now check that, for any $\tilde{\tau}\in \Sigma_{E,\infty}$,
\begin{equation}\label{E:stability-FV}
F_A(M^\circ_{\sigma^{-1}\tilde{\tau}})\subseteq M^{\circ}_{\tilde{\tau}},\quad \textrm{and} \quad V_A(M^{\circ}_{\tilde{\tau}})\subseteq M^{\circ}_{\sigma^{-1}\tilde{\tau}}.
\end{equation}
Note that we are using the genuine but not essential Frobenius and Verschiebung here.
 We distinguish several cases:

\begin{itemize}
\item  $\tilde\tau,\sigma^{-1}\tilde\tau\notin \tilde\Delta(\ttT)^+$. Then  $M^\circ_{\tilde\tau}=\tcD^\circ_{A,\tilde\tau}$ and $M^\circ_{\sigma^{-1}\tilde\tau}=\tcD^\circ_{A,\sigma^{-1}\tilde\tau}$. Hence \eqref{E:stability-FV} is clear.

\item  $\tilde{\tau}\in \tilde\Delta(\ttT)^+$ and $\sigma^{-1}\tilde{\tau}\notin \tilde\Delta(\ttT)^+$. Then we have $M^{\circ}_{\sigma^{-1}\tilde{\tau}}=\tcD^\circ_{A,\sigma^{-1} \tilde\tau}$ and $M^\circ_{\tilde{\tau}}=p^{-1}F_A(\tcD_{A,\sigma^{-1}\tilde\tau}^\circ)$. Hence $F_A(M^{\circ}_{\sigma^{-1}\tilde{\tau}})\subseteq M^\circ_{\tilde{\tau}}$ is trivial, and $V_A( M^\circ_{\tilde{\tau}})=M^{\circ}_{\sigma^{-1}\tilde{\tau}}$.

\item $\tilde{\tau},\sigma^{-1}\tilde{\tau}\in \tilde\Delta(\ttT)^+$. Let $n$ be positive integer for $\tilde{\tau}$ as in  Lemma~\ref{L:distance to T'}. Then we have
$$
M^{\circ}_{\tilde{\tau}}=p^{-1} F_{A,\es}^n(\tcD^\circ_{A, \sigma^{-n}\tilde{\tau}})=
 F_{A,\es}\big( p^{-1} F_{A,\es}^{n-1}(\tcD^\circ_{A, \sigma^{-n}\tilde{\tau}})\big) =   F_{A,\es}\big(M^{\circ}_{\sigma^{-1}\tilde\tau} \big).
$$
The inclusions \eqref{E:stability-FV} are clear from this.

\item $\tilde{\tau}\notin \tilde\Delta(\ttT)^+$ and $\sigma^{-1}\tilde{\tau}\in \tilde\Delta(\ttT)^+$. In this case, $\sigma^{-1}\tilde\tau$ must be in $\ttT_E$ by Lemma~\ref{L:distance to T'}. Hence, we have $M^{\circ}_{\tilde\tau}=\tcD^\circ_{A,\tilde\tau}$ and $M^\circ_{\sigma^{-1}\tilde{\tau}}=p^{-1}\tilde\omega^\circ_{A,\sigma^{-1}\tilde{\tau}}$ as remarked above. We see thus $F_A(M^\circ_{\sigma^{-1}\tilde{\tau}})=M^\circ_{\tilde{\tau}}$ and $V_A(M^\circ_{\tilde{\tau}})=p M^\circ_{\sigma^{-1}\tilde{\tau}}$.

\end{itemize}

Consequently, if we put $M^{\circ}=\bigoplus_{\tilde\tau\in \Sigma_{E,\infty}}M^\circ_{\tilde\tau}$ and $M=(M^{\circ})^{\oplus 2}$, then $M$ is a Dieudonn\'e module, and $\tcD_{A}\subseteq M\subseteq p^{-1}\tcD_A$ with induced $F$ and $V$ on $M$. Consider the quotient $M/\tcD_A$. It corresponds to   a finite subgroup scheme $K$ of $A[p]$ stable under the action of $\cO_D$ by the covariant Dieudonn\'e theory. We put $C=A/K$ and let   $\phi_A:A\ra C$ denote the natural quotient so that the induced map $\phi_{A,*}: \tcD_{A}\ra \tcD_{C}$ is  identified with the natural inclusion $\tcD_{A}\hra M$.  The morphisms $F_C$ and $V_C$ on $\tcD_C$ are induced from those on $\tcD_{A}[1/p]$. 
    It is clear that $C$ is equipped with a natural action $\iota_C$ by $\cO_D$, and  $\phi_{A}$ satisfies conditions (iii) and (iv) for the moduli space $Y'_{\ttT}$. Conversely, if $C$ exists, the conditions (iii) and (iv) imply that $\tcD_{C,\tilde\tau}^{\circ}$ has to coincide with $M^\circ_{\tilde\tau}$. Therefore, $C$ is uniquely determined by $A$.
We finally remark that, by construction, $\tilde \calD_{C, \tilde \tau}^\circ / \tilde \calD_{A, \tilde \tau}^\circ$ is isomorphic to $k$ if $\tilde \tau \in \tilde \Delta(\ttT)^+$ and trivial otherwise.

We now construct the abelian variety $B$ and the isogeny $\phi_B:B\ra C$. Similarly to the construction for $C$, we will first define a $W(k)$-lattice $N^\circ=\bigoplus_{\tilde{\tau}\in \Sigma_{E,\infty}}N^\circ_{\tilde{\tau}}\subseteq \tcD_{C}^{\circ}$, with $N^\circ_{\tilde\tau}=\tcD^{\circ}_{C,\tilde{\tau}}$ unless $\tilde\tau\in \tilde\Delta(\ttT)^-$.  
 In the exceptional case, we put $N^\circ_{\tilde{\tau}}={F}_{A,\es}^n(\tcD^\circ_{C,\sigma^{-n}\tilde{\tau}})$, where $n$ is the positive integer given by Lemma~\ref{L:distance to T'}. 
 Here, we view $\tcD^{\circ}_{C,\sigma^{-n}\tilde{\tau}}$ as a lattice of $\tcD^{\circ}_{A,\sigma^{-n}\tilde\tau}[1/p]$ so that ${F}_{A,\es}^n(\tcD^\circ_{C,\sigma^{-n}\tilde{\tau}})$ makes sense.
Note once again that, if $\tilde{\tau}\in \tilde  \Delta(\ttT)^- \cap \ttT_E$, then  $n$ equals to $n_{\tau}$ defined in \ref{S:partial-Hasse}, and we have
$
N^\circ_{\tilde{\tau}}=\tilde{\omega}^\circ_{C, \tilde{\tau}}\simeq \tilde{\omega}^\circ_{A, \tilde{\tau}},
$
since $h_{\tilde{\tau}}(A)$ vanishes. We now check that $N^\circ$ is stable under $F_C$ and $V_C$, i.e. $F_{C}(N^\circ_{\sigma^{-1}\tilde\tau})\subseteq N^\circ_{\tilde{\tau}}$ and $V_C(N^\circ_{\tilde\tau})\subseteq N^\circ_{\sigma^{-1}\tilde\tau}$ for all $\tilde\tau\in \Sigma_{E,\infty}$. The same arguments for  $M$ above work verbatim in this case (with $\tilde \Delta(\ttT)^-$ in places of $\tilde \Delta(\ttT)^+$).
Again, we point out that, by construction, $\tilde \calD_{C, \tilde \tau}^\circ / \tilde \calD_{B, \tilde \tau}^\circ$ is isomorphic to $k$ if $\tilde \tau \in \tilde \Delta(\ttT)^-$ and trivial otherwise.

Therefore,  $N=(N^\circ)^{\oplus 2}$ is a Dieudonn\'e module such that the inclusions $p\tcD_{C}\subseteq N\subseteq \tcD_C$ respect the Frobenius and Verschiebung actions. 
In particular, the Dieudonn\'e submodule $N/p\tcD_{C}\subset \tcD_{C}/p\tcD_{C}$ is the covariant Dieudonn\'e module of a closed finite subgroup scheme $H\subset C[p]$ stable under the action of $\cO_D$. We put $B=C/H$, and define $\phi_{B}:B\ra C$ to be the isogeny such that the composite $C\ra B=C/H\xra{\phi_B} C$ is the multiplication by $p$. Then the induced morphism $\phi_{B,*}:\tcD_{B}\ra \tcD_{C}$ is identified with the natural inclusion $N\subseteq \tcD_{C}$. It is clear  that $B$ is equipped with a natural action by $\cO_D$, and the condition (v) for the moduli space $Y'_{\ttT}$ is satisfied. Conversely, if the abelian variety $B$ exists, then condition (v) implies that $\tcD_{B}^\circ$ has to be $N^\circ$ defined above. This means that $B$ is uniquely determined by $C$, thus by $A$.

To see condition (ix) of the moduli space $Y'_{\ttT}$, we consider  the quasi-isogeny:
\[
\lambda_B: B\xra{\phi_B}C\xleftarrow{\phi_A}A\xra{\lambda_A}A^\vee
\xleftarrow{\phi_{A}^\vee}C^\vee\xra{\phi_B^\vee}B^\vee.
\]
We have to show that $\lambda_B$ is a genuine isogeny, and verify that it satisfies conditions (b2) and (b3) in Theorem~\ref{T:unitary-shimura-variety-representability} for the Shimura variety $\bfSh_{K'_\ttT}(G'_{\tilde \ttS(\ttT)})$.
It is equivalent to proving that, when viewing $\tcD_{B, \tilde \tau}^\circ$ as a $W(k)$-lattice of $\tcD_{A, \tilde \tau}^\circ[1/p]$ via the quasi-isogeny $B\xra{\phi_B}C\xleftarrow{\phi_A}A$, the perfect alternating pairing
\[
\langle\ ,\ \rangle_{\lambda_A, \tilde \tau}: \tcD^\circ_{A, \tilde \tau}[1/p]\times \tcD^\circ_{A, \tilde \tau^c}[1/p]\ra W(k)[1/p]
\]
for $\tilde \tau \in \Sigma_{E, \infty/\gothp}$ induces  a perfect pairing of $\tcD_{B,\tilde{\tau}}^\circ\times \tcD_{B,\tilde\tau^c}^\circ\ra W(k)$ if  $\gothp$ is not of type $\beta^\sharp$ for $\ttS(\ttT)$, and induces an inclusion $\tilde \calD^\circ_{B, \tilde \tau^c} \subset \tilde \calD^{\circ, \vee}_{B, \tilde \tau}$ with quotient equal to $k$ if $\gothp$ is of type $\beta^\sharp$ for $\ttS(\ttT)$.
We discuss case by case.
\begin{itemize}
\item 
if $\gothp$ is of type $\beta^\sharp$ for $\ttS$, then both $\phi_A$ and $\phi_B$ induce isomorphisms on the $\gothp$-divisible groups and the statement is clear in this case.

\item  If $\gothp$ is of Case $\beta2$,  $\tilde\Delta(\ttT)^+_{/\gothp}=\emptyset$.
By the construction of $B$, we have $\tcD_{B,\tilde\tau}^{\circ}=\tcD_{A,\tilde\tau}^{\circ}$ unless $\tilde\tau\in \tilde\Delta(\ttT)^-_{/\gothp}$; in the latter case, $\tcD_{B,\tilde\tau}^{\circ}=\tilde{F}^n_{\es,\tilde\tau}(\tcD_{A,\sigma^{-n}\tilde\tau}^\circ)$ is a submodule of $\tcD_{A,\tilde\tau}^\circ$ with quotient isomorphic to $k$. 
Note that $\Sigma_{E, \infty/\gothp} = \tilde \Delta(\ttT)^-_{/\gothp} \coprod  (\tilde \Delta(\ttT)^{-}_{/\gothp})^c$. This implies that the pairing $\langle \ ,\ \rangle_{\lambda_A, \tilde \tau}$ induces an inclusion $\tilde \calD^\circ_{B, \tilde \tau^c} \subset \tilde \calD^{\circ, \vee}_{B, \tilde \tau}$ with quotient equal to $k$.

\item In all other cases, we have $\tilde \Delta(\ttT)^+_{/\gothp} = (\tilde \Delta(\ttT)^-_{/\gothp})^c$.  So
\begin{equation}\label{E:Dieud-B-2}
\tcD_{B,\tilde\tau}^\circ=
\begin{cases}
\tcD_{A,\tilde{\tau}}^\circ
 &\text{if }\tilde{\tau}\notin (\tilde\Delta(\ttT)^+_{/\gothp}\cup \tilde\Delta(\ttT)^-_{/\gothp});\\
p^{-1}{F}_{A,\es}^n(\tcD_{A,\sigma^{-n}\tilde\tau}^\circ) &\text{if }\tilde\tau\in \tilde\Delta(\ttT)^+_{/\gothp};\\
{F}_{A,\es}^n(\tcD_{A,\sigma^{-n}\tilde\tau}^\circ) &\text{if }\tilde{\tau}\in \tilde\Delta(\ttT)^-_{/\gothp}.
\end{cases}
\end{equation}
It is clear that $\langle\ ,\ \rangle_{\lambda_A}$ induces a perfect pairing on $\tcD_{B,\tilde{\tau}}^\circ\times\tcD_{B,\tilde\tau^c}^\circ$ if $\tilde\tau\notin (\tilde\Delta(\ttT)^+_{/\gothp}\cup \tilde\Delta(\ttT)^-_{/\gothp})$. If $\tilde\tau\in (\tilde\Delta(\ttT)^+_{/\gothp}\cup \tilde\Delta(\ttT)^-_{/\gothp})$, the perfect duality between $\tilde \calD_{B, \tilde \tau}^\circ$ and $\tilde \calD_{B, \tilde \tau^c}^\circ$ follows from the equality 
\[
\langle p^{-1}  F_{A,\es}^n u, F_{A,\es}^n v\rangle_{\lambda_A, \tilde \tau} = \langle u, v \rangle_{\lambda_A, \sigma^{-n} \tilde \tau}^{\sigma^n},
\]
for all $u\in \tcD^{\circ}_{A,\sigma^{-n}\tilde\tau}$ and $v\in \tcD^{\circ}_{A,\sigma^{-n}\tilde\tau^c}$.

\end{itemize}
This completes the verification of condition (viii) of the moduli space $Y'_\ttT$ and conditions (b2) and (b3) in Theorem~\ref{T:unitary-shimura-variety-representability} for $\bfSh_{K'_\ttT}(G'_{\tilde \ttS(\ttT)})$. It is also clear that $\lambda_B$ induces the involution $*_{\ttS(\ttT)}$ on $\cO_D=\cO_{D_{\ttS(\ttT)}}$.

We now check that the abelian variety $B$ has the correct signature required by the moduli space $\bfSh_{K'_{\ttT}}(G'_{\tilde \ttS(\ttT)})$.  For convenience of future reference, we put this into the following.

\begin{lemma}
\label{L:dimension count}
In the setup above, that is, knowing
\[
\dim \Coker(\phi_{A, *, \tilde \tau}) = \delta_{\tilde \Delta(\ttT)^+}(\tilde \tau)\quad \dim \Coker(\phi_{B, *, \tilde \tau}) = \delta_{\tilde \Delta(\ttT)^-}(\tilde \tau),
\]
where $\delta_\bullet(?)$ is  $1$ if $? \in \bullet$ and is $0$ if $? \notin \bullet$, we have $\dim \omega_{B^\vee/k, \tilde \tau}^\circ = 2- s_{\ttT, \tilde \tau}$ for all $\tilde \tau \in \Sigma_{E, \infty}$ if and only if $\dim \omega_{A^\vee/k, \tilde \tau}^\circ = 2- s_{ \tilde \tau}$ for all $\tilde \tau \in \Sigma_{E, \infty}$, with the numbers $s_{\ttT, \tilde \tau}$ defined as in Subsection~\ref{S:tilde S(T)}.
\end{lemma}
\begin{proof}
This is a simple dimension count.  We prove the sufficiency, and the necessity follows by reversing the argument.
Using the signature condition for the Shimura variety $X'_\ttT$, we have
\[
s_{\tilde \tau} = \dim_k \big( \tilde \calD^\circ_{A, \tilde \tau} \big/ V(\tilde\calD^\circ_{A, \sigma\tilde \tau} ) \big).
\]
Comparing this with the abelian variety $B$, we have
\[
\dim_k\frac{\tilde \calD^\circ_{B, \tilde \tau} }{ V(\tilde\calD^\circ_{B, \sigma\tilde \tau} )} =   \dim_k \frac{\tilde \calD^\circ_{A, \tilde \tau}}{ V(\tilde\calD^\circ_{A, \sigma\tilde \tau} )} - \big(
\dim_k \frac{ \tilde \calD^\circ_{C, \tilde \tau} }{ \tilde\calD^\circ_{B,\tilde \tau}}
-\dim_k \frac{ \tilde \calD^\circ_{C, \tilde \tau} }{ \tilde\calD^\circ_{A,\tilde \tau}}
\big)
+
\big(
\dim_k \frac{ \tilde \calD^\circ_{C, \sigma\tilde \tau} }{ \tilde\calD^\circ_{B,\sigma \tilde\tau}}
-\dim_k \frac{ \tilde \calD^\circ_{C, \sigma\tilde \tau} }{ \tilde\calD^\circ_{A,\sigma\tilde \tau}}
\big);
\]
here we used the fact that the quotient $\tilde \calD^\circ_{C, \sigma\tilde \tau} / \tilde \calD_{B, \sigma \tilde \tau}^\circ$ has the same dimension as $V(\tilde \calD^\circ_{C, \sigma\tilde \tau}) / V(\tilde \calD_{B, \sigma \tilde \tau}^\circ)$ and the same for $A$ in place of $B$ because $V$ is equivariant.
Using our construction of the abelian varieties $B$ and $C$, we deduce
\begin{equation}
\label{E:dimension count}
\dim_k \big( \tilde \calD^\circ_{B, \tilde \tau} \big/ V(\tilde\calD^\circ_{B, \sigma\tilde \tau} ) \big)= s_{\tilde \tau} - \big( \delta_{\tilde \Delta(\ttT)^-}(\tilde \tau) - \delta_{\tilde \Delta(\ttT)^+}(\tilde \tau) \big) + \big(\delta_{\tilde \Delta(\ttT)^-}(\sigma\tilde \tau) - \delta_{\tilde \Delta(\ttT)^+}(\sigma\tilde \tau)\big).
\end{equation}
Using the definition of $\tilde \Delta(\ttT)^\pm$, one checks case-by-case that the expression \eqref{E:dimension count} is equal to $s_{ \ttT, \tilde \tau}$.  We will only indicate the proof when $\tilde \tau \in \Sigma_{E, \infty/\gothp}$ for $\gothp$ in Case $\alpha1$, and leave the other cases as an exercise for the interested readers.
Indeed, under the notation from Subsections~ \ref{S:Delta-pm}, when $\gothp\in \Sigma_p$ is of type $\alpha1$, $\tilde \Delta(\ttT)_{/\gothp}^\pm = \coprod_i \tilde C_i^\pm$.
Then 
\[
\delta_{\tilde \Delta(\ttT)^+}(\tilde \tau)-\delta_{\tilde \Delta(\ttT)^+}(\sigma\tilde \tau)
=\left\{
\begin{array}{ll}
1&\textrm{if } \tilde \tau \textrm{ is one of }\sigma^{-a_1}\tilde \tau_i^c, \sigma^{-a_3}\tilde \tau_i^c, \dots;
\\
-1&\textrm{if } \tilde \tau \textrm{ is one of }\sigma^{-a_2}\tilde \tau_i^c, \sigma^{-a_4}\tilde \tau_i^c, \dots; \\
0&\textrm{otherwise};
\end{array}
\right.
\]
\[
\textrm{and}\
\delta_{\tilde \Delta(\ttT)^-}(\tilde \tau)-\delta_{\tilde \Delta(\ttT)^-}(\sigma\tilde \tau)
=\left\{
\begin{array}{ll}
1&\textrm{if } \tilde \tau \textrm{ is one of }\sigma^{-a_1}\tilde \tau_i, \sigma^{-a_3}\tilde \tau_i, \dots;
\\
-1&\textrm{if } \tilde \tau \textrm{ is one of }\sigma^{-a_2}\tilde \tau_i, \sigma^{-a_4}\tilde \tau_i, \dots; \\
0&\textrm{otherwise}.
\end{array}
\right.
\]
Putting these two formulas together and using the notation from Subsection~\ref{S:tilde S(T)}, we have
\[
\big(\delta_{\tilde \Delta(\ttT)^+}(\tilde \tau) -\delta_{\tilde \Delta(\ttT)^+}(\sigma\tilde \tau)
 \big) -\big(
\delta_{\tilde \Delta(\ttT)^-}(\tilde \tau)  -\delta_{\tilde \Delta(\ttT)^-}(\sigma \tilde \tau)\big)
= \delta_{\tilde \ttT'}(\tilde \tau^c)-\delta_{\tilde \ttT'}(\tilde \tau).
\]
This implies that \eqref{E:dimension count} is equal to $s_{ \ttT, \tilde \tau}$, and concludes the proof of Lemma~\ref{L:dimension count}.
\end{proof}

We now continue our proof of Lemma~\ref{L:Y_T=X_T-1} (as part of the proof of Proposition~\ref{P:Y_S=X_S}).
To fulfill condition (vi) of the moduli space $Y'_\ttT$, the tame level structure on $B$ is chosen and determined as the composite
\[
\beta^p_{K'^p_\ttT}:  \widehat \Lambda_{\ttS(\ttT)}^{(p)} \xrightarrow{\theta_\ttT^{-1}}
 \widehat \Lambda_{\ttS}^{(p)}
\xrightarrow{\alpha} T^{(p)}(A) \xrightarrow{T^{(p)}(\phi_A)}
T^{(p)}(C) \xrightarrow{T^{(p)}(\phi_B)^{-1}} T^{(p)}(B),
\]
where both $T ^{(p)}(\phi_A)$ and $T^{(p)}(\phi_B)$ are  isomorphisms because $\phi_A$ and $\phi_B$ are $p$-isogenies.

It remains to show that there exists a unique collection of subgroups $\beta_{p}$ satisfying \ref{T:unitary-shimura-variety-representability}(c2) for $\bfSh_{K'_{\ttT}}(G'_{\tilde \ttS(\ttT)})$ and properties (vii) and (viii) of $Y'_{\ttT}$. So the corresponding prime $\gothp \in \Sigma_p$ is either of type $\alpha^\sharp$ for $\ttS$ or in the Case $\alpha2$ of Subsection~\ref{S:quaternion-data-T}.
In the former case, we have $\ttT_{/\gothp} = \emptyset$, which forces $\Delta(\ttT)^\pm_{/\gothp} = \emptyset$ by definition.  So the induced morphisms $\phi_{A,\gothp}: A[\gothp^\infty] \to C[\gothp^\infty]$ and $\phi_{B,\gothp}: B[\gothp^\infty] \to C[\gothp^\infty]$ are isomorphisms.
Now, condition (vii) of the moduli space $Y'_\ttT$ determines that the level structure $\beta_\gothp$ is taken to be $\phi_{B,\gothp}^{-1}\big( \phi_{A, \gothp}(\alpha_\gothp)\big)$.

If $\gothp$ is in Case $\alpha 2$ of Subsection~\ref{S:quaternion-data-T}, the prime $\gothp$ splits into two primes $\gothq$ and $\bar\gothq$ in $E$.  
Using the polarization $\lambda_B$, we just need to show that there exists a unique subgroup $H_{\gothq}\subseteq B[\gothq]$ satisfying condition (vii).
Since $s_{\ttT, \tilde \tau} = 0$ or $2$ for $\tilde \tau \in \Sigma_{E,\infty/\gothp}$, both $ F_{B,\es,\tilde \tau}$ and $ V_{B,\es, \tilde \tau}$ are isomorphisms.
 We define a one-dimensional $k$-vector subspace $\cD_{H_{\gothq}}^\circ\subseteq \tcD_{B,\tilde\tau}^\circ/p\tcD_{B,\tilde\tau}^\circ$ for each $\tilde\tau\in \Sigma_{E,\infty/\gothq}$ as follows:
\begin{itemize}
\item If $\tilde\tau\in \tilde\Delta(\ttT)^-_{/\gothp}$, then $\tilde \calD^\circ_{B, \tilde \tau}$ is contained in $  \tilde \calD^\circ_{C, \tilde \tau}\cong\tilde \calD^\circ_{A, \tilde \tau}$ with quotient isomorphic to $k$. Put $\cD^{\circ}_{H_{\gothq}}=p\tcD^\circ_{A,\tilde\tau}/p\tcD^\circ_{B,\tilde\tau}$. 

\item If $\tilde\tau\notin \tilde\Delta(\ttT)^{-}_{/\gothp}$, let $n\in \NN$ be the least positive integer such that $\sigma^{-n}\tilde\tau\in \tilde\Delta(\ttT)^-_{/\gothp}$ (such $n$ exists because $\tilde \tau_0$ in Subsection~\ref{S:tilde S(T)} belongs to $\tilde \Delta(\ttT)^-_{/\gothp}$). Put 
$\cD^{\circ}_{H_{\gothq},\tilde\tau}=F^{n}_{B,\es}(\cD^{\circ}_{H_{\gothq},\sigma^{-n}\tilde\tau})$.

\end{itemize}
Put $\cD_{H_{\gothq}}=\bigoplus_{\tilde\tau\in \Sigma_{E,\infty/\gothq}}\cD^{\circ,\oplus 2}_{H_{\gothq}, \tilde\tau}$.
 Using the vanishing of the partial Hasse invariants $\{h_{\tilde\tau}(A): \tilde\tau\in {\ttT}_{E,\infty/\gothp} \}$, one checks easily that $\cD_{H_{\gothq}}\subseteq \cD_{B[\gothq]}$ is a Dieudonn\'e submodule. We define $H_{\gothq}\subseteq B[\gothq]$ as the finite subgroup scheme corresponding to $\cD_{H_{\gothq}}$ by covariant Dieudonn\'e theory. Then $\cD_{H_{\gothq}}$ is canonically identified with the kernel of the induced map
\[
\phi_{\gothq,*} \colon \cD_{B}=H^{\dR}_1(B/k)\ra \cD_{B/H_{\gothq}}=H^{\dR}_1((B/H_{\gothq})/k).
\]
Therefore,  $H_{\gothq}$ satisfies condition (viii) of the moduli space $Y'_\ttT$. This shows the existence of $H_{\gothq}$. For the uniqueness,  the condition (viii) forces the choice of $\cD_{H_{\gothq}, \tilde{\tau}}^{\circ}$ for $\tilde\tau\in \tilde{\Delta}(\ttT)^-_{/\gothp}$ and the stability under $F_B$ and $V_B$ forces the choice at other $\tilde \tau$'s. This concludes the proof that $Y'_{\ttT}\ra X'_{\ttT}$ induces a bijection on closed points.
\end{proof}

\begin{lemma}\label{L:Y_T-X_T-tangent}
The map $\eta_1: Y'_{\ttT}\ra X'_{\ttT}$ induces an isomorphism of tangent spaces at every closed point.
\end{lemma}
\begin{proof}
Let $y=(A, \iota_A,\lambda_A,  \alpha_{K'}, B, \iota_B, \lambda_B, \beta_{K'_\ttT}, C, \iota_C; \phi_A, \phi_B)$ be a closed point of $Y'_{\ttT}$ with values in a perfect field $k$, and $x=(A, \iota_A,\lambda_A,  \alpha_{K'})$ be its image in $X'_{\ttT}$. We have to show that $Y'_{\ttT}\ra X'_{\ttT}$ induces an isomorphism of $k$-vector spaces between tangent spaces: $T_{Y'_{\ttT}, y}\xra{\cong}T_{X'_{\ttT}, x}$.

Set $\II=\Spec(k[\epsilon]/\epsilon^2)$. By deformation theory,  $T_{X'_{\ttT},x}$ is identified with the $\II$-valued points $x_{\II}=(A_{\II}, \iota_{A,\II}, \lambda_{A,\II}, \alpha_{K',\II})$ of $X'_{\ttT}$ with reduction $x\in X'_{\ttT}(k)$ modulo $\epsilon$.
  In the proof of Proposition~\ref{Prop:smoothness}, we have seen that giving an $x_{\II}$ is equivalent to giving, for each $\tilde\tau\in \Sigma_{E,\infty}$, a direct factor  $\omega^{\circ}_{A^\vee,\II,\tilde\tau}\subseteq H^{\cris}_1(A/k)_{\II,\tilde\tau}^\circ$ that  lifts $\omega^\circ_{A^\vee,\tilde\tau}\subseteq H^{\dR}_1(A/k)^{\circ}_{\tilde \tau}$ and satisfies the following properties:
\begin{itemize}
\item[(a)] If $\tilde\tau\in \Sigma_{E,\infty/\gothp}$ with $\gothp$ not of type $\beta^\sharp$ for $\ttS$, then $\omega_{A^\vee,\II,\tilde\tau}^{\circ}$ and $\omega_{A^\vee,\II,\tilde\tau^c}^\circ$ are orthognal complements of each other under the perfect pairing
\[
H^{\cris}_1(A/k)_{\II,\tilde\tau}^\circ\times H^{\cris}_1(A/k)_{\II,\tilde\tau^c}^\circ \ra k[\epsilon]/\epsilon^2
\]
induced by the polarization $\lambda_A$.

\item[(b)] If $\tilde\tau \in \tilde \ttS_\infty$, then $\omega_{A^\vee,\II,\tilde\tau}^\circ=0$ and $\omega_{A^\vee,\II,\tilde\tau^c}^\circ = H^{\cris}_1(A/k)^{\circ}_{\II,\tilde\tau^c}$.

\item[(c)] If $\tilde\tau$ restricts to $\tau\in \ttT$, then $\omega_{A^\vee,\II,\tilde\tau}^{\circ}$ has to be $F_{A,\es}^{n_{\tau}}(H^\cris_1(A^{(p^{n_{\tau}})}/k)_{\II,\tilde\tau}^\circ)$, where $n_\tau$ is as introduced in Subsection~\ref{S:partial-Hasse} and $F_{A,\es}^{n_\tau}$ on the crystalline homology are defined in the same way as $F_{A, \es}^{n_\tau}$ on the de Rham homology as in Notation~\ref{N:essential frobenius and verschiebung}.
 Since we are in characteristic $p$, we have $F_{A,\es}^{n_{\tau}}(H^\cris_1(A^{(p^{n_{\tau}})}/k)_{\II,\tilde\tau}^\circ)=\omega^\circ_{A^\vee,\tilde\tau}\otimes k[\epsilon]/\epsilon^2$.
\end{itemize}
Note also that, the crystal nature of $H^{\cris}_1(A/k)$ implies that there is a canonical isomorphism
$$H^{\cris}_1(A/k)_{\II} \cong H^{\dR}_1(A/k)\otimes_k k[\epsilon]/\epsilon^2.$$

 We have to show that, given such an $x_{\II}$, or equivalently given the liftings $\omega_{A^\vee, \II,\tilde\tau}^\circ$ as above, there exist unique 
 $(B_{\II},\iota_{B,\II}, \lambda_{B,\II}, \beta_{K'_\ttT, \II}, C_{\II}, \iota_{C, \II}; \phi_{A,\II},\phi_{B,\II})
 $ over $\II$ deforming
 $(B, \iota_B,\lambda_B, \beta_{K'_\ttT}, C, \iota_C; \phi_A, \phi_B)$ such that $(A_{\II}, \iota_{A,\II}, \lambda_{A,\II}, \alpha_{K',\II}, B_{\II},  \iota_{B,\II},\lambda_{B,\II},\beta_{K'_\ttT, \II}, C_{\II}, \iota_{C, \II}; \phi_{A,\II},\phi_{B,\II})$ is an $\II$-valued point of $Y'_{\ttT}$.

We start with $C_{\II}$. To show its existence, it suffices to construct, for each $\tilde{\tau}\in \Sigma_{E,\infty}$, a direct factor $\omega_{C^\vee,\II,\tilde\tau}^\circ\subseteq H^{\cris}_1(C/k)_{\II,\tilde\tau}^\circ$ that lifts $\omega^\circ_{C^\vee, \tilde\tau}\subseteq\cD_{C,\tilde\tau}^\circ\cong  H^{\dR}_1(C/k)_{\tilde\tau}^\circ$.
\begin{itemize}
\item
When neither $\tilde \tau$ nor $\sigma \tilde \tau $ belongs to $ \tilde \Delta(\ttT)^+$, 
$\phi_{A, *, ?}: H^{\dR}_{1}(A/k)^\circ_{?}\xra{\sim}H^{\dR}_1(C/k)^\circ_{?}$  is an isomorphism for $?=\tilde\tau, \sigma\tilde\tau$.
We take $\omega_{C^\vee,\II,\tilde\tau}^\circ\subseteq H^{\cris}_1(C/k)_{\II,\tilde\tau}^\circ$ to be the image of $\omega^\circ_{A^\vee, \II,\tilde\tau}\subseteq H^{\cris}_1(A/k)_{\II,\tilde\tau}^\circ$ under the induced morphism $\phi^\cris_{A, *, \tilde  \tau}$ on the crystalline homology.

\item
When either one of $\tilde\tau$ and $\sigma \tilde \tau$ belongs to $\tilde \Delta(\ttT)^+$, an easy dimension count argument similar to Lemma~\ref{L:dimension count} (using Lemma~\ref{L:property of Delta}) shows that $\omega_{C^\vee, \tilde \tau}^\circ$ is either $0$ or of rank $2$. So there is a unique obvious such lift $\omega_{C^\vee,\II,\tilde\tau}^\circ$.
\end{itemize}

This finishes the construction of $\omega^\circ_{C^\vee,\II,\tilde\tau}$ for all $\tilde\tau$, hence one gets a deformation $C_{\II}$ of $C$, carrying a natural action of $\calO_D$. It is clear that the map $\phi^\cris_{A,*}: H^{\cris}_1(A/k)^\circ_{\II,\tilde\tau}\ra H^{\cris}_1(C/k)^\circ_{\II,\tilde\tau}$ sends $\omega^\circ_{A^\vee,\II,\tilde\tau}$ to $\omega^\circ_{C^\vee,\II,\tilde\tau}$. Hence, $\phi_{A}$ deforms to an $\calO_D$-equivariant isogeny of abelian schemes $\phi_{A,\II}:A_{\II}\ra C_{\II}$ by \cite[2.1.6.9]{lan}. 

We check now that $\phi_{A_\II}$ satisfies condition (iv) of the moduli space $Y'_{\ttT}$. We note that the map $\phi_{A_{\II},*}: H^{\dR}_1(A_{\II}/\II)\ra H^{\dR}_1(C_{\II}/\II)$ is canonically identified with $\phi_{A,*}^{\cris}: H^\cris_1(A/k)_{\II} \to  H^{\cris}_1(C /k)_\II$ by crystalline theory, which is in turn isomorphic to the base change of $\phi_{A,*}:H^\dR_1(A/k)\ra H^{\dR}_1(C/k)$ via  $k\hra k[\epsilon]/\epsilon^2$. Let  $\tilde\tau\in \tilde \Delta(\ttT)^+$. Since the Frobenius on $k[\epsilon]/\epsilon^2$ factors as
$$
k[\epsilon]/\epsilon^2\twoheadrightarrow k\xra{x\mapsto x^p} k\hra k[\epsilon]/\epsilon^2,
$$
we see that
$$
F_{A,\es}^n(H^{\dR}_1(A_{\II}^{(p^n)}/\II)^\circ_{\tilde\tau})
=F_{A,\es}^n(H^{\dR}_1(A^{(p^n)}/k)^\circ_{\tilde\tau})\otimes_k k[\epsilon]/\epsilon^2.
$$
Hence, the kernel of $\phi_{A_{\II},*,\tilde\tau}: H^{\dR}_1(A_{\II}/\II)^\circ_{\tilde\tau}\ra H^\dR_1(C_{\II}/\II)^\circ_{\tilde\tau}$ coincides with $F_{A,\es}^n(H^{\dR}_1(A_{\II}^{(p^n)}/\II)^\circ_{\tilde\tau})$, since it is the case after reduction modulo $\epsilon$. This shows that $\phi_{A_\II}$ satisfies the condition (iv). Conversely, it is clear that, if $C_{\II}$ and $\phi_{A_\II}$ satisfy condition (iv), then they have to be of the form as above.

We show now that  there exists a unique deformation $(B_{\II}, \phi_{B_\II})$ over $\II$ of $(B,\phi_B)$ satisfying condition (vi) of the moduli space $Y'_\ttT$. To construct $B_{\II}$, one has to specify, for each $\tilde\tau\in \Sigma_{E,\infty}$ a subbundle $\omega_{B^\vee, \II, \tilde\tau}^\circ\subseteq H^{\cris}_1(B/k)_{\II,\tilde\tau}^\circ$ lifting $\omega^\circ_{B^\vee,\tilde\tau}\subseteq H^{\dR}_1(B/k)^\circ_{\tilde\tau}$.
Similar to the discussion above,
\begin{itemize}

\item
If neither $\tilde \tau$ nor $\sigma \tilde \tau$ belong to $\tilde \Delta(\ttT)^-$, then $\phi_{B, *, ?}: H^{\dR}_{1}(B/k)^\circ_{?}\xra{\sim}H^{\dR}_1(C/k)^\circ_{?}$  is an isomorphism for $?=\tilde\tau, \sigma\tilde\tau$.
We take $\omega_{B^\vee,\II,\tilde\tau}^\circ\subseteq H^{\cris}_1(B/k)_{\II,\tilde\tau}^\circ$ to be the image of $\omega^\circ_{C^\vee, \II,\tilde\tau}\subseteq H^{\cris}_1(C/k)_{\II,\tilde\tau}^\circ$ under the induced morphism $\phi_{B, *, \tilde  \tau}^{-1}$ on the crystalline homology.

\item 
If at least one of $\tilde \tau$ and $\sigma \tilde \tau$ belongs to $\tilde \Delta(\ttT)^-$, then an easy dimension count argument similar to Lemma~\ref{L:dimension count} (using Lemma~\ref{L:property of Delta}) shows that $\omega^\circ_{B^\vee, \tilde \tau}$ is either $0$ or of rank $2$. There is a unique obvious such lift $\omega^\circ_{B^\vee, \II, \tilde \tau}$.
\end{itemize}
This defines  $\omega_{B^\vee,\II, \tilde\tau}^\circ$ for all $\tilde\tau\in \Sigma_{E,\infty}$. Hence, one gets a deformation $B_{\II}$ of $B$ over $k[\epsilon]/\epsilon^2$. It is immediate from the construction that the action of $\cO_D$ lifts to $B_{\II}$, and $\phi^\cris_{B,*,\tilde\tau}: H^{\cris}_1(B/k)^{\circ}_{\II, \tilde\tau}\ra H^{\cris}_1(C/k)^\circ_{\II,\tilde\tau}$ sends $\omega_{B^\vee,\II, \tilde\tau}^\circ$ to $ \omega_{C^\vee, \II,\tilde\tau}^\circ$ for all $\tilde\tau\in \Sigma_{E,\infty}$. Hence, $\phi_{B}:B\ra C$ deforms to an $\calO_D$-equivariant isogeny $\phi_{B_\II}: B_{\II}\ra C_{\II}$.  In the same way as for $\phi_{A_{\II}}$, we prove  that $\phi_{B_\II}$ satisfies condition (v) of the moduli space $Y'_{\ttT}$, and conversely the condition (v) determines $B_{\II}$ uniquely.

Let $\langle\ ,\ \rangle_{\lambda_{B}}:H^{\cris}_1(B/k)_{\II}^\circ\times H^\cris_1(B/k)^\circ_{\II}\ra k[\epsilon]/\epsilon^2$ be the pairing induced by the polarization $\lambda_B$. To prove that $\lambda_{B}$ deforms (necessarily uniquely) to a polarization $\lambda_{B_\II}$ on $B_{\II}$, it suffices to check that $\langle\ ,\ \rangle_{\lambda_B}^\cris$ vanishes on  $\omega_{B^\vee, \II, \tilde\tau}^\circ\times \omega_{B^\vee,\II, \tilde\tau^c}^\circ$ for all $\tilde\tau\in \Sigma_{E,\infty}$ (cf. \cite[2.1.6.9, 2.2.2.2, 2.2.2.6]{lan}):

\begin{itemize}
\item If   $\tau=\tilde\tau|_F$ lies in  $\ttS(\ttT)_\infty$, this is trivial, because one of $\omega^{\circ}_{B^\vee, \II, \tilde\tau}$ and $\omega_{B^\vee, \II, \tilde\tau^c}^\circ$ is equal to $0$ and the other one is equal to $H^{\cris}_1(B/k)_{\II,\tilde\tau}^\circ$ by construction.

\item If  $\tau=\tilde\tau|_F$ is not in $\ttS(\ttT)_\infty$, then the natural isomorphism $H^\cris_1(B/k)^{\circ}_{\II,\star}\cong H^\cris_1(A/k)^\circ_{\II, \star}$ sends $\omega^\circ_{B^\vee, \II, \star}$ to $\omega^\circ_{A^\vee,\II, \star}$ for $\star=\tilde\tau, \tilde\tau^c$.  The vanishing of $\langle \ ,\ \rangle_{\lambda_B}$ on $\omega^\circ_{B^\vee,\II,\tilde\tau}\times \omega_{B^\vee, \II, \tilde\tau^c}^\circ$ follows from the similar statement with $B$ replaced by $A$.
\end{itemize}
Therefore, we see that $\lambda_B$ deforms to a polarization $\lambda_{B_{\II}}$ on $B_\II$. Since $\lambda_{B_\II,*}^\dR:H^{\dR}_1(B/\II)\ra H^{\dR}_1(B^\vee/\II)$ is canonically identified with $\lambda_{B,*}^\cris: H^{\cris}_1(B/k)_{\II}\ra H^{\cris}_1(B^\vee/k)_{\II}$, which is in turn identified with the base change of $\lambda^\dR_{B,*}$ via $k\hra k[\epsilon ]/\epsilon^2$, it is clear that condition (ii) regarding the polarization is preserved by the deformation $\lambda_{B_{\II}}$.

It remains to prove that $\beta_{K'_\ttT}$ deforms to
$\beta_{K'_\ttT, \II}$.  The deformation of the tame level structure is automatic; the deformation of the subgroup at $p$-adic places of type $\alpha^\sharp$ and $\alpha2$ is also unique, by the same argument as in Theorem~\ref{T:unitary-shimura-variety-representability}.
\end{proof}

\subsection{A lift of $I_\ttT$}
\label{S:tilde IT}
Recall that $I_\ttT$ is the subset $\ttS(\ttT)_\infty - (\ttS_\infty \cup \ttT)$ defined in Theorem~\ref{T:main-thm}.
We use $\tilde I_{\ttT}$ to denote the subset of complex embeddings of $E$ consisting of the unique lift $\tilde \tau$ of every element $\tau \in I_\ttT$, for which $\tilde \tau^c \in \tilde \ttS(\ttT)_\infty$.  We describe this set  explicitly as follows.

We write $I_{\ttT/\gothp} = I_\ttT \cap \Sigma_{\infty/\gothp}$ and $\tilde I_{\ttT/\gothp} = \tilde I_\ttT \cap \Sigma_{E, \infty/\gothp}$ for $\gothp \in \Sigma_p$. They are empty set unless $\gothp$ is of type $\alpha 1$ or $\beta 1$.
When $\gothp$ is of type $\alpha1$ or $\beta1$, using  the notation of Subsection~\ref{S:quaternion-data-T},
$I_{\ttT/\gothp}$ consists of $\sigma^{-m_i-1}\tau_i$ for all $i$ such that $\#(C_i \cap \ttT_{/\gothp})$ is odd.
In the notation of Subsection~\ref{S:tilde S(T)}, the set $\tilde I_{\ttT/\gothp}$ consists of  $\sigma^{-a_{r_i}}\tilde \tau_i$ for all $i$ such that $\#(C_i \cap \ttT_{/\gothp})$ is odd.
We remark that, in either case, for any $\tilde \tau$ lifting a place $\tau \in I_{\ttT}$, $\tilde \tau \notin \tilde \Delta(\ttT)^+ \cup \tilde \Delta(\ttT)^-$.

\subsection{Isomorphism of  $Y'_\ttT$ with $Z'_{\ttT}$}
\label{S:Y_T=Z_T}
Let $Z'_\ttT$ be the moduli space over $k_0$ representing the functor that
takes a locally noetherian $k_0$-scheme $S$ to the set of isomorphism classes of tuples $(B, \iota_B, \lambda_B, \beta_{K'_{\ttT}}, J^\circ)$, where
\begin{itemize}
\item[(i)]
$(B, \iota_B, \lambda_B, \beta_{K'_{\ttT}})$ is an $S$-valued point of   $\bfSh_{K'_{\ttT}}(G'_{\tilde \ttS(\ttT)})$.

\item[(ii)]
$J^\circ$ is the collection of sub-bundles $J^\circ_{\tilde\tau}\subseteq H^{\dR}_1(B/S)_{\tilde\tau}^\circ$ locally free of rank $1$ for each   $\tilde\tau\in \tilde I_\ttT$.\end{itemize}

It is clear that $Z'_{\ttT}$ is a $(\PP^1)^{I_{\ttT}}$-bundle over $\bfSh_{K'_{\ttT}}(G'_{\tilde \ttS(\ttT)})$.

We define a morphism $\eta_2: Y'_{\ttT}\ra Z'_{\ttT}$ as follows:
Let $S$ be a locally noetherian  $k_0$-scheme, and $x=(A, \iota_{A},\lambda_A, \alpha_{K'}, B, \iota_B, \lambda_B, \beta_{K'_\ttT}, C, \iota_C;\phi_A, \phi_B)$ be an $S$-valued point of $Y'_{\ttT}$. We define $\eta_2(x)\in Z'_{\ttT}(S)$ to be the isomorphism class  of the tuple $(B, \iota_B, \lambda_B, \beta_{K'_{\ttT}}, J^\circ)$,
where $J^\circ_{\tilde \tau}$ is given by $\phi_{B,*,\tilde\tau}^{-1}\circ \phi_{A,*,\tilde\tau}(\omega_{A^\vee,\tilde\tau}^\circ)$ for the isomorphisms
\[
\xymatrix{
H^{\dR}_1(A/S)^\circ_{\tilde\tau}\ar[r]^{\phi_{A,*,\tilde\tau}}_{\cong} &H^{\dR}_1(C/S)^\circ_{\tilde\tau} & H^{\dR}_1(B/S)^\circ_{\tilde\tau}\ar[l]^{\cong}_{\phi_{B,*, \tilde\tau}}.
}\]
Note that $\tilde \tau \notin \tilde \Delta(\ttT)^\pm$ implies that both $\phi_{A, *, \tilde \tau}$ and $\phi_{B, *, \tilde \tau}$ are isomorphisms.

\begin{prop}\label{P:Y_T=Z_T}
The morphism $\eta_2: Y'_{\ttT}\ra Z'_{\ttT}$ is an isomorphism.
\end{prop}

We note that Theorem~\ref{T:main-thm-unitary} follows immediately from this Proposition and Proposition~\ref{P:Y_S=X_S}.

\begin{proof}

As in the proof of Proposition~\ref{P:Y_S=X_S}, it suffices to prove that $\eta_2$ induces a bijection on  the closed points  and on tangent spaces.

\textbf{Step I.} We show first that $\eta_2$ induces a bijection on closed points. Let $z=(B,\iota_B, \lambda_B, \beta_{K'_\ttT}, J^\circ)$ be a closed point of $Z'_{\ttT}$ with values in $k=\Fpb$. We have to show that there exists a \emph{unique} point  $y=(A, \iota_{A}, \lambda_A, \alpha_{K'}, B, \iota_B, \lambda_B, \beta_{K'_\ttT}, C, \iota_C;\phi_A, \phi_B)\in Y'_{\ttT}(k)$ with $\eta_2(y)=z$. To prove this, we basically reverse the  construction in  the proof of Lemma~\ref{L:Y_T=X_T-1}.

We start by reconstructing $C$ from $B$ and $J^\circ$. We denote by $\tcD_B=(\tcD^{\circ}_B)^{\oplus 2}$ the covariant Dieudonn\'e module of $B$, and by $\tcD^\circ_B=\bigoplus_{\tilde\tau\in \Sigma_{E,\infty}}\tcD_{B,\tilde\tau}^\circ$ the canonical decomposition according to the $\calO_E$-action. We construct a Dieudonn\'e submodule  $M^\circ=\bigoplus_{\tau\in \Sigma_{E,\infty}} M^\circ_{\tilde\tau}\subseteq \tcD_{B}^\circ [1/p]$ with $\tcD^\circ_{B}\subseteq M^\circ\subseteq p^{-1}\tcD^\circ_{B}$ as follows. 
Let $\tilde\tau\in \Sigma_{E,\infty/\gothp}$ with $\gothp\in \Sigma_p$. If $\tilde\tau\notin \tilde \Delta(\ttT)^-$, we put $M^\circ_{\tilde\tau}=\tcD^\circ_{B,\tilde\tau}$. To define $M^\circ_{\tilde\tau}$ in the other case, we separate the discussion according to the type of $\gothp$.

\begin{itemize}
\item (Case $\alpha 1$ and $\beta 1$)
Recall our notation from Subsections~\ref{S:quaternion-data-T}, \ref{S:tilde S(T)}, \ref{S:Delta-pm}, and \ref{S:tilde IT}.  There are two subcases according to the parity of $\#(C_i \cap \ttT_{/\gothp})$, where $C_i$ is a chain of $\ttS_{\infty/\gothp}\cup \ttT_{/\gothp}$ as in Subsection~\ref{S:quaternion-data-T}.  (It should not be confused with the abelian variety $C$.)

\begin{itemize}

\item When $r_i=\#(C_i \cap \ttT_{/\gothp})$ is odd, $\sigma^{-m_i-1}\tilde\tau_i\in \tilde I_{\ttT/\gothp}$ so that  $J^\circ_{\sigma^{-m_i-1}\tilde\tau_i}$ is defined. 
In this case, all $\tau = \sigma^{-\ell}\tau_i$ belong to $ \ttS(\ttT)_{\infty/\gothp}$ for $0 \leq \ell  \leq m_i+1$; so $s_{\ttT, \sigma^{-\ell} \tilde \tau_i} \in \{0,2\}$ and the essential Frobenii
\[
\xymatrix{
 F_{B,\es}^{m_i+1-\ell}:\
\tcD^\circ_{B,\sigma^{-m_i-1}\tilde\tau_i}\ar[r]_-{ F_{B,\es}}^-{\cong}
&\tcD^\circ_{B,\sigma^{-m_i}\tilde\tau_i}\ar[r]^-{\cong}_-{ F_{B,\es}} &\cdots \ar[r]^-{\cong}_-{ F_{B, \es}}&\tcD^{\circ}_{B,\sigma^{-\ell}\tilde\tau_i}
}\]
are isomorphisms for such an $\ell$.
If $a_j \leq \ell < a_{j+1}$ for some odd number $j$, we put
\[
M^\circ_{\sigma^{-\ell}\tilde\tau_i}=p^{-1} F_{B,\es}^{m_i+1-\ell}(\tilde J^\circ_{\sigma^{-m_i-1}\tilde\tau_i}),
\]
where $\tilde J^\circ_{\sigma^{-m_i-1}\tilde\tau_i}$ denotes the inverse image in $\tcD_{B,\sigma^{-m_i-1}\tilde\tau_i}^\circ$ of $J^\circ_{\sigma^{-m_i-1}\tilde\tau_i}\subseteq \cD^\circ_{B,\sigma^{-m_i-1}\tilde\tau_i}$  under the natural reduction map modulo $p$. For other $\ell$, we have already defined $M^\circ_{\sigma^{-\ell}\tilde\tau_i}$ to be $\tilde \calD^\circ_{B, \sigma^{-\ell}\tilde\tau_i}$.

\item When $r_i=\#(C_i\cap \ttT_{/\gothp})$ is even, there is no $J^\circ$ involved in this construction.
Note that
all $\tau = \sigma^{-\ell}\tau_i$ belong to $\ttS(\ttT)_{\infty/\gothp}$ for $0 \leq \ell \leq m_i$. So $s_{\ttT, \sigma^{-\ell} \tilde \tau_i} \in \{0,2\}$ and in the sequence of  essential Frobenii
\[
\xymatrix{\quad\quad
 F_{B,\es}^{m_i-\ell+1}:\
\tcD^{\circ}_{B,\sigma^{-m_i-1}\tilde\tau_i}\ar[r]_-{F_{B}}&\tcD^\circ_{B,\sigma^{-m_i}\tilde\tau_i}\ar[r]_-{ F_{B,\es}}^-{\cong}
&\tcD^\circ_{B,\sigma^{-m_i+1}\tilde\tau_i}\ar[r]^-{\cong}_-{ F_{B,\es}} &\cdots \ar[r]^-{\cong}_-{ F_{B, \es}}&\tcD^{\circ}_{B,\sigma^{-\ell}\tilde\tau_i},
}
\]
all the maps except the first one are isomorphisms.
If $a_j \leq \ell < a_{j+1}$ for some odd number $j$, we put
$$
M^\circ_{\sigma^{-\ell}\tilde\tau_i}=p^{-1} F_{B,\es}^{m_i-\ell+1} (\tcD^{\circ}_{B,\sigma^{-m_i-1}\tilde\tau_i});
$$
then we have  $\dim_k(M^\circ_{\sigma^{-\ell}\tilde\tau_i}/\tcD^\circ_{B,\sigma^{-\ell}\tilde\tau_i})=1$, since the cokernel of  $F_B: \tcD^\circ_{B,\sigma^{-m_i-1}\tilde\tau_i}\ra \tcD^{\circ}_{B,\sigma^{-m_i}\tilde\tau_i}$ has dimension $1$, as $s_{\ttT,\sigma^{-m_i-1}\tilde\tau_i}=1$.
(For other $\ell$, we have already defined $M^\circ_{\sigma^{-\ell}\tilde\tau_i}$ to be $\tilde \calD^\circ_{B, \sigma^{-\ell}\tilde\tau_i}$.)
\end{itemize}

\item (Case $\alpha 2$) In this case, $\gothp$ is a prime of type $\alpha^\sharp$ for $\Sh_{K_{\ttT}}(G_{\ttS(\ttT)})$, and it splits into two primes $\gothq$ and $\bar\gothq$ in $E$. Let $H_{\gothq}\subseteq B[\gothq]$ be the closed subgroup scheme given in the data $\beta_{K'_\ttT}$. 
Let $H_{\bar \gothq}$ be its annihilator under the Weil pairing between $B[\gothq]$ and $B[\bar \gothq]$ induced by $\lambda_B$. (We collectively write $H_\gothp$ for $H_\gothq \times H_{\bar \gothq}$.)
Let $\cD^\circ_{H_\gothp}=\bigoplus_{\tilde\tau\in \Sigma_{E,\infty/\gothp}}\cD^\circ_{H_\gothp,\tilde\tau} \subseteq \cD_{B}^\circ$ be the reduced covariant Dieudonn\'e module of $H_\gothp = H_{\gothq} \times H_{\bar \gothq}$. Then each $\cD^\circ_{H_{\gothp},\tilde\tau}$ is necessarily one-dimensional over $k$ for all $\tilde\tau\in \Sigma_{E,\infty/\gothp}$. For $\tilde\tau\in \tilde\Delta(\ttT)^-_{/\gothp}$, we define
$$
M^\circ_{\tilde\tau} =p^{-1}\tilde\cD^\circ_{H_\gothp,\tilde\tau},
$$
where  $\tcD^\circ_{H_\gothp,\tilde\tau}$ denotes the inverse image in $\tcD^\circ_{B,\tilde\tau}$ of the subspace $\cD^\circ_{H_{\gothp},\tilde\tau}\subseteq \cD_{B,\tilde\tau}^\circ$.  (We have defined $M^\circ_{\tilde \tau} = \tilde D^\circ_{B, \tilde \tau}$ for $\tilde \tau \notin \tilde \Delta(\ttT)^-_{/\gothp}$ before.)

\item(Case $\beta 2$)  In this case, $\gothp$ is a prime of type $\beta^\sharp$ for $\Sh_{K_{\ttT}}(G_{\ttS(\ttT)})$.  For $\tilde\tau\in \tilde \Delta(\ttT)^-_{/\gothp}$, let $\lambda_{B,*,\tilde\tau}: \cD^{\circ}_{B,\tilde\tau}\ra \cD^\circ_{B^\vee,\tilde\tau}$ be the morphism induced by the polarization $\lambda_B$. By Theorem~\ref{T:unitary-shimura-variety-representability}(b3), $J^\circ_{\tilde\tau}:=\Ker(\lambda_{B,*,\tilde\tau})$ is a $k$-vector space of dimension $1$.
We set $M^\circ_{\tilde\tau}=p^{-1}\tilde J^\circ_{\tilde\tau}$ for such $\tilde \tau$, where $\tilde J^\circ_{\tilde \tau}$ is the preimage of $J^\circ_{\tilde \tau}$ under the reduction map $\tilde \calD^\circ_{B, \tilde \tau} \to \calD^\circ_{B, \tilde \tau}$. 
Note that when viewing $\tcD^\circ_{B^\vee,\tilde\tau}$ as a lattice of $\tcD^\circ_{B,\tilde\tau}[1/p]$ using the polarization, we have $M^\circ_{\tilde\tau} =\tcD^\circ_{B^\vee,\tilde\tau}$.
(We have defined $M^\circ_{\tilde \tau} = \tilde D^\circ_{B, \tilde \tau}$ for $\tilde \tau \notin \tilde \Delta(\ttT)^-_{/\gothp}$ before.)

\end{itemize}

This concludes the definition of $M^\circ\subseteq p^{-1}\tcD^\circ_{B}$. One checks easily that $M^\circ$ is stable under $F_B$ and $V_B$. Consider the quotient Dieudonn\'e modules
$$
M/\tcD_{B}=(M^\circ/\tcD^\circ_{B})^{\oplus 2}\subseteq p^{-1}\tcD_{B}/\tcD_B\cong \cD_B.
$$
 Then $M/\tcD_B$ corresponds to a closed finite group scheme $G\subseteq B[p]$ stable under the action of $\cO_D$. We put $C=B/G$ with the induced $\calO_D$-action, and define $\phi_B:B\ra C$ as the canonical $\calO_D$-equivariant isogeny. Then the natural induced map $\phi_{B,*}:\tcD^\circ_{B}\ra \tcD^\circ_{C}$ is identified with the inclusion $\tcD^\circ_B\hra M^\circ$. 

We now construct $A$ from $C$. Similar to above, we first define a $W(k)$-lattice $L^\circ = \bigoplus_{\tilde \tau\in \Sigma_{E, \infty}} L^\circ_{\tilde \tau} \subseteq \tilde \calD_C^\circ$, with $L^\circ_{\tilde \tau} = \tilde \calD_{C, \tilde \tau}^\circ$ unless $\tilde \tau \in \tilde \Delta(\ttT)^+$.
If $\tilde \tau \in \tilde \Delta(\ttT)^+$, then the corresponding $p$-adic place $\gothp \in \Sigma_p$ cannot be of type $\beta2$ or $\beta^\sharp$. 
In this case, we identify $\tcD_{B}^\circ[1/p]$ with $\tcD_C^\circ[1/p]$ so that $\tcD_B^\circ$ and $\tcD^\circ_C$ are both viewed as  $W(k)$-lattices in $\tcD_B^\circ[1/p]$. The polarization $\lambda_B$ induces a perfect pairing 
\[
\langle\ ,\ \rangle_{\lambda_B}\colon \tcD^\circ_{B,\tilde\tau}[1/p]\times \tcD^\circ_{B,\tilde\tau^c}[1/p]\ra W(k)[1/p],
\]
which induces a perfect pairing between $\tcD^\circ_{B,\tilde\tau}$ and $\tcD^\circ_{B,\tilde\tau^c}$.
We put (for $\tilde \in \tilde \Delta(\ttT)^+$)
 $$
L^\circ_{\tilde\tau}=\tilde \calD^{\circ,\vee}_{C,\tilde\tau^c}: =\{v\in \tcD^\circ_{C,\tilde\tau}[1/p]: \langle v, w\rangle_{\lambda_B}\in W(k),\text{ for all } w\in \tilde \calD^\circ_{C,\tilde\tau^c} \}.
$$
Note that $\tilde \tau \in \tilde \Delta(\ttT)^+$ always implies that $\tilde \tau^c \in \tilde \Delta(\ttT)^-- \tilde \Delta(\ttT)^+$. So $\tilde \calD^\circ_{C, \tilde \tau} = \tilde \calD^\circ_{B, \tilde \tau}$ and $\tilde \calD^\circ_{C, \tilde \tau^c} \supset \tilde \calD^\circ_{B, \tilde \tau^c}$ with quotient isomorphic to $k$.
This implies that $L^\circ_{\tilde \tau} \subseteq \tilde \calD_{C, \tilde \tau}^\circ$ with quotient isomorphic to $k$.

 As usual, one verifies that $L^\circ$ is stable under $F_B$ and $V_B$ (because it is equal to either $\tilde \calD_{C, \tilde \tau}^\circ$ or $\tilde \calD_{C, \tilde \tau^c}^{\circ, \vee}$ in various cases), and we put $L=(L^\circ)^{\oplus 2}$. The quotient Dieudonn\'e module $L/p\tcD_{C}$ corresponds to a closed subgroup scheme $K\subseteq C[p]$ stable under the action of $\cO_D$. We put $A=C/K$ equipped with the induced $\calO_D$-action, and define $\phi_A\colon A\ra C$ as the canonical $\calO_D$-equivariant isogeny with kernel $C[p]/K$. Then $\phi_{A,*}:\tcD_{A}\ra \tcD_C$ is identified with the natural inclusion $L\hra \tcD_C$.

We define $\lambda_A:A\ra A^\vee$ to be the quasi-isogeny:
\[
\lambda_A: A\xra{\phi_A} C \xleftarrow{\phi_B} B\xra{\lambda_B}B^\vee\xleftarrow{\phi_{B}^\vee} C^\vee
\xra{\phi_A^\vee}A^\vee,
\]
and we will verify that $\lambda_A$ is  a genuine isogeny (hence a polarization since $\lambda_B$ is) satisfying condition (b) of the moduli space $\bfSh_{K'_p}(G'_{\tilde \ttS})$ as in Theorem~\ref{T:unitary-shimura-variety-representability}. 
We may identify $\tcD_{A}^\circ[1/p]$ and $\tcD^\circ_{A^\vee}[1/p]$ with $\tcD_{B}^\circ[1/p]$, and view both $\tcD^\circ_{A,\tilde\tau}$ and $\tcD^\circ_{A^\vee,\tilde\tau}$ as lattices of $\tcD^\circ_{B,\tilde\tau}[1/p]$.
It suffices to show that we have a natural inclusion
$$
\tcD^\circ_{A^\vee, \tilde\tau} \subseteq (\tcD^\circ_{A,\tilde\tau^c})^\vee= \big\{v\in \tcD^\circ_{B,\tilde\tau}[1/p]: \langle v, w\rangle_{\lambda_B}\in W(k)\text{ for all }w\in \tcD^\circ_{A,\tilde\tau^c}\big\},
$$
which is an isomorphism unless $\tilde \tau$ induces a $p$-adic place of type $\beta^\sharp$ for $\Sh_{K}(G_{\ttS})$ in which case it is an inclusion with quotient $k$.
\begin{itemize}
\item
By the construction of $A$, this is clear for $\tilde \tau \in \tilde \Delta(\ttT)^+$ and hence for all their complex conjugates (as the duality is reciprocal).
\item
For all places $\tilde \tau \in \Sigma_{E,\infty/ \gothp}$ such that  $\gothp$ is  not of type $\beta2$ and $\tilde \tau \notin \tilde \Delta(\ttT)^\pm$, we know that $\tilde \tau^c \notin \tilde \Delta(\ttT)^\pm$.  So $\tilde \calD^\circ_{A, ? } = \tilde \calD^\circ_{B,?}$ for $? = \tilde \tau, \tilde \tau^c$ under the identification.  The  statement is clear. Note that this includes the case that $\gothp$ is a prime of type $\beta^\sharp$ for $\Sh_{K}(G_{\ttS})$.
\item
The only case left is when $\tilde \tau \in \Sigma_{E, \infty/\gothp}$ for $\gothp$ of type $\beta2$. In this case, $\tcD^\circ_{A,\tilde\tau}=\tcD^\circ_{C,\tilde\tau}$ which is the dual of $\tcD^\circ_{C,\tilde\tau^c}$ for all $\tilde\tau\in \Sigma_{E,\infty/\gothp}$ by construction.
\end{itemize}

This concludes the verification of that $\lambda_A$ is isogeny satisfying condition (b) of Theorem~\ref{T:unitary-shimura-variety-representability} for the moduli space $\bfSh_{K'_p}(G'_{\tilde \ttS})$.

We now  define the level structure $\alpha_{K'}=\alpha^p\alpha_p$ on $A$. For the prime-to-$p$ level structure $\alpha^p$, we define it to be the $K_{\ttT}'$-orbit of the isomorphism class:
\[
\alpha^p\colon\Lambda^{(p)}\xrightarrow[\beta^p]{\sim} T^{(p)}(B)\xrightarrow[\phi_{B, *}]{\sim}  T^{(p)}(C) \xleftarrow[\phi_{A,*}]{\sim} T^{(p)}(A).
\]
We take the closed subgroup scheme $\alpha_\gothp\subseteq A[\gothp]$ for each $\gothp\in \Sigma_p$ of type $\alpha^\sharp$ for $\ttS$ (and hence for $\ttS(\ttT)$) to be the subgroup scheme corresponding to $\beta_\gothp$ under the sequence of isomorphisms (note that $\tilde \Delta(\ttT)^\pm_{/\gothp} = \emptyset$ for $\gothp$ of type $\alpha^\sharp$)
of $p$-divisible groups
$$
A[\gothp^\infty] \xrightarrow[\ \cong\ ]{\phi_{A}} C[\gothp^\infty]
\xleftarrow[\ \cong\ ]{\phi_B} B[\gothp^\infty].$$
 It is clear that $\alpha_{K'}$ verifies condition (c) in Theorem~\ref{T:unitary-shimura-variety-representability}.

This finishes the construction of all the data $y=(A,\iota_A,\lambda_A,\alpha_{K'}, B, \iota_B,\lambda_B, \beta_{K'_\ttT}, C, \iota_C;\phi_A,\phi_B)$. To see that $y$ is indeed a $k$-point of $Y'_{\ttT}$, we have to check that $y$ satisfies the conditions (i)-(ix) for $Y'_{\ttT}$ in Subsection~\ref{S:moduli-Y_S}. Conditions (ii), (iii), and (vi)-(ix) being clear from our construction, it remains to check (i), (iv) and (v).
Moreover, the  Kottwitz signature condition Theorem~\ref{T:unitary-shimura-variety-representability}(a) follows from Lemma~\ref{L:dimension count} immediately.  So for property (i), it remains to show that the partial Hasse invariant $h_{\tilde\tau}(A)$ vanishes if $\tau = \tilde\tau|_F\in \ttT$.

We now check these properties (i), (iv), and (v) in various cases.
For this, we identify $\tilde \calD_A^\circ[\frac1p]$, $\tilde \calD_B^\circ[\frac 1p]$, and $\tilde \calD_C^\circ[\frac 1p]$ via $\phi_{A, *}$ and $\phi_{B, *}$.

(1) Assume $\gothp$ is a prime of case $\alpha1$ or $\beta1$.
We keep the notation as before.
If $\tilde \tau$ does not lift a place belonging to some chain $C_i$ inside $\ttS_{\infty/\gothp}\cup \ttT_{/\gothp}$, the conditions (i), (iv), and (v) trivially hold.
So we assume $\tilde \tau|_F \in C_i$ for some $C_i$.

If $r_i=\#(\ttT_{/\gothp}\cap C_i)$ is odd, unwinding our earlier construction gives, for $0 \leq \ell \leq m_i+1$, 
\begin{align*}
&\tilde \calD_{A, \sigma^{-\ell}\tilde \tau_i}^\circ = \tilde \calD_{C, \sigma^{-\ell}\tilde \tau_i}^\circ = \left\{
\begin{array}{ll}
p^{-1} F_{B,\es}^{m_i+1-\ell}(\tilde J^\circ_{\sigma^{-m_i-1}\tilde \tau_i}) & \textrm{ if }a_j \leq \ell< a_{j+1} \textrm{ for some odd }j,
\\
\tilde \calD_{B, \sigma^{-\ell}\tilde \tau_i}^\circ & \textrm{ otherwise;}
\end{array}
\right.
\\
&\tilde \calD_{A, \sigma^{-\ell}\tilde \tau_i^c}^\circ = \left\{
\begin{array}{ll}
p F_{B,\es}^{m_i+1-\ell}(\tilde J^{\circ,\perp}_{\sigma^{-m_i-1}\tilde \tau_i}) & \textrm{ if }a_j \leq \ell< a_{j+1} \textrm{ for some odd }j,
\\
\tilde \calD_{B, \sigma^{-\ell}\tilde \tau_i^c}^\circ & \textrm{ otherwise; and}
\end{array}
\right.\\
& \tilde \calD_{C, \sigma^{-\ell}\tilde \tau_i^c}^\circ = \tilde \calD_{B, \sigma^{-\ell}\tilde \tau_i^c}^\circ \textrm{ for all } \ell.
\end{align*}
For condition (v) in Subsection~\ref{S:moduli-Y_S}, it is  trivial unless $\tilde \tau = \sigma^{-\ell} \tilde \tau_i$ for some $\ell \in [a_j, a_{j+1})$ with $j$ odd. In the exceptional case, it is equivalent to proving that (for the $n$ as in condition (v))
\[
\tilde \calD^\circ_{B, \sigma^{-\ell} \tilde \tau_i}  =  F_{A, \es}^n (\tilde \calD_{A, \sigma^{-\ell - n} \tilde \tau_i}^\circ).
\]
Note that  $n = a_{j+1}-\ell$,  it follows that $\tcD_{A,\sigma^{-\ell-n}\tilde\tau_i}^{\circ}=\tcD_{B,\sigma^{-a_{j+1}}\tilde\tau_i}^{\circ}$ by definition. As $s_{\ttT, \sigma^{-a_{j+1}}\tilde\tau_i}=2$ and $s_{\sigma^{-a_{j+1}}\tilde\tau_i}=1$,  
    $F_{A,\es}^n(\tilde \calD_{A, \sigma^{-\ell - n} \tilde \tau_i}^\circ)$ coincides with $F_{B,\es}^n(\tilde \calD_{B, \sigma^{-a_{j+1}}\tilde\tau_i}^\circ )$ by the definition of essential Frobenius.
The desired equality follows from the fact that $F_{B,\es}^n(\tilde \calD_{B, \sigma^{-a_{j+1}}\tilde\tau_i}^\circ )=\tcD^\circ_{B,\sigma^{-\ell}\tilde\tau_i}$.

Similarly, condition (iv) is trivial unless $\tilde \tau = \sigma^{-\ell} \tilde \tau_i^c$ for some $\ell \in [a_j, a_{j+1})$ with $j$ odd. In the exceptional case, it is equivalent to the following equality (for the $n$ as in condition (iv))
\[
p\tilde \calD^\circ_{B, \sigma^{-\ell} \tilde \tau_i^c}  = \tilde F_{A, \es}^n (\tilde \calD_{A, \sigma^{-\ell - n} \tilde \tau_i^c}^\circ).
\]
But $n = a_{j+1}-\ell$ by definition; so $\tilde \calD_{A, \sigma^{-\ell - n} \tilde \tau_i^c}^\circ = \tilde \calD^\circ_{B, \sigma^{-a_{j+1}}\tilde \tau_i^c}$. 
Since $s_{\ttT, \sigma^{-a_{j+1}}\tilde \tau^c_i} = 0$ and $s_{\sigma^{-a_{j+1}}\tilde \tau^c_i} = 1$, the essential Frobenius of $A$ at $\sigma^{-a_{j+1}}\tilde \tau^c_i$ is defined to be $F_A$ while that of $B$ at $\sigma^{-a_{j+1}}\tilde \tau^c_i$ is defined to be $V_{B}^{-1}$. Therefore,  $F_{A, \es, \sigma^{-\ell}\tilde \tau_i^c}^n$ is the same as $pF_{B, \es, \sigma^{-\ell}\tilde \tau_i^c}^n$. The equality above is now clear.

We now check the vanishing of partial Hasse invariants $h_{\sigma^{a_j}\tilde\tau_i}(A)$  with $1\leq j\leq r_i-1$. 
By Lemma~\ref{Lemma:partial-Hasse}, it suffices to show that,  for any $j = 1, \dots, r_i-1$ and setting $a_0 = -1$, the image of 
\[
 F_{A, \es}^{a_{j+1}- a_{j-1}}: \tcD^\circ_{A, \sigma^{-a_{j+1}} \tilde \tau_i} \ra  \tcD^\circ_{A, \sigma^{-a_{j-1}}\tilde \tau_i} 
\]
is contained in $p\tcD^\circ_{A, \sigma^{-a_{j-1}}\tilde \tau_i}$. First,  regardless of the parity of $j$, we find easily that
$ F_{A, \es}^{a_{j+1}- a_{j-1}}=pF_{B, \es}^{a_{j+1}- a_{j-1}}$ as maps from $\tcD^\circ_{A, \sigma^{-a_{j+1}} \tilde \tau_i}$ to $\tcD^\circ_{A, \sigma^{-a_{j-1}} \tilde \tau_i}$
by checking carefully the dependence of the essential Frobenii on the signatures.
Now if $j$ is odd, then 
 $$
 \tcD^{\circ}_{A,\sigma^{-\ell}\tilde\tau_i}=\tcD^{\circ}_{B,\sigma^{-\ell}\tilde\tau_i}\quad \text{for  }\ell=a_{j+1} \textrm{ and }a_{j-1}.
 $$
Hence one gets $F_{A, \es}^{a_{j+1}- a_{j-1}}(\tcD^{\circ}_{A,\sigma^{-a_{j+1}}\tilde\tau_i})=p\tcD^\circ_{A, \sigma^{-a_{j-1}} \tilde \tau_i}$, since $F_{B, \es}^{a_{j+1}- a_{j-1}}(\tcD^{\circ}_{B,\sigma^{-a_{j+1}}\tilde\tau_i})=\tcD^\circ_{B, \sigma^{-a_{j-1}} \tilde \tau_i}$.
 If $j$ is even, then 
 \[
 \tcD^{\circ}_{A,\sigma^{-\ell}\tilde\tau_i}=p^{-1} F_{B,\es}^{m_i+1-\ell}(\tilde J^\circ_{\sigma^{-m_i-1}\tilde \tau_i}),\quad \text{for } \ell=a_{j+1}\textrm{ and } a_{j-1}.
 \]
It is also clear that $F_{A, \es}^{a_{j+1}- a_{j-1}}(\tcD^{\circ}_{A,\sigma^{-a_{j+1}}\tilde\tau_i})=p\tcD^\circ_{A, \sigma^{-a_{j-1}} \tilde \tau_i}$.

If $r_i=\#(\ttT_{/\gothp}\cap C_i)$ is even, all conditions can be proved in exactly the same way, except replacing $ F_{B, \es}^{m_i+1-\ell}(\tilde J^\circ_{\sigma^{-m_i-1}\tilde \tau_i})$ by 
$ F_{B,\es}^{m_i-\ell} (F_B(\tcD^{\circ}_{B,\sigma^{-m_i-1}\tilde\tau_i}))$ and the proof of  the vanishing of the last Hasse invariant $h_{\sigma^{-a_{r_i}}\tilde \tau_i}(A)$ needs a small modification.
In fact,  we have 
\[
\tcD^{\circ}_{A,\sigma^{-\ell}\tilde\tau_i}=\begin{cases}
\tcD^{\circ}_{B,\sigma^{-\ell}\tilde\tau_i} & \text{for } a_{r_i}\leq \ell\leq m_i+1,\\
 p^{-1} F_{B,\es}^{m_i-\ell} (F_B(\tcD^{\circ}_{B,\sigma^{-m_i-1}\tilde\tau_i})) &\text{for } a_{r_{i}-1}\leq \ell < a_{r_i}.
 \end{cases}
 \]
Note that the number $n_{\sigma^{-a_{r_{i}}}\tau_i}$ defined in \ref{S:partial-Hasse} is equal to $m_i+1-a_{r_i}$, and the essential Frobenius $F_{A,\es}: \tcD^{\circ}_{A,\sigma^{-a_{r_i}}\tilde \tau_i}\ra \tcD^{\circ}_{A,\sigma^{-a_{r_i}+1}\tilde \tau_i}$ is simply $F_A$.  
We have
\[
 F_{A}
 F_{A, \es}^{m_i+1-a_{r_i}}(\tilde \calD^\circ_{A, \sigma^{-m_i-1}\tilde \tau_i} ) = F_{A, \es}^{m_i+2-a_{r_i}} (\tilde \calD^\circ_{A, \sigma^{-m_i-1}\tilde \tau_i} ) = p\tilde \calD^\circ_{A, \sigma^{-a_{r_i}+1}\tilde\tau_i}.
\]
This verifies the vanishing of $h_{\sigma^{-a_{r_i}}\tilde \tau_i}(A)$.

(2) Assume that $\gothp$ is a prime of Case $\alpha2$.
 We write $H_\gothp = H_{\gothq}\oplus H_{\bar\gothq}$ and $\tcD^\circ_{H_{\gothq},\tilde\tau}\subseteq \tcD^\circ_{B,\tilde\tau}$ as before.  We have
\[
\tcD^\circ_{A,\tilde\tau}=
\begin{cases}p^{-1}\tcD^\circ_{H_{\gothp},\tilde\tau} &\textrm{if }\tilde\tau \in \tilde \Delta(\ttT)_{/\gothp}^-,
\\
p(\tcD^\circ_{H_{\gothp},\tilde\tau^c})^\vee &\textrm{if }\tilde\tau \in \tilde \Delta(\ttT)_{/\gothp}^+,
\\
\tcD^\circ_{B,\tilde\tau} &\text{otherwise},
\end{cases}
\quad \textrm{and} \quad
\tcD^\circ_{C,\tilde\tau}=
\begin{cases}p^{-1}\tcD^\circ_{H_{\gothp},\tilde\tau} &\textrm{if }\tilde\tau \in \tilde \Delta(\ttT)_{/\gothp}^-,
\\
\tcD^\circ_{B,\tilde\tau} &\text{otherwise}.
\end{cases}
\]
The same arguments as in (1) allows us to check conditions (i), (iv) and (v).

(3) Assume now that $\gothp$ is prime of Case $\beta2$ in Subsection~\ref{S:quaternion-data-T}.  For each $\tilde\tau\in \Sigma_{E,\infty/\gothp}$, let  $\lambda_{B,*,\tilde\tau}: \tcD^\circ_{B,\tilde\tau}\ra \tcD^\circ_{B^\vee,\tilde\tau}$ be the map induced by the polarization $\lambda_B$.
By condition (b3) of Theorem~\ref{T:unitary-shimura-variety-representability}, its cokernel has  dimension $1$ over $k$.
When viewing $\tcD^\circ_{B^\vee,\tilde\tau}$ as a lattice of $\tcD^\circ_{B,\tilde\tau}[1/p]$ via $\lambda^{-1}_{B,*,\tilde\tau}$, we have
\[
\tcD^\circ_{A,\tilde\tau} = \tcD^\circ_{C,\tilde\tau}=
\begin{cases}
\tcD^\circ_{B^\vee,\tilde\tau} &\text{if }\tilde\tau\in \tilde \Delta(\ttT)_{/\gothp}^-,\\
\tcD^\circ_{B,\tilde\tau} &\text{otherwise}.
\end{cases}
\]
The same argument as in (1) allows us to check conditions (i), (iv) and (v).  This then concludes the proof of Step I.

\vspace{10pt}

\textbf{Step II:} Let $y=(A, \iota_A,\lambda_A,  \alpha_{K'}, B,\iota_B, \lambda_B,  \beta_{K'_\ttT}, C, \iota_C; \phi_A, \phi_B)\in Y'_{\ttT}$ be a closed point with values in $k=\Fpb$, and $z=\eta_2(y)=(B, \iota_B, \lambda_B, \beta_{K'_\ttT}, J^\circ)\in Z'_{\ttT}$. We prove that $\eta_2: Y_{T}'\ra Z'_{\ttT} $ induces a bijection of tangent spaces $\eta_{2,y}: T_{Y'_{\ttT}, y}\xra{\cong} T_{Z'_{\ttT},z}$. We follow the same strategy as in Lemma~\ref{L:Y_T-X_T-tangent}.

Set $\II=\Spec(k[\epsilon]/\epsilon^2)$. The tangent space $T_{Z'_{\ttT},z}$ is identified with the set of deformations $z_{\II}=(B_{\II}, \iota_{B,\II}, \lambda_{B,\II}, \beta_{K'_\ttT,\II}, J^\circ_{\II})\in Z'_{\ttT}(\II)$ of $z$, where $J^\circ_{\II}$ is the collection of sub-bundles $J^\circ_{\II,\tilde\tau}\subseteq H^\dR_{1}(B_{\II}/\II)^\circ_{\tilde\tau}=H^{\cris}_{1}(B/k)_{\II,\tilde\tau}^{\circ}$ for each $\tilde\tau\in \tilde I_\ttT$. We have to show that every point $z_\II$ lifts uniquely to a deformation $y_{\II}=(A_{\II}, \iota_{A_\II}, \lambda_{A_{\II}},  \alpha_{K',\II}, B_{\II}, \iota_{B_{\II}},\lambda_{B_{\II}},  \beta_{K'_\ttT,\II}, C_{\II}, \iota_{C_{\II}}; \phi_{A_{\II}}, \phi_{B_{\II}})\in Y'_{\ttT}(\II)$ with $\eta(y_{\II})=z_{\II}$.

We start with $C_{\II}$ and $\phi_{B_{\II}}$. For $\tilde\tau\in \Sigma_{E,\infty}$, denote by $$
\phi_{B,*,\tilde\tau}^{\cris}: H^\cris_1(B/k)^{\circ}_{\II,\tilde\tau}\ra H^\cris_1(C/k)^\circ_{\II,\tilde\tau}
$$
 the natural morphism induced by $\phi_{B}$, and by $\phi^{\dR}_{B,*,\tilde\tau}$  the analogous map between the de Rham homology $H^{\dR}_1$. The crystalline nature of $H^{\cris}_1$ implies that $\phi^{\cris}_{B,*,\tilde\tau}=\phi^{\dR}_{B,*,\tilde\tau}\otimes_k k[\epsilon]/\epsilon^2$.
  To construct $C_{\II}$ and $\phi_{B,\II}$ it suffices to specify, for each $\tilde\tau\in \Sigma_{E,\infty}$, a sub-bundle $\omega^\circ_{C^\vee,\II,\tilde\tau}\subseteq H^{\cris}_1(C/k)^\circ_{\II,\tilde\tau}$ which lifts $\omega^\circ_{C^\vee,\tilde\tau}$ and satisfies
  \begin{equation}\label{E:omega-inclusion-C}
  \phi^{\cris}_{B,*,\tilde\tau}(\omega^{\circ}_{B^\vee_{\II},\tilde\tau})\subseteq \omega^\circ_{C^\vee,\II,\tilde\tau}.
  \end{equation}
 We distinguish a few cases:
  \begin{enumerate}
  \item
  If neither $\tilde\tau$ nor $\sigma\tilde\tau$ belong to $\tilde\Delta(\ttT)^-$, both $\phi^{\cris}_{B,*,\tilde\tau}$ and $\phi^{\cris}_{B,*,\sigma\tilde\tau}$ are isomorphisms.
  It follows that $\phi_{B,*,\tilde\tau}^{\dR}(\omega^\circ_{B^\vee,\tilde\tau})=\omega^\circ_{C^\vee, \tilde\tau}$. Hence we  have to take $\omega^{\circ}_{C^\vee,\II, \tilde\tau}=\phi^{\cris}_{B,*, \tilde\tau}(\omega^\circ_{B^\vee_{\II},\tilde\tau})$.

\item If both $\tilde \tau, \sigma \tilde \tau \in \tilde \Delta(\ttT)^-$, then $\tau=\tilde\tau|_{F}\in \ttS_{\infty/\gothp}$ by Lemma~\ref{L:property of Delta}.  A simple dimension count similar to Lemma~\ref{L:dimension count} implies that
  $
  \dim_{k}(\omega^\circ_{C^\vee, \tilde\tau})=\dim_{k}(\omega^\circ_{B^\vee,\tilde\tau})\in \{0,2\}.
  $
  We  take $\omega^{\circ}_{C^\vee,\II, \tilde\tau}$ to be $0$ or $H^{\cris}_1(C/k)^\circ_{\II, \tilde\tau}$ correspondingly, and \eqref{E:omega-inclusion-C} trivially holds.

\item
If $\tilde \tau \in \tilde I_\ttT$, property of the morphism $\eta_2$ forces $\omega^{\circ}_{C^\vee, \II,\tilde\tau}=\phi^{\cris}_{B,*,\tilde\tau}(J^\circ_{B_{\II},\tilde\tau})$.
\item
For all other $\tilde \tau$, $\tilde \tau|_F$ must belong to $\ttT$.  Let $n$ be the number associated to $\tilde\tau $  as in Lemma~\ref{L:distance to T'}.  
By the vanishing of the partial Hasse invariant on $A$ at $\tilde \tau$, we have $\omega^\circ_{C^\vee, \tilde \tau} = F_{C,\es}^n(H^{\dR}_1(C/k)^\circ_{\sigma^{-n}\tilde \tau}).$
We take
\[
\omega^\circ_{C^\vee, \II, \tilde \tau} = F_{C,\es}^n(H_1^\cris(C^{(p^n)}/k)^\circ_{\II, \tilde \tau}).
\]
This is not a forced choice now, but it will become one when we have constructed the lift $A_\II$ and require $A_\II$ to have vanishing partial Hasse invariant.

Since $F_{B,\es}^n: H_1^\cris(B^{(p^n)}/k)^\circ_{\II, \tilde \tau} \to H_1^\cris(B/k)^\circ_{\II, \tilde \tau}$ is an isomorphism, we conclude that \eqref{E:omega-inclusion-C} holds for $\tilde \tau$.

  \end{enumerate}

     We now construct $A_\II$ and the isogeny $\phi_{A_{\II}}:A_{\II}\ra C_{\II}$.  As usual, we have to specify, for each $\tilde\tau\in \Sigma_{E,\infty}$,
     a sub-bundle $\omega^\circ_{A^\vee,\II,\tilde\tau}\subseteq H^{\cris}_1(A/k)^\circ_{\II,\tilde\tau}$ that lifts $\omega^{\circ}_{A^\vee,\tilde\tau}$ and satisfies $\phi_{A,*,\tilde\tau}^\cris(\omega^\circ_{A^\vee,\II,\tilde\tau})\subseteq \omega^\circ_{C^\vee,\II,\tilde\tau}$.
 Let $\gothp\in \Sigma_p$ be the prime such that $\tilde\tau\in \Sigma_{E,\infty/\gothp}$.

\begin{itemize}
\item If neither $\tilde \tau$ nor $\sigma \tilde \tau$ belong to $\tilde \Delta(\ttT)^+$, then $\phi_{A, *, \tilde \tau}^\dR$ and hence $\phi^\cris_{A, *, \tilde \tau}$ is an isomorphism.
We are forced to take  $\omega^{\circ}_{A^\vee,\II, \tilde\tau}=(\phi^{\cris}_{A,*,\tilde\tau})^{-1}(\omega^\circ_{C^\vee, \II, \tilde\tau})$. In particular, if $\tilde\tau\in \tilde I_{\ttT}$,  we have $\omega^\circ_{A^\vee,\II, \tilde\tau}=(\phi^{\cris}_{A,*,\tilde\tau})^{-1}\phi^{\cris}_{B,*,\tilde\tau}(J^{\circ}_{B_{\II},\tilde\tau})$.

\item In all other cases, we must have $\tilde \tau \in \Sigma_{E, \infty/\gothp}$ for $\gothp$ not of type $\beta2$ or $\beta^\sharp$.
Then we have to take $\omega^\circ_{A^\vee,\II,\tilde\tau}$ to be the orthogonal complement of $\omega^\circ_{A^\vee,\II, \tilde\tau^c}$ (which is already defined in the previous case) under the perfect pairing
     \[
     \langle\ ,\ \rangle_{\lambda_A}\colon H^{\cris}_1(A/k)^\circ_{\II, \tilde\tau}\times H^{\cris}_{1}(A/k)^\circ_{\II,\tilde\tau^c}\ra k[\epsilon]/\epsilon^2
     \]
induced by the polarization $\lambda_A$. It is clear that $\omega^\circ_{A^\vee, \II, \tilde\tau}$ is a lift of $\omega^\circ_{A^\vee, \tilde\tau}$.
It remains to show that $\phi^\cris_{A, *, \tilde \tau}(\omega^\circ_{A^\vee, \II, \tilde \tau}) \subseteq \omega^\circ_{C^\vee, \II, \tilde \tau}$.
We consider the following commutative diagram
\begin{equation}
\label{E:dualization ACB}
\xymatrix@C=5pt{
H_1^\cris(A/k)^\circ_{\II,\tilde{\tau}}
\ar[d]^{\phi^{\cris}_{A,*,\tilde\tau}}
 & \times & H_1^\cris(A/k)^\circ_{\II,\tilde{\tau}^c}\ar[d]^{\phi^{\cris}_{A,*,\tilde\tau^c}}_\cong
 \ar[rrrr]^-{\langle\ , \ \rangle_{\lambda_A}} &&&& \II
 \\
 H_1^\cris(C/k)^\circ_{\II,\tilde{\tau}} & \times & H_1^\cris(C/k)^\circ_{\II,\tilde{\tau}^c}
 \\
H_1^\cris(B/k)^\circ_{\II,\tilde{\tau}} \ar[u]_{\phi^{\cris}_{B,*,\tilde\tau}}^\cong
& \times & 
\ar[u]_{\phi^{\cris}_{B,*,\tilde\tau^c}}
H_1^\cris(B/k)^\circ_{\II,\tilde{\tau}^c}\ar[rrrr]^-{\langle\ , \ \rangle_{\lambda_B}} &&&& \II,\ar@{=}[uu]
}
\end{equation}
where both duality pairings are perfect.
By our choice of $\tilde \tau$, we have $\tilde \tau, \sigma \tilde \tau \notin \tilde \Delta(\ttT)^-$ and $\tilde \tau^c, \sigma \tilde \tau^c \notin \tilde \Delta(\ttT)^+$; so both $\phi^\cris_{A, *, \tilde \tau^c}$ and $\phi^\cris_{B, *, \tilde \tau}$ in \eqref{E:dualization ACB}
are isomorphisms and they induce isomorphisms on the reduced  differentials.
Using the diagram \eqref{E:dualization ACB} of perfect duality, the inclusion $\phi^\cris_{A, *, \tilde \tau}(\omega^\circ_{A^\vee, \II, \tilde \tau}) \subseteq \omega^\circ_{C^\vee, \II, \tilde \tau}$ is equivalent to the inclusion $\phi^\cris_{B, *, \tilde \tau^c}(\omega^\circ_{B,\II, \tilde \tau^c}) \subseteq 
\omega^\circ_{C,\II, \tilde \tau^c}$, which was already checked.
 \end{itemize}

By construction, the $\tilde\tau$-partial Hasse invariant of  $A_\II$ vanishes if $\tilde\tau\in \tilde\Delta(\ttT)^-$ and $\tilde\tau|_F\in \ttT$. The duality guarantees the vanishing of Hasse invariants at their conjugate places.  This condition conversely forces the uniqueness of our choice of $C_\II$ and $A_\II$, as indicated earlier.
From the construction, $\omega^\circ_{A^\vee,\II}=\bigoplus_{\tilde\tau\in \Sigma_{E,\infty}}\omega^\circ_{A^\vee,\II,\tilde\tau}$ is isotropic  under the paring on $H^\cris_{1}(A/k)^\circ_{\II}$ induced by $\lambda_A$.
This concludes checking condition (1) of Subsection~\ref{S:moduli-Y_S}.

The lift of the level structure $\alpha_{K', \II}$ is automatic for the tame part, and can be done in a unique way as in the proof of Theorem~\ref{T:unitary-shimura-variety-representability}.

It then remains to check that  $\phi_{A_\II}$ and $\phi_{B_\II}$ satisfy conditions (iv) and (v) of Subsection \ref{S:moduli-Y_S}. For condition (v),
it is obvious except when $\tilde \tau \in \tilde \Delta(\ttT)^-$, in which case,
\[
\mathrm{Im}(\phi^\cris_{B, *, \tilde \tau}) = \mathrm{Im}(\phi^\dR_{B, *, \tilde \tau}) \otimes_k \II, \quad \textrm{and} \quad
\phi_{A, *, \tilde \tau}^\cris(\mathrm{Im}(F_{\es,A_\II, \tilde \tau}^n)) = \phi_{A, *, \tilde \tau}^\dR(\mathrm{Im}(F_{\es,A, \tilde \tau}^n)) \otimes_k \II,
\]
where $n>1$ is the number determined in Lemma~\ref{L:distance to T'}.
So condition (v) for the lift follows from that for $\phi_A: A \to C$.  (Note that $n\geq 1$ implies that the image of the essential image is determined by the reduction.)
Exactly the same argument proves condition (iv).

This concludes Step II of the proof of Proposition~\ref{P:Y_T=Z_T}.
\end{proof}

\subsection{End of proof of Theorem~\ref{T:main-thm-unitary}}
\label{S:End-of-proof}  
Statement (1) of Theorem~\ref{T:main-thm-unitary} follows from Propositions~\ref{P:Y_S=X_S} and \ref{P:Y_T=Z_T}. Statements (2) and (3) are clear from the proof of (1). 
It remains to prove statement (4), namely the compatibility of the twisted partial Frobenius. We use $X'_{\tilde \ttS,\ttT}$, $Y_{\tilde \ttS,\ttT}'$ and $Z'_{\tilde \ttS,\ttT}$ to denote the original $X'_{\ttT}$, $Y'_{\ttT}$ and $Z'_{\ttT}$  in \eqref{S:moduli-Y_S} to indicate their dependence on $\tilde \ttS$. We will define a twisted partial Frobenius 
\[
\gothF'_{\gothp^2,\ttS}:Y'_{\tilde \ttS,\ttT}\ra Y'_{\sigma^2_{\gothp}\tilde \ttS,\sigma^2_{\gothp}\ttT}
\] 
compatible via $\eta_1$ with the $\gothF'_{\gothp^2,\tilde \ttS}$ on $\bfSh_{K'}(G'_{\tilde \ttS})_{k_0}$ defined in Subsection~\ref{S:partial Frobenius}. 
For an $S$-valued point  $x=(A, \iota_A,\lambda_A,  \alpha_{K'}, B, \iota_B, \lambda_B, \beta_p, C, \iota_C; \phi_A, \phi_B)$ of $Y'_{\tilde \ttS,\ttT}$, its image 
\[
\gothF'_{\gothp^2,\tilde \ttS}(x)=(A', \iota_{A'},\lambda_{A'},  \alpha'_{K'}, B', \iota_{B'}, \lambda_{B'}, \beta'_p, C', \iota_{C'}; \phi_{A'}, \phi_{B'}).
\] 
is given as follows.
Here,  for $G=A,B,C$, we put $G'=(G/\Ker_{G,\gothp^2})\otimes_{\cO_F}\gothp$ where $\Ker_{G',\gothp^2}$ is the $\gothp$-component of the $p^2$-Frobenius of $G$.
 The induced structures $(\iota_{A'},\lambda_{A'}, \alpha'_{K'},\iota_{B'},\lambda_{B'},\beta'_{p})$ are defined in the same way as in \ref{S:partial Frobenius}.
  The isogenies $\phi_{A'}: A'\ra C'$ and $\phi_{B'}:B'\ra C'$ are constructed from $\phi_{A}$ and $\phi_B$ by the functoriality of $p^2$-Frobenius. 
We have to  prove that the induced map on de Rham homologies $\phi_{A',*,\tilde\tau}$ and $\phi_{B',*,\tilde\tau}$ satisfy the required conditions in (v) and (vi) of \ref{S:moduli-Y_S}.  If $\tilde\tau\in \Sigma_{E,\infty/\gothp'}$ with $\gothp'\neq \gothp$, this is clear, because the $p$-divisible groups $G'[\gothp'^{\infty}]$ are canonically identified with $G[\gothp'^{\infty}]$ for $G=A,B,C$. Now consider the case $\gothp'=\gothp$. As in  the proof of Lemma~\ref{L:partial Frobenius vs partial Hasse inv},  for $G=A,B,C$, the  $p$-divisible group $G'[\gothp^{\infty}]$ is isomorphic to the base change of $G[\gothp^{\infty}]$ via the $p^2$-Frobenius on $S$. One deduces thus isomorphisms of de Rham homology groups
\begin{equation}\label{E:isom-dR-ABC}
H^{\dR}_1(G'/S)^{\circ}_{\tilde\tau}=(H^{\dR}_1(G/S)^{\circ}_{\sigma^{-2}\tilde\tau})^{(p^2)},
\end{equation}
which is compatible with $F$ and $V$ as $\tilde\tau\in \Sigma_{E,\infty/\gothp}$ varies, and compatible with $\phi_{A',*,\tilde\tau}$ and $\phi_{B',*,\tilde\tau}$ by functoriality. Hence, the required properties on $\phi_{A',*,\tilde\tau}$ and $\phi_{B',*,\tilde\tau}$ follow from those on $\phi_{A,*,\sigma^{-2}\tilde\tau}$ and $ \phi_{B,*,\sigma^{-2}\tilde\tau}$.
 This finishes the construction of $\gothF'_{\gothp^2}$ on $Y'_{\tilde \ttS,\ttT}$.
Via the isomorphism $\eta_2:Y'_{\tilde \ttS,\ttT}\xra{\sim}Z'_{\tilde \ttS,\ttT}$ proved in \ref{P:Y_T=Z_T}, $\gothF'_{\gothp^2,\ttS}$ induces a map $\gothF'_{\gothp^2,\ttS}:Z'_{\tilde \ttS,\ttT}\ra Z'_{\sigma^2_{\gothp}\tilde \ttS,\sigma^2_{\gothp}\ttT}$.
   Let $z=(B,\iota_{B},\lambda_{B},\beta_{K'_{\ttT}}, J^{\circ})$ be an $S$-valued point of $Z'_{\tilde \ttS,\ttT}$ as described in \ref{S:Y_T=Z_T}. Then its image $\gothF'_{\gothp^2,\tilde \ttS}(z)$ is given by $(B',\iota_{B'},\lambda_{B'},\beta'_{K'_{\ttT}}, J^{\circ,\prime})\in Z'_{\sigma^2_{\gothp}\ttS,\sigma^{2}\ttT}$, where $(B',\iota_{B'},\lambda_{B'}, \beta'_{K'_{\ttT}})$ are defined as in \ref{S:partial Frobenius}, and $J^{\circ,\prime}$ is the collection of line bundles 
$J^{\circ,\prime}_{\tilde\tau}\subseteq H^{\dR}_1(B'/S)^{\circ}_{\tilde\tau}$ for each $\tilde\tau\in \bigcup_{\gothp'\in\Sigma_{p}}\sigma^2_{\gothp}(\tilde I_{\ttT_{\gothp'}})$ given as follows. For $\tilde\tau\in \tilde I_{\ttT_{\gothp'}}$ with $\gothp'\neq \gothp$, we have $J^{\circ,\prime}_{\tilde\tau}=J^{\circ}_{\tilde\tau}$ since $H^{\dR}_1(B'/S)^{\circ}_{\tilde\tau}$ is canonically identified with $H^{\dR}_1(B/S)^{\circ}_{\tilde\tau}$. For $\tilde\tau\in \sigma_\gothp^2(\tilde I_{\ttT_{\gothp}})$, we have
   $J^{\circ,\prime}_{\tilde\tau}=(J^{\circ}_{\sigma^{-2}\tilde\tau})^{(p^2)}$, which makes sense thanks to the isomorphism \eqref{E:isom-dR-ABC} for $G=B'$. After identifying $\bfSh_{K'}(G'_{\tilde \ttS})_{k_0,\ttT}=X'_{\tilde \ttS,\ttT}$ with $Z'_{\tilde \ttS,\ttT}$, the projection $\pi_{\ttT}:Z'_{\tilde \ttS,\ttT}\to \bfSh_{K'_{\ttT}}(G'_{\tilde \ttS(\ttT)})_{k_0}$ is given by $(B,\iota_{B},\lambda_{B},\beta_{K'_{\ttT}}, J^{\circ})\mapsto (B,\iota_{B},\lambda_{B},\beta_{K'_{\ttT}})$. It is clear that we have a commutative diagram:
   
   \[
\xymatrix@C=50pt{
Z'_{\tilde \ttS,\ttT}\ar[r]_-{\xi^\mathrm{rel}} \ar[rd]_{\pi_{\ttT}}
\ar@/^15pt/[rr]^-{\gothF'_{\gothp^2, \tilde \ttS}}
&
\gothF'^*_{\gothp^2, \ttS(\ttT)}(Z'_{\sigma^2_{\gothp}\tilde \ttS,\sigma^2_{\gothp}\ttT}) \ar[d] \ar[r]_-{\gothF'^*_{\gothp^2, \tilde \ttS(\ttT)}}
&
Z'_{\sigma^{2}_{\gothp}\tilde \ttS,\sigma^2_{\gothp}\ttT} \ar[d]^{\pi_{\sigma_{\gothp}^2\ttT}}
\\
&
\bfSh_{K'_\ttT}(G'_{\tilde \ttS(\ttT)})_{k_0}
\ar[r]^-{\gothF'_{\gothp^2, \tilde \ttS(\ttT)}} &
\bfSh_{K'_\ttT}(G'_{\sigma_\gothp^2(\tilde \ttS(\ttT))})_{k_0},
}
\]
where $\xi^{\mathrm{rel}}$ is given by $(B,\iota_{B},\lambda_{B},\beta_{K'_{\ttT}}, J^{\circ})\mapsto (B,\iota_{B},\lambda_{B},\beta_{K'_{\ttT}}, J^{\circ,\prime})$ with $J^{\circ,\prime}$ defined above. This proves  statement (4)  immediately.

\section{Ampleness of modular line bundles}
\label{Section:GO divisors}
In this section, we suppose that $F\neq \Q$. We will apply Theorem~\ref{T:main-thm-unitary} to prove some necessary conditions for the ampleness  of certain modular line bundles on quaternionic/unitary Shimura varieties. 
In this section, let $X' = \bSh_{K'}(G_{\tilde \ttS}')_{k_0}$ be a unitary Shimura variety over $k_0$ considered in Subsection~\ref{S:GO-notation}. This is a smooth and quasi-projective variety over $k_0$, and projective if $\ttS_{\infty}\neq\emptyset$.
Let $(\bfA',\iota,\lambda, \alpha_{K'})$ be the universal abelian scheme over $X'$.
For each $\tilde\tau\in \Sigma_{E,\infty}$, the $\cO_{X'}$-module $\omega^{\circ}_{\bfA'^\vee/X', \tilde \tau}$ is   locally free of rank $2-s_{\tilde\tau}$. It is a line bundle if $\tilde \tau|_F$  belongs to $ \Sigma_{\infty}-\ttS_\infty$.


\subsection{Rational Picard group}

For a variety $Y$ over $k_0$, we write $\Pic(Y)_\Q$ for $\Pic(Y)\otimes_\ZZ \QQ$.  For a line bundle  $\calL$ on $Y$, we denote by $[\calL]$ its class in $\Pic (Y)_{\Q}$.

\begin{lemma}\label{L:omega-Pic}
\emph{(1)} For any $\tilde\tau\in\Sigma_{E,\infty}$ lifting a place  $\tau \in \Sigma_\infty-\ttS_\infty$, we have equalities
\[
[\omega^\circ_{\bfA'^\vee/X',\tilde\tau}]=[\omega^\circ_{\bfA'^\vee/X',\tilde\tau^c}]=[\omega^\circ_{\bfA'/X',\tilde\tau}]=[\omega^\circ_{\bfA'/X',\tilde\tau^c}].
\]

\emph{(2)} For any $\tilde\tau\in \Sigma_{E,\infty}$, we have $[\wedge^2_{\cO_{X'}} H_{1}^\dR(\bfA'/X')^\circ_{\tilde\tau}]=0.$

\emph{(3)}
Let $X'^*$ denote the minimal compactification of $X'$ (which is just $X'$ if $\ttS_{\infty} \neq \emptyset$).
Then the natural morphism $j: \Pic(X'^*) \to \Pic(X')$ is injective.
Moreover, for each $\tilde \tau \in \Sigma_{E, \infty}$ lifting a place  $\tau \in  \Sigma_\infty - \ttS_\infty$,  $[\omega^\circ_{\bfA'^\vee/X', \tilde \tau}]$ belongs to the image of $j_\QQ: \Pic(X'^*)_\QQ \to \Pic(X')_\Q$.

\end{lemma}

\begin{proof}
(1)  Suppose that $\tau \in \Sigma_{\infty/\gothp} - \ttS_{\infty/\gothp}$ for $\gothp \in \Sigma_\gothp$.  Clearly, $\gothp$ is not of type $\alpha^\sharp$ or $\beta^\sharp$.
The equality  $[\omega^\circ_{\bfA'^\vee/X',\tilde\tau}]=[\omega^\circ_{\bfA'/X',\tilde\tau^c}]$ follows from the isomorphism $\omega^\circ_{\bfA'^\vee/X',\tilde\tau}\cong \omega^\circ_{\bfA'/X',\tilde\tau^c} $ thanks to the polarization $\lambda$ on $\bfA'$.
 To prove the equality $[\omega^\circ_{\bfA'^\vee/X',\tilde\tau}]=[\omega^\circ_{\bfA'^\vee/X',\tilde\tau^c}]$, we consider the partial Hasse invariants $h_{\tilde\tau}\in \Gamma\big(X', (\omega^\circ_{\bfA'^\vee/X',\sigma^{-n_\tau}\tilde\tau})^{\otimes p^{n_{\tau}}}\otimes \omega^{\circ,\otimes (-1)}_{\bfA'^\vee/X',\tilde\tau}\big)$,
 and  $h_{\tilde\tau^c}$ defined similarly with $\tilde\tau$ replaced by $\tilde\tau^c$. By Lemma~\ref{Lemma:partial-Hasse} and  Proposition~\ref{Prop:smoothness}, the vanishing of $h_{\tilde\tau}$ and $h_{\tilde\tau^c}$ define the same divisor: $X'_{\tau}\subseteq X'$.
 Hence, for each $\tilde\tau\in \Sigma_{E,\infty}$ lifting some $ \tau \in \Sigma_{\infty/\gothp} -\ttS_{\infty/\gothp}$, we have an equality
\begin{equation}\label{E:equality-tau-tau-c}
p^{n_{\tau}} [\omega^\circ_{\bfA'^\vee/X',\sigma^{-n_\tau}\tilde\tau}]- [\omega^{\circ}_{\bfA'^\vee/X',\tilde\tau}] =p^{n_{\tau}} [\omega^{\circ}_{\bfA'^\vee/X',\sigma^{-n_\tau}\tilde\tau^{c}}]-[\omega^{\circ}_{\bfA'^\vee/X',\tilde\tau^c}].
\end{equation}
Let $C$ be the square matrix with coefficients in $\Q$, whose rows and columns are labeled by those places $\tilde \tau \in \Sigma_{E, \infty}$ lifting a place $\tau \in \Sigma_{\infty/\gothp} - \ttS_{\infty/\gothp}$, and whose $(\tilde \tau_1, \tilde \tau_2)$-entry is
\[
c_{\tilde \tau_1, \tilde \tau_2}=\begin{cases} -1&\text{if } \tilde \tau_1=\tilde \tau_2,\\
p^{n_{\tau_2}}&\text{if }\tilde \tau_1=\sigma^{-n_{\tau_2}}\tilde \tau_2,\\
0&\text{otherwise}.
\end{cases}
\]
One checks easily that $C$ is invertible, hence it follows from  \eqref{E:equality-tau-tau-c} that $[\omega^\circ_{\bfA'^\vee/X',\tilde\tau}]=[\omega^\circ_{\bfA'^\vee/X',\tilde\tau^c}]$.

(2) Assume first that $\tilde\tau\in \Sigma_{E,\infty}$ lifts some $\tau \in \Sigma_{\infty/\gothp}-\ttS_{\infty}$. From the Hodge filtration $0\ra \omega^\circ_{\bfA'^\vee,\tilde\tau}\ra H^\dR_1(\bfA'/X')^\circ_{\tilde\tau}\ra \Lie(\bfA'/X')^\circ_{\tilde\tau}\ra 0$, one deduces that
 \[
 [\wedge^2_{\cO_{X'}}H^{\dR}_1(\bfA'/X')_{\tilde\tau}^\circ]=[\omega^\circ_{\bfA'^\vee/X',\tilde\tau}]+[\Lie(\bfA'/X')_{\tilde\tau}^\circ].
 \]
 Then statement (2) follows from (1) and that $[\Lie(\bfA'/X')^\circ_{\tilde\tau}]=-[\omega^\circ_{\bfA'/X',\tilde\tau}]=-[\omega^\circ_{\bfA'^\vee/X',\tilde\tau^c}]$. Consider now the  case when $\tilde\tau\in \Sigma_{E,\infty}$ lifts some $\tau \in \ttS_{\infty/\gothp}$ for a place $\gothp$ of type $\alpha$ or $\beta$. Then there is an integer $m\geq 1$ such that $\sigma^m\tau\in \Sigma_{\infty}-\ttS_{\infty}$ and $\sigma^{i}\tau\in \ttS_{\infty}$ for all $0\leq i\leq m-1$, and we have a sequence of isomorphisms
 \[
 \xymatrix{
 H^\dR_1(\bfA'/X')^{\circ, (p^m)}_{\tilde\tau}\ar[r]^-{F_{\bfA',\es}}_-{\cong} &H^\dR_1(\bfA'/X')^{\circ, (p^{m-1})}_{\sigma\tilde\tau}\ar[r]^-{F_{\bfA',\es}}_-{\cong} &\cdots\ar[r]^-{F_{\bfA',\es}}_-{\cong} &H^\dR_1(\bfA'/X')^\circ_{\sigma^m\tilde\tau}.
 }
 \]
 From this, one gets
 $$
 p^m[\wedge^2_{\cO_{X'}}H^\dR_1(\bfA'/X')^\circ_{\tilde\tau}]=[\wedge^2_{\cO_{X'}}H^\dR_1(\bfA'/X')^\circ_{\sigma^m\tilde\tau}]=0.$$
 
Finally, if $\tilde\tau\in \Sigma_{E,\infty/\gothp}$ for a place $\gothp$ of type $\alpha^\sharp$ or $\beta^\sharp$ and if $m$ is the inertia degree of $\gothp$ over $p$, then the sequence of isomorphism
 \[
 \xymatrix{
 H^\dR_1(\bfA'/X')^{\circ, (p^{2m})}_{\tilde\tau}\ar[r]^-{F_{\bfA',\es}}_-{\cong} &H^\dR_1(\bfA'/X')^{\circ, (p^{2m-1})}_{\sigma\tilde\tau}\ar[r]^-{F_{\bfA',\es}}_-{\cong} &\cdots\ar[r]^-{F_{\bfA',\es}}_-{\cong} &H^\dR_1(\bfA'/X')^\circ_{\sigma^{2m}\tilde\tau}
 }
 \]
gives rise to an equality
\[
p^{2m}[\wedge^2_{\cO_{X'}}H^\dR_1(\bfA'/X')^\circ_{\tilde\tau}]=[\wedge^2_{\cO_{X'}}H^\dR_1(\bfA'/X')^\circ_{\sigma^{2m}\tilde\tau}]=[\wedge^2_{\cO_{X'}}H^\dR_1(\bfA'/X')^\circ_{\tilde\tau}],
\]
as $\sigma^{2m}\tilde \tau = \tilde \tau$.  This forces $[\wedge^2_{\cO_{X'}}H^\dR_1(\bfA'/X')^\circ_{\tilde\tau}] = 0$.

(3) If $X'$ is a Shimura curve, then $X'^*=X'$ as $F\neq \Q$.  Assume now that $X'$ has dimension at least $2$. The injectivity of $j: \Pic(X'^*) \to \Pic(X')$ follows from the fact that $X'^*$ is normal \cite[Proposition~7.2.4.3]{lan},  and that the complement $X'^*-X'$  has codimension $\geq 2$.
Recall from (1) that the partial Hasse invariant defines the class as described in \eqref{E:equality-tau-tau-c}.
Note that the inverse matrix of $C$ has all entries positive.
It follows that each $[\omega^\circ_{\bfA'^\vee/X', \tilde\tau_0}]$ for $\tilde \tau_0$ lifting $\tau_0 \in \Sigma_\infty- \ttS_\infty$ is a positive linear combination of $\calO_{X'}(X'_\tau)$'s.
Let $X'^{\mathrm{n-ord}}=\bigcup_{\tau\in \Sigma_{\infty}-\ttS_{\infty}}X'_{\tau}\subseteq X'$  be the union of the Goren-Oort strata of codimension $1$. 
Since $X'^{\mathrm{n-ord}}$ is closed in $X'^*$ and is disjoint from the cusps, each line bundle $\calO_{X'}(X'_\tau)$ extends to a line bundle $\calO_{X'^*}(X'_\tau)$.
By linear combination, each $[\omega^\circ_{\bfA'^\vee/X', \tilde\tau_0}]$ extends to a class in $\Pic(X'^*)_\QQ$.
\end{proof}

\begin{notation}
For any $\tau \in \Sigma_{\infty}-\ttS_{\infty}$, we put  $[\omega_{\tau}]=[\omega^\circ_{\bfA'/X',\tilde\tau}]$ for simplicity, where $\tilde \tau$ is a lift of $\tau$. This is a well defined element in $\Pic (X'^*)_\Q$ by Lemma~\ref{L:omega-Pic}.
\end{notation}

\begin{prop}
\label{P:normal bundle}
Let $\gothp$ be a $p$-adic place such that $\#(\Sigma_{\infty/\gothp} - \ttS_{\infty/\gothp}) >1$.
When $\ttT$ consists of a single element $\tau \in \Sigma_{\infty/\gothp}$ with $\gothp$ not of type $\beta2$, let $n_\tau = n_{\tau,\ttS}$ be as in Subsection~\ref{S:partial-Hasse}.  Let $N_{X'_{\ttT}}(X')$ denote the normal bundle of the embedding  $X'_{\ttT} \hookrightarrow X'$. Then the equality $[N_{X'_{\ttT}}(X')]=[\calO(-2p^{n_\tau})]$ holds  in $\Pic(X'_{\ttT})_{\Q}$,  where $\calO(1)$ is the canonical quotient bundle of the $\PP^1$-bundle $\pi_\tau: X'_{\ttT}=\bfSh_{K'}(G'_{\tilde \ttS})_{k_0, \ttT}  \to \bfSh_{K'}(G'_{\tilde \ttS(\ttT)})_{k_0} $, and $\calO(-2p^{n_\tau})$ is the dual of  $\calO(1)^{\otimes2 p^{n_\tau}}$.
\end{prop}
\begin{proof}
By the construction in 
Subsection~\ref{S:Y_T=Z_T}, the set $\tilde I_\ttT = \{\sigma^{-n_\tau}\tilde \tau\}$ for a specific lift $\tilde \tau$ of $\tau$.  We have
\[
J^\circ_{\sigma^{-n_\tau}\tilde \tau} = \phi_{B, *, \sigma^{-n_\tau}\tilde \tau}^{-1} \circ \phi_{A, *, \sigma^{-n_\tau}\tilde \tau}(\omega^\circ_{\bfA^\vee, \sigma^{-n_\tau}\tilde \tau}),
\]
in terms of the moduli description of $Y'_\ttT\cong X'_\ttT$.
So the restriction of $[\omega_{ \sigma^{-n_\tau} \tau}]$ to $X'_{\ttT}$ is $[\calO(-1)]$.
The Goren-Oort stratum $X'_{\ttT}$ is defined as the zero locus of 
\[
h_{\tilde \tau}: \omega^\circ_{\bfA'^\vee/X', \tilde \tau} \to (\omega^\circ_{\bfA'^\vee / X', \sigma^{-n_\tau} \tilde \tau})^{\otimes p^{n_\tau}}.
\]
So firstly, the class of $N_{X'_{\ttT}}(X')$ in $\Pic(X'_{\ttT})_{\Q}$ is given by the restriction of $p^{n_\tau}[\omega_{\sigma^{-n_\tau}\tau}] - [\omega_{\tau}]$ to $X'_{\ttT}$, and secondly, on $X'_{\ttT}$ we have an isomorphism
\[
\omega^\circ_{\bfA'^\vee/X', \tilde \tau} \xrightarrow{\cong} H_1^\dR(\bfA'/X')^{\circ, (p^{n_\tau})}_{\tilde \tau} /  (\omega^\circ_{\bfA'^\vee / X', \sigma^{-n_\tau} \tilde \tau})^{\otimes p^{n_\tau}}.
\]
This implies that $[\omega_\tau]$ equals to $-p^{n_\tau}[\omega_{\sigma^{-n_\tau}\tau}]$ in $\Pic(X'_{\ttT})_{\Q}$.

To sum up, we have equalities in $\Pic(X'_{\ttT})_{\Q}$:
\begin{align*}
[N_{X'_{\ttT}}(X')]&=p^{n_\tau}[\omega_{\sigma^{-n_\tau}\tau}] - [\omega_{\tau}]\\
&=2p^{n_\tau}[\omega_{\sigma^{-n_\tau}\tau}]=[\calO(-2p^{n_\tau})].\tag*{\qedhere}
\end{align*} 
\end{proof}

\begin{theorem}\label{T:ampleness}
Let $\underline t=(t_{\tau})\in \Q^{\Sigma_{\infty}-\ttS_{\infty}}$. If the element $[\omega^{\underline t}]=\sum_{\tau\in \Sigma_{\infty}-\ttS_{\infty}}t_{\tau}[\omega_{\tau}]$ of $\Pic(X')_{\Q}$ is ample, then 
  \begin{equation}\label{E:condition-ample}
 p^{n_{\tau}}t_{\tau}> t_{\sigma^{-n_\tau}\tau}\quad \text{(and $t_{\tau}>0$) for all }\tau.
  \end{equation}
Here, we put the second condition in parentheses, because it follows from the first one.
\end{theorem}

\begin{proof}
 Assume that $[\omega^{\underline t}]$ is ample. Let $\tau\in \Sigma_{\infty}-\ttS_{\infty}$, and $\gothp\in \Sigma_p$ be the prime of $F$ such that $\tau\in \Sigma_{\infty/\gothp}$.  We distinguish two cases:
\begin{itemize}
\item $\Sigma_{\infty/\gothp}-\ttS_{\infty/\gothp}=\{\tau\}$. Condition \eqref{E:condition-ample} for $\tau$ is  simply $t_{\tau}>0$. We consider the GO-stratum $X'_{\ttT_{\tau}}$ with $\ttT_{\tau}=\Sigma_{\infty}-(\ttS_{\infty}\cup \{\tau\})$. Then $X'_{\ttT_{\tau}}$ is isomorphic to  a Shimura curve by Theorem~\ref{T:main-thm-unitary}, and let $i_{\tau}: X'_{\ttT_{\tau}}\ra X'$ denote the canonical embedding. For any $\tilde\tau'\in \Sigma_{E,\infty}-\ttS_{E,\infty}$ with restriction $\tau'=\tilde\tau'|_{F}\neq \tau$. Let
\begin{equation}\label{E:definition-F-tau}
F_{\bfA',\es,\tilde\tau'}^{n_{\tau'}}: H^{\dR}_1(\bfA'^{(p^{n_{\tau}})}/X')^\circ_{\tilde\tau'}\ra H^{\dR}_1(\bfA'/X')^\circ_{\tilde\tau'}.
\end{equation}
be the $n_{\tau'}$-th iteration of the essential Frobenius in Subsection~\ref{S:partial-Hasse}. We  always have $\Ker(F_{\bfA',\es,\tilde\tau'}^{n_{\tau'}})=(\omega^\circ_{\bfA'^\vee/X',\sigma^{-n_{\tau'}}\tilde\tau'})^{(p^{n_{\tau'}})}$.
The vanishing of $h_{\tilde\tau'}$ on $X'_{\ttT_{\tau}}$ is equivalent to
$$
\im(F_{\bfA',\es,\tilde\tau'}^{n_{\tau'}})|_{X'_{\ttT_{\tau}}}=(\omega^\circ_{\bfA'^\vee/X',\tilde\tau'})|_{X'_{\ttT_{\tau}}}.
$$
Therefore, one deduces an equality in $\Pic(X'_{\ttT_{\tau}})_{\Q}$:
\[
p^{n_{\tau}}i_{\tau}^*[\omega_{\sigma^{-n_{\tau'}}\tau'}]+i^*_{\tau}[\omega_{\tau'}]=0.
\]
Letting $\tau'$ run through the set $\Sigma_{\infty/\gothq}-\ttS_{\infty/\gothq}$ with $\gothq\neq \gothp$, one obtains that $i_{\tau}^*[\omega_{\tau'}]=0$. Therefore,  we have $i_{\tau}^*[\omega^{\underline t}]=t_{\tau} i_{\tau}^*[\omega_{\tau}]$, which is ample on $X'_{\ttT_\tau}$ since so is $[\omega^{\underline t}]$ on $X'_{\ttT_\tau}$ by assumption. By the ampleness of $\det(\omega)=\bigotimes_{\tilde\tau'\in \Sigma_{E,\infty}}\omega^\circ_{\bfA',\tilde\tau'}$ on $X'$ and hence on  $X'_{\ttT_{\tau}}$, we see that $i^*_{\tau}[\omega_{\tau}]$ is ample on $X'_{\ttT_{\tau}}$. It follows  that $t_{\tau}>0$.

\item $\Sigma_{\infty/\gothp}-\ttS_{\infty/\gothp}\neq\{\tau\}$. Consider the GO-stratum $X'_{\tau}$ given by $\ttT=\{\tau\}$, then $\ttS(\ttT)=\ttS\cup \{\tau,\sigma^{-n_{\tau}}\tau\}$ in the notation of Subsection~\ref{S:quaternion-data-T}. By Propositions~\ref{P:Y_S=X_S} and \ref{P:Y_T=Z_T}, $X'_\tau$ is isomorphic to a $\PP^1$-bundle over $\bSh_{K'_{\tau}}(G'_{\tilde \ttS(\ttT)})_{k_0}$. Let $\pi: X'_{\tau}\ra \bSh_{K'_{\tau}}(G'_{\tilde \ttS(\ttT)})_{k_0}$ denote the natural projection. The ampleness of $[\omega^{\underline t}]$ on $X'$ implies the ampleness of its restriction to each closed fiber $\PP^1_s$ of $\pi$. 
By the proof of Proposition~\ref{P:normal bundle}, we have
\[
\omega^\circ_{\bfA'^\vee/X',\tilde\tau'}|_{\PP^1_{s}}\simeq
 \begin{cases}
 \cO_{\PP^1_{s}}(-1) &\text{if $\tau'=\sigma^{-n_{\tau}}\tau$};\\
 \cO_{\PP^1_{s}}(p^{n_{\tau}}) &\text{if } \tau'=\tau;\\
 \cO_{\PP^1_{s}} &\text{otherwise}.
 \end{cases}
\]
The relation~\eqref{E:condition-ample} follows immediately.\qedhere
\end{itemize}
\end{proof}

Since the Hilbert modular varieties and the unitary Shimura varieties have the same neutral geometric connected components, the following is  an immediate corollary of Theorem~\ref{T:ampleness}.
\begin{cor}\label{C:ampleness Hilbert}
Let $X$ denote the special fiber of the Hilbert modular variety, and $X^*$ its minimal compactification.
Then for each $\tau \in \Sigma_\infty$, the class $[\omega_\tau] \in \Pic(X)_\Q$ uniquely extends to a class $[\omega_\tau] \in \Pic(X^*)_\Q$. Moreover, $[\omega^{\underline t}] = \sum_{\tau \in \Sigma_\infty}t_\tau [\omega_\tau]$ is ample only when $t_{\tau}>0$ and  $p t_\tau > t_{\sigma^{-1}\tau}$ for all $\tau \in \Sigma_\infty$. 
\end{cor}

For the converse to Theorem~\ref{T:ampleness}, we have the following
\begin{conj}
\label{C:ampleness conjecture}
The conditions in Theorem~\ref{T:ampleness} and Corollary~\ref{C:ampleness Hilbert} are also sufficient for $[\omega^{\underline t}]$ to be ample.
\end{conj}

\begin{remark}
In the case of Hilbert modular surface, Corollary~\ref{C:ampleness Hilbert} and the sufficiency of the condition were proved by Andreatta and Goren \cite[Theorem~8.1.1]{andreatta-goren}, which relies heavily on some intersection theory on surfaces.  It seems difficult to generalize their method.

Using our global geometric description, it seems possible to prove, for small inertia degrees (at least when all inertia degrees are $\leq 5$), Conjecture~\ref{C:ampleness conjecture} using variants of Nakai-Moishezon criterion.  The combinatorics becomes complicated when the inertia degree is large.
\end{remark}

\section{Link morphisms}\label{Section:links}
We will introduce certain generalizations of the partial Frobenius morphisms, called the \emph{link morphisms}, on unitary Shimura varieties associated to quaternionic ones. 
These morphisms appear naturally when considering the restriction of the projection maps $\pi_{\ttT}$  in Theorem~\ref{T:main-thm-unitary} to other Goren-Oort strata. The explicit descriptions of these morphisms are essential for the application considered  in the forthcoming paper \cite{tian-xiao3}. 
For simplicity,  \emph{we will assume that $p$ is inert of degree $g$ in the totally real field $F$.}
Denote by $\gothp$ the unique prime of $F$ above $p$. 
Let $E$ be a CM extension of $F$. 
If $\gothp$ splits in $E$, fix  a prime $\gothq$ of $E$ above $\gothp$, and denote the other prime by $\bar \gothq$; if $\gothp$ is inert in $E$, we denote by $\gothq$ the unique prime of $E$ above $\gothp$.

\subsection{Links}\label{S:links}
We introduce some combinatorial objects. 
Let $n\geq 1$ be an integer.
Put $n$ points aligned equi-distantly on a horizontal section of a vertical cylinder. We label the $n$ points by the elements of $\Z/n\Z$ so that the $(i+1)$-st point is next to the $i$th point on the right. 
Let $S$ be a subset of the $n$ points above. To such an $S$, we associate a graph as follows: 
we start left to right with the plot labeled $0\in \Z/n\Z$,  and draw a \emph{plus sign} if the element  is in $S$, and a \emph{node} if it is in $\Z/n\Z-S$. 
We call such a picture a \emph{band of length $n$} associated to $S$.
 For instance, if $n=5$ and $S=\{1,3\}$, then the band is 
$
\psset{unit=0.3}
\begin{pspicture}(-.5,-0.3)(4.5,0.3)
\psset{linecolor=black}
\psdots(0,0)(2,0)(4,0)
\psdots[dotstyle=+](1,0)(3,0)
\end{pspicture}$.

Let $S'$ be another subset of $\Z/n\Z$ of the same cardinality as $S$.
 Then a link \emph{link} $\eta: S\ra S'$ is a graph of the following kind: put the band attached to $S$ on the top of the band for $S'$ in the same  cylinder; draw non-intersecting curves from  each of the nodes from the top band to a  node on the bottom band. We say a curving is turning to the \emph{left} (resp. to the \emph{right}) if it is so as moving from the top band to the bottom band.
If a curve travels $m$-numbers of points (of both plus signs and nodes) to the \emph{right} (resp. \emph{left}), we say the displace of this curve is $m$ (resp. $-m$).
When both $S$ and $S'$  are equal to  $\Z/n\Z$ (so that there are no nodes at all), then we say  that $\eta:S\ra S'$ is the trivial link.
 We define the \emph{total displacement} of a link $\eta$ as the sum of the displacements of all curves  in $\eta$. For example, if $n=5$, $S=\{1, 3\}$ and $S'=\{1, 4\}$, then 
\begin{equation}
\label{E:left turn link}
\psset{unit=0.3} \eta = 
\begin{pspicture}(-.5,-0.3)(5,2.3)
\psset{linecolor=red}
\psset{linewidth=1pt}
\psbezier(0,2)(1,1)(3,1)(3,0)
\psbezier(2,2)(2,1)(3.5,1.5)(4.5,0.5)
\psbezier(-0.5,1.3)(0.5,.3)(2,1)(2,0)
\psarc{-}(-0.5,0){0.5}{0}{90}
\psarc{-}(4.5,2){0.5}{180}{270}
\psset{linecolor=black}
\psdots(0,2)(2,2)(4,2)
\psdots(0,0)(2,0)(3,0)
\psdots[dotstyle=+](1,2)(3,2)
\psdots[dotstyle=+](1,0)(4,0)
\end{pspicture}.
\end{equation}
is a link from $S$ to $S'$, and   its total displacement is $v(\eta)=3+3+3=9$. 

For a link $\eta: S\ra S'$, we denote by $\eta^{-1}: S'\ra S$ the link obtained by flipping the picture about the equator of the cylinder. For two links $\eta: S\ra S'$ and $\eta':S'\ra S''$, we define the composition of $\eta'\circ\eta: S\ra S''$ by putting the picture of $\eta$ on the top of the picture of $\eta'$ and identify the nodes corresponding  to $\eta'$.  It is obvious that $v(\eta^{-1})=-v(\eta)$ and $v(\eta'\circ\eta)=v(\eta')+v(\eta)$.

\subsection{Links for a subset of places of $F$ or $E$}
We return to the setup of Notation~\ref{S:Notation-for-the-paper}, and recall that $p$ is inert in $F$. 
 We fix an isomorphism $\Sigma_{\infty}\cong \Z/g\Z$ so that $i\mapsto i+1$ corresponds to the action of Frobenius on $\Sigma_{\infty}$.

  For an even subset $\ttS$ of places in $F$, we have the \emph{band} for $\ttS$ when applying Subsection~\ref{S:links} to  the subset $\ttS_{\infty}$ of $\Sigma_{\infty}$. 
Let $\ttS'$ be another even subset of places of $F$ such that $\#\ttS_{\infty}=\#\ttS'_{\infty}$ and  $\ttS'$ contains the same finite places of $F$ as $\ttS$ does.
   A link  $\eta$ from the band for $\ttS$ to that for $\ttS'$ is denoted by $\eta: \ttS\ra \ttS'$. When $\ttS'=\ttS$ and $\ttS_{\infty}=\Sigma_{\infty}$,  $\eta:\ttS\ra\ttS'$ is necessarily the trivial link (so that there are no curves at all).  
    
    The Frobenius action on $\Sigma_{\infty}$ defines a link   $\sigma: \ttS\ra \sigma(\ttS)$, in which all curves turn to the right with displacement 1; the total displacement of this link $\sigma$ is $v(\sigma)=g-\#\ttS_{\infty}$.
     Here, $\sigma(\ttS)$ denotes the subset of places of $F$ whose finite part is the same  as $\ttS$, and whose infinite part is  the image of Frobenius on $\ttS_{\infty}$. 

\begin{notation}\label{S:notation-n-tau}
Recall the definition of $n_\tau$ for $\tau \in \Sigma_\infty - \ttS_\infty$ from Notation~\ref{N:n tau}.
For simplicity, we write $\tau^-$ for $\sigma^{-n_\tau} \tau$, and we use $\tau^+$ to denote the unique place in $\Sigma_\infty-\ttS_\infty$ such that $\tau = (\tau^+)^-  = \sigma^{-n_{\tau^+}}\tau^+$.
 When there are several $\ttS$ involved, we will write $n_{\tau}(\ttS)$ for $n_{\tau}$ to emphasize its dependence on $\ttS$. 
     
\end{notation}

\subsection{Link morphisms}\label{S:link-morphisms}

 Let $\eta: \ttS\ra \ttS'$ be a link of two even subsets of places of $F$. 
 If $\ttS_{\infty}\neq \Sigma_{\infty}$,  we  denote by $m(\tau)$  the displacement of the curve  at $\tau$ in the link $\eta$ for each $\tau\in \Sigma_{\infty}-\ttS_{\infty}$,
and put $m(\tau) =0$ for $\tau \in \ttS_\infty$.

   Let $\bfSh_{K'}(G'_{\tilde \ttS})_{k_0}$ and $\bfSh_{K'}(G'_{\tilde\ttS'})_{k_0}$ denote  the special fibers of some unitary Shimura varieties  of type considered in Subsection~\ref{S:PEL-Shimura-data}. (There is no restriction on the signatures, i.e. the sets $\tilde \ttS_\infty$ and $\tilde  \ttS'_\infty$ that lift $\ttS_\infty$ and $\ttS'_\infty$; but we fix them.)
    Here, we have fixed (compatible) isomorphisms $\cO_D: = \cO_{D_{\ttS}}\cong \cO_{D_{\ttS'}}\cong  \rmM_{2\times 2}(\cO_E)$ and $G'_{\tilde\ttS}(\AAA^{\infty})\cong G'_{\tilde\ttS'}(\AAA^{\infty})$, and regard $K'$ as an open compact subgroups of both of the groups; this is possible because $\ttS$ and $\ttS'$ have the same finite part, and the argument in Lemma~\ref{L:compare D_S with D_S(T)} applies verbatim in this situation.
     Note that $K'_p$ is  assumed to be hyperspecial as in Subsection~\ref{S:PEL-Shimura-data}.
     Let $\bfA'_{\tilde\ttS, k_0}$ be the universal abelian scheme over $\bfSh_{K'}(G'_{\tilde \ttS})_{k_0}$.
      For a point $x$ of $\bfSh_{K'}(G'_{\tilde\ttS})_{k_0}$ with values in a perfect field $k(x)$, we denote by $\bfA'_{\tilde\ttS,x}$ the base change  of $\bfA'_{\tilde\ttS}$ to $x$, and $\tcD(\bfA'_{\tilde\ttS, x})^{\circ}$ the reduced part of the \emph{covariant} Dieudonn\'e module of $\bfA'_{\tilde\ttS,x}$ (cf. Subsection~\ref{S:GO-notation} and the proof of Lemma~\ref{L:Y_T=X_T-1}).
       For each $\tilde\tau\in \Sigma_{E, \infty} -  \ttS_{E,\infty}$ 
   we have the essential Frobenius  map defined in \ref{N:essential frobenius and verschiebung}:
   \[
   F_{\bfA',\es}: \tcD(\bfA'_{\tilde \ttS,x})^{\circ}_{\sigma^{-1}\tilde\tau}\ra \tcD(\bfA'_{\tilde \ttS,x})_{\tilde\tau}^{\circ}.
   \]      
    Finally,   recall  that a \emph{$p$-quasi-isogeny} of abelian varieties means  a quasi-isogeny of the form $f_1\circ f_2^{-1}$, where  $f_1$ and $f_2$ are isogenies  of $p$-power order. The degree of the quasi-isogeny is $\deg f_1 / \deg f_2$.

 \begin{defn}\label{D:link-morphism}
 Assume that $m(\tau)\geq 0$ for each $\tau\in \Sigma_\infty$, i.e. all curves (if any) in $\eta$ are either straight lines or all turning to the right. 
    Let $n$ be an integer. 
    If $\gothp$ is inert in $E$, we assume that $n=0$.   A \emph{link morphism of indentation degree $n$} associated to  $\eta$ on $\bfSh_{K'}(G'_{\tilde\ttS})_{k_0}$ (if exists)   is a morphism of varieties 
   $$
     \eta'_{(n),\sharp}: \bfSh_{K'}(G'_{\tilde\ttS})_{k_0}\ra \bfSh_{K'}(G'_{\tilde\ttS'})_{k_0}
     $$
     together with a  $p$-quasi-isogeny of abelian varieties 
   $$\eta'^{\sharp}_{(n)}: \bfA'_{\tilde\ttS, k_0}\ra \eta'^{*}_{(n),\sharp}(\bfA'_{\tilde\ttS',k_0}),$$ 
   such that the following conditions are satisfied:
   \begin{itemize}
   \item[(1)]  $\eta'_{(n),\sharp}$ induces a  bijection  on geometric points.
   \item[(2)] The quasi-isogeny $\eta'^{\sharp}_{(n)}$ is compatible with the actions of $\cO_D$, level structures, and the polarizations on both abelian varieties.
\item[(3)] 
There exists, for  each $\tilde\tau\in \Sigma_{E, \infty} -  \ttS_{E,\infty}$, some $t_{\tilde\tau}\in \Z$, such that, for every $\overline \FF_p$-point $x$ of $\bfSh_{K'}(G'_{\tilde \ttS})_{k_0}$ with image $x'=\eta'_{(n),\sharp}(x)$, 
   \[
  \eta'^{\sharp}_{(n), *}\big(F_{\es,\bfA'_{\tilde\ttS,x}}^{m(\tau)}(\tcD(\bfA'_{\tilde \ttS,x})^\circ_{\tilde\tau})\big)=p^{t_{\tilde\tau}}\tcD(\bfA'_{\tilde\ttS',x'})^\circ_{\sigma^{m(\tau)}\tilde\tau},
   \]
where $\tau\in \Sigma_{\infty}$ is the image of $\tilde\tau$.

   \item[(4)]  The quasi-isogeny of $\gothq$-divisible group 
   \[
   \eta'^{\sharp}_{(n),\gothq}:\bfA'_{\tilde\ttS, k_0}[\gothq^{\infty}]\ra \eta'^{*}_{(n),\sharp}\bfA'_{\tilde\ttS', k_0}[\gothq^{\infty}]
   \]
   has degree $p^{2n}$. Here, our convention for $\gothq$ is as at the beginning of this section. In particular, if $\gothp$ splits in $E$, then the quasi-isogeny on the  $\bar\gothq^{\infty}$-divisible groups
   \[
   \eta'^{\sharp}_{(n),\bar\gothq}:\bfA'_{\tilde\ttS, k_0}[\bar\gothq^{\infty}]\ra \eta'^*_{(n),\sharp}\bfA'_{\tilde\ttS', k_0}[\bar\gothq^{\infty}]
   \]
     has necessarily  degree $p^{-2n}$.  Here, the exponent $2n$ is due to the fact that $\bfA'_{k_0}[\gothq^{\infty}]$ is two copies of its reduced part $\bfA'_{k_0}[\gothq^{\infty}]^{\circ}$.
   
      \end{itemize}
      \end{defn}

        Let $\eta_i:\ttS_i\ra \ttS_{i+1}$ for $i=1,2$ be two links with all curves turning to the right, and 
  let $(\eta'_{i,\sharp},\eta'^{\sharp}_i)$ for $i=1,2$ be a link morphism of indentation degree $n_i$ on $\bfSh_{K'_p}(G'_{\tilde \ttS_i})_{k_0}$  attached to $\eta_{i}$. 
The composition of $(\eta'_{2,\sharp},\eta'^{\sharp}_{2})$ with $(\eta'_{1,\sharp},\eta'^{\sharp}_{1})$ defined by  
\[
\eta'_{12,\sharp}: \bfSh_{K'_p}(G'_{\tilde\ttS_1})_{k_0}\xra{\eta'_{1,\sharp}}\bfSh_{K'_p}(G'_{\tilde\ttS_2})_{k_0}\xra{\eta'_{2,\sharp}}\bfSh_{K'_p}(G'_{\tilde\ttS_3})_{k_0}
\]
and 
\[
\eta'^{\sharp}_{12}: \bfA'_{\tilde\ttS_1,k_0}\xra{\eta'^{\sharp}_{1}}\eta'^{*}_{1,\sharp}(\bfA'_{\tilde\ttS_2,k_0})\xra{\eta'^{*}_{1,\sharp}(\eta'^{\sharp}_{2})}\eta'^{*}_{12,\sharp}(\bfA'_{\tilde\ttS_3,k_0}),
\]
is a link morphism attached to the composed link $\eta_{12} :=\eta_2\circ\eta_1$ with  indentation degree $n_1+n_2$.   
   
    \subsection{Variants}
   The formulation of link morphisms on $\bfSh_{K'}(G'_{\tilde \ttS})_{k_0}$ is  compatible with changing   the tame level $K'^p$. 
    By taking the inverse limit of $K'^p$,  one can  define a link morphism on $\bfSh_{K'_p}(G'_{\tilde \ttS})_{k_0}$ associated to $\eta$ in the obvious way.  
       
  One can  define similarly a link morphism  of indentation degree $n$ on $\bfSh_{K''_p}(G''_{\tilde \ttS})_{k_0}$ as  a pair $(\eta''_{(n), \sharp}, \eta''_{(n),\sharp})$, where 
  \[
  \eta''_{(n),\sharp}: \bfSh_{K''_p}(G''_{\tilde\ttS})_{k_0}\ra \bfSh_{K''_p}(G''_{\tilde\ttS'})_{k_0}
  \]
  is a morphism of varieties and 
  \[
  \eta''^{\sharp}_{(n)}: \bfA''_{\tilde\ttS, k_0}\ra \eta''^{*}_{(n),\sharp}(\bfA''_{\tilde\ttS',k_0})
  \]
  is a  $p$-quasi-isogeny of abelian schemes such that same conditions (1)-(4) in \ref{D:link-morphism} are satisfied (except the primes is replaced by double primes). Here, $\bfA''_{\tilde\ttS, k_0}$ is the family of abelian varieties constructed in Subsection~\ref{S:abel var in unitary case}.

      \begin{example}
     (1) Consider the second iteration of the Frobenius link $\sigma^2=\sigma_{\gothp}^2: \ttS\ra \sigma^{2}(\ttS)$. The twisted (partial) Frobenius map \eqref{E:twist-partial-Frob}
      \[
      \gothF'_{\gothp^2}:
       \bfSh_{K'}(G'_{\tilde \ttS})_{k_{0}} \to \bfSh_{K'}(G'_{\sigma^2 \tilde \ttS})_{k_{0}}
      \]
      together with the isogeny $\eta'_{\gothp^2}$ 
defined in \eqref{E:universal-isog-Frob} is a link morphism associated to $\sigma^2$; the indentation degree is $0$ if $\gothp$ is inert in $E/F$, and is $2\#\tilde \ttS_{\infty/\bar \gothq} - 2\#\tilde \ttS_{\infty/ \gothq}$ if $\gothp$ splits in $E/F$. 
      
      (2)  Assume that $\ttS_\infty = \Sigma_\infty$ and $\gothp\notin\ttS$ (so that $\gothp$  splits in $E$ by our choice of $E$).
        The Shimura variety $\bfSh_{K'}(G'_{\tilde \ttS})_{k_0}$ is just a finite union of closed points.
         Let $\tau_0\in \Sigma_{\infty}$, and  $\tilde \tau_0\in \tilde \ttS_\infty$ be  the  lift of $\tau_0$ with signature $s_{\tilde\tau_0}=0$. 
          We assume that  $\sigma^{-1}\tilde \tau_0 \notin \tilde \ttS_\infty$ (so  that $\sigma^{-1}\tilde \tau_0^c \in \tilde \ttS_\infty$).
          Let $\tilde \ttS'$ denote the subset of places of $F$ containing the same finite places as $ \tilde \ttS$ and such that  $\tilde \ttS'_\infty = \tilde \ttS_\infty \cup \{\tilde \tau_0^c, \sigma^{-1}\tilde\tau_0\} \backslash \{ \tilde \tau_0, \sigma^{-1}\tilde \tau_0^c\}$.
          Let $\ttS'$ be the subset of places of $F$ defined by the restriction of $\tilde\ttS'_{\infty}$.
Then there exists a link morphism  $(\delta'_{\tau_0,\sharp}, \delta'^{\sharp}_{\tau_0})$ from $\bfSh_{K'}(G'_{\tilde \ttS})_{k_0}$ to  $\bfSh_{K'}(G'_{\tilde \ttS'})_{k_0}$ associated to the trivial link $\ttS\ra \ttS'$ defined as follows; its indentation is $0$ if $\gothp$ is inert in $E/F$,  is $2$ if $\gothp$ splits in $E/F$ and $\tilde \tau$ induces the $p$-adic place $\gothq$, and is $-2$ if $\gothp$ splits in $E/F$ and $\tilde \tau$ induces the $p$-adic place $\bar \gothq$.

It suffices to define $\delta'_{\tau_0,\sharp}$ on the geometric closed points, as both Shimura varieties are zero-dimensional.
For each $\overline \FF_p$-point $x = (A, \iota ,\lambda_{A}, \alpha_{K'})  \in \bfSh_{K'}(G'_{\tilde \ttS})(\overline \FF_p)$,
let $\tilde \calD^\circ_{A,\tilde \tau}$ denote the $\tilde\tau$-component of the reduced  covariant Dieudonn\'e module of $A$ for each $\tilde \tau \in  \Sigma_{E, \infty}$. 
 We put $M^\circ_{\tilde \tau} = \tilde \calD^\circ_{A,\tilde \tau}$ for $\tilde \tau \neq \tilde \tau_0, \tilde \tau_0^c$ and 
 $$M^\circ_{\tilde \tau_0^c} = p\tilde \calD^\circ_{A,\tilde \tau_0^c},\quad M^\circ_{\tilde \tau_0} = \frac 1p\tilde \calD^\circ_{A,\tilde \tau_0}\subseteq \tcD^{\circ}_{A,\tilde\tau_0}[\frac{1}{p}].
 $$
It is straightforward to check that the signature condition implies that $M_{\tilde\tau}$'s are stable under the actions of Frobenius and Verschiebung of $\tcD^{\circ}_{A,\tilde\tau_0}[1/p]$.  
As in the proof of Proposition~\ref{P:Y_S=X_S}, this gives rise to an abelian variety $B$ of dimension $4g$ with an action of $\calO_D$ and an $\calO_D$-quasi-isogeny $\phi:B \to A$ such that the induced morphism on Dieudonn\'e modules  $\phi_*: \tcD^{\circ}_{B,\tilde\tau}\ra \tcD^{\circ}_{A,\tilde\tau}[1/p]$ is identified with the natural inclusion $M^\circ_{\tilde\tau}\hra \tcD^{\circ}_{A,\tilde\tau}[1/p]$ for all $\tilde\tau\in \Sigma_{E,\infty}$.
The polarization  $\lambda_{A}$ induces naturally a polarization $\lambda_{B}$ on $B$ such that $\lambda_B=\phi^\vee\circ \lambda_A\circ\phi$, since $M^\circ_{\tilde\tau}$ is the dual lattice of $M^\circ_{\tilde\tau^c}$ for every $\tilde\tau\in \Sigma_{E,\infty}$.
When $\gothp$ is of type $\alpha_2$, then $K'_p$ is the Iwahoric subgroup and 
the  level structure at $\gothp$ is equivalent to the data of a collection of submodules $L^\circ_{\tilde \tau} \subset \tilde \calD^\circ_{A,\tilde \tau}$ for $\tilde \tau \in \Sigma_{E, \infty/\gothq}$ which are stable under the action of Frobenius and Verschiebung morphisms and such that  $\tilde \calD^\circ_{A,\tilde \tau} / L^\circ_{\tilde \tau}$ is a one-dimensional vector space over $\overline \FF_p$ for each $\tilde \tau$.
This then gives rise to a level structure at $\gothp$ for $B_x$ by taking $L'^\circ_{\tilde \tau} = L^\circ_{\tilde \tau}$ if $\tilde \tau \neq \tilde \tau_0$, and $L'^\circ_{\tilde \tau_0} = p^{-1}L^\circ_{\tilde \tau_0}$. 
It is clear that other level structures of $A$ transfer to that of $B$ automatically.
This then defines a morphism 
$$
\delta'_{\tau_0,\sharp}: \bfSh_{K'}(G'_{\tilde \ttS})_{k_0} \to \bfSh_{K'}(G'_{\tilde \ttS'})_{k_0}.
$$
One checks easily  that one can reverse the construction to recover $A$ from $B$. So $\delta'_{\tau_0,\sharp}$ is an isomorphism, and there exists a $p$-quasi-isogeny 
$$
\delta'^{\sharp}_{\tau_0}: \bfA'_{\tilde \ttS, k_0} \to (\delta'_{\tau_0,\sharp})^*\bfA'_{\tilde \ttS', k_0},
$$
whose base change to $x$ is $\phi^{-1}: A\ra B$ constructed above. 
It is evident by construction that $(\delta'_{\tau_0,\sharp},\delta'^{\sharp}_{\tau_0})$ is a link morphism  of the prescribed indentation  associated to the trivial link $\ttS\ra \ttS'$.

      \end{example}

   The following proposition will play a crucial role in our application in \cite{tian-xiao3}.
   \begin{prop}\label{P:uniqueness-link-morphism}
   For a given link $\eta: \ttS\ra \ttS'$ with all curves (if any) turning to the right and an integer $n\in \Z$ (with $n=0$ if $\gothp$ is inert in $E$), there exists at most one link morphism  of  indentation degree $n$ from $\bfSh_{K'_p}(G'_{\tilde\ttS})_{k_0}$  to $\bfSh_{K'_p}(G'_{\tilde\ttS'})_{k_0}$ (or from $\bfSh_{K''_p}(G''_{\tilde\ttS})_{k_0}$ to $\bfSh_{K''_p}(G''_{\tilde\ttS'})_{k_0}$ ) associated to $\eta$. 
   \end{prop}
   
   \begin{proof}
   Since $\bfSh_{K'_p}(G'_{\tilde\ttS})_{k_0}$ and $\bfSh_{K''_p}(G''_{\tilde\ttS})_{k_0}$ have canonically isomorphic neutral connected component (and the restrictions of $\bfA'_{\tilde\ttS,k_0}$ and $\bfA''_{\tilde\ttS,k_0}$ to this neutral connected component are also canonically isomorphic), it suffices to treat the case of $\bfSh_{K'_p}(G'_{\tilde\ttS})_{k_0}$.
   Let $(\eta'_{i,\sharp}, \eta'^{ \sharp}_{i})$ for $i=1,2$ be two link morphisms of indentation degree $n$ associated to $\eta$.    By the moduli property of $\bfSh_{K'_p}(G'_{\tilde\ttS'})_{k_0}$, it suffices to show that  the $p$-quasi-isogeny of abelian varieties 
   \[
  \phi: \eta'^{ *}_{1,\sharp}(\bfA'_{\tilde\ttS',k_0}) \xleftarrow{\eta'^{\sharp}_{1}}\bfA'_{\tilde\ttS, k_0}\xra{\eta'^{\sharp}_{2}} \eta'^{ *}_{2,\sharp}(\bfA'_{\tilde\ttS',k_0})
   \]
   is an isomorphism.
   By \cite[Proposition~2.9]{rapoport-zink}, the locus where  $\phi$ is an isomorphism is a closed subscheme of $\bfSh_{K'}(G'_{\tilde\ttS})_{k_0}$. 
    As $\bfSh_{K'_p}(G'_{\tilde\ttS})_{k_0}$ is a reduced variety, $\phi$ is an isomorphism if and only if it is so after base changing to every $\overline \FF_p$-point of $\bfSh_{K'_p}(G'_{\tilde\ttS})_{k_0}$. 
    
 Let $x$ be an $\overline \FF_p$-point of $\bfSh_{K'}(G'_{\tilde\ttS})_{k_0}$, and   put $x_i=\eta'_{i,\sharp}(x)$ for $i=1,2$. 
Consider first the case $\ttS_{\infty}\neq \Sigma_{\infty}$. 
By condition~\ref{D:link-morphism}(3), there exists  an integer $u_{\tilde\tau}$ for each $\tilde\tau\in \Sigma_{E,\infty/\gothq}-\ttS'_{E,\infty/\gothq}$ such that
    \[
    \phi_{x,*}\big(\tcD(\bfA_{\tilde\ttS',x_1})^{\circ}_{\tilde\tau}\big)=p^{u_{\tilde\tau}}\tcD(\bfA_{\tilde\ttS',x_2})^{\circ}_{\tilde\tau}.
    \]
    We claim that $u_{\tilde\tau}$ must be $0$ for all $\tilde\tau$. 
    Note that the cokernel of 
     \[
     F^{n_{\tau}}_{\es,\bfA_{\tilde\ttS',x_i}}: \tcD(\bfA_{\tilde\ttS',x_i})^{\circ}_{\sigma^{-n_{\tau}}\tilde\tau}\ra \tcD(\bfA_{\tilde\ttS',x_i})^{\circ}_{\tilde\tau}
     \]
   has dimension $1$ over $k(x_i)$ for $i=1,2$. Since $\phi_{x,*}$ commutes with $F^{n_{\tau}}_{\es}$, we see that $u_{\tilde\tau}=u_{\sigma^{-n_{\tau}}\tilde\tau}$.
    Consequently, for all $\tilde\tau\in \Sigma_{E,\infty/\gothq}$, $u_{\tilde\tau}$  takes the same value, which we denote by $u$.
   However,   both $\eta'^{\sharp}_{1, \gothq}$ and $\eta'^{\sharp}_{2, \gothq}$ have degree $p^{2n}$ by condition~\ref{D:link-morphism}(4).
    It follows that 
    $$\phi_{x,\gothq}:\eta'^{ *}_{1,\sharp}(\bfA'_{\tilde\ttS',x_1})[\gothq^{\infty}]\ra \eta'^{ *}_{2,\sharp}(\bfA'_{\tilde\ttS',x_2})[\gothq^{\infty}] 
    $$ 
    is a quasi-isogeny of degree $0$, which forces $u$ to be $0$. Hence $\phi_{*}$ is an isomorphism. 
    
    When $\ttS_{\infty}=\Sigma_{\infty}$, we have  similarly  an integer $u_{\tilde\tau}$ for all $\tilde\tau\in \Sigma_{E,\infty}$ such that $\phi_{*}\big(\tcD(\bfA_{\tilde\ttS',x_1})^{\circ}_{\tilde\tau}\big)=p^{u_{\tilde\tau}}\tcD(\bfA_{\tilde\ttS',x_2})^{\circ}_{\tilde\tau}$. 
    Since $\bfSh_{K'}(G'_{\tilde\ttS})_{k_0}$ and $\bfSh_{K'}(G'_{\tilde\ttS'})_{k_0}$ are both zero-dimensional and 
    $$F_{\es,\bfA'_{\tilde\ttS,x_i}}: \tcD(\bfA_{\tilde\ttS',x_i})^{\circ}_{\sigma^{-1}\tilde\tau}\ra \tcD(\bfA_{\tilde\ttS',x_i})_{\tilde\tau}^{\circ}
    $$ 
    is an  isomorphism for all $\tilde\tau\in \Sigma_{E,\infty}$ and $i=1,2$,   the commutativity of $\phi_*$ with essential Frobenii show that $u_{\tilde\tau}=u_{\sigma^{-1}\tilde\tau}$.  The same arguments as above show that $\phi_*$ is an isomorphism.     
\end{proof}
   
   \begin{remark}
  This Proposition does not guarantee the existence of the  link morphism associated to a given link. However, we do expect the link morphisms to exist in general (for links with all curves turning to the right).
   \end{remark}

\subsection{Link morphisms and Hecke operators}
Assume  $\ttS_\infty = \Sigma_\infty$ so that $\gothp$ is of type $\alpha$ or $\alpha^{\sharp}$ and  the band associated to $\ttS$ consists of only plus signs. 
Let $\gothq$ and $\bar\gothq$ be the two primes of $E$ above $\gothp$.
We will focus on the compatibility of link morphisms  with the Hecke operators at $\gothq$, whose definition we recall now.

We have the following description
$$
G''_{\tilde \ttS}(\Q_p) \cong \GL_2(F_{\gothp})\times_{F_{\gothp}^{\times}}(E_{\gothq}^{\times}\times E_{\bar\gothq}^{\times})\xra{\sim} \GL_2(E_{\gothq})\times F_{\gothp}^{\times},
$$
where the last isomorphism is given by $(g, (\lambda_1,\lambda_2))\mapsto(g\lambda_1, \det(g)\lambda_{1}\lambda_2)$ for $g\in \GL_2(F_{\gothp})$ and $\lambda_1\in E_{\gothq}^{\times}$ and $\lambda_2\in E_{\bar\gothq}^{\times}$. 
Then $G'_{\tilde \ttS}(\Q_p)$ is the subgroup $\GL_2(E_{\gothq})\times \Q_p^{\times}$ of $G''_{\tilde \ttS}(\Q_p)$.
Let $\gamma_{\gothq}$ (resp. $\xi_\gothq$)  be the element of $G'_{\tilde \ttS}(\AAA^{\infty})$ which is equal to  
$$(\begin{pmatrix}p^{-1}&0\\0&p^{-1}\end{pmatrix},1)\in \GL_2(E_\gothq)\times\Q_p^{\times}
\quad \textrm{(resp. }(\begin{pmatrix}p^{-1} &0\\ 0&1\end{pmatrix}, 1)\in \GL_2(E_\gothq)\times\Q_p^{\times} \ )$$
at $p$
and is equal to $1$ at other places.
Assume that  $K'\subseteq G'_{\tilde \ttS}(\AAA^{\infty})$  is hyperspecial at $p$, i.e. $K'_p=\GL_{2}(\cO_{E_{\gothq}})\times \ZZ_p^{\times}$. 
We use $S_{\gothq}$ and $T_{\gothq}$ to denote the  Hecke correspondences on $\bfSh_{K'}(G'_{\tilde\ttS})$ defined by $K'\gamma_\gothq K'$ and $K'\xi_\gothq K'$, respectively. 
Explicitly, if $\Iw'_p= \Iw_{\gothq}\times \ZZ_p^{\times}\subseteq G'_{\tilde\ttS}(\Q_p)$ with $\Iw_{\gothq}\subseteq \GL_2(\cO_{E_{\gothq}})$ the standard Iwahoric subgroup reducing to upper triangular matrices when modulo $p$, then  the Hecke correspondence $T_{\gothq}$ is given by the following  diagram:
\begin{equation}\label{E:Hecke-T_q}
\xymatrix{
& \bfSh_{K'^p
\Iw'_p}(G'_{\tilde\ttS}) \ar[rd]^{\pr_2}\ar[ld]_{\pr_1}\\
\bfSh_{K'}(G'_{\tilde\ttS})&& \bfSh_{K'}(G'_{\tilde\ttS}),
}
\end{equation}
where $\pr_1$ is the natural projection, and $\pr_2$ is induced by the right multiplication by $\xi_{\gothq}$.
 Note that $S_{\gothq}$ is  an automorphism of $\bfSh_{K'}(G'_{\tilde\ttS})$, and  there is a natural $p$-quasi-isogeny of universal abelian schemes 
 $$
 \Phi_{S_{\gothq}}: \bfA'_{\tilde \ttS}\ra S_{\gothq}^*\bfA'_{\tilde\ttS}
 $$ 
 compatible with all structures such that the induced quasi-isogeny of $p$-divisible groups  $\Phi_{S_{\gothq}}[\gothq^{\infty}]: \bfA'_{\tilde\ttS}[\gothq^{\infty}]\ra (S_{\gothq}^*\bfA'_{\tilde\ttS})[\gothq^{\infty}]$ is the canonical isogeny with kernel $\bfA'_{\tilde\ttS}[\gothq]$.

 Similarly, the elements $\gamma_{\gothq}$ and $\xi_{\gothq}$ induce Hecke correspondences on $\bfSh_{K''}(G''_{\tilde\ttS})$, which we denote still by $S_\gothq$ and $T_\gothq$ respectively. 

\begin{prop}
Assume that $\ttS_\infty = \Sigma_\infty$.
 Let  $\tilde \ttS_\infty$ and $\tilde \ttS'_\infty$ be two different choices of signatures in Subsection~\ref{S:CM extension}.
Suppose that there exists a link morphism $(\eta'_\sharp,\eta'^{\sharp})$ from  $ \bfSh_{K'}(G'_{\tilde \ttS})_{k_0}$ to $ \bfSh_{K'}(G'_{\tilde \ttS'})_{k_0}$ (of some indentation) associated to the trivial link $\ttS \to \ttS'$, where $K'_p= \GL_2(\cO_{E_{\gothq}})\times \Z_p^{\times}\subseteq G'(\Q_p)$ is hyperspecial. 
Then $(\eta'_{\sharp}, \eta'^{\sharp})$ lifts uniquely to a link morphism $(\eta'_{\sharp, \Iw}, \eta'^{\sharp}_{\Iw})$ on  $\bfSh_{K'^p\Iw'_p}(G'_{\tilde\ttS})_{k_0}$ such that  the following commutative diagrams are Cartesian:
\[
\xymatrix{
\bfSh_{K'}(G'_{\tilde \ttS})_{k_0} \ar[d]^{\eta'_\sharp} & \bfSh_{K'^p\Iw'_p}(G'_{\tilde \ttS})_{k_0} \ar[l]_{\pr_1}\ar[d]^{\eta'_{\sharp,\Iw}} \ar[r]^{\pr_2} & \bfSh_{K'}(G'_{\tilde \ttS})_{k_0}\ar[d]^{\eta'_{\sharp}} & \bfSh_{K'}(G'_{\tilde \ttS})_{k_0}\ar[d]^{\eta'_\sharp} \ar[r]^{S_\gothq} & \bfSh_{K'}(G'_{\tilde \ttS})_{k_0}\ar[d]^{\eta'_\sharp}
\\
\bfSh_{K'}(G'_{\tilde \ttS'})_{k_0} & \bfSh_{K'^p\Iw'_{p}}(G'_{\tilde \ttS'})_{k_0}\ar[l]_{\pr_1} \ar[r]^{\pr_2} & \bfSh_{K'}(G'_{\tilde \ttS'})_{k_0} & \bfSh_{K'}(G'_{\tilde \ttS'})_{k_0}  \ar[r]^{S_\gothq} & \bfSh_{K'}(G'_{\tilde \ttS'})_{k_0},
}
\]
where   the top and the bottom lines of the left diagram are the Hecke correspondences $T_{\gothq}$ defined above.
The same holds for the link morphism $(\eta''_\sharp,\eta''^{\sharp})$: $\bfSh_{K''}(G''_{\tilde \ttS})_{k_0}\to \bfSh_{K''}(G''_{\tilde \ttS'})_{k_0}$.
\end{prop}
\begin{proof}
Note that $S_\gothq$ is in fact an isomorphism of Shimura varieties; so the compatibility with $S_\gothq$-action follows from the uniqueness of link morphism by  Proposition~\ref{P:uniqueness-link-morphism}.

We prove now the existence of the lift  $(\eta'_{\sharp,\Iw},\eta'^{\sharp}_{\Iw})$, whose uniqueness is proved in  \ref{P:uniqueness-link-morphism}.
 Let $x=(A,\iota,\lambda, \alpha_{K'^p})$ be a point of $\bfSh_{K'}(G'_{\tilde\ttS})(\overline \FF_p)$. Put $x'=\eta'_{\sharp}(x)=(A',\iota',\lambda',\alpha'_{K'^p})$.
 By Definition~\ref{D:link-morphism}(3), for any $\tilde\tau\in \Sigma_{E,\infty}$, there exists $t_{\tilde\tau}\in \ZZ$ independent of $x$ such that $\eta'^{\sharp}_{*}(\tcD^{\circ}_{A,\tilde\tau})=p^{t_{\tilde\tau}}\tcD^{\circ}_{A',\tilde\tau}$.
  Fix a $\tilde\tau_0\in \Sigma_{E,\infty/\gothq}$. 
 Giving a point $y$ of $\bfSh_{K'^p\Iw'_{p}}(G'_{\tilde\ttS})_{k_0}$ with $\pr_1(y)=x$ is equivalent to giving a $W(\overline \FF_p)$-submodule $\tilde H^\circ_{\tilde\tau_0}\subseteq \tcD^{\circ}_{A,\tilde\tau_0}$ such that $F_{\es, A}^{g}(\tilde H^\circ_{\tilde\tau_0})=\tilde H^\circ_{\tilde\tau_0}$ and $\tcD^{\circ}_{A,\tilde\tau_0}/\tilde H^\circ_{\tilde\tau_0}$ is one-dimensional over $\overline \FF_p$.
 We put $\tilde H'^\circ_{\tilde\tau_0}=p^{-t_{\tilde\tau_0}}\eta'^{\sharp}_*(\tilde H^\circ_{\tilde\tau_0})\subseteq \tcD^{\circ}_{A',\tilde\tau_0}$.   
 Then one sees easily that the quotient $\tcD^{\circ}_{A',\tilde\tau_0}/\tilde H'^\circ_{\tilde\tau_0}$ is one-dimensional over $\overline \FF_p$ and  $\tilde H'^\circ_{\tilde\tau_0}$ is fixed by $F_{\es, A'}^g$. 
  This gives rise to a point $y'$ of $\bfSh_{K'^p\Iw'_p}(G'_{\tilde\ttS'})_{k_0}$ with $\pr_1(y')=x'$. 
  One defines thus $\eta_{\sharp,\Iw}'(y)=y'$, and the quasi-isogeny $\eta'^{\sharp}_{\Iw}$ as the pull-back of $\eta'^{\sharp}$ via $\pr_1$.
   It is clear by construction that $\eta'_{\sharp}\circ\pr_{1}=\pr_1\circ \eta'_{\sharp,\Iw}$.
   
   It remains to prove that $\eta'_{\sharp}\circ \pr_2=\pr_2\circ\eta'_{\sharp,\Iw}$. Let  $y=(A,\iota, \lambda, \alpha_{K'^p}, \tilde H^\circ_{\tilde\tau_0})\in \bfSh_{K'^p\Iw_p}(G'_{\tilde\ttS})(\overline \FF_p)$ be a point above $x$ as above. 
   We put $\tcD_{A,\gothq}^{\circ}: =\bigoplus_{\tilde\tau\in \Sigma_{E,\infty/\gothq}} \tcD_{A,\tilde\tau}^{\circ}$, and we define   $\tcD_{A,\bar\gothq}^{\circ}$ similarly with $\gothq$ replaced by $\bar\gothq$.
  Then  $\tilde H^\circ_{\gothq}: =\frac{1}{p}\bigoplus_{i=0}^{g-1} F_{\es, A}^{i}(\tilde H^\circ_{\tilde\tau_0})$ is a $W(\overline \FF_p)$-lattice of $\tcD^{\circ}_{A,\gothq}[1/p]$  stable under the action of $F$ and $V$.
     Let  $\tilde H^\circ_{\bar\gothq}=\tilde H^{\circ,\vee}_{\gothq}\subseteq \tcD^{\circ}_{A,\bar\gothq}[1/p]$ denote   the dual lattice of $\tilde H^\circ_{\gothq}$ under the perfect pairing between $\tcD^{\circ}_{A,\gothq}[1/p]$ and $\tcD^{\circ}_{A,\bar\gothq}[1/p]$ induced by $\lambda$. By the theory of  Dieudonn\'e modules,   there exists a unique abelian variety $B$ equipped with $\cO_D$-action $\iota_B$ together with an $\cO_D$-linear $p$-quasi-isogeny $\phi\colon B\ra A$ such that $\phi_*(\tcD_{B}^{\circ})$ is identified with the lattice $\tilde H^\circ_{\gothq}\oplus \tilde H^\circ_{\bar\gothq}$ of $\tcD_{A}^{\circ}[1/p]$.
      Note that $B$ satisfies the signature condition of $\bfSh_{K'}(G_{\tilde\ttS})_{k_0}$.
Since $\tilde H^\circ_{\gothq}$ and $\tilde H^\circ_{\bar\gothq}$ are dual to each other, the quasi-isogeny  $\lambda_B = \phi^{\vee}\circ\lambda\circ\phi\colon B\ra B^{\vee}$ is  a prime-to-$p$ polarization $\lambda_B$ on $B$. 
       We equip moreover $B$ with the $K'^p$-level structure $\beta_{K'^p}$ such that   $\alpha_{K'^p}=\phi\circ \beta_{K'^p}$.
Thus  $z: =(B,\iota_B,\lambda_B, \beta_{K'^p})$ gives rise to an $\overline \FF_p$-point of $\bfSh_{K'}(G'_{\tilde\ttS})_{k_0}$, and we have  $z=\pr_2(y)$ by the moduli interpretation of $\pr_2$.

Let $(B',\iota_{B'},\lambda_{B'},\beta'_{K'^p})$ denote the image $\eta'_{\sharp}(z)$. 
Then  $\tcD^{\circ}_{B',\tilde\tau_0}$ is identified via $(\eta'^{\sharp}_{*})^{-1}$ with the lattice $p^{-t_{\tilde\tau_0}} \tcD^{\circ}_{B,\tilde\tau_0}$ of $\tcD^{\circ}_{B,\tilde\tau_0}[1/p]$, hence with the lattice $p^{-t_{\tilde\tau_0}-1}\tilde H_{\tilde\tau_0}$ of $\tcD^{\circ}_{A,\tilde\tau_0}[1/p]$.
  By our construction of $y'=\eta'_{\sharp, \Iw}(y)$, it is easy to see that, if $\pr_2(y')=(B'',\iota_{B''},\lambda_{B''},\beta''_{K'^p})$, then $\tcD^{\circ}_{B'',\tilde\tau_0}$ can be canonical identified with $\tcD^{\circ}_{B',\tilde\tau_0}$ as lattices of $\tcD^{\circ}_{A,\tilde\tau_0}[1/p]$. 
 Since   other components $\tcD^{\circ}_{B',\tilde\tau}$ or $\tcD^{\circ}_{B'',\tilde\tau}$ for $\tilde\tau\in \Sigma_{E,\infty}$ are determined from $\tcD^{\circ}_{B',\tilde\tau_0}$ and $\tcD^{\circ}_{B'',\tilde\tau_0}$ by the same rules (i.e. stability under the essential Frobenius  and the duality), we see that $B'$ is canonically isomorphic to $B''$, compatible with all structures. This concludes the proof of $\pr_2\circ\eta'_{\sharp,\Iw}=\eta'_{\sharp}\circ\pr_2$.
\end{proof}

For the rest of this paper, we discuss two topics, whose proofs are nested together.
One topic is to understand the behavior of the description of the Goren-Oort strata under the link morphisms; the other  is to understand the restriction of the $\PP^1$-bundle description of the Goren-Oort strata to other Goren-Oort strata.

\begin{prop}
\label{P:restriction of GO strata}
Let $\tau \in \Sigma_{\infty} - \ttS_{\infty}$ be a place such that $\tau^-\neq \tau$, and let $\pi_\tau: \bfSh_{K'}(G'_{\tilde \ttS})_{k_0,\tau} \to \bfSh_{K'}(G'_{\tilde \ttS(\tau)})_{k_0}$ be  the  $\PP^1$-bundle fibration given by  Theorem~\ref{T:main-thm-unitary} for the Goren-Oort stratum defined by the vanishing of the partial Hasse invariant at $\tau$. Let $\ttT$ be a subset of $\Sigma_\infty - \ttS_\infty$ containing $\tau$.

\begin{itemize}
\item[(1)]
If $\tau^+ \notin \ttT$, then we put $\ttT_{\tau} = \ttT\backslash \{\tau, \tau^-\}$ and we have a commutative diagram
\begin{equation}
\label{E:good commutative diagram}
\xymatrix{
\bfSh_{K'}(G'_{\tilde \ttS})_{k_0, \ttT} \ar@{^{(}->}[r] \ar[d]&
\bfSh_{K'}(G'_{\tilde \ttS})_{k_0, \tau} \ar[d]^{\pi_\tau}
\\
\bfSh_{K'}(G'_{\tilde \ttS(\tau)})_{k_0, \ttT_{\tau}} \ar@{^{(}->}[r]&
\bfSh_{K'}(G'_{\tilde \ttS(\tau)})_{k_0}.
}
\end{equation}
If $\tau^- \in \ttT$ the left vertical arrow is an isomorphism.
If $\tau^- \notin \ttT$, this diagram is  Cartesian.

\item[(2)]
If $\tau, \tau^- \in \ttT$ and $\tau^+\neq \tau^-$, then we put $\ttT_{\tau} = \ttT\backslash \{\tau, \tau^-\}$ and $\pi_\tau$ induces a natural isomorphism
\begin{equation}
\label{E:easy projection}
\pi_\tau\colon \bfSh_{K'}(G'_{\tilde \ttS})_{k_0, \ttT} \to \bfSh_{K'}(G'_{\tilde \ttS(\tau)})_{k_0, \ttT_{\tau}}
\end{equation}
\end{itemize}
Moreover, all descriptions above are compatible with the natural quasi-isogenies on universal abelian varieties, and analogous results hold for $\bfSh_{K''}(G''_{\tilde \ttS})_{k_0}$.
\end{prop}

\begin{proof}
The statements for $\bfSh_{K''}(G''_{\tilde \ttS})_{k_0}$ follow from those analogs for $\bfSh_{K'}(G'_{\tilde\ttS})_{k_0}$ by \ref{S:transfer math obj} (or in this case more explicitly by \ref{S:abel var in unitary case}). Thus, we will  just prove the proposition  for $\bfSh_{K'}(G'_{\tilde\ttS})_{k_0}$.

(1) If $\tau^+ \notin \ttT$, the prime $\gothp$ must be of type $\alpha1$ or $\beta1$.
By the proof of Proposition~\ref{P:Y_S=X_S}, the natural quasi-isogeny $\phi: \pi_\tau^*(\bfA'_{\tilde \ttS(\tau), k_0}) \to \bfA'_{\tilde \ttS,k_0}|_{ \bfSh_{K'}(G'_{\tilde\ttS})_{k_0, \tau}}$ induces an isomorphism on the (reduced) differential forms at $\tilde \tau'$ for all $\tilde \tau'\in \Sigma_{E,\infty}$ \emph{not} lifting $\tau, \sigma^{-1}\tau, \dots, \sigma^{-n_\tau}\tau = \tau^-$.
So $\pi_\tau$ induces a Cartesian square
\[
\xymatrix{
\bfSh_{K'}(G'_{\tilde \ttS})_{k_0, \ttT_{\tau}\cup\{\tau\}} \ar@{^{(}->}[r] \ar[d]^{\pi_{\tau}}&
\bfSh_{K'}(G'_{\tilde \ttS})_{k_0, \tau} \ar[d]^{\pi_\tau}
\\
\bfSh_{K'}(G'_{\tilde \ttS(\tau)})_{k_0, \ttT_{\tau}} \ar@{^{(}->}[r]&
\bfSh_{K'}(G'_{\tilde \ttS(\tau)})_{k_0}
}
\]
This already proves (1) in case $\tau^- \notin \ttT$. 
Suppose now $\tau^-\in \ttT$. 
 By Proposition~\ref{P:Y_T=Z_T}, $\bfSh_{K'}(G'_{\tilde \ttS})_{k_0, \ttT_{\tau}\cup\{\tau\}}$ is the moduli space of tuples $(B,\iota_B, \lambda_B, \beta_{K'_{\ttT_{\tau}}}; J^{\circ}_{\tilde\tau^-})$, where 
 \begin{itemize}
 \item $(B,\iota_B, \lambda_B, \beta_{K'_{\ttT_{\tau}}})$ is a point  of $\bfSh_{K'}(G'_{\tilde \ttS(\tau)})_{k_0,\ttT_{\tau}}$ with values in a scheme $S$ over $k_0$;
 \item   $J^{\circ}_{\tilde\tau^-}$ is a sub-bundle of  $H^{\dR}_1(B/S)^{\circ}_{\tilde\tau^-}$ of rank 1 (here, $\tilde\tau^-\in \Sigma_{E,\infty}$ is the specific  lift of $\tau^-$ defined  in Subsection~\ref{S:tilde IT}).
 \end{itemize}
  Then the closed subscheme $\bfSh_{K'}(G'_{\tilde \ttS})_{k_0, \ttT}$ of $\bfSh_{K'}(G'_{\tilde \ttS})_{k_0, \ttT_{\tau}\cup\{\tau\}}$ is defined by the condition: 
 $$
 J^\circ_{\tilde\tau^-} =F_{B,\es}^{n_{\tau^-}}\big(H_1^\dR(B^{(p^{n_{\tau^-}})} / S)^{\circ}_{\sigma^{-n_{\tau^-}}\tilde\tau^{-}}\big).
 $$  
This shows that the restriction of $\pi_{\tau}:\bfSh_{K'}(G'_{\tilde \ttS})_{k_0, \ttT_{\tau}\cup\{\tau\}}\ra \bfSh_{K'}(G'_{\tilde \ttS})_{k_0, \ttT_{\tau}}$ to  $\bfSh_{K'}(G'_{\tilde \ttS})_{k_0, \ttT}$ is an isomorphism.

(2)
When $\tau^+\notin \ttT$, this was proved in (1). 
Assume now $\tau^+ \in \ttT$. 
To complete the proof, it suffices to prove that, for a $k_0$-scheme $S$,  the $\tau^+$-th partial Hasse invariant vanishes at an $S$-point  $x=(A, \iota_A, \lambda_A, \alpha_{K'})\in \bfSh_{K'}(G'_{\tilde \ttS})_{k_0,\{\tau\}}$ if and only if it vanishes at  $\pi_{\tau}(x)=(B, \iota_B, \lambda_B, \beta_{K'})\in \bfSh_{K'}(G'_{\tilde\ttS(\tau)})_{k_0}$.
Let $\tilde\tau$ be the lift of $\tau$ contained in $\tilde\Delta(\tau)^+$  (See Subsection~\ref{S:Delta-pm} for the notation $\tilde\Delta(\tau)^+$). Put $\tilde\tau^+=\sigma^{n_{\tau^+}}\tilde\tau$ and $\tilde\tau^-=\sigma^{-n_{\tau}}\tilde\tau$.
By Lemma~\ref{Lemma:partial-Hasse}, it suffices   to show that 
\begin{align}
\label{E:equivalent condition for triple Hasse invariant}
& 
F_{A,\es}^{n_{\tau^+}} \big( H_1^\dR(A/S)^{\circ, (p^{n_{\tau^+}})}_{\tilde \tau} \big) = \omega^\circ_{A^\vee/S, \tilde \tau^+},\\
\nonumber
 \Leftrightarrow \ & F_{B,\es }^{n_{\tau^+}+ n_\tau + n_{\tau^-}}\big(
H_1^\dR(B/S)^{\circ, (p^{n_{\tau^+}+ n_\tau + n_{\tau^-}})}_{\sigma^{-n_{\tau^-}}\tilde \tau^-}
\big) =  \omega^\circ_{B^\vee/S, \tilde \tau^+}.
\end{align}
But this follows from the following three facts.
\begin{itemize}
\item[(a)]
By the  definition of essential Frobenius in \ref{N:essential frobenius and verschiebung}, one deduces a commutative diagram 
\[
\xymatrix{H_1^\dR(A/S)_{\tilde \tau}^{\circ,(p^{n_{\tau^+}})}\ar[rr]^{F_{A,\es}^{n_{\tau^+}}}\ar[d]_{\phi_{*,\tilde\tau}} && H^{\dR}_1(A/S)_{\tilde\tau^+}^\circ \ar[d]^{\phi_{*,\tilde\tau^+}}_{\cong}\\
H_1^\dR(B/S)_{\tilde \tau}^{\circ,(p^{n_{\tau^+}})}\ar[rr]^{F_{B,\es}^{n_{\tau^+}}} && H^{\dR}_{1}(B/S)^{\circ}_{\tilde\tau^+},
}
\]
\item[(b)]
It follows from  condition (v) of the moduli description  in Subsection~\ref{S:moduli-Y_S} that 
\[
\phi_{*, \tilde\tau}(H_1^\dR(A/S)_{\tilde \tau}^\circ) = F_{ B,\es}^{n_{\tau^-}+ n_\tau} \big(H_1^\dR(B/S)^{\circ, (p^{ n_\tau + n_{\tau^-}})}_{\sigma^{-n_{\tau^-}}\tilde \tau^-}
\big).
\]
\item[(c)]
The condition  $\tau^-\neq \tau^+$ implies that the quasi-isogeny  $\phi:A\ra B$ induces an isomorphism $\phi_{*, \tilde\tau^+}:H^{\dR}_{1}(A/S)^{\circ}_{\tilde \tau^+}\cong H^{\dR}_1(B/S)^{\circ}_{\tilde \tau^+}$ preserving the Hodge filtrations, in particular identifying the submodules $\phi_{*, \tilde \tau}(\omega^\circ_{A^\vee/S, \tilde \tau^+}) = \omega^\circ_{B^\vee/S, \tilde \tau^+}$.\qedhere
\end{itemize}
\end{proof}

\subsection{Compatibility of link morphisms and the description of Goren-Oort stata}

We first recall that, although the subset $\ttS(\tau)$ is completely determined by $\ttS$ and $\tau$ as in Subsection~\ref{S:quaternion-data-T}, the lift $\tilde \ttS(\tau)_{\infty}$, which consists of all  $\tilde\tau'\in \Sigma_{E,\infty}$ with signature $s_{\tilde\tau'}=0$ (see Subsection~\ref{S:CM extension}), depends on an auxiliary choice in Subsection~\ref{S:tilde S(T)}:  a lift $\tilde\tau$ of $\tau$ to be contained in $\tilde\ttS(\tau)_{\infty}$. 
We assume that $\#(\Sigma_\infty-\ttS_\infty) \geq 2$. If $\gothp$ splits as $\gothq\bar\gothq$ in $E$ for a fixed place $\gothq$, then the $\tilde\tau$ contained  in $\tilde\ttS(\tau)_{\infty}$ is always  chosen to be the one in $\Sigma_{E,\infty/\gothq}$.
  If $\gothp$ is inert in $E$, then there are two possible choices: $\tilde\tau$ and its conjugate $\tilde\tau^c$ for a fixed lift $\tilde\tau$ of $\tau$.  In the latter case, we denote by $\tilde\tau$ the lift of $\tau$ contained in $\tilde\ttS(\tau)_{\infty}$, by $\tilde\ttS(\tau)'=(\ttS(\tau),\ttS(\tau)'_{\infty})$  the  lift of $\ttS(\tau)$ such that $\tilde\tau^c\in \tilde\ttS(\tau)'_{\infty}$, and let  $\pi_{\tau}'\colon \bfSh_{K'}(G'_{\tilde \ttS})_{k_0,\tau} \to \bfSh_{K'}(G'_{\tilde \ttS(\tau)'})_{k_0}$ be the corresponding $\PP^1$-bundle. The following proposition says that $\pi_{\tau}$ and $\pi_{\tau}'$ differ from each other by a link isomorphism.

   \begin{prop}\label{P:ambiguity-signature}
   Assume that $\gothp$ is inert in $E$.
Then there exists a link isomorphism 
$$(\eta'_{ \tilde\ttS(\tau), \tilde \ttS(\tau)', \sharp},\eta'^{\sharp}_{\tilde\ttS(\tau),\tilde\ttS(\tau)'})\colon  \bfSh_{K'}(G'_{\tilde \ttS(\tau)})_{k_0}\xra{\sim} \bfSh_{K'}(G'_{\tilde \ttS(\tau)'})_{k_0}
$$
 (of indentation degree $0$)  associated to the identity link $\eta: \ttS(\tau)\ra \ttS(\tau)$ such that the  diagram 
\[
\xymatrix{
& \bfSh_{K'}(G'_{\tilde \ttS})_{k_0, \tau}
\ar[dl]_{\pi_\tau}
\ar[dr]^{\pi'_\tau}
\\
\bfSh_{K'}(G'_{\tilde \ttS(\tau)})_{k_0} \ar[rr]_{\eta_{\tilde \ttS(\tau), \tilde \ttS(\tau)', \sharp}}^{\cong}
&& 
\bfSh_{K'}(G'_{\tilde \ttS(\tau)'})_{k_0}
}
\]
commutes.
Similar  statements hold for $\bfSh_{K''}(G''_{\tilde \ttS(\tau)})_{k_0}$ and $\bfSh_{K''}(G''_{\tilde\ttS(\tau)'})_{k_0}$.
   \end{prop}
   \begin{proof}
Consider  the closed subvariety $\bfSh_{K'}(G'_{\tilde\ttS})_{k_0, \{\tau, \tau^-\}}$ of $\bfSh_{K'}(G'_{\tilde\ttS})_{k_0, \tau}$. 
Then by  Proposition~\ref{P:restriction of GO strata} (2) (note that $\gothp$ being inert in $E$ implies that $\tau^+\neq \tau^-$), $\pi|_{\bfSh_{K'}(G'_{\tilde\ttS})_{k_0, \{\tau, \tau^-\}}}$ is an isomorphism. We put 
$$\eta'_{\tilde \ttS(\tau),\tilde \ttS(\tau)',\sharp}\colon =\pi'_{\tau}\circ(\pi|_{\bfSh_{K'}(G'_{\tilde\ttS})_{k_0, \{\tau, \tau^-\}}})^{-1} $$
    and define $\eta'^{\sharp}_{\tilde \ttS(\tau),\tilde \ttS(\tau)'}$ as the pull-back via $(\pi_{\tau}|_{\bfSh_{K'}(G'_{\tilde\ttS})_{k_0, \{\tau^-, \tau\}}})^{-1}$ of the quasi-isogeny
    \[
    \pi_{\tau}^*\bfA'_{\tilde \ttS(\tau),k_0}\ra \bfA'_{\tilde\ttS,k_0}|_{\bfSh_{K'}(G'_{\tilde\ttS})_{k_0,\{\tau,\tau^-\}}}\ra\pi'^*_{\tau}\bfA_{\tilde\ttS(\tau)',k_0},
    \] 
    where the two quasi-isogenies are given by Theorem~\ref{T:main-thm-unitary}(2).
   By Proposition~\ref{P:restriction of GO strata}(2) again, $\pi'_{\tau}|_{\bfSh_{K'}(G'_{\tilde\ttS})_{k_0, \{\tau,\tau^-\}}}$ is an isomorphism, hence so is $\eta'_{\tilde \ttS(\tau),\tilde \ttS(\tau)',\sharp}$.
    It remains to show that $(\eta'_{\tilde \ttS(\tau),\tilde \ttS(\tau)',\sharp}, \eta'^{\sharp}_{\tilde \ttS(\tau),\tilde \ttS(\tau)'})$ is  a link morphism associated to the identity link on $\ttS$. 
    Let $x$ be a geometric point of  $\bfSh_{K'}(G'_{\tilde\ttS(\tau)})_{k_0}$, and $x'=\eta'_{\tilde \ttS(\tau),\tilde \ttS(\tau)',\sharp}(x)$. 
    By construction, it is easy to see that the quasi-isogeny $\eta'^{\sharp}_{\tilde \ttS(\tau),\tilde \ttS(\tau)'}$ induces an isomorphism $\tcD(\bfA'_{\tilde\ttS(\tau), k_0,x})^{\circ}_{\tilde\tau'}\xra{\sim} \tcD(\bfA'_{\tilde\ttS(\tau)', k_0,x'})^{\circ}_{\tilde\tau'}$ for $\tilde\tau'\neq \sigma^a\tilde\tau, \sigma^a\tilde\tau^c$ with $a = 0, \dots, n_\tau-1$. In the exceptional cases, we have 
    \[
    \eta'^{\sharp}_{\tilde \ttS(\tau),\tilde \ttS(\tau)'}(\tcD(\bfA'_{\tilde\ttS(\tau), k_0,x})^{\circ}_{\tilde\tau'})=\begin{cases}
    p\tcD(\bfA'_{\tilde\ttS(\tau)', k_0,x'})^{\circ}_{\tilde\tau'} &\text{for }\tilde\tau'=\sigma^a\tilde\tau \textrm{ for }a = 0, \dots, n_\tau-1;\\
   \frac{1}{p} \tcD(\bfA'_{\tilde\ttS(\tau)', k_0,x'})^{\circ}_{\tilde\tau'} &\text{for }\tilde\tau'=\sigma^a\tilde\tau^c\textrm{ for }a = 0, \dots, n_\tau-1.
    \end{cases} 
    \]
    Hence, $(\eta'_{\tilde \ttS(\tau),\tilde \ttS(\tau)',\sharp}, \eta'^{\sharp}_{\tilde \ttS(\tau),\tilde \ttS(\tau)'})$ verifies Definition~\ref{D:link-morphism}.
  \end{proof}
  
  The following Lemma will be needed in the proof of the main result of this section. 
  
  \begin{lemma}\label{L:non-vanishing-chi}
Assume that $\ttS_\infty \neq \emptyset$.
Let $\chi(\bfSh_{K'}(G'_{\tilde\ttS})_{\overline \FF_p}): =\sum_{i=0}^{+\infty}(-1)^i\dim H^{i}_\et(\bfSh_{K'}(G'_{\tilde\ttS})_{\overline \FF_p},  \Q_{\ell})$ denote the Euler-Poincar\'e characteristic of $\bfSh_{K'}(G'_{\tilde\ttS})_{\overline \FF_p}$ for some fixed prime $\ell\neq p$.
 Then we have $\chi(\bfSh_{K'}(G'_{\tilde\ttS})_{\overline \FF_p})\neq 0$.
\end{lemma}
   \begin{proof}
The assumption $\ttS_\infty \neq \emptyset$ implies that the Shimura variety $\bfSh_{K'}(G'_{\tilde \ttS})$ is proper.
Consider the integral model $\bfSh_{K'}(G'_{\tilde\ttS})$ over $\cO_{\tilde\wp}$, and choose an embedding $\cO_{\tilde\wp}\hra \CC$. By the proper base change theorem and the standard comparison theorems, we have 
   $$
\chi(\bfSh_{K'}(G'_{\tilde\ttS})_{\overline \FF_p})=\chi(\bfSh_{K'}(G'_{\tilde\ttS})_{\CC}): =\sum_{i=0}^{\infty} (-1)^{i}\dim H^{i}_\mathrm{sing}(\bfSh_{K'}(G'_{\tilde\ttS})_{\CC}, \CC),
   $$
where $H^{i}_\mathrm{sing}(\bfSh_{K'}(G'_{\tilde\ttS})_{\CC}, \CC)$ denotes the  singular cohomology of $\bfSh_{K'}(G'_{\tilde\ttS})_{\CC}$ for the usual complex topology.
For each $\tilde\tau$ lifting an element of $\Sigma_{\infty}-\ttS_{\infty}$, put $\omega_{\tilde\tau}^{\circ}=\omega_{\bfA'_{\tilde\ttS,\CC},\tilde\tau}^{\circ}$ and   $\gotht_{\tilde\tau}^{\circ}=\Lie(\bfA'_{\tilde\ttS,\CC})^\circ_{\tilde\tau}=\omega^{\circ,\vee}_{\tilde\tau}$ to simplify the notation.
     They are line bundles over $\bfSh_{K'}(G'_{\tilde\ttS})_{\CC}$.
   We have a Hodge filtration
   \[
   0\ra \omega^{\circ}_{\tilde\tau^c}\ra H^{\dR}_1(\bfA'_{\tilde\ttS,\CC}/\bfSh_{K'}(G'_{\tilde\ttS})_{\CC})^{\circ}_{\tilde\tau}\ra \gotht_{\tilde\tau}^{\circ}\ra 0.
   \]
Note that $H^{\dR}_1(\bfA'_{\tilde\ttS,\CC}/\bfSh_{K'}(G'_{\tilde\ttS})_{\CC})_{\tilde\tau}^{\circ}$ is equipped with the integrable  Gauss-Manin connection so that its   Chern classes are trivial by classical Chern-Weil theory. 
  One obtains thus
   \[
  \begin{cases} c_1(\omega_{\tilde\tau^c}^{\circ})c_1(\gotht_{\tilde\tau}^{\circ})=c_2(H^{\dR}_1(\bfA'_{\tilde\ttS,\CC}/\bfSh_{K'}(G'_{\tilde\ttS})_{\CC})^{\circ}_{\tilde\tau})=0,\\
  c_1(\omega_{\tilde\tau^c}^{\circ})+c_1(\gotht^{\circ}_{\tilde\tau})=c_1(H^{\dR}_1(\bfA'_{\tilde\ttS,\CC}/\bfSh_{K'}(G'_{\tilde\ttS})_{\CC})^{\circ}_{\tilde\tau})=0,
  \end{cases}
  \Longrightarrow 
  \begin{cases}
  c_1(\omega^{\circ}_{\tilde\tau})^2=0,\\
   c_1(\omega^{\circ}_{\tilde\tau})=c_1(\omega_{\tilde\tau^c}^{\circ}),
  \end{cases}
   \]
   where $c_i(\calE)\in H^{2i}_\mathrm{sing}(\bfSh_{K'}(G'_{\tilde\ttS})_{\CC},\CC)$ denotes the $i$-th Chern class of a vector bundle $\calE$.
Let $\calT$ denote the tangent bundle of $\bfSh_{K'}(G'_{\tilde\ttS})_\CC$, and put $\det(\omega): =\bigotimes_{\tilde\tau\in \Sigma_{E,\infty}-\ttS_{E, \infty}}\omega^{\circ}_{\tilde\tau}$.
By Corollary~\ref{C:deformation}, we get 
\[
c_d(\calT)=\prod_{\tau\in \Sigma_{\infty}-\ttS_{\infty}}(-2c_1(\omega_{\tilde\tau}^{\circ})),
\]
where $\tilde\tau\in \Sigma_{E,\infty}$ is an arbitrary lift of $\tau$, and $d=\#\Sigma_{\infty}-\#\ttS_{\infty}$ is the dimension of $\bfSh_{K'}(G'_{\tilde\ttS})_{\CC}$. Note that $c_1(\omega^{\circ}_{\tilde\tau})^2=0$ and $c_1(\omega^\circ_{\tilde \tau}) = c_1(\omega^\circ_{\tilde \tau^c})$ implies that 
\[
(c_1(\det(\omega))^d = 2^d d! \prod_{\tau\in \Sigma_{\infty}-\ttS_{\infty}}(c_1(\omega_{\tilde\tau}^{\circ})) = (-1)^{d}{d!}
c_d(\calT).
\]
   It is well known that $\det(\omega)$ is ample (see \cite{lan}, for instance),  hence it follows that $c_d(\calT)\neq 0$.
   On the other hand, there exists a canonical isomorphism 
  $$\Tr\colon  H^{2d}_\mathrm{sing}(\bfSh_{K'}(G'_{\tilde\ttS})_{\CC},\CC)\xra{\sim} \CC,$$
   which sends the cycle class of a point to $1$. 
   The Lemma follows immediately from the non-vanishing of $c_d(\calT)$ and the well-known fact that  $\Tr(c_d(\calT))=\chi(\bfSh_{K'}(G'_{\tilde\ttS})_{\CC})$. 
   \end{proof}

We state now the main result of this section, which will play a crucial role in our application to Tate cycles in \cite{tian-xiao3}. 

\begin{theorem}
\label{T:link and Hecke operator}
Keep the same notation as in Proposition~\ref{P:restriction of GO strata}, that is, let $\tau \in \Sigma_{\infty} - \ttS_{\infty}$ be a place such that $\tau^-$ is different from $\tau$ and let $\ttT$ be a subset of $\Sigma_\infty - \ttS_\infty$ containing $\tau$.

\begin{itemize}
\item[(1)]
If $\tau^- \notin \ttT$ and $\tau,\tau^+ \in \ttT$, 
we put $\ttT_{\tau^+} = \ttT \backslash \{\tau, \tau^+\}$.
 Let $\eta = \eta_{\tau^-\ra \tau^+}: \ttS(\tau^+)\ra \ttS(\tau)$ be the link  given by straight lines except sending $\tau^-$ to $\tau^+$ (to the right) with displacement $v(\eta)=n_{\tau}+n_{\tau^+}$:
\[
\psset{unit=0.3}
 \begin{pspicture}(-1.2,-0.4)(16,2.4)
\psset{linecolor=red}
\psset{linewidth=1pt}
\psline{-}(0,0)(0,2)
\psline{-}(14.4,0)(14.4,2)
\psbezier(4.8,2)(4.8,0)(9.6,2)(9.6,0)
\psset{linecolor=black}
\psdots(0,0)(0,2)(4.8,2)(9.6,0)(14.4,0)(14.4,2)
\psdots[dotstyle=+](1,0)(-1,0)(-1,2)(1,2)(3.8,0)
(3.8,2)(4.8,0)(5.8,0)(5.8,2)(8.6,0)(8.6,2)(10.6,0)(9.6,2)(10.6,2)(13.4,0)(13.4,2)(15.4,0)(15.4,2)
\psset{linewidth=.1pt}
\psdots(1.7,0)(1.7,2)(2.4,0)(2.4,2)(3.1,0)(3.1,2)(6.5,0)(7.2,0)(7.9,0)(6.5,2)(7.2,2)(7.9,2)(11.3,0)(12,0)(12.7,0)(11.3,2)(12,2)(12.7,2).
\end{pspicture}
\]
\begin{enumerate}
\item[(a)] We consider the composition  
\[
\eta'_{ \sharp}:
\bfSh_{K'}(G'_{\tilde \ttS(\tau^+)}) \xleftarrow[\cong]{\pi_{\tau^+}} \bfSh_{K'}(G'_{\tilde \ttS})_{k_0, \{\tau, \tau^+\}} \xrightarrow{i_{\tau^+}}
\bfSh_{K'}(G'_{\tilde \ttS})_{k_0, \tau}
\xrightarrow{\pi_\tau} \bfSh_{K'}(G'_{\tilde \ttS(\tau)})_{k_0},
\] and let $\eta'^{\sharp}$ denote  the natural quasi-isogeny of abelian varieties on 
 $ \bfSh_{K'}(G'_{\tilde \ttS(\tau^+)})_{k_0}$ given by  (the pull-back via $(\pi_{\tau^+|_{\bfSh_{K'}(G'_{\tilde \ttS})_{k_0, \{\tau, \tau^+\}}}})^{-1,*}$ of) 
  \[
(\pi_{\tau^+}|_{\bfSh_{K'}(G'_{\tilde \ttS})_{k_0, \{\tau, \tau^+\}}})^*\bfA'_{\tilde\ttS(\tau^+),k_0} \leftarrow i_{\tau^+}^*(\bfA'_{\tilde\ttS,k_0}|_{\bfSh_{K'}(G'_{\tilde\ttS})_{k_0, \tau}})\ra i_{\tau^+}^*\pi_{\tau}^* (\bfA'_{\tilde\ttS(\tau),k_0}).
 \]
Then $(\eta'_{ \sharp}, \eta'^{\sharp})$ is   the link morphism  associated to the link  $\eta$ of indentation degree $n=n_{\tau^+}-n_{\tau}$ if $\gothp$ splits in $E/F$ and $n=0$ if $\gothp$ is inert in $E/F$. Moreover, the following diagram 
\begin{equation}
\label{E:diagram involving link}
\xymatrix@C=55pt{
\bfSh_{K'}(G'_{\tilde \ttS})_{k_0, \ttT} \ar@{^{(}->}[r] \ar[d]_\cong
&
\bfSh_{K'}(G'_{\tilde \ttS})_{k_0, \{\tau, \tau^+\}} \ar@{^{(}->}[r]^-{i_{\tau^+}} \ar[d]^{\pi_{\tau^+}|_{\bfSh_{K'}(G'_{\tilde \ttS})_{k_0, \{\tau, \tau^+\}}}}_\cong
&
\bfSh_{K'}(G'_{\tilde \ttS})_{k_0, \tau} \ar[d]^{\pi_\tau}
\\
\bfSh_{K'}(G'_{\tilde \ttS(\tau^+)})_{k_0, \ttT_{\tau^+}} \ar@{^{(}->}[r]&
\bfSh_{K'}(G'_{\tilde \ttS(\tau^+)})_{k_0} \ar[r]^{\eta'_{ \sharp}}&
\bfSh_{K'}(G'_{\tilde \ttS(\tau)})_{k_0},
}
\end{equation}
is commutative, where the two vertical isomorphisms are given by Proposition~\ref{P:restriction of GO strata}(2).

\item[(b)] For $\tilde\tau'\in \Sigma_{E,\infty} \backslash \ttS_{E, \infty}$,  the quasi-isogeny $\eta'^{\sharp}:\bfA'_{\tilde\ttS(\tau^+), k_0}\ra \eta'^{*}_{\sharp} \bfA'_{\tilde\ttS(\tau), k_0}$ induces a canonical isomorphism 
\[
\eta'^*_{\sharp}(\Lie(\bfA'_{\tilde\ttS(\tau), k_0})_{\tilde\tau'}^{\circ}\cong 
\begin{cases}
\Lie(\bfA'_{\tilde\ttS(\tau^+), k_0})^{\circ, (p^{v(\eta)})}_{\sigma^{-v(\eta)}\tilde\tau'} &\text{if $\tilde\tau'$ is a lifting of  $\tau^+$,}\\
\Lie(\bfA'_{\tilde\ttS(\tau^+), k_0})^{\circ}_{\tau'} &\text{otherwise.}
\end{cases}
\]

\item[(c)] The morphism  $\eta'_{\sharp}$ is finite flat of degree $p^{v(\eta)} = p^{n_\tau+n_{\tau^+}}$.

\end{enumerate} 
 \item[(2)]
Assume $\Sigma_{\infty} - \ttS_{\infty} = \{\tau, \tau^-\}= \ttT$ (so that $\tau^+ = \tau^-$ and $\gothp$ is of type $\alpha2$ for $\ttT$). 
 Then
there exists a link morphism  $(\eta_{\sharp},\eta^{\sharp}): \bfSh_{K'}(G'_{\tilde \ttS(\tau^-)})_{k_0}\ra \bfSh_{K'}(G'_{\tilde \ttS(\tau)})_{k_0}$ of indentation degree $2(g-n_{\tau})=2n_{\tau^-}$ associated with the trivial link $\eta:\ttS(\tau^-)\ra \ttS(\tau)$ such that  
the  diagram
\[
\xymatrix{
& \bfSh_{K'}(G'_{\tilde \ttS})_{k_0, \{\tau, \tau^-\}} \ar[dl]_{\pi_\tau} \ar[dr]^{\pi_{\tau^-}}
\\
\bfSh_{K'}(G'_{\tilde \ttS(\tau)})_{k_0} && \bfSh_{K'}(G'_{\tilde \ttS(\tau^-)})_{k_0}\ar[rr]^{\eta_\sharp}&&
\bfSh_{K'}(G'_{\tilde \ttS(\tau)})_{k_0}.
}
\]
coincides with the Hecke correspondence $T_{\gothq}$ \eqref{E:Hecke-T_q} if we identify $\bfSh_{K'}(G'_{\tilde \ttS})_{k_0, \{\tau, \tau^-\}}$  with $\bfSh_{K'^p\Iw'_p}(G'_{\tilde\ttS(\tau)})_{k_0}$ via the isomorphism given by  Theorem~\ref{T:main-thm-unitary}.
\end{itemize}
All descriptions above are compatible with the natural tame Hecke operator actions, and  similar results  apply to $\bfSh_{K''}(G''_{\tilde \ttS})_{k_0}$.
\end{theorem}

\begin{proof}
The statements for $\bfSh_{K''}(G''_{\tilde \ttS})_{k_0}$ follow from those analogs for $\bfSh_{K'}(G'_{\tilde\ttS})_{k_0}$ by \ref{S:transfer math obj} (or in this case more explicitly by \ref{S:abel var in unitary case}). Thus, we will  just prove the theorem  for $\bfSh_{K'}(G'_{\tilde\ttS})_{k_0}$.

(1)(a) The commutativity of the left square of \eqref{E:diagram involving link} is tautological, and the commutativity of the right square was proved in Proposition~\ref{P:restriction of GO strata}(1). 
 It remains to show that  $\pi_\tau \circ (\pi_{\tau^+}|_{\bfSh_{K'}(G'_{\tilde \ttS})_{k_0, \{\tau, \tau^+\}}})^{-1}$ is  the link morphism  $\eta'_{\sharp}$ on $\bfSh_{K'}(G'_{\tilde\ttS(\tau^+)})_{k_0}$ associated to the link $\eta=\eta_{\tau^-\ra \tau^+}$.

Let $\tilde\tau\in \Sigma_{E,\infty}$ (resp. $\tilde\tau^+$) denote the lift of $\tau$ (resp. $\tau^+$) contained in $\tilde\ttS(\tau)_{\infty}$ (resp. $\tilde\ttS(\tau^+)_{\infty}$).
By Subsection~\ref{S:tilde S(T)}, we have $\tilde\tau=\sigma^{-n_{\tau^+}}\tilde\tau^+$ if $\gothp$ is splits in $E$. 
If $\gothp$ is inert in $E$, it is also harmless to assume $\tilde\tau=\sigma^{-n_{\tau^+}}\tilde\tau^+$
  in view of Propositions~\ref{P:ambiguity-signature} and \ref{P:compatibility of link and GO}.
   Put $\tilde \tau^- =\sigma^{-n_\tau} \tilde\tau$. Let  $y=(B, \iota_{B}, \lambda_{B}, \beta'_{K'})$ be an $S$-point of  $\bfSh_{K'}(G'_{\tilde \ttS(\tau^+)})_{k_0}$ for a locally noetherian  $k_0$-scheme $S$.
 Then the preimage of $y$ under  $ \pi_{\tau^+}|_{\bfSh_{K'}(G'_{\tilde \ttS})_{k_0, \{\tau, \tau^+\}}}$ is given by  $x=(A, \iota_A, \lambda_A, \alpha_{K'})\in \bfSh_{K'}(G'_{\tilde\ttS})_{k_0, \{\tau,\tau^+\}}$, for which  there exists   a quasi-isogeny $\phi: B \to A$ such that $\phi_{*, \tilde \tau'}: H^{\dR}_1(B/S)^{\circ}_{\tilde\tau'}\ra H^{\dR}_1(A/S)^{\circ}_{\tilde\tau'}$ is  a well-defined  isomorphism for all $\tilde \tau' \in \Sigma_{E, \infty}$ except for $\tilde \tau'=\sigma^a\tilde\tau$ or $ \sigma^{a}\tilde\tau^c$ with $a = 1, \dots, n_{\tau^+}$. Note that $\sigma^a \tilde{\tau} \in \tilde \Delta(\tau^+)^-$ and $\sigma^a \tilde{\tau}^c \in \tilde \Delta(\tau^+)^+$. So in the exceptional cases, $\phi_{*,\sigma^a\tilde\tau}$ and $(\phi^{-1})_{*, \sigma^a\tilde \tau^c}$ are  well defined, where $\phi^{-1}\colon A\ra B$ denotes the  quasi-isogeny inverse to $\phi$,  and  we have (by the proof of Proposition~\ref{P:Y_T=Z_T})
\begin{align}
\label{E:Ker = F = Lie}
&\quad \Ker(\phi_{*, \sigma^a\tilde \tau}) =  F_{B,\es}^{n_\tau + a} \big( H^\dR_1(B/S)^{\circ, (p^{n_\tau+a})}_{\tilde \tau^-} \big)\cong \Lie(B/S)^{\circ, (p^{n_{\tau}+a})}_{\tilde\tau^-},\text{ and }\\
\label{E:Im = F = Lie}
&  \im ((\phi^{-1})_{*, \sigma^a\tilde \tau^c}) = F_{B,\es}^{n_\tau + a} \big( H^\dR_1(B/S)^{\circ, (p^{n_\tau+a})}_{\tilde \tau^{-,c}} \big)\cong \Lie(B/S)^{\circ,(p^{n_{\tau}+a})}_{\tilde\tau^{-,c}}.
\end{align}
If  $\pi_\tau$ sends $x=(A, \iota_A, \lambda_A, \alpha_{K'})$ to  $z=(B', \iota_{B'}, \lambda_{B'}, \beta_{K'}) \in \bfSh_{K'}(G'_{\tilde \ttS(\tau)})(S)$, there is a quasi-isogeny $\psi: A \to B'$ such that $\psi_{*, \tilde \tau'}$ is an isomorphism for all $\tilde \tau' \in\Sigma_{E, \infty}$ except for $\tilde \tau'=\sigma^b\tilde\tau^-$ or $\sigma^b\tilde\tau^{-,c}$ with $b =1, \dots, n_\tau$. In the exceptional cases, $\psi_{*, \sigma^b\tilde \tau^{-,c}}$ and $(\psi^{-1})_{*, \sigma^b\tilde \tau^{-}}$ are well defined, and 
we have (by the proof of Proposition~\ref{P:Y_S=X_S} or rather the moduli description  in Subsection~\ref{S:moduli-Y_S})
\begin{align*}
&\quad 
\Ker(\psi_{*, \sigma^b\tilde \tau^{-,c}}) =  F_{A,\es}^{ b} \big( H^\dR_1(A/S)^{\circ, (p^{b})}_{\tilde \tau^{-,c}} \big),\text{ and }\\
&  \im ((\psi^{-1})_{*, \sigma^b\tilde \tau^{-}}) = F_{A,\es}^{b} \big( H^\dR_1(A/S)^{\circ, (p^{b})}_{\tilde \tau^{-}} \big).
\end{align*}
By definition, we have $\eta'_{\sharp}(y)=z$, and the composed quasi-isogeny $\psi\circ\phi: B\ra B'$ is nothing but the base change of $\eta'^{\sharp}$ to $S$.
For later reference, we remark that $\psi$ and $\phi$ induces isomorphisms 
\begin{equation}\label{E:Lie-algebra-equality}
\Lie(B/S)^{\circ}_{\tilde\tau'}\cong \Lie(A/S)^{\circ}_{\tilde\tau'}\cong \Lie(B'/S)^{\circ}_{\tilde\tau'}
\end{equation}
 for all $\tilde\tau'$ with restriction $\tau'\in \Sigma_{\infty}-\ttS_{\infty}$  different from $\tau^-,\tau, \tau^+$, and  
\begin{small}
\begin{align}\label{E:isom-Lie-algebras}
\Lie(B/S)^{\circ,(p^{n_{\tau}+n_{\tau^+}})}_{\tilde\tau^-}&\xra[\cong]{\eqref{E:Ker = F = Lie}}\Ker(\phi_{*,\tilde\tau^+}) \cong \Coker(\phi_{*,\tilde\tau^+})\cong \Lie(A/S)^{\circ}_{\tilde\tau^+}\xra[\cong]{\psi_{*, \tilde \tau^+}} \Lie(B'/S)^{\circ}_{\tilde\tau^+},\\
\label{E:isom-Lie-algebras1}
\Lie(B/S)^{\circ,(p^{n_{\tau}+n_{\tau^{+}}})}_{\tilde\tau^{-,c}}&\xra[\cong]{\eqref{E:Im = F = Lie}}\mathrm{Im}((\phi^{-1})_{*,\tilde\tau^{+,c}})\cong \Lie(A/S)^{\circ}_{\tilde\tau^{+,c}} \xra[\cong]{\psi_{*, \tilde \tau^{+,c}}} \Lie(B'/S)^{\circ}_{\tilde\tau^{+,c}}.
\end{align}
\end{small}

Consider the case when $S=\Spec(k)$ with $k$ a perfect  field containing $k_0$. Denote by $\tcD^{\circ}_{B,\tilde\tau'}$ the $\tilde\tau'$-component of the reduced covariant Dieudonn\'e module of $B$. 
From the discussion above, one sees easily that $\tcD^{\circ}_{B',\tilde\tau}=\tcD^{\circ}_{B,\tilde\tau'}$ for all $\tilde\tau'\in \Sigma_{E,\infty}$ expect for $\tilde\tau'\in\{\sigma^a\tilde\tau, \sigma^{a}\tilde\tau^c\;|\; 1\leq a\leq n_{\tau^+}\}\cup \{\sigma^b\tilde\tau^-, \sigma^b\tilde\tau^{-,c}\;|\;1\leq b\leq  n_{\tau}\}$. In the exceptional cases, we have 
\[
(\psi\circ\phi)^{-1}_*\tcD^{\circ}_{B',\tilde\tau'}=\begin{cases}
p^{-1}F_{B,\es}^{n_{\tau}+a}(\tcD^{\circ}_{B,\tilde\tau^-}) &\text{if }\tilde\tau'=\sigma^a\tilde\tau;\\
F_{B,\es}^{n_{\tau}+a}(\tcD^{\circ}_{B,\tilde\tau^{-,c}}) &\text{if }\tilde\tau'=\sigma^a\tilde\tau^c;\\
p^{-1}F_{B,\es}^{b}(\tcD^{\circ}_{B,\tilde\tau^{-,c}}) &\text{if }\tilde\tau'=\sigma^b\tilde\tau^{-,c};\\
F_{B,\es}^b(\tcD^{\circ}_{B,\tilde\tau^{-}}) &\text{if }\tilde\tau'=\sigma^b\tilde\tau^-.
\end{cases}
\]
Since the essential Frobenius $F_{B,\es}$ is bijective after inverting $p$, one sees easily that $\tcD^{\circ}_{B}$ can be recovered from $\tcD^{\circ}_{B'}$.
 This implies immediately that $\eta'_{\sharp}: \bfSh_{K'}(G'_{\tilde \ttS(\tau^+)})_{k_0} \ra \bfSh_{K'}(G'_{\tilde \ttS(\tau)})_{k_0} $ induces a bijection on $k$-valued points, i.e. $\eta'_{\sharp}$ verifies condition (1) in Definition~\ref{D:link-morphism}.
By the discussion above,  it is also obvious that conditions (2) and (3) of Definition~\ref{D:link-morphism} are also verified for $(\eta'_{\sharp},\eta'^{\sharp})$. Finally, from the formulas for $\tcD^{\circ}_{B,\tilde\tau}$, one sees easily that the degree of the quasi-isogeny 
 \[(\psi\circ\phi)_{\gothq}: B[\gothq^{\infty}]\ra B'[\gothq^{\infty}]\]
 is $2(n_{\tau^+}-n_{\tau})$ if $\gothp$ splits in $E$, and is $0$ if $\gothp$ is inert in $E$. This shows that $(\eta'_{\sharp},\eta'^{\sharp})$ is the link morphism associated to $\eta$ with the said indentation degree.
 
 Statement (1)(b) follows from the isomorphisms \eqref{E:Lie-algebra-equality}, \eqref{E:isom-Lie-algebras}, and \eqref{E:isom-Lie-algebras1} applied to the case when  $B$ is  the universal abelian scheme $\bfA'_{\tilde\ttS(\tau^+),k_0}$. 
 It remains to prove  (1)(c). The morphism $\eta'_{\sharp}$  is clearly quasi-finite, and  hence finite because $\bfSh_{K'}(G'_{\tilde\ttS(\tau^+)})_{k_0}$ is proper by Proposition~\ref{Prop:smoothness}.  
Since both $\bfSh_{K'}(G'_{\tilde\ttS(\tau)})_{k_0}$ and $\bfSh_{K'}(G'_{\tilde\ttS(\tau^+)})_{k_0}$ are regular, we conclude  by \cite[Theorem 23.1]{matsumura} that $\eta'_{\sharp}$ is flat at every point of $\bfSh_{K'}(G'_{\tilde\ttS(\tau)})_{k_0}$. 
It remains to see that the degree of $\eta'_{\sharp}$ is $p^{v(\eta)}$. For this, one might be able to argue geometrically, but we found it more convenient to compare the top Chern classes of the tangent bundle.

Let $\calT_{\tau}$ and $\calT_{\tau^+}$ denote respectively the tangent bundle of $\bfSh_{K'}(G'_{\tilde\ttS(\tau)})_{k_0}$ and $\bfSh_{K'}(G'_{\tilde\ttS(\tau^+)})_{k_0}$, and let $d=\#\Sigma_{\infty}-\#\ttS(\tau)_{\infty}$ be the common dimension of these  Shimura varieties.
 Fix a prime $\ell\neq p$.
 For a vector bundle $\calE$ over a proper and smooth $k_0$-variety $X$, we denote by $c_i(\calE)\in H^{2i}_{\et}(X_{\overline \FF_p},\Q_{\ell})(i)$ the $i$-th Chern class of $\calE$. 
 By Proposition~\ref{Prop:smoothness}, we have 
 \[
 c_{d}(\calT_{\tau})=\prod_{\tau'\in \Sigma_{\infty}-\ttS(\tau)_{\infty}}c_1 \bigg(\Lie(\bfA'_{\tilde\ttS(\tau),k_0})^{\circ}_{\tilde\tau'}\otimes\Lie(\bfA'_{\tilde\ttS(\tau),k_0})^{\circ}_{\tilde\tau'^c}\bigg),
 \]
 where $\tilde\tau',\tilde\tau'^c\in \Sigma_{E,\infty}$ denote the two liftings of $\tau'$.
 A similar formula for $c_d(\calT_{\tau^+})$ holds with $\tau$ replaced  by $\tau^+$.
By (1)(b), we have   
\[
\eta'^{*}_{\sharp}c_{d}(\calT_{\tau})=c_d(\eta'^{*}_{\sharp}\calT_{\tau})=p^{v(\eta)}c_d(\calT_{\tau^+}).
\]
Let  
$$
\Tr_{?}\colon H^{2d}_{\et}(\bfSh_{K'}(G'_{\tilde\ttS(?)})_{\overline\FF_p},\Q_{\ell})(d)\xra{\sim} \Q_{\ell} \quad \text{for }?=\tau,\tau^+
$$ 
be  the $\ell$-adic trace map. 
  Then  we have 
  \[
  \deg(\eta'_{\sharp}) \Tr_{\tau}(c_d(\calT_{\tau}))=\Tr_{\tau}(\eta'^{*}_{\sharp}c_d(\calT_{\tau}))=p^{v(\eta)} \Tr_{\tau^+}(c_d(\calT_{\tau^+})).
  \]
  It is well known that $\Tr_{\tau}(c_d(\calT_{\tau}))=\chi(\bfSh_{K'}(G'_{\tilde\ttS(\tau)})_{\overline \FF_p})$ (see \cite[Expos\'e VII, Corollaire 4.9]{SGA5}), where 
  $$
  \chi(\bfSh_{K'}(G'_{\tilde\ttS(\tau)})_{\overline \FF_p}): =\sum_{i=0}^{2d}(-1)^i\dim H^{i}_{\et}(\bfSh_{K'}(G'_{\tilde\ttS(\tau)})_{\overline \FF_p},\Q_{\ell})
  $$
 is the ($\ell$-adic) Euler-Poincar\'e characteristic.   Hence, one obtains 
 \[
 \deg(\eta'_{\sharp}) \cdot \chi(\bfSh_{K'}(G'_{\tilde\ttS(\tau)})_{\overline \FF_p})=p^{v(\eta)} \cdot \chi(\bfSh_{K'}(G'_{\tilde\ttS(\tau^+)})_{\overline \FF_p})
 \]
 Since $\eta'_{\sharp}$ is  purely inseparable,  we have $\chi(\bfSh_{K'}(G'_{\tilde\ttS(\tau)})_{\overline \FF_p})=\chi(\bfSh_{K'}(G'_{\tilde\ttS(\tau')})_{\overline \FF_p})$. By Lemma~\ref{L:non-vanishing-chi},  we have  $\chi(\bfSh_{K'}(G'_{\tilde\ttS(\tau)}))\neq 0$, and hence $\deg(\eta'_{\sharp})=p^{v(\eta)}$.

(2) 
Note that $\gothp$ splits in $E$, and  fix a prime $\gothq$ of $E$ dividing $\gothp$.
As before, we denote by $\tilde\tau$ and $\tilde\tau^-$ the liftings of $\tau$ and $\tau^-$ in $\Sigma_{E,\infty/\gothq}$ respectively.

We define first a link morphism $(\eta_{\sharp},\eta^{\sharp}): \bfSh_{K'}(G'_{\tilde\ttS(\tau^-)})_{k_0}\ra \bfSh_{K'}(G'_{\tilde\ttS(\tau)})_{k_0}$ of indentation degree $p^{2(g-n_{\tau})}$ as follows.
Let $y=(B',\iota_{B'},\lambda_{B'},\beta_{K'})$  be an $\overline \FF_p$-point of $\bfSh_{K'}(G'_{\tilde\ttS(\tau^-)})_{k_0}$ so that $\dim \omega^\circ_{B'^\vee/k_0, \tilde \tau} = 0$ and $\dim \omega^\circ_{B'^\vee/k_0, \tilde \tau^c} = 2$. We define first a lattice $M^{\circ}_{\tilde\tau'}$ of $\tcD^{\circ}_{B',\tilde\tau'}[1/p]$ for each $\tilde\tau'\in \Sigma_{E,\infty}$ as follows: We put
\[
M^{\circ}_{\sigma^{i}\tilde\tau}=
\begin{cases}\frac{1}{p}\tcD^{\circ}_{B',\sigma^{i}\tilde\tau} &\text{for } i=1,\cdots, n_{\tau^-}=g-n_{\tau},\\
\tcD^{\circ}_{B',\sigma^i\tilde\tau} &\text{for } i=n_{\tau^-}+1,\cdots,g,
\end{cases}
\]
and $M^{\circ}_{\sigma^i\tilde\tau^{c}}=M^{\circ,\vee}_{\sigma^i\tilde\tau}$. One checks easily that $M=\bigoplus_{\tilde\tau'\in \Sigma_{E,\infty}} M^{\circ,\oplus 2}_{\tilde\tau'}$ is stable under the action of Frobenius and Verschiebung homomorphisms, hence  a Dieudonn\'e submodule of $\tcD_{B'}[1/p]$. 
As in the proof of Proposition~\ref{P:Y_S=X_S}, this gives rise to an abelian variety $B''$ equipped with an action by $\cO_{D}$ and an $\cO_D$-equivariant quasi-isogeny $\phi: B' \to B''$ such that the induced morphism on the Dieudonn\'e module $\phi^{-1}: \tilde \calD_{B''} \to \tilde \calD_{B'}[1/p]$ is identified with the inclusion   $M \subset \tcD_{B'}[1/p]$.
Since the lattice $M\subseteq \tcD_{B'}[1/p]$ is self-dual by construction, the polarization $\lambda_{B'}$ induces a prime-to-$p$ polarization $\lambda_{B''}$ on $B''$ such that $\lambda_{B'}=\phi^\vee\circ \lambda_{B''}\circ \phi$.
 Finally, the $K'^p$-level structure $\beta_{K'}$ on $B'$ induces naturally a $K'^p$-level structure $\beta''_{K'}$ on $B''$.
  Moreover, an easy computation shows that $\dim \omega^\circ_{B''^\vee/k_0, \tilde \tau} = 2$ and $\dim \omega^\circ_{B''^\vee/k_0, \tilde \tau^c} = 0$, and $\dim \omega^\circ_{B''^\vee/k_0, \tilde \tau'}=\dim \omega^\circ_{B'^\vee/k_0, \tilde \tau'}$ at other $\tilde \tau'$.
   Thus,   $(B'',\iota_{B''},\lambda_{B''},\beta''_{K'})$ is a point  of $\bfSh_{K'}(G'_{\tilde\ttS(\tau)})_{k_0}$. Let 
      $$
   \eta_{\sharp}\colon \bfSh_{K'}(G'_{\tilde\ttS(\tau^-)})_{k_0} \to \bfSh_{K'}(G'_{\tilde\ttS(\tau)})_{k_0}
   $$ 
   be the map sending $y=(B',\iota_{B'},\lambda_{B'},\beta_{K'})\mapsto (B'',\iota_{B''},\lambda_{B''},\beta''_{K'})$, and let 
   \[\eta^{\sharp}\colon \bfA'_{\tilde\ttS(\tau^-), k_0}\ra \eta_{\sharp}^*\bfA'_{\tilde\ttS(\tau), k_0}
   \]
   be the $p$-quasi-isogeny whose base change to each $y$ is $\phi: B'\ra B''$ constructed above.
    Then it is clear by construction that $(\eta_{\sharp}, \eta^{\sharp})$ is the link morphism of indentation degree ${2(g-n_{\tau})}$ associated to the trivial link from $\ttS(\tau^-)$ to $\ttS(\tau)$.


Denote by  $\pi_{\{\tau,\tau^-\}}: \bfSh_{K'}(G'_{\tilde\ttS})_{k_0,\{\tau,\tau^-\}}\xra{\sim} \bfSh_{K'^p\Iw'_{p}}(G'_{\tilde\ttS(\tau)})_{k_0}$ the isomorphism given by Theorem~\ref{T:main-thm-unitary}. 
Let $x=(A, \iota_A, \lambda_A, \alpha_{K'})$ be an $\overline \FF_p$-point of  $\bfSh_{K'}(G'_{\tilde \ttS})_{k_0,\{\tau, \tau^-\}}$.
 Then its image $(B,\iota_{B},\lambda_{B},\beta_{K'^\gothp},\beta_{\gothp})$ under $\pi_{\{\tau,\tau^-\}}$  is characterized as follows: 
\begin{itemize}
\item[(a)] There exists an $\cO_D$-equivariant $p$-quasi-isogeny $\phi: B\ra A$ such that $\phi$ induces an isomorphism  $\phi_{*}\colon \tcD^{\circ}_{B,\tilde\tau'}\xra{\sim} \tcD^{\circ}_{A,\tilde\tau'}$ for $\tilde\tau'$ different from $ \sigma^{i}\tilde\tau^-$ with $i=1,\dots, n_{\tau}$ and   their complex conjugates.
 In the exceptional cases, we have 
\begin{align*}
&\phi_{*}(\tcD^{\circ}_{B,\sigma^i\tilde\tau^-})=\im(F_{A,\es}^{i}\colon \tcD^{\circ}_{A,\tilde\tau^-}\ra \tcD^{\circ}_{A,\sigma^i\tilde\tau^-}),
\quad \text{and} \\
&
\phi_{*}(\tcD^{\circ}_{B,\sigma^i\tilde\tau^{-,c}})=\frac{1}{p}\im(F_{A,\es}^i\colon\tcD^{\circ}_{A,\tilde\tau^{-,c}}\ra \tcD^{\circ}_{A,\sigma^i\tilde\tau^{-,c}}).
\end{align*}
\item[(b)] We have $\lambda_{B}=\phi^\vee\circ\lambda_{A}\circ\phi$, and $\beta_{K'^p}=\phi\circ\alpha_{K'^p}$.
\item[(c)]  Let $M^\circ_{\tilde\tau}\subseteq \tcD_{B,\tilde\tau}^{\circ}/p\tcD^{\circ}_{B,\tilde\tau}$ be the one-dimensional subspace given by the image of $p\tcD^{\circ}_{A,\tilde\tau}$ via $\phi_{*}^{-1}$. 
Then $M_{\tilde\tau}^\circ$ is stable under $F_{B,\es}^g$, and  $M:=\bigoplus_{i=0}^{g-1}F_{B,\es}^i(M^\circ_{\tilde\tau})^{\oplus 2}$ is a Dieudonn\'e submodule of $\calD_{B[\gothq]}=\bigoplus_{i=0}^{g-1}\tcD_{B,\sigma^{i}\tilde\tau}/p\tcD_{B,\sigma^i\tilde\tau}$. Let $H_{\gothq}$ be the subgroup scheme  of $B[\gothq]$ associated to $M$. Then the 
 Iwahoric level structure of $B$ at $p$  is given by $\beta_{\gothp}=H_{\gothq}\oplus H_{\bar\gothq}$, where $H_{\bar\gothq}\subseteq B[\bar\gothq]$ is the orthogonal complement of $H_{\gothq}$ under the natural duality between $B[\gothq]$ and $B[\bar\gothq]$. 
\end{itemize}
It is clear that the image of $x$ under $\pi_{\tau}$ is $(B,\iota_{B},\lambda_{B},\beta_{K'^p})$ by forgetting the Iwahoric level structure at $p$ of $\pi_{\{\tau,\tau^-\}}(x)$. This shows that,  via the isomorphism $\pi_{\{\tau,\tau^-\}}$, the map $\pi_{\tau}|_{\bfSh_{K'}(G'_{\tilde \ttS})_{k_0,\{\tau,\tau^-\}}}$  coincides with the projection $\pr_1$ in \eqref{E:Hecke-T_q}.

To finish the proof of (2), it remains  to show that the composition
$$
\eta_{\sharp}\circ\pi_{\tau^-}\circ\pi_{\{\tau,\tau^-\}}^{-1}: \bfSh_{K'^p\Iw'_p}(G'_{\tilde \ttS(\tau)})_{k_0}\ra \bfSh_{K'}(G'_{\tilde\ttS(\tau)})_{k_0}
$$
 is the second projection $\pr_2$ in \eqref{E:Hecke-T_q}. 
 Let $x=(A,\iota_{A},\lambda_{A},\alpha_{K'})$ be a point of $\bfSh_{K'}(G'_{\tilde\ttS})_{k_0,\{\tau,\tau^-\}}$ with image $\pi_{\{\tau,\tau^-\}}(x)=(B,\iota_B,\lambda_B,\beta_{K'^\gothp},\beta_{\gothp})$, together with the $p$-quasi-isogeny $\phi\colon B\ra A$ as described above. 
The image of $(A,\iota_{A},\lambda_{A},\alpha_{K'})$ under $\pi_{\tau^-}$ is given by 
  $(B', \iota_{B'}, \lambda_{B'}, \beta_{K'}) \in \bfSh_{K'}(G'_{\tilde \ttS(\tau^-)})_{k_0}$, which admits a quasi-isogeny $\psi: B'\ra A$ compatible with all structures such that $\psi_{*}: \tcD^{\circ}_{B',\tilde\tau'}\cong \tcD^{\circ}_{A,\tilde\tau'}$ for all $\tilde\tau'$ except for those $\tilde\tau'$'s lifting $\sigma^i\tau$ for any $i = 1, \dots, g-n_{\tau}=n_{\tau^-}$.  
In the exceptional cases, we have 
$$
\psi_{*}( \tcD^{\circ}_{B',\sigma^i\tilde\tau})=\im(F_{A,\es}^i\colon \tcD^{\circ}_{A,\tilde\tau}\ra \tcD^{\circ}_{A,\sigma^i\tilde\tau})\quad\text{and}\quad
 \psi_*(\tcD^{\circ}_{B',\sigma^i\tilde\tau^c})=\frac1 p\im(F_{A,\es}^i\colon \tcD^{\circ}_{A,\tilde\tau^c}\ra \tcD^{\circ}_{A,\sigma^i\tilde\tau^c}).
$$ 
Consider the composed isogeny $\phi^{-1}\circ\psi\colon B'\ra B$.
 Let $\tilde M^\circ_{\tilde\tau}\subseteq \tcD^{\circ}_{B,\tilde\tau}$ be the inverse image of the one-dimensional subspace $M^\circ_{\tilde\tau}\subseteq \tcD^{\circ}_{B,\tilde\tau}/p\tcD^{\circ}_{B,\tilde\tau}$ given by the image of $p\tcD^{\circ}_{A,\tilde\tau}$ as in (c) above. Then, we have 
 \[
 (\phi^{-1}\circ\psi)_*(\tcD^{\circ}_{B',\sigma^{i}\tilde\tau})=
 \begin{cases}
 F_{B,\es}^i(\tilde M^\circ_{\tilde\tau})&\text{for }i=1,\cdots, n_{\tau^-}=g-n_{\tau};\\
 \frac1pF_{B,\es}^i(\tilde M^\circ_{\tilde\tau}) &\text{for }i=n_{\tau^-}+1,\cdots, g;
 \end{cases}
 \] 
 and  $(\phi^{-1}\circ\psi)_*(\tcD^{\circ}_{B',\sigma^{i}\tilde\tau^c})$ is the orthogonal complement of $(\phi^{-1}\circ\psi)_*(\tcD^{\circ}_{B',\sigma^{i}\tilde\tau})$. 
 
 Let $(B'',\iota_{B''},\lambda_{B''},\beta''_{K'})$ be the image of $(B', \iota_{B'}, \lambda_{B'}, \beta_{K'}) \in \bfSh_{K'}(G'_{\tilde \ttS(\tau^-)})_{k_0}$ under $\eta_{\sharp}$. 
 Then, the composed $p$-quasi-isogeny $B\xra{\psi^{-1}\circ\phi} B'\xra{\eta^{\sharp}}B''$, identifies   $\tcD^{\circ}_{B'',\sigma^i\tilde\tau}$ with the lattice 
  $$
  \frac{1}{p}F_{B,\es}^{i}(\tilde M^\circ_{\tilde\tau})\subseteq \tcD^{\circ}_{B,\sigma^i\tilde\tau}[1/p],\quad \text{for all }i=1,\cdots, g.
  $$ 
Thus, one sees immediately that the map $(B,\iota_B, \lambda_B, \beta_{K'^\gothp},\beta_{\gothp})\mapsto (B'',\iota_{B''},\lambda_{B''},\beta''_{K'})$ is nothing but the second projection in \eqref{E:Hecke-T_q}. This finishes the proof of (2). 
\end{proof}

   Our last  proposition explains the compatibility of the description of the GO-divisors as in Theorem~\ref{T:main-thm-unitary}  with respect to the link morphism, especially to the link morphism appearing  in Theorem~\ref{T:link and Hecke operator}(1). 

\begin{prop}
\label{P:compatibility of link and GO}
Assume that $\#\Sigma_\infty - \#\ttS_\infty \geq 2$. 
Let $\tau_0\in \Sigma_{\infty}-\ttS_{\infty}$, and  $\eta: \ttS \to \ttS'$ be a link such that all curves are straight lines except for (possibly) one curve turning to the right, linking $\tau_0\in \Sigma_{\infty}-\ttS_{\infty}$ with $\tau'_0 = \eta(\tau_0) = \sigma^{m(\tau_0)}\tau_0$ for some integer $m(\tau_0)\geq 0$.  Assume that the link morphism $(\eta'_\sharp, \eta'^\sharp): \bfSh_{K'}(G'_{\tilde \ttS})_{k_0} \to  \bfSh_{K'}(G'_{\tilde \ttS'})_{k_0}$ of (some) indentation degree $n \in \ZZ$ associated to $\eta$ exists.
 The setup automatically implies that $\tau^+_0 = \tau'^+_0$ and $\tau_0^- = \tau'^-_0$. 
 Let $\tau\in \Sigma_\infty - \ttS_\infty$. 
 Then the following statements hold:
 \begin{enumerate}
 \item[(1)] The link morphism $\eta'_\sharp$ sends the GO-divisor $\bfSh_{K'}(G'_{\tilde \ttS})_{k_0, \tau}$ into the GO-divisor $\bfSh_{K'}(G'_{\tilde \ttS'})_{k_0, \eta(\tau)}$. 
 
 \item[(2)] Let $\eta_{\tau}:  \ttS(\tau) \to  \ttS'(\eta(\tau))$ denote the link given by removing from $\eta$ the two curves attached to $\tau$ and $\tau^-$.
  Then there exists a link morphism $(\eta'_{\tau,\sharp},\eta'^{\sharp}_{\tau})$ (of some indentation degree $m$) from $\bfSh_{K'}(G'_{\tilde \ttS(\tau)})_{k_0}$ to $\bfSh_{K'}(G'_{\tilde \ttS'(\eta(\tau))})_{k_0}$, associated to the link $\eta_\tau$
such that we have the  following commutative diagram of Shimura varieties
\begin{equation}\label{E:induced-link-morphism}
\xymatrix@C=60pt{
\bfSh_{K'}(G'_{\tilde \ttS})_{k_0, \tau} \ar[r]^-{\pi_{\tau}}  \ar[d]^{\eta'_{\sharp}} &
\bfSh_{K'}(G'_{\tilde \ttS(\tau)})_{k_0} \ar[d]^{\eta'_{\tau,\sharp}}
\\
  \bfSh_{K'}(G'_{\tilde \ttS'})_{k_0,\eta(\tau)}
  \ar[r]^-{\pi_{\eta(\tau)}}
  & \bfSh_{K'}(G'_{\tilde \ttS'(\eta(\tau))})_{k_0},
}
\end{equation}
and a similar commutative diagram of quasi-isogenies of universal abelian varieties. Moreover, the indentation degree of the link morphism $\eta'_{\tau, \sharp}$ is given by
 $$m=n+n_{\tau}-n_{\eta(\tau)}(\ttS')=\begin{cases}
0 &\text{if $\gothp$ is inert in $E/F$;}\\
n &\text{if $\gothp$ splits in $E/F$ and $\tau\neq \tau_0, \tau_0^+$;}\\
n-m(\tau_0) &\text{if $\gothp$ splits in $E/F$ and $\tau=\tau_0$;}\\
n+m(\tau_0) &\text{if $\gothp$ splits in $E/F$ and $\tau=\tau_0^+$.}
\end{cases}
$$

\item[(3)] Suppose moreover that the link $\eta$ and the link morphism $(\eta'_{\sharp}, \eta'^{\sharp})$ are  those appearing  in Theorem~\ref{T:link and Hecke operator}(1) (so our $\tilde\ttS$ being $\tilde\ttS(\tau^+)$, $\ttS'$ being $\tilde\ttS(\tau)$, and  $\tau_0$ being $\tau^-$ therein, respectively).
If $\cO_{\tau}(1)$ (resp. $\cO_{\eta(\tau)}(1)$) denotes the tautological  quotient line bundle on $\bfSh_{K'}(G'_{\tilde\ttS})_{k_0, \tau}$  (resp. on $\bfSh_{K'}(G'_{\tilde\ttS'})_{k_0, \eta(\tau)}$) for the $\PP^1$-fibration $\pi_{\tau}$ (resp. $\pi_{\eta(\tau)}$), then we have a canonical isomorphism 
\begin{equation}\label{E:link-line-bundles}
\eta'^{*}_{\sharp}(\cO_{\eta(\tau)}(1))\cong 
\begin{cases}
\cO_{\tau}(1) &\text{if } \tau\neq \tau_0^+;\\
\cO_{\tau}(p^{m(\tau_0)}) &\text{if } \tau=\tau_0^+.
\end{cases}
\end{equation}
 Moreover, the induced  link morphism $\eta'_{\tau, \sharp}$ is finite flat of degree $p^{v(\eta_{\tau})}$, i.e. it is an isomorphism if $\tau\in \{\tau^{+}_0, \tau_0\}$, and it is finite flat of degree $p^{m(\tau_0)}$ if $\tau\notin \{ \tau_0, \tau_0^{+}\}$.
 \end{enumerate}

The analogous results hold for link morphisms for $\bfSh_{K''}(G''_{\tilde \ttS})_{k_0}$'s.

\end{prop}
\begin{proof}
(1) Since $\bfSh_{K'}(G'_{\tilde\ttS})_{k_0, \tau}$ is reduced, it suffices to prove that $\eta'_{\sharp}$ sends every $\overline \FF_p$-point of $\bfSh_{K'}(G'_{\tilde\ttS})_{k_0, \tau}$ to $\bfSh_{K'}(G'_{\tilde\ttS'})_{k_0,\eta(\tau)}$.
Let $x=(A,\iota, \lambda, \alpha_{K'^p})$ be an $\overline \FF_p$-point of $\bfSh_{K'}(G'_{\tilde\ttS})_{k_0, \tau}$, and $\eta'_{\sharp}(x)=(A',\iota',\lambda',\alpha'_{K'^p})$
be its image. Let $\tilde\tau\in \Sigma_{E,\infty}$ be a place above $\tau$, and put $\tilde\tau^-=\sigma^{-n_{\tau}}\tilde\tau$ and $\tilde\tau^+ =\sigma^{n_{\tau^+}}\tilde\tau$.
By Lemma~\ref{Lemma:partial-Hasse},  the condition  $h_{\tau}(A)=0$ is equivalent to $F_{\es,A}^{n_{\tau}}(\tcD^{\circ}_{A,\tilde\tau^-})=\tilde\omega_{A^\vee,\tilde\tau}^{\circ}$, where $\tilde\omega^{\circ}_{A^\vee,\tilde\tau}\subseteq \tcD^{\circ}_{A,\tilde\tau}$ denotes the inverse image of $\omega^{\circ}_{A^\vee,\tilde\tau}\subseteq \calD^{\circ}_{A,\tilde\tau}$. 
The latter condition is in turn equivalent to  that $F_{\es, A}^{n_{\tau^+}}\circ F_{\es,A}^{n_{\tau}}(\tcD^{\circ}_{A,\tilde\tau^-})=p\tcD^{\circ}_{A,\tilde\tau^+}$.
We set $\eta(\tilde\tau^-)=\sigma^{m(\tau^-)}\tilde\tau^-$, where $m(\tau^-)$ the displacement of the curve in $\eta$ connecting $\tau^-$ and $\eta(\tau^-)$ (which equals to $0$ except when $\tau^-=\tau_0$); similarly, we put $\eta(\tilde\tau^+) = \sigma^{m(\tau^+)}\tilde \tau^+$.
 Since $\eta'^{\sharp}_{*}: \tcD^{\circ}_{A}[1/p]\ra \tcD^{\circ}_{A'}[1/p]$ commutes with  Frobenius and Verschiebung homomorphisms, one sees easily  from  condition (3) in Definition~\ref{D:link-morphism} that  $F_{\es,A'}^{n_{\eta(\tau^+)}(\ttS')} \circ F_{\es,A'}^{n_{\eta(\tau)}(\ttS')}(\tcD^{\circ}_{A',\eta(\tilde\tau^-)})=p^u\tcD^{\circ}_{A',\eta(\tilde\tau^+)}$ for some integer $u\in \Z$. Here, $n_{\eta(\tau)}(\ttS')$ is the integer  defined in Notation~\ref{N:n tau} associated to $\tau$ for the set $\ttS'$.
 But  $F_{\es,A'}^{n_{\eta(\tau^+)}(\ttS')} \circ F_{\es,A'}^{n_{\eta(\tau)}(\ttS')}(\tcD^{\circ}_{A',\eta(\tilde\tau^-)})$ is  a $W(\overline \FF_p)$-sublattice of $\tcD^{\circ}_{A,\eta(\tilde\tau^+)}$ with quotient of  length $2$ over $W(\overline \FF_p)$. Hence, the integer $u$ has to be $1$.
  By the same reasoning using Lemma~\ref{Lemma:partial-Hasse}, this is equivalent to saying that $h_{\eta(\tau)}(A')=0$, or equivalently $\eta'_{\sharp}(x)\in \bfSh_{K'}(G'_{\tilde\ttS'})_{k_0,\eta(\tau)}$.

(2)
Assume first that $\#\Sigma_{\infty}-\#\ttS_{\infty}>2$.
 By  Proposition~\ref{P:restriction of GO strata}(2), $\pi_{\tau}|_{\bfSh_{K'}(G'_{\tilde\ttS})_{k_0, \{\tau, \tau^-\}}}$ is an isomorphism.
 We define 
\[\eta'_{\tau,\sharp}\colon= \pi_{\eta(\tau)}\circ \eta'_{\sharp}\circ (\pi_{\tau}|_{\bfSh_{K'}(G'_{\tilde\ttS})_{k_0, \{\tau, \tau^-\}}})^{-1}
\] 
and $\eta'^{\sharp}_{\tau}$ as the pull-back via $(\pi_{\tau}|_{\bfSh_{K'}(G'_{\tilde\ttS})_{k_0, \{\tau^-, \tau\}}})^{-1}$ of the quasi-isogeny 
\[
\pi_{\tau}^{*}\bfA'_{\tilde\ttS(\tau), k_0}\xra{\phi_{\tau}^{-1}} (\bfA'_{\tilde\ttS, k_0}|_{\bfSh_{K'}(G'_{\tilde\ttS})_{k_0, \tau}})\xra{\eta'^{\sharp}} \eta'^{*}_{\sharp} (\bfA'_{\tilde\ttS',k_0}|_{\bfSh_{K'}(G'_{\tilde\ttS'})_{k_0,\eta(\tau)}})\xra{\eta'^*_\sharp(\phi_{\eta(\tau)})} \eta'^*_{\sharp}\pi_{\eta(\tau)}^*(\bfA'_{\tilde\ttS(\eta(\tau)), k_0})
\]
 of abelian schemes on $\bfSh_{K'}(G'_{\tilde\ttS})_{k_0, \tau}$.
  Here,  the first and the third quasi-isogenies are given by Theorem~\ref{T:main-thm-unitary}(2).
It is clear that the diagram \eqref{E:induced-link-morphism} is commutative.
 It remains to show that $(\eta'_{\tau,\sharp},\eta'^{\sharp}_{\tau})$ defines a link morphism. 
Conditions (1) and (2) of Definition~\ref{D:link-morphism} being clear, condition (3) can be verified by a tedious but straightforward computation.
To see condition (4) on the indentation degree, we need only to discuss the case when $\gothp$ splits in $E/F$.
In this case, $\phi^{-1}_\tau[\gothq^\infty]$ has degree $p^{2n_{\tau}}$,
$\eta'^\sharp[\gothq^\infty]$ has degree $p^{2n}$, and $\phi_{\eta(\tau)}[\gothq^\infty]$ has degree $p^{-2n_{\eta(\tau)}(\ttS')}$. 
So the total  degree of quasi-isogeny of $\eta'^{\sharp}_{\tau}$ is $p^{2m}=p^{2n+2n_{\tau}-2n_{\eta(\tau)}(\ttS')}$.
A case-by-case discussion proves the condition (4) of Definition~\ref{D:link-morphism} on indentation degrees.

Assume now $\#\Sigma_{\infty}-\#\ttS_{\infty}=2$ so that $\Sigma_{\infty}-\ttS_{\infty}=\{\tau_0, \tau_0^+=\tau_0^{-}\}$. This implies that   $\gothp$ splits  as $\gothq\bar\gothq$ in $E$.
 Since the Shimura variety $\bfSh_{K'}(G'_{\tilde\ttS(\tau)})_{k_0}$ is zero-dimensional, we just need to define the desired link morphism $(\eta'_{\tau,\sharp}, \eta'^{\sharp}_{\tau})$ on $\overline \FF_p$-points.
   For each $\tilde\tau'\in \Sigma_{E,\infty}$ with restriction $\tau'\in \{\tau_0,\tau_0^-\}$, let $t_{\tilde\tau'}\in \Z$ denote the integer as in Definition~\ref{D:link-morphism}(3) attached to  $\tilde\tau'$ for the link morphism $(\eta'_{\sharp},\eta'^{\sharp})$.
   Let $y=(B,\iota_{B},\lambda_B, \beta_{K'^p})$ be an $\overline \FF_p$-point of $\bfSh_{K'}(G'_{\tilde\ttS(\tau)})_{k_0}$. 
    We now distinguish two cases:
 \begin{itemize}
 \item[(a)] Consider first the case $\tau=\tau_0$. 
 Let $\tilde\tau^-_0$ be the lift of $\tau_0$ in $ \Sigma_{E,\infty/\gothq}$.
We define $M^\circ_{\tilde\tau_0^-}=p^{-t_{\tilde\tau_0}} \tcD^{\circ}_{B,\tilde\tau_0}$, and  
   $M^\circ_{\sigma^i\tilde\tau_0^-}=p^{-\delta_i}F^{i}(M^\circ_{\tilde\tau_0^-})$ for each integer $i$ with $1\leq i\leq g-1$, where $\delta_i$ denotes the number of integers $j$ with $1\leq j\leq i$ such that $\sigma^{j}\tilde\tau_0^-\in \tilde\ttS'(\eta(\tau))_{\infty}$. 
   Put  $M^\circ_{\gothq}=\bigoplus_{0\leq i\leq g-1}M^\circ_{\sigma^i\tilde\tau_0^-}$, and let $M^\circ_{\bar\gothq}\subseteq \tcD^{\circ}_{B,\bar\gothq}[1/p]$ be the dual lattice of $M^\circ_{\gothq}$ with respect to the pairing induced by $\lambda_B$.
    Then $M^\circ:= M^\circ_{\gothq}\oplus M^\circ_{\bar\gothq}$ is a Dieudonn\'e submodule of $\tcD^{\circ}_{B}[1/p]$. 
   By the same argument as in the proof of  Proposition~\ref{P:Y_S=X_S}, there exists a unique abelian variety $B'$ equipped with an $\cO_D$-action $\iota_{B'}$ together with a $p$-quasi-isogeny $\phi: B\ra B'$ such that the induced map  $\phi^{-1}_{*}: \tcD^{\circ}_{B'}\ra \tcD^{\circ}_{B}[1/p]$ is identified with the natural inclusion $M^\circ\hra  \tcD^{\circ}_{B}[1/p]$. 
   As usual, since $M$ is a self-dual lattice, $\lambda_{B}$ induces a prime-to-$p$ polarization $\lambda_{B'}$ such that $\phi^{\vee}\circ\lambda_{B'}\circ\phi=\lambda_{B}$. 
   We equip $B'$ with the  $K'^p$-level structure $\beta'_{K'^p}=\phi\circ\beta_{K'^p}$. By the construction, one sees also easily that $B'$ satisfies the necessary signature condition so that $y'=(B',\iota_{B'},\lambda_{B'},\beta'_{K'^p})$ is a point of $\bfSh_{K'}(G'_{\tilde\ttS'(\eta(\tau))})_{k_0}$.
  We define 
  $$\eta'_{\tau, \sharp}\colon \bfSh_{K'}(G'_{\tilde\ttS(\tau)})_{k_0}\ra \bfSh_{K'}(G'_{\tilde\ttS'(\eta(\tau))})_{k_0},\quad \text{and}\quad \eta'^{\sharp}_{\tau}\colon \bfA'_{\tilde\ttS(\tau),k_0}\ra \bfA'_{\tilde\ttS'(\eta(\tau)), k_0}
  $$ 
   by  $\eta'_{\tau, \sharp}(y)= y'$ and $\eta'^{\sharp}_{\tau, y}=\phi$. 
   It is evident that $(\eta'_{\tau,\sharp},\eta'^{\sharp}_{\tau})$ is a link morphism. 
   It remains to check the commutativity of the diagram \eqref{E:induced-link-morphism}. 
   Let $x=(A,\iota_{A},\lambda_A,\alpha_{K'^p})\in \bfSh_{K'}(G'_{\tilde\ttS})_{k_0,\tau}$ be a point above $y$,  $x'=(A',\iota_{A'},\lambda_{A'},\alpha'_{K'^p})\in \bfSh_{K'}(G'_{\tilde\ttS'})_{k_0,\eta(\tau)}$  be the image of $x$ under $\eta'_{\sharp}$. We need to prove that $\pi_{\eta(\tau)}(x')=y'$. Let $y''=(B'',\iota_{B''},\lambda_{B''},\beta''_{K'^p})\in \bfSh_{K'}(G'_{\tilde\ttS'(\eta(\tau))})_{k_0,\eta(\tau)}$ denote temporarily the point $\pi_{\eta(\tau)}(x')$.
    Denote by $\psi\colon A\ra B$ and $\psi'\colon A'\ra B''$ be the  quasi-isogenies given by Theorem~\ref{T:main-thm-unitary}.  
Then we have  $\psi_*  (\tcD^{\circ}_{A,\tilde\tau_0^-})= \tcD^{\circ}_{B,\tilde\tau_0^-}$ and $\psi'_{*}(\tcD^{\circ}_{A',\tilde\tau_0^-})=\tcD^{\circ}_{B'',\tilde\tau_0^-}$ by Subsection~\ref{S:moduli-Y_S}, and $\eta'^{\sharp}_{*}(\tcD^{\circ}_{A,\tilde\tau_0^-})=p^{t_{\tilde\tau_0^-}}\tcD^{\circ}_{A',\tilde\tau_0^-}$ by Definition~\ref{D:link-morphism}(3) as $\eta$ is a straight line at $\tau_0^-$.  
Consider the quasi-isogeny $\psi'\circ\eta'^{\sharp}\circ\psi^{-1}\circ\phi^{-1}: B'\ra B''$. It induces an isomorphism   
  $\tcD^{\circ}_{B',\tilde\tau_0^-}\xra{\sim} \tcD^{\circ}_{B'',\tilde\tau_0^-}$. 
   But the other components of the Dieudonn\'e modules  are determined by that at $\tilde\tau_0^-$, as $\ttS'(\eta(\tau))_\infty  = \Sigma_\infty$.
   It follows that $\psi'\circ\eta'^{\sharp}\circ\psi^{-1}\circ\phi^{-1}: B'\ra B''$ is an isomorphism compatible with all structures, i.e. $y''=y'$.
  The computation of the indentation degree is the same as that in the case when $\#\Sigma_\infty - \#\ttS_\infty >2$.
       
   \item[(b)] In the case  $\tau=\tau_0^+ = \tau_0^-$, the construction is similar. 
   Let  $\tilde\tau_0$ be the lift of $\tau_0$ in $\Sigma_{E,\infty/\gothq}$, and $\eta(\tilde\tau_0)=\sigma^{m(\tau_0)}\tilde\tau_0$. 
Put $M^\circ_{\eta(\tilde\tau_0)}:=p^{-t_{\tilde\tau_0}}F^{m(\tau_0)}(\tcD^{\circ}_{B,\eta(\tilde\tau_0)})$, and $M^\circ_{\sigma^{i}\eta(\tilde\tau_0)}:=p^{-\delta_i}F^{i}(M^\circ_{\eta(\tilde\tau_0)})$ for each integer $i$ with $1\leq i\leq g-1$, where $\delta_i$ is the number of integers $j$ with $1\leq j\leq i$ such that $\sigma^j\eta(\tilde\tau_0)\in \tilde\ttS'(\eta(\tau))_{\infty}$. 
Let $M^\circ_{\gothq}=\bigoplus_{0\leq i\leq g-1}M^\circ_{\sigma^{i}\eta(\tilde\tau_0)}$, and $M^\circ_{\bar\gothq}\subseteq \tcD^{\circ}_{B,\bar\gothq}[1/p]$ be the dual lattice. 
   As in case (a) above, such a lattice $M^\circ:=M^\circ_\gothq \oplus M^\circ_{\bar \gothq}$ gives rise to an  $\overline \FF_p$-point $y'=(B',\iota_{B'},\lambda_{B'},\beta'_{K'^p})$ together with a $p$-isogeny $\phi: B\ra B'$ compatible with all structures. We define the desired link morphism $(\eta'_{\tau,\sharp}, \eta'^{\sharp}_{\tau})$ such that $\eta'_{\tau,\sharp}(y)=y'$ and $\eta'^\sharp_{\tau,y}=\phi$.
    The commutativity of \eqref{E:induced-link-morphism} is proved by the same arguments as in (a). We leave the details to the reader. 
   \end{itemize}
   
 (3) 
We note that, if $\tilde\tau^-$ is the lifting of $\tau^-$ not contained in $\tilde\ttS(\tau)_{\infty}$, then we have a canonical isomorphism  $\cO_{\tau}(1)\cong \Lie(\bfA'_{\tilde\ttS, k_0})^{\circ}_{\tilde\tau^-}$ by the construction of $\pi_{\tau}$. Similarly, one has $\cO_{\eta(\tau)}(1)\cong \Lie(\bfA'_{\tilde\ttS',k_0})^{\circ}_{\eta(\tilde\tau^-)}$. Now the isomorphism \eqref{E:link-line-bundles} follows immediately from Theorem~\ref{T:link and Hecke operator}(1)(b).
We prove now the second part of (3). 
By the construction of $\eta'_{\tau}$, it follows from Theorem~\ref{T:link and Hecke operator}(1)(b) that  $\eta'^{\sharp}_{\tau}\colon \bfA'_{\tilde\ttS(\tau),k_0}\ra \eta'^*_{\tau,\sharp}(\bfA'_{\tilde\ttS(\eta(\tau)),k_0})$ induces, for any $\tilde\tau'\in \Sigma_{E,\infty}$ lifting an element $\Sigma_{\infty}-\ttS(\eta(\tau))$,  an isomorphism 
\[
\eta'^{*}_{\tau, \sharp}(\Lie(\bfA'_{\tilde\ttS(\eta(\tau)),k_0})^{\circ}_{\tilde\tau'})\cong 
\begin{cases}
\Lie(\bfA'_{\tilde\ttS(\tau), k_0})^{\circ, (p^{m(\tau_0)})}_{\sigma^{-m(\tau_0)}\tilde\tau'} &\text{if $\tilde\tau'$ lifts $\eta(\tau_0)$,}\\
\Lie(\bfA'_{\tilde\ttS(\tau), k_0})^{\circ}_{\tilde\tau'} &\text{otherwise},
\end{cases}
\]
 If $\tau\in \{\tau_0, \tau_0^{+}\}$ or equivalently $\tau_0\in \ttS(\tau)$, then the first case above never happens. Therefore, by Proposition~\ref{Prop:smoothness}, we see that $\eta'_{\tau, \sharp}$ induces an isomorphism of tangent spaces between $\bfSh_{K'}(G'_{\tilde\ttS(\tau)})_{k_0}$ and $\bfSh_{K'}(G'_{\tilde\ttS(\eta(\tau))})_{k_0}$. 
 Since $\eta'_{\tau,\sharp}$ is bijective on the closed points by the definition of link morphism, $\eta'_{\tau, \sharp}$ is actually an isomorphism. If $\tau\notin \{\tau_0, \tau_0^+\}$ or equivalently $\tau_0\notin \ttS(\tau)$, we conclude by the  same arguments as in the proof of Theorem~\ref{T:link and Hecke operator}(1)(d). 
 \end{proof}



\end{document}